\documentclass[12pt,oneside]{book}
\usepackage[utf8]{inputenc}
\usepackage{amsmath,amssymb,amsthm}
\usepackage{graphicx}
\usepackage{hyperref}
\usepackage[shortlabels]{enumitem}
\usepackage{algorithm}
\usepackage[noend]{algpseudocode}  
\usepackage{booktabs}
\usepackage{longtable}
\usepackage{multirow}
\usepackage{geometry}
\usepackage{titlesec}
\usepackage{tikz}  
\usepackage{tkz-graph} 
\usepackage{tkz-berge} 
\usepackage{titlesec}
\usetikzlibrary{3d}         
\usetikzlibrary{positioning} 
\usetikzlibrary{arrows.meta} 
\usepackage{amsmath}        
\usepackage{amssymb}        

\newtheorem{theorem}{Theorem}[chapter]
\newtheorem{lemma}[theorem]{Lemma}
\newtheorem{proposition}[theorem]{Proposition}
\newtheorem{corollary}[theorem]{Corollary}
\newtheorem{definition}[theorem]{Definition}
\newtheorem{example}[theorem]{Example}
\newtheorem{remark}[theorem]{Remark}

\newtheorem{conjecture}[theorem]{Conjecture}

\newcommand{\Z}{\mathbb{Z}}

\newcommand{\F}{\mathbb{F}}
\newcommand{\Cay}{\mathrm{Cay}}
\providecommand{\Lift}{\mathcal{L}}
\providecommand{\Restr}{\mathcal{R}}
\providecommand{\GFT}{\widehat{\mathcal{G}}}
\providecommand{\Filt}{\mathrm{Filt}}
\providecommand{\supp}{\mathrm{supp}}
\providecommand{\Id}{\mathrm{Id}}
\newcommand{\partintro}[1]{%
	\par\addvspace{1.5\baselineskip}%
	\noindent\emph{#1}\par\addvspace{0.5\baselineskip}%
}

\DeclareMathOperator{\diam}{diam}

\title{Minimal Isometric Embeddings of Graphs into Abelian Groups: \\ Theory, Algorithms, and Applications to Signal Processing over Networks}
\author{Rigobert Fokam}
\date{\today}

\begin{document}
	\usetikzlibrary{positioning}
	
	\frontmatter
	
	\begin{titlepage}
		\begin{center}
			
			\begin{tabular}{@{}p{0.45\textwidth}@{}p{0.45\textwidth}@{}}
				\raggedright
				\includegraphics[height=3.5cm]{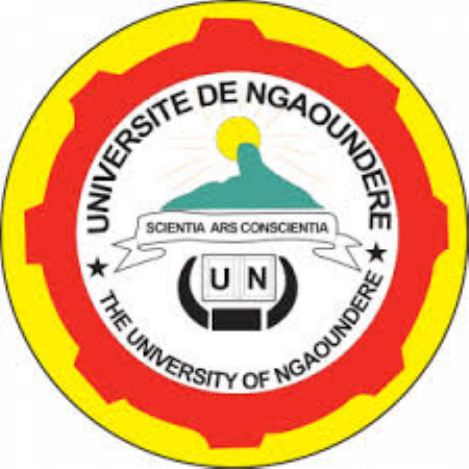} & 
				\raggedleft
				\includegraphics[height=3.5cm]{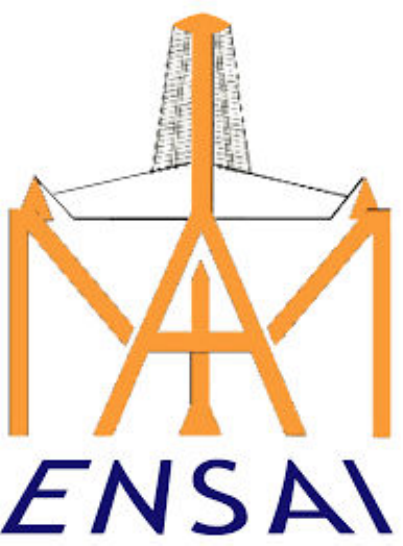} \\
			\end{tabular}
			
			\vspace{0.2cm} 
			
			{\LARGE \textbf{THE UNIVERSITY OF NGAOUNDERE}} \\
			\vspace{0.1cm} 
			{\large \textbf{NATIONAL SCHOOL OF AGRO-INDUSTRIAL SCIENCES}} \\
			
			\vspace{1cm} 
			
			{\Huge \textbf{Ph.D. Dissertation}} \\
			
			\vspace{1cm} 
			
			\begin{minipage}{0.95\textwidth}
				\centering
				{\Huge \textbf{Minimal Isometric Embeddings of Graphs into Abelian Groups: Theory, Algorithms, and Applications to Signal Processing over Networks}}
			\end{minipage}
			
			\vspace{1.5cm} 
			
			\begin{minipage}{0.95\textwidth}
				\centering
				\large
				\textbf{Author:} \\
				{\LARGE \textbf{RIGOBERT FOKAM SOUOP}} \\
				\textbf{Reg. N° 08M057EN}\\
				\vspace{0.3cm} 
				\textbf{Supervisor:} \\
				{\Large \textbf{Prof. LAURENT BITJOKA}}
			\end{minipage}
			
			\vspace{1cm} 
			
			\begin{minipage}{0.8\textwidth}
				\centering
				\normalsize 
				A Dissertation submitted to the \\
				\vspace{0.2cm} 
				\textbf{DEPARTMENT OF ELECTRICAL ENGINEERING, AUTOMATION AND ENERGY} \\
				\vspace{0.1cm} 
				Doctoral Unit of Applied Physics and Engineering \\
				\vspace{0.5cm} 
				In Partial Fulfillment of the Requirements for the Degree \\
				\vspace{0.2cm} 
				\textbf{Doctor of Philosophy} \\
				\textbf{(Electrical Engineering and Computer Science)} \\
				\vspace{0.5cm} 
				\textbf{June 2026}
			\end{minipage}
			
		\end{center}
	\end{titlepage}
	
	\cleardoublepage
	\thispagestyle{empty}
	
	\begin{center}
		{\Large \textbf{APPROVAL PAGE}}
	\end{center}
	
	\vspace{1cm}
	
	This dissertation titled \textbf{"Minimal Isometric Embeddings of Graphs into Abelian Groups: Theory, Algorithms, and Applications to Signal Processing over Networks"} by \textbf{Rigobert Fokam Souop} is approved by the examining committee:
	
	\vspace{2cm}
	
	\begin{flushright}
		\begin{minipage}{0.6\textwidth}
			\centering
			\dotfill \\
			\textbf{Prof. Laurent Bitjoka} \\
			Supervisor \\
			The University of Ngaoundere
		\end{minipage}
	\end{flushright}
	
	\vspace{1.5cm}
	
	\begin{flushright}
		\begin{minipage}{0.6\textwidth}
			\centering
			\dotfill \\
			\textbf{[Name of Committee Member]} \\
			Title \\
			Institution
		\end{minipage}
	\end{flushright}
	
	\vspace{1.5cm}
	
	\begin{flushright}
		\begin{minipage}{0.6\textwidth}
			\centering
			\dotfill \\
			\textbf{[Name of Committee Member]} \\
			Title \\
			Institution
		\end{minipage}
	\end{flushright}
	
	\vspace{1.5cm}
	
	\begin{flushright}
		\begin{minipage}{0.6\textwidth}
			\centering
			\dotfill \\
			\textbf{[Name of Committee Member]} \\
			Title \\
			Institution
		\end{minipage}
	\end{flushright}
	
	\vspace{1.5cm}

	\vspace{2cm}
	
	\begin{center}
		Date of Defense: \dotfill
	\end{center}
	
	\cleardoublepage
	\thispagestyle{empty}
	
	\begin{center}
		\vspace*{5cm}
		{\Large \textit{To my family, teachers, and all those who believe in the power of knowledge.}}
	\end{center}
	
	\vfill
	
	\begin{center}
		{\large \textit{"Mathematics is the language in which God has written the universe."}} \\
		{\large \textit{-- Galileo Galilei}}
	\end{center}
	
	\cleardoublepage
	\chapter*{Acknowledgments}
	\thispagestyle{empty}
	
	\addcontentsline{toc}{chapter}{Acknowledgments}
	
	I would like to express my deepest gratitude to:
	
	\begin{itemize}
		\item \textbf{Prof. Laurent Bitjoka}, my supervisor, for his invaluable guidance, patience, and unwavering support throughout this research journey. His expertise and encouragement were fundamental to the completion of this work.
		
		\item The members of my doctoral committee for their insightful feedback and constructive criticism that greatly improved this dissertation.
		
		\item The faculty and staff of the Department of Electrical Engineering, Automation, and Energy at the University of Ngaoundere for providing an excellent research environment.
		
		\item My colleagues at the Laboratory of Energy, Signal, Imaging and Automation (LESIA) for stimulating discussions and collaborative spirit.
		
		\item Solutum Engineering, to whom I am grateful for stepping in with timely support at a critical juncture of this work.
		
		\item My family for their unconditional love, support, and patience during the long hours of research and writing.
		
		\item All those who, directly or indirectly, contributed to the realization of this work.
	\end{itemize}
	
	\vspace{1cm}
	
	\begin{flushright}
		\textit{Rigobert Fokam Souop} \\
		Ngaoundere, 2026
	\end{flushright}

	\chapter*{Abstract}
	\addcontentsline{toc}{chapter}{Abstract}
	
	This doctoral dissertation develops a framework for embedding arbitrary connected graphs isometrically into Cayley graphs of Abelian groups, with applications to harmonic analysis on networks. The work addresses the challenge of representing irregular graph-structured data within highly symmetric algebraic hosts, on which classical Fourier theory applies verbatim rather than by analogy. The theoretical core is twofold. First, we introduce binary relations $\varphi$, $\Phi$, and $\Psi$ on graph edges that detect \emph{metric parallelism} --- a strict generalization of the Djokovi\'c--Winkler relation beyond bipartite and partial-cube structures --- together with a \emph{transitive prune} operation that converts these relations into candidate same-generator edge partitions. Second, we prove the \emph{Cocycle/Quotient Labeling Theorem}: any edge partition induces a most-generic consistent vertex labeling as a GF(2) quotient of dimension $k = t - \mathrm{rank}(A)$, where $A$ is the cycle--class parity matrix; the labeling is conflict-free by construction and can fail only by shortcuts, never by stretching. Combined with a shortcut-repair loop whose terminal case is the provably isometric spanning-tree embedding, this yields a universal algorithm: every connected graph $G$ embeds isometrically into a Cayley graph of $\mathbb{Z}_2^{k}$ with $k \leq n-1$, and the algorithm was verified exhaustively --- with independent reconstruction of every host --- on all 995 connected graphs with at most seven vertices. We complement the construction with a bounds theory: $k \geq \max(\mathrm{diam}(G), \lceil \log_2 n \rceil)$ in general; the window $[\max(\mathrm{diam}, \lceil\log_2 n\rceil),\, n-1]$ is attained at both ends; stars satisfy $k_{\min}(K_{1,q}) = \lceil\log_2 q\rceil + 1$, exponentially below the naive dimension; odd cycles require $k = n-1$, so the universal upper bound is tight. Against exhaustive exact search, the algorithm attains the true minimum on 29 of the 30 connected graphs with at most five vertices, with the unique exception completely characterized. Chapter 3 generalizes the quotient machinery from GF(2) to $\mathbb{Z}$ via the Smith Normal Form of the cycle--class matrix, producing embeddings into products of cyclic groups (torus skeletons). The primary application domain is harmonic analysis on networks: embedding graphs into Cayley graphs of Abelian groups provides a rigorous foundation for Fourier analysis, convolution, and wavelet transforms on graph signals, preserving translation--modulation duality, convolution theorems, and Plancherel identities that matrix-based graph signal processing lacks. We name this framework \emph{Group Embedding-based Graph Signal Processing} (GE-GSP), built on the group-embedding graph Fourier transform (GE-GFT) and graph wavelet transform (GE-GWT). Experiments cover structured graph families, the standard graph-signal-processing benchmark domains (rings, paths, grids, circulants, sensor graphs, the Zachary Karate Club), and real-world networks, with every embedding certified isometric. Applications extend to network design, error-correcting codes, parallel computing, and quantum information processing.
	
	\textbf{Keywords:} Graph isometric embedding, Cayley graphs, $\varphi$ relation, $\Phi$ relation, $\Psi$ relation, cocycle condition, quotient labeling, transitive prune, torus skeleton, embedding dimension bounds, harmonic analysis, graph signal processing, Fourier transform, wavelet transform, GE-GSP, GE-GFT, GE-GWT.

	\chapter*{Résumé}
	\addcontentsline{toc}{chapter}{Résumé}
	
	Cette thèse de doctorat développe un cadre pour le plongement isométrique de graphes connexes arbitraires dans les graphes de Cayley de groupes abéliens, avec pour application principale l'analyse harmonique des signaux sur graphes. Le travail aborde le défi de représenter des données à support irrégulier dans des hôtes algébriques hautement symétriques, sur lesquels la théorie de Fourier classique s'applique littéralement et non par simple analogie. Le cœur théorique est double. D'une part, nous introduisons les relations binaires $\varphi$, $\Phi$ et $\Psi$ sur les arêtes, qui détectent le \emph{parallélisme métrique} --- une généralisation stricte de la relation de Djokovi\'c--Winkler au-delà des graphes bipartis et des cubes partiels --- ainsi que l'opération de \emph{séparation transitive} qui convertit ces relations en partitions candidates des arêtes par générateur commun. D'autre part, nous démontrons le \emph{théorème d'étiquetage quotient (condition de cocycle)} : toute partition des arêtes induit un étiquetage cohérent le plus générique comme quotient sur GF(2), de dimension $k = t - \mathrm{rang}(A)$ où $A$ est la matrice de parité cycles--classes ; cet étiquetage est sans conflit par construction et ne peut échouer que par raccourci, jamais par étirement. Combiné à une boucle de réparation des raccourcis dont le cas terminal est le plongement d'arbre couvrant, prouvé isométrique, ceci fournit un algorithme universel : tout graphe connexe $G$ se plonge isométriquement dans un graphe de Cayley de $\mathbb{Z}_2^{k}$ avec $k \leq n-1$, vérifié exhaustivement --- avec reconstruction indépendante de chaque hôte --- sur les 995 graphes connexes d'au plus sept sommets. Nous établissons en outre une théorie des bornes : $k \geq \max(\mathrm{diam}(G), \lceil \log_2 n \rceil)$ en général ; la fenêtre $[\max(\mathrm{diam}, \lceil\log_2 n\rceil),\, n-1]$ est atteinte à ses deux extrémités ; les graphes étoiles vérifient $k_{\min}(K_{1,q}) = \lceil\log_2 q\rceil + 1$, exponentiellement en dessous de la dimension naïve ; les cycles impairs exigent $k = n-1$, la borne supérieure universelle est donc optimale. Face à une recherche exacte exhaustive, l'algorithme atteint le minimum vrai sur 29 des 30 graphes connexes d'au plus cinq sommets, l'unique exception étant complètement caractérisée. Le chapitre 3 généralise la notion de quotient de GF(2) à $\mathbb{Z}$ via la forme normale de Smith de la matrice cycles--classes, produisant des plongements dans des produits de groupes cycliques (squelettes de tore). Le domaine d'application principal est l'analyse harmonique sur les réseaux : le plongement dans les graphes de Cayley de groupes abéliens fonde rigoureusement l'analyse de Fourier, la convolution et les transformées en ondelettes des signaux sur graphes, en préservant la dualité translation--modulation, les théorèmes de convolution et les identités de Plancherel qui font défaut aux méthodes matricielles. Nous désignons ce cadre par \emph{Group Embedding-based Graph Signal Processing} (GE-GSP), fondé sur la transformée de Fourier (GE-GFT) et la transformée en ondelettes (GE-GWT) par plongement de groupe. Les expériences couvrent les familles structurées, les domaines de référence du traitement du signal sur graphes (graphs cycliques, graphs chemins, graphs grilles, graphes circulants, graphes de capteurs, graph du club de karaté de Zachary) et des réseaux réels, chaque plongement étant certifié isométrique. Les applications s'étendent à la conception de réseaux, aux codes correcteurs, au calcul parallèle et au traitement quantique de l'information.
	
	\textbf{Mots-clés:} Plongement isométrique de graphes, graphes de Cayley, relation $\varphi$, relation $\Phi$, relation $\Psi$, condition de cocycle, étiquetage quotient, séparation transitive, squelette de tore, bornes de la dimension du plongement, analyse harmonique, traitement du signal sur graphes, transformée de Fourier, transformée en ondelettes, GE-GSP, GE-GFT, GE-GWT.
	
	\tableofcontents
	\listoffigures
	\listoftables
	\listofalgorithms
	
	\chapter*{List of Symbols and Notation}
	\addcontentsline{toc}{chapter}{List of Symbols and Notation}
	\markboth{LIST OF SYMBOLS}{LIST OF SYMBOLS}
	
	\noindent\textbf{Graphs and metrics}
	\begin{longtable}{p{3.1cm}p{11.5cm}}
		$G=(V,E)$ & finite connected undirected graph, $n=|V|$ vertices, $m=|E|$ edges \\
		$d_G(u,v)$ & shortest-path (geodesic) distance between vertices $u,v$ in $G$ \\
		$\diam(G)$ & diameter of $G$, $\max_{u,v} d_G(u,v)$ \\
		$N_i,\ N(i,k)$ & neighborhood of vertex $i$; vertices within $k$ hops of $i$ \\
		$G\square H$ & Cartesian product of graphs $G$ and $H$ \\
		$Q_p$ & $p$-dimensional hypercube, $Q_p = K_2^{\square p}$ \\
		$C_m,\ P_k,\ K_n$ & cycle on $m$ vertices, path on $k$ vertices, complete graph on $n$ \\
		$K_{r,s}$ & complete bipartite graph \\ $\mathrm{CL}_n$ & circular ladder graph \\
		$C_n(d_1,\dots)$ & circulant graph on $\Z_n$ with connection set $\{\pm d_i\}$ \\
		$\mathrm{idim}(G)$ & isometric (hypercube) dimension of a partial cube \\
	\end{longtable}
	
	\noindent\textbf{Groups, characters, embeddings}
	\begin{longtable}{p{3.1cm}p{11.5cm}}
		$\Gamma$ & finite abelian group, the embedding host; $N=|\Gamma|$ \\
		$\Z_2^k$ & elementary abelian $2$-group of rank $k$ (binary host) \\
		$\Z_{N_1}\!\times\!\cdots\!\times\!\Z_{N_d}$ & general abelian host (product of cyclic groups) \\
		$\Cay(\Gamma,S)$ & Cayley graph of $\Gamma$ with symmetric generating set $S$ \\
		$\phi,\ \lambda$ & isometric embedding $V(G)\to\Gamma$; vertex labeling \\
		$\varepsilon$ & excursion ratio $|V(G)|/|\Gamma|$; $\varepsilon=1$ iff onto \\
		$\widehat{\Gamma}$ & dual group (characters) of $\Gamma$, $\widehat\Gamma\cong\Gamma$ \\
		$\chi_k(g)$ & character $\prod_j \omega_{N_j}^{k_j g_j}$, $\omega_{N}=e^{2\pi i/N}$ \\
		$\nu(G)$ & minimum host order over isometric abelian-Cayley embeddings \\
		$k_{\min}(G)$ & minimum binary dimension over isometric $\Z_2^k$ embeddings \\
	\end{longtable}
	
	\noindent\textbf{Relations, partitions, and the quotient core}
	\begin{longtable}{p{3.1cm}p{11.5cm}}
		$\varphi,\ \Phi,\ \Psi$ & metric-parallelism relations on edges (Chapters 2--3) \\
		$\theta$ & Djokovi\'c--Winkler relation \\
		$\mathcal{P}=\{F_1,\dots,F_t\}$ & partition of $E(G)$ into $t$ generator classes \\
		$A$ & cycle--class matrix: parity over $\F_2$ (Ch.\,2), signed over $\Z$ (Ch.\,3) \\
		$\rho=\mathrm{rank}(A)$ & rank of the cycle--class matrix \\
		$k=t-\rho$ & binary quotient dimension (Quotient Labeling Theorem) \\
		$\mathrm{SNF}$ & Smith Normal Form; diagonal invariants $d_1,\dots,d_\rho$ \\
		$\Gamma_{\mathrm{univ}}$ & universal quotient group $\Z^t/\mathrm{rowlattice}(A)$ \\
	\end{longtable}
	\newpage
	\noindent\textbf{Harmonic analysis on graphs (Part II)}
	\begin{longtable}{p{3.1cm}p{11.5cm}}
		$\Lift,\ \Restr$ & lift $\mathbb{C}^{V(G)}\to\mathbb{C}^\Gamma$ (zero-extend); restriction \\
		$\GFT$ & group graph Fourier transform, $\GFT s = \widehat{\Lift s}$ \\
		$\hat f,\ \hat{\tilde s}$ & Fourier transform on $\Gamma$ \\
		$\tau_h,\ T_h$ & host translation by $h$; graph translation $\Restr\,\tau_h\,\Lift$ \\
		$\mu_k$ & modulation by character $\chi_k$ \\
		$\ast$ & group convolution \\
		$\Filt_a$ & graph filtering operator with kernel $a$ \\
		$\psi_{j,h}$ & group wavelet at scale $j$, center $h$; tight-frame atom \\
		$L=D-A$ & graph Laplacian (spectral GSP comparison); $\lambda_\ell,u_\ell$ its eigenpairs \\
		$\supp(\cdot)$ & support of a signal \\
	\end{longtable}
	
	\noindent\textbf{Standard symbols.} $\F_2$ the two-element field; $\Z$ the
	integers; $\mathbb{C}$ the complex numbers; $\Id$ the identity; $\delta_v$
	the indicator (delta) at vertex $v$; $[\,\cdot\,]$ the Iverson bracket;
	$\lceil\cdot\rceil$ the ceiling; $|\cdot|$ cardinality or word length as
	indicated by context.
	
	\mainmatter
	\chapter*{General Introduction}
	\addcontentsline{toc}{chapter}{General Introduction}
	\label{chap:introduction}
	
	\section{Research Context and Motivation}
	The advent of the digital era has precipitated an unprecedented explosion in the volume and variety of data generated globally \cite{Hey2009}. Unlike traditional data that assumes independence or simple Euclidean geometry, modern complex systems—ranging from social interaction networks and biological protein structures to transportation grids and power systems—produce data that is fundamentally relational and irregular \cite{Barabasi2016,Watts2002,Newman2018}.Such networks span social, biological, and infrastructural domains. The structural properties of these systems, such as small-world and community patterns, were first modeled in seminal works \cite{Watts1998,Hagmann2008b} and have since been explored across social and brain networks \cite{Easley2010,Bullmore2009,Huang2018}. Similarly, computational techniques have advanced from early graph algorithms \cite{Lovasz2012,Klein1991} to modern web-scale frameworks capable of handling streaming and dynamic relational data \cite{Leskovec2014,Muthukrishnan2005,DasSarma2010,Bhattacharya2015}. Graphs provide the most natural mathematical abstraction for representing such data, where nodes denote entities and edges encode relationships or interactions \cite{Bondy1976,West2001,Diestel2017}.
	
	However, the irregularity of real-world graphs presents a challenging obstacle to mathematical analysis \cite{Shuman2013,Ortega2018}. Unlike Euclidean spaces or regular grids, graphs lack a canonical coordinate system, a uniform notion of translation, and a natural basis for harmonic analysis \cite{Hammond2011,Sandryhaila2013}. This irregularity complicates the direct application of classical signal processing tools, such as Fourier transforms, wavelets, and convolution operators, which rely heavily on the underlying symmetries of the data domain \cite{Mallat2009,Katznelson2004}.
	
	This dissertation addresses a fundamental question at the intersection of discrete mathematics and signal processing: \textbf{How can we embed arbitrary graphs into highly symmetric algebraic structures to enable rigorous mathematical analysis of network-structured data?} Our central hypothesis is that by isometrically embedding graphs into Cayley graphs of algebraic groups—specifically Abelian groups—we can inherit the rich analytical machinery from group theory and abstract harmonic analysis for application to irregular graph-structured data \cite{Godsil2001,Rudin1962,Folland1995}.
	
	\begin{figure}[htbp]
		\centering
		\begin{tikzpicture}[
			node distance=2cm,
			box/.style={rectangle, draw=black, fill=blue!10, very thick, minimum size=1cm, align=center},
			arrow/.style={->, >=stealth, very thick}
			]
			\node[box] (irregular) {Irregular Graph\\$G=(V,E)$};
			\node[box, right=3cm of irregular, fill=red!10] (cayley) {Cayley Graph\\$\text{Cay}(\Gamma, S)$};
			\node[below=0.5cm of irregular, text width=4cm, align=center, font=\footnotesize] {Chaotic structure\\No harmonic basis};
			\node[below=0.5cm of cayley, text width=4cm, align=center, font=\footnotesize] {Group symmetry\\Fourier basis available};
			
			\draw[arrow] (irregular) -- node[above, align=center] {Isometric \\ Embedding} node[below] {$\phi: G \hookrightarrow \Gamma$} (cayley);
		\end{tikzpicture}
		\caption{Illustration of the embedding concept: An irregular graph $G$ (left) is isometrically embedded into a symmetric Cayley graph $\text{Cayley}(\Gamma, S)$ (right), enabling harmonic analysis.}
		\label{fig:concept}
	\end{figure}
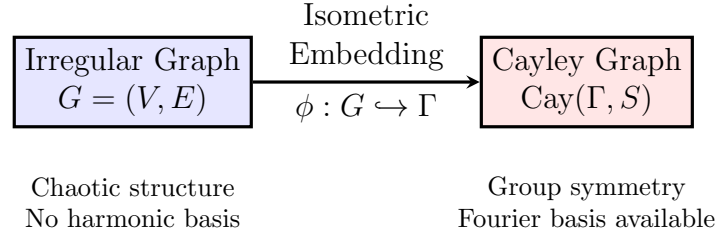
	
	Our core strategy is \emph{representation by symmetry}. We seek to embed a given graph into a highly symmetric host structure that inherently supports harmonic analysis. The ideal hosts for this purpose are the \textbf{Cayley graphs of Abelian groups} \cite{Godsil2001,Biggs1993,Lubotzky1994}. These graphs generalize regular grids and tori and possess a well-defined algebraic Fourier theory based on group characters \cite{Terras1999,Sternberg2004}. 
	
	However, a graph can often be embedded into many such groups. A naive embedding into a very large group, while theoretically possible, would be computationally inefficient and would dilute the analogy to classical signal processing \cite{Graham1985}. Therefore, the critical challenge becomes finding not just \emph{any} embedding, but the \textbf{minimal isometric embedding}—the one into the smallest possible group (or the one with the smallest dimension) that preserves graph distances exactly \cite{Chepoi2000,Imrich2010}. This pursuit of minimality is the golden thread unifying the theoretical, algorithmic, and applied contributions of this thesis.
	
	The motivation for this research is twofold. From a theoretical perspective, we seek to establish deeper connections between graph theory and group theory through the study of isometric embeddings into Cayley graphs, extending classical results by Sabidussi \cite{Sabidussi1958} and others \cite{Babai1996}. From an applied perspective, we aim to develop practical tools for harmonic analysis on networks, addressing limitations in existing graph signal processing (GSP) methods \cite{Shuman2013,Stankovic2019} while maintaining fidelity to the classical signal processing theory developed by Brillouin, Wiener, and others \cite{Brillouin1953,Wiener1949}; spectral thresholding for denoising follows the classical soft-thresholding principle \cite{Donoho1995}.
	
	\section{Related Work and Positioning}
	
	\subsection{Graph Theory and Metric Geometry}
	The mathematical study of graph embeddings has a rich history intersecting with metric geometry and group theory. Graphs can be viewed as metric spaces where the distance is defined by the length of the shortest path \cite{Bollobas1998,Chartrand2006}. 
	
	\subsubsection{Isometric Embeddings and Product Graphs}
	A fundamental problem is determining when a graph $G$ can be embedded isometrically (distance-preservingly) into a product graph, such as the hypercube $\mathbb{Z}_2^n$ or the torus $\mathbb{Z}_k^n$. Graham and Pollak \cite{Graham1985,Graham1978} initiated the study of embeddings into cubes, leading to the concept of \textit{partial cubes}. A graph is a partial cube if it is isometric to a subgraph of a hypercube \cite{Winkler1984}. These graphs are characterized by the Djoković-Winkler relation $\Theta$, which identifies pairs of edges that lie on a common shortest path of a specific type \cite{Djokovic1973}. This relation partitions the edge set into equivalence classes, each corresponding to a dimension in the hypercube \cite{Imrich2010}.
	
	However, restricting embeddings to hypercubes ($\mathbb{Z}_2^n$) limits the application to binary structures. Generalizing to $\mathbb{Z}_k^n$ (toroidal grids) is more challenging but offers richer structural representations \cite{Chepoi2000,Ekim2011}. Recent works in metric graph theory have expanded these concepts to \textit{partial torus graphs}, characterizing graphs that embed into products of cycles \cite{Chepoi2000,Paul2016}. Despite these advances, the problem of finding the minimal dimension for the smallest Abelian group containing an arbitrary graph remains largely open and computationally difficult \cite{Chiba1985,Hell2004}.
	
	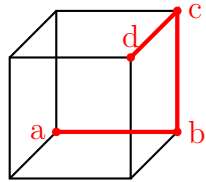
\begin{figure}[htbp]
		\centering
		\begin{tikzpicture}[scale=0.8]
			\coordinate (A) at (0,0,0);
			\coordinate (B) at (2,0,0);
			\coordinate (C) at (2,2,0);
			\coordinate (D) at (0,2,0);
			\coordinate (E) at (0,0,2);
			\coordinate (F) at (2,0,2);
			\coordinate (G) at (2,2,2);
			\coordinate (H) at (0,2,2);
			
			\draw[thick] (A)--(B)--(C)--(D)--cycle;
			\draw[thick] (E)--(F)--(G)--(H)--cycle;
			\draw[thick] (A)--(E) (B)--(F) (C)--(G) (D)--(H);
			
			\node at (0,1.5,-3.5) {\small Hypercube ($Q_3$)};
			
			\coordinate (S1) at (0,0,0);
			\coordinate (S2) at (2,0,0);
			\coordinate (S3) at (2,2,0);
			\coordinate (S4) at (2,2,2);
			
			\draw[ultra thick, red] (S1)--(S2)--(S3)--(S4);
			\fill[red] (S1) circle (2pt) node[left] {a};
			\fill[red] (S2) circle (2pt) node[right] {b};
			\fill[red] (S3) circle (2pt) node[right] {c};
			\fill[red] (S4) circle (2pt) node[above] {d};
			
		\end{tikzpicture}
		\caption{Visualization of a partial cube embedding (red) into a host hypercube $Q_3$.}
		\label{fig:partial_cube}
	\end{figure}
	
	\subsubsection{Cayley Graphs and Group Theory}
	Cayley graphs link group theory to graph theory \cite{Cayley1878}. Given a group $\Gamma$ and a generating set $S$ (inverse-closed, excluding the identity), the Cayley graph $\text{Cay}(\Gamma, S)$ has vertices as group elements and edges connecting $g$ to $gs$ for all $s \in S$ \cite{Godsil2001,Biggs1993}. These graphs are always vertex-transitive, meaning the graph looks identical from the perspective of any vertex \cite{Sabidussi1958}. 
	
	Sabidussi's theorem \cite{Sabidussi1958} states that a graph is a Cayley graph if and only if its automorphism group contains a regular subgroup \cite{Seress2003}. This implies that studying embeddings into Cayley graphs is equivalent to studying group representations that respect the graph metric. While extensive literature exists on recognizing Cayley graphs \cite{Li2002,Godsil1983}, the inverse problem—finding a "small" Cayley graph that contains a given arbitrary graph—is less explored.
	
	\subsection{Harmonic Analysis on Graphs}
	The application of signal processing techniques to graph data has emerged as a vibrant field known as Graph Signal Processing (GSP) \cite{Shuman2013,Ortega2018,Stankovic2019}.
	
	\subsubsection{Spectral Graph Theory}
	The dominant paradigm in GSP is spectral graph theory. Chung \cite{Chung1997} and others established the foundations using the graph Laplacian $\mathcal{L}$. The Graph Fourier Transform (GFT) is defined via the eigenvectors of $\mathcal{L}$ \cite{Shuman2013}. This approach draws analogies to the classical Fourier transform, where eigenvectors represent "frequency" components varying slowly over the graph \cite{Zhang2018}. 
	
	However, spectral methods suffer from significant limitations:
	\begin{itemize}
		\item \textbf{Localization:} Eigenvectors of irregular graphs are typically not localized in space, making interpretation difficult \cite{Hammond2011,Coifman2006}.
		\item \textbf{Shift Invariance:} Unlike Euclidean domains, graphs lack a shift operator, leading to the lack of a generalized convolution theorem \cite{Sandryhaila2013}.
		\item \textbf{Instability:} Small perturbations in graph structure can lead to large changes in the eigenvectors (eigengap problem) \cite{VonLuxburg2007}.
	\end{itemize}
	
	\subsubsection{Alternative Approaches}
	To address these limitations, researchers have explored wavelet transforms on graphs. Hammond et al. \cite{Hammond2011} and Shuman et al. \cite{Shuman2011} designed spectral wavelets using the heat kernel and polynomial approximations of the Laplacian. Another approach is the "Aggregation" method, which builds multiscale representations by graph coarsening \cite{Gavish2010}.
	
	Algebraic Signal Processing (ASP), proposed by Puschel and Moura \cite{Puschel2003, Puschel2008}, provides a rigorous algebraic framework by defining signal models based on algebras. This approach perfectly captures classical signal processing but requires the data domain to support the structure of a module over an algebra \cite{Sandryhaila2014}. This brings us back to the group-theoretic view: if the graph is a Cayley graph, the signal space is naturally a module over the group algebra, allowing for a perfect translation of classical harmonic analysis \cite{Moura2015}.
	
	\begin{figure}[htbp]
		\centering
		\begin{tikzpicture}[
			scale=0.9, transform shape,
			cell/.style={rectangle, draw=black, minimum size=1cm},
			operator/.style={circle, draw=red!80, fill=red!20, very thick}
			]
			\node[cell] (sig1) at (0, 4) {$x[n]$};
			\node[operator] (fft1) at (2.5, 4) {DFT};
			\node[cell] (freq1) at (5, 4) {$X[k]$};
			\node[below=0.2cm of sig1] {\small Time};
			\node[below=0.2cm of freq1] {\small Freq};
			\draw[->, thick] (sig1) -- (fft1);
			\draw[->, thick] (fft1) -- (freq1);
			\node[above=0.5cm of fft1] {\textbf{Classical SP} ($\mathbb{Z}$)};
			
			\node[cell, fill=blue!10] (sig2) at (0, 0) {$f(v)$};
			\node[operator, fill=blue!10, dashed] (lap) at (2.5, 0) {$\mathcal{L}$?};
			\node[cell, fill=blue!10] (freq2) at (5, 0) {$\hat{f}$?};
			\node[below=0.2cm of sig2] {\small Vertex Domain};
			\node[below=0.2cm of freq2] {\small Spectral Domain};
			\draw[->, thick] (sig2) -- (lap);
			\draw[->, thick] (lap) -- (freq2);
			\node[above=0.5cm of lap] {\textbf{Graph SP} (Irregular)};
			
			\draw[<->, dashed, thick, blue] (0.5, -0.4) -- node[below] {Embedding} (4.5, -0.4);
		\end{tikzpicture}
		\caption{The gap between Classical Signal Processing (rigid algebra) and Graph Signal Processing (irregular). This thesis proposes an embedding strategy to align the bottom row with the top row's structure.}
		\label{fig:gsp_comparison}
	\end{figure}
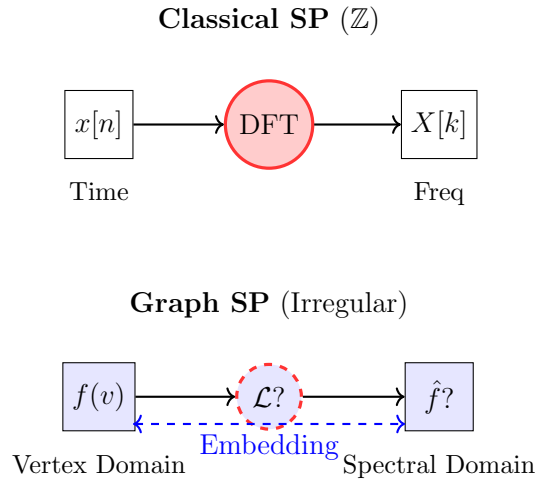
	
	\subsection{Gap and Positioning}
	There is a distinct disconnect between the rich theory of graph embeddings and the practical needs of signal processing. Existing embedding techniques often focus on dimensionality reduction (e.g., MDS \cite{Borg2005}, t-SNE \cite{Maaten2008}) which distorts distances, or theoretical characterizations (partial cubes) that do not generalize to arbitrary Abelian groups \cite{Chepoi2000}. Concurrently, GSP methods often abandon group-theoretic rigor to handle irregularity, resulting in "approximate" convolutions and translations \cite{Sandryhaila2013}.
	
	This dissertation bridges this gap by developing \textbf{minimal isometric embeddings} into Abelian groups \cite{Biggs1993}. By mapping a graph $G$ to a group $\Gamma$, we effectively discretize the graph onto a regular structure. This allows us to treat graph signals as functions on a group, $\Gamma \to \mathbb{C}$, thereby enabling the direct application of harmonic analysis \cite{Katznelson2004}.
	
	\section{Research Objectives and methodological approach}
	
	\subsection{Research Objectives}
	The primary objectives of this research are:
	
	\begin{enumerate}
		\item To develop a comprehensive theory of isometric graph embeddings into Cayley graphs of algebraic structures, with particular focus on finitely generated Abelian groups \cite{Babai1996,Alspach1997}.
		
		\item To introduce novel mathematical relations ($\varphi$, $\Phi$, $\Psi$) that characterize when and how graphs can be embedded into Cayley graphs, effectively generalizing the concept of partial cubes to partial tori \cite{Djokovic1973,Winkler1984,Chepoi2000}.
		
		\item To design efficient algorithms for computing compact embeddings, with proven optimality on identified graph classes (partial cubes, even cycles, grids, the Petersen graph, complete graphs of order $2^t$), a measured optimality profile on exhaustively searchable instances, and certified isometry on every output \cite{Cormen2009}.
		
		\item To apply the embedding framework to harmonic analysis on graphs, establishing rigorous foundations for Fourier analysis, convolution, and wavelet transforms on network-structured data without sacrificing the mathematical properties enjoyed by Euclidean signal processing \cite{Katznelson2004,Mallat2009}.
		
		\item To validate the theoretical framework through algorithmic implementations and applications to real-world networks, demonstrating the utility of the approach in denoising, classification, and compression tasks \cite{Leskovec2007,Newman2003}.
	\end{enumerate}
	
	\subsection{Methodological Approach}
	This research employs an integrated methodology combining \cite{Cormen2009,Folland1995}:
	\begin{itemize}
		\item \textbf{Theoretical Mathematics:} Graph theory \cite{Bondy1976}, group theory \cite{Rotman1995}, representation theory \cite{Fulton1991}, metric geometry \cite{Burago2001}
		
		\item \textbf{Algorithm Design:} Graph algorithms \cite{Cormen2009}, computational complexity \cite{Arora2009}, optimization
		
		\item \textbf{Signal Processing:} Harmonic analysis \cite{Katznelson2004}, Fourier theory \cite{Folland1995}, wavelet theory \cite{Mallat2009}
		
		\item \textbf{Experimental Validation:} Algorithm implementation, performance benchmarking, case studies
	\end{itemize}

	\section{Core contributions}
	
	\subsection{Theoretical Contributions to Graph Embedding Theory}
	\begin{itemize}
		\item \textbf{Novel Binary Relations:} We introduce $\varphi$, $\Phi$, and $\Psi$ relations on graph edges that generalize and extend the classical Djoković-Winkler $\Theta$ relation \cite{Djokovic1973,Winkler1984}. These relations detect structural patterns that determine embeddability into Cayley graphs.
		
		\item \textbf{Torus Skeleton Theorem:} We prove that every connected graph contains a spanning partial torus graph (its \emph{torus skeleton}) that encodes its embedding structure into product groups \cite{Chepoi2000,Imrich2010}.
		
		\item \textbf{Compact Embedding Theory with a Bounds Frontier:} We prove that every connected graph embeds isometrically into a Cayley graph of $\mathbb{Z}_2^{k}$ with $k \leq n-1$, that $k \geq \max(\mathrm{diam}(G), \lceil \log_2 n \rceil)$ always, and that both ends of this window are attained (odd cycles require $k = n-1$; hypercubes, even cycles and complete graphs $K_{2^t}$ attain the lower bound). We further determine the exact minimum for stars, $k_{\min}(K_{1,q}) = \lceil \log_2 q\rceil + 1$. For general graphs the framework computes compact embeddings whose dimension is governed by the rank of the cycle--class parity matrix; exact global minimization is conjectured NP-hard and our algorithm is a heuristic with a fully characterized failure mode \cite{Deza1997}.
		
		\item \textbf{Transitive Prune Operation:} We introduce a novel operation for extracting transitive structure from the $\varphi$ and $\Phi$ relations, enabling the identification of candidate same-generator edge classes.
		
		\item \textbf{Cocycle/Quotient Labeling Theorem and a Universal Algorithm:} We prove that any edge partition induces a most-generic consistent vertex labeling as a GF(2) quotient of dimension $k = t - \mathrm{rank}(A)$, where $A$ is the cycle--class parity matrix; the labeling is conflict-free by construction and can only err by shortcuts, never stretches. Combined with a shortcut-repair loop whose terminal case is the provably isometric naive embedding, this yields the first $\varphi$-based embedding algorithm that succeeds on \emph{every} connected graph, verified exhaustively (with independent reconstruction of every host) on all 995 connected graphs with at most seven vertices.
	\end{itemize}
	
	\subsection{Algorithmic Contributions}
	\begin{itemize}
		\item \textbf{Embedding Algorithms:} We introduce novel algorithms running in $O\bigl(n(n+m) + m^2 + R(2^k m + n^2)\bigr)$ time to compute compact embeddings into binary groups (Chapter 2) and arbitrary Abelian groups (Chapter 3). For partial cubes and structured benchmark families, a single round suffices and the polynomial terms dominate; our asymptotic evaluation strictly adheres to standard computational and complexity models \cite{Cormen2009,Arora2009}.
		
		\item \textbf{Complexity Analysis:} We provide a complete worst-case complexity analysis and derive dedicated optimization strategies to ensure efficient practical implementation.
		
		\item \textbf{Experimental Framework:} We implement the proposed algorithms and validate their performance (Chapter 4) on both synthetic graph families and established real-world networks from the SNAP repository \cite{Leskovec2007}.
	\end{itemize}

	\subsection{Applied Contributions to Harmonic Analysis on Graphs}
	\begin{itemize}
		\item \textbf{Fourier Analysis Framework:} By embedding graphs into Cayley graphs of Abelian groups, we establish a rigorous foundation for Fourier analysis on graphs that preserves fundamental properties of classical Fourier analysis \cite{Katznelson2004,Folland1995} (Chapter 5).
		
		\item \textbf{Wavelet Transform Construction:} We extend the framework to multi-scale analysis through wavelet transforms on graphs, addressing limitations of existing spectral graph theory approaches \cite{Mallat2009,Daubechies1992} (Chapter 6).
		
		\item \textbf{Comparative Analysis:} We demonstrate through theoretical analysis and examples how our group-theoretic approach overcomes limitations of existing graph signal processing methods based on graph matrices \cite{Shuman2013,Hammond2011}.
	\end{itemize}
	
	\subsection{Broader Applications}
	Beyond harmonic analysis, the embedding framework has various applications:
	\begin{itemize}
		\item \cite{Leighton1992} Network design and fault-tolerant architectures
		\item \cite{MacWilliams1977} Error-correcting codes and coding theory
		\item \cite{Dally1990} Parallel computing and load balancing
		\item \cite{Nielsen2010} Quantum information processing
		\item \cite{Hamilton2020} Machine learning and graph representation learning
	\end{itemize}

	\section{Structure of the Manuscript}
	
	This dissertation is organized to guide the reader progressively from the foundational mathematical concepts through theoretical innovations to practical signal processing applications. The manuscript is divided into two main parts reflecting the dual focus on graph embedding theory and its applications to harmonic analysis. The two parts are bookended by a General Introduction and a General Conclusion.
	
	\subsection{General Introduction and Preliminaries}
	The thesis begins with the present General Introduction, which outlines the research context, motivation, and the specific objectives of the work. Subsequently, \textbf{Chapter 1} (Preliminaries) establishes the necessary mathematical framework. This chapter reviews the fundamental concepts in graph theory—including graph products and metric geometry—the algebraic foundations of Cayley graphs and generating sets, and the essential principles of harmonic analysis, specifically Fourier filter banks and wavelets on groups.
	
	\subsection{Part I: Theory of Graph Embeddings}
	The first part of the thesis is dedicated to the development of a comprehensive theory for the isometric embedding of arbitrary graphs into algebraic structures. It consists of the following chapters:
	
	\begin{itemize}
		\item \textbf{Chapter 2: Compact Binary Embeddings} -- This chapter addresses the fundamental case of embedding graphs into binary groups $\mathbb{Z}_2^k$. We introduce the $\varphi$ relation, the transitive prune, and the Cocycle/Quotient Labeling Theorem; present a universal embedding algorithm with a shortcut-repair loop; and establish the bounds theory $\max(\mathrm{diam}\,G, \lceil\log_2 n\rceil) \leq k_{\min}(G) \leq n-1$, with exact values for stars and odd cycles.
		
		\item \textbf{Chapter 3: Minimal Abelian Embeddings (Torus Skeletons)} -- This constitutes the core theoretical contribution. We introduce the advanced $\Phi$ and $\Psi$ binary relations and prove the Torus Skeleton Theorem. This chapter generalizes the embedding framework to arbitrary Abelian groups, enabling a more general representation of graph structures.
		
		\item \textbf{Chapter 4: Applications of part I} – Evaluates embedding algorithms and explores applications beyond harmonic analysis 
	\end{itemize}
	
	\subsection{Part II: Applications to Signal Processing on Graphs}
	The second part of the thesis leverages the embedding theory developed in Part I to address challenges in signal processing on graphs. It demonstrates how the group-theoretic properties of Cayley graphs facilitate advanced harmonic analysis on networks:
	
	\begin{itemize}
		\item \textbf{Chapter 5: Fourier Analysis via Graph Embeddings} -- We develop a rigorous Graph Fourier Transform (GFT) grounded in the algebraic structure of the embedding group. This chapter establishes the convolution theorem and explores duality properties, drawing direct parallels to classical signal processing.
		
		\item \textbf{Chapter 6: Wavelet Analysis via Graph Embeddings} -- Extending the framework to multi-scale representations, this chapter presents the construction of wavelet bases on embedded graphs. We utilize group-theoretic lifting schemes to define scalable transforms suitable for analyzing data at varying resolutions.
		
		\item \textbf{Chapter 7: Applications of part II} -- The practical efficacy of the proposed framework is validated through extensive experimentation. We apply our methods to real-world datasets, including social and biological networks, demonstrating improvements in signal reconstruction and denoising tasks.
	\end{itemize}
	
	\subsection{General Conclusion}
	The thesis concludes with a General Conclusion and Future Work. This final chapter synthesizes the theoretical and practical contributions of the research, discusses their implications for the fields of graph theory and signal processing, and proposes promising avenues for future research, such as extensions to non-Abelian embeddings and the analysis of dynamic graphs.
	
	\section{Broader Significance and Impact}
	This research makes several significant contributions \cite{Babai1996,Shuman2013}:
	
	\begin{itemize}
		\item \textbf{Bridging Mathematical Fields:} Establishes new connections between graph theory \cite{Bondy1976}, group theory \cite{Rotman1995}, and signal processing \cite{Folland1995}.
		
		\item \textbf{Theoretical Foundations:} Provides rigorous mathematical foundations for graph embedding into algebraic structures \cite{Godsil2001}.
		
		\item \textbf{Practical Algorithms:} Develops efficient algorithms with certified isometric outputs, proven optimality on identified graph classes, and clearly characterized heuristic behavior elsewhere \cite{Cormen2009}.
		
		\item \textbf{Applied Framework:} Creates a comprehensive framework for harmonic analysis on networks with applications across science and engineering \cite{Coifman2006}.
		
		\item \textbf{Educational Value:} Presents a cohesive mathematical narrative that can advance education in discrete mathematics and signal processing.
	\end{itemize}
	
	The work presented in this dissertation addresses fundamental questions at the intersection of discrete mathematics and signal processing \cite{Chung1997,Shuman2013}, providing both theoretical insights and practical tools for analyzing the complex network-structured data that characterizes our digital age \cite{Barabasi1999,Newman2003}.
	
	\chapter{Preliminaries}
	\label{chap:preliminaries}
	
	This chapter provides the rigorous mathematical foundation required for the subsequent development of the dissertation. We establish the notation and review fundamental concepts from graph theory, group theory, and abstract harmonic analysis. We pay particular attention to the algebraic structures that facilitate the transition between irregular graph domains and regular harmonic analysis frameworks.
	
	\section{Graph Theory Fundamentals}
	\label{sec:graph-basics}
	
	Throughout this dissertation, we work exclusively with finite, simple, and connected graphs unless explicitly stated otherwise. Let $G = (V, E)$ denote a graph where $V$ is the vertex set and $E \subseteq \binom{V}{2}$ is the edge set \cite{Diestel2017,West2001}.
	
	\subsection{Graph Families and Basic Structures}
	We define several canonical graph families that will serve as building blocks or comparative benchmarks for our embedding theorems.
	
	\begin{itemize}
		\item \textbf{Path and Cycle Graphs:} The path graph $P_n$ is the tree on $n$ vertices with exactly two vertices of degree 1 (leaves) and $n-2$ vertices of degree 2. The cycle graph $C_n$ is obtained by connecting the two leaves of $P_n$, resulting in a 2-regular connected graph \cite{Bondy1976}.
		\item \textbf{Complete and Bipartite Graphs:} The complete graph $K_n$ has an edge between every distinct pair of vertices. A graph is \textit{bipartite} if its vertex set can be partitioned into two sets $V_1$ and $V_2$ such that every edge connects a vertex in $V_1$ to one in $V_2$. The complete bipartite graph $K_{m,n}$ is the bipartite graph with $|V_1|=m$, $|V_2|=n$ and all possible edges between the sets \cite{Harary1969}.
		\item \textbf{Trees and Spanning Trees:} A tree is a connected acyclic graph. A \textit{spanning tree} of a graph $G$ is a subgraph that is a tree and includes all vertices of $G$. The existence of spanning trees is fundamental to connectivity and embedding algorithms \cite{Graham1985}.
		\item \textbf{The Petersen Graph:} A specific, highly symmetric graph on 10 vertices. It serves as a counterexample in many graph theory problems and is the smallest cubic graph of girth 5. It is not a partial cube, illustrating the complexity of embedding arbitrary graphs \cite{Holton1992}.
	\end{itemize}
	
	\begin{figure}[htbp]
		\centering
		\begin{tikzpicture}[scale=1, auto, swap]
			\foreach \x in {0,1,2}
			\foreach \y in {0,1,2} {
				\node[circle, draw, fill=blue!20, inner sep=1.5pt] (n\x\y) at (\x*1.5,\y*1.5) {};
			}
			
			\foreach \y in {0,1,2} {
				\draw (n0\y) -- (n1\y);
				\draw (n1\y) -- (n2\y);
			}
			
			\foreach \x in {0,1,2} {
				\draw (n\x0) -- (n\x1);
				\draw (n\x1) -- (n\x2);
			}
			
			\node[below right] at (0,-0.5) {Grid $P_3 \square P_3$};
			
			
			\begin{scope}[xshift=6cm, yshift=1.5cm]
				\foreach \i in {1,...,5} {
					\node[circle, draw, fill=red!20, inner sep=1.5pt] (u\i) at ({72*(\i-1)+90}:2) {};
				}
				\foreach \i [evaluate=\i as \j using {int(mod(\i,5)+1)}] in {1,...,5}
				\draw[thick] (u\i) -- (u\j);
				
				\foreach \i in {1,...,5} {
					\node[circle, draw, fill=red!20, inner sep=1.5pt] (v\i) at ({72*(\i-1)+90}:1) {};
					\draw (u\i) -- (v\i);
				}
				
				\foreach \i [evaluate=\i as \j using {int(mod(\i+1,5)+1)}] in {1,...,5}
				\draw[thick] (v\i) -- (v\j);
				
				\node[below] at (0,-2) {Petersen Graph};
			\end{scope}
		\end{tikzpicture}
		\caption{Left: A Cartesian product grid graph. Right: The Petersen graph, featuring the characteristic pentagram inner structure.}
		\label{fig:families}
	\end{figure}
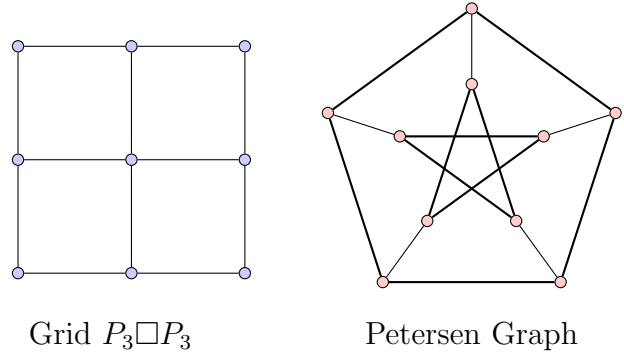
	
	\subsection{Graph Operations and Automorphisms}
	Graph operations allow us to construct complex structures from simpler ones. The most relevant for this thesis is the \textbf{Cartesian product}.
	For graphs $G$ and $H$, the Cartesian product $G \square H$ has vertex set $V(G) \times V(H)$. Two vertices $(g, h)$ and $(g', h')$ are adjacent if either $g=g'$ and $hh' \in E(H)$, or $h=h'$ and $gg' \in E(G)$ \cite{Imrich2010,Hammack2012}.
	Iterating products yields $n$-dimensional grids: $P_{k_1} \square \dots \square P_{k_n}$. When the factors are cycles $C_{k_i}$, the result is a toroidal grid, or torus.
	
	The symmetry of a graph is captured by its \textbf{automorphism group} $\text{Aut}(G)$, consisting of all permutations of $V(G)$ that preserve adjacency \cite{Godsil2001}. A graph is \textit{vertex-transitive} if $\text{Aut}(G)$ acts transitively on vertices. Cayley graphs are the canonical examples of vertex-transitive graphs \cite{Sabidussi1958}.
	
	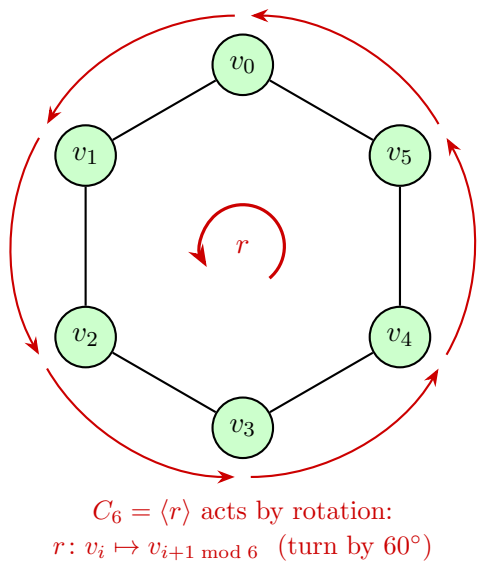
\begin{figure}[htbp]
		\centering
		\begin{tikzpicture}[
			vtx/.style={circle, draw, thick, fill=green!20, minimum size=8mm, inner sep=0pt},
			rot/.style={->, red!80!black, thick, >=Stealth},
			font=\small
			]
			\foreach \i/\ang in {0/90, 1/150, 2/210, 3/270, 4/330, 5/30} {
				\node[vtx] (v\i) at (\ang:2.4) {$v_{\i}$};
			}
			\draw[thick] (v0)--(v1)--(v2)--(v3)--(v4)--(v5)--(v0);
			\foreach \a/\b in {90/150, 150/210, 210/270, 270/330, 330/390, 30/90} {
				\draw[rot, shorten >=3pt, shorten <=3pt] (\a:3.05) arc (\a:\b:3.05);
			}
			\draw[rot, very thick] (-50:0.55) arc (-50:210:0.55);
			\node[red!80!black] at (0,0) {$r$};
			\node[red!80!black, align=center, font=\footnotesize] at (0,-3.75)
			{$C_6=\langle r\rangle$ acts by rotation:\\[1pt]
				$r\colon v_i \mapsto v_{i+1\bmod 6}$ \ (turn by $60^\circ$)};
		\end{tikzpicture}
		\caption{The cyclic group $C_6=\langle r\rangle$ acting on the $6$-cycle
			graph $C_6$ by rotation. The six rim edges form the cycle itself; the
			generator $r$ rotates the hexagon by $60^\circ$, sending each vertex $v_i$
			to its neighbour $v_{i+1}$ (indices mod $6$). Because $r$ carries edges to
			edges, it is a graph automorphism.}
		\label{fig:automorphism}
	\end{figure}
	
	\subsection{Algebraic Representations of Graphs}
	\label{sec:algebraic-graph-theory}
	Beyond the combinatorial definition, graphs admit powerful algebraic representations via matrices. These matrices play a central role in Graph Signal Processing (GSP) and spectral graph theory.
	
	Let $A$ be the \textbf{adjacency matrix} of a graph $G$ on $n$ vertices, defined by:
	\[
	A_{ij} = \begin{cases} 
		1 & \text{if } \{v_i, v_j\} \in E(G), \\
		0 & \text{otherwise}.
	\end{cases}
	\]
	
	Let $D$ be the \textbf{degree matrix}, a diagonal matrix where $D_{ii} = \deg(v_i)$. The \textbf{Laplacian matrix} $L$ is defined as $L = D - A$. This operator is central to spectral graph theory and the definition of the Graph Fourier Transform (GFT) \cite{Chung1997,Cvetkovic2010}. The normalized Laplacian is often defined as $\mathcal{L} = D^{-1/2} L D^{-1/2}$.
	
	These matrices satisfy:
	\begin{itemize}
		\item $A$ encodes the immediate connectivity.
		\item $L$ acts as a difference operator: $(Lf)_i = \sum_{j \sim i} (f(i) - f(j))$.
		\item The eigenvectors of $L$ provide an orthonormal basis for signals on $G$, generalizing the Fourier basis to graphs \cite{Chung1997}.
	\end{itemize}
	
	For readers unfamiliar with the formal treatment of graph products and embeddings, we recommend the comprehensive work by Hammack \cite{Hammack2012} and the classic texts by Bondy and Murty \cite{Bondy1976} or West \cite{West2001}.
	
	\subsection{Isometric and Graph Embeddings}
	A fundamental notion in metric graph theory is the preservation of distances.
	
	\begin{definition}[Graph Metrics]
		The \textbf{shortest path distance} $d_G(u,v)$ is the length of the shortest path connecting $u,v \in V(G)$. The metric space $(V, d_G)$ defines the \textbf{graph metric} \cite{Chartrand2006}.
	\end{definition}
	
	\begin{definition}[Isometric Embedding]
		A graph $G$ embeds \textbf{isometrically} into $H$ if there exists a mapping $\phi: V(G) \to V(H)$ such that $d_H(\phi(u), \phi(v)) = d_G(u,v)$ for all $u,v \in V(G)$ \cite{Burago2001}.
	\end{definition}
	
	A particularly important class of graphs that embed into hypercubes $Q_n$ (the Cartesian product of $n$ copies of $K_2$) are the \textbf{partial cubes} \cite{Djokovic1973}. These graphs correspond to binary strings where the Hamming distance equals the graph distance. While powerful, the partial cube restriction is limiting; this thesis generalizes this to Abelian groups (products of cycles), known as partial torus graphs \cite{Chepoi2000}.
	
	\section{Group Theory and Symmetry}
	\label{sec:group-theory}
	
	We assume the reader is familiar with basic group definitions. For a comprehensive introduction to group theory, including homomorphisms, actions, and structure theorems, we refer the reader to the standard texts by Rotman \cite{Rotman1995}, Dummit and Foote \cite{Dummit2004}, and Herstein \cite{Herstein1975}. Here we focus on aspects relevant to signal processing and graph symmetry.
	
	\subsection{Finite Abelian Groups}
	A group $\Gamma$ is \textbf{Abelian} (commutative) if $gh = hg$ for all $g,h \in \Gamma$. The Fundamental Theorem of Finite Abelian Groups states that any such group is isomorphic to a direct product of cyclic groups of prime-power order \cite{Rotman1995}:
	\[
	\Gamma \cong \mathbb{Z}_{n_1} \times \dots \times \mathbb{Z}_{n_d}.
	\]
	In the context of signal processing, we often view $\Gamma$ as the index set for a signal. The additive group $\mathbb{Z}_N$ (cyclic) and the vector space $\mathbb{Z}_2^n$ (hypercube) are the most common examples.
	
	\subsection{Representations of Finite Groups}
	To perform Fourier analysis on groups, we utilize representation theory \cite{Serre1977}. For deeper study of harmonic analysis on finite groups, see Folland \cite{Folland1995} or Terras \cite{Terras1999}.
	
	\begin{definition}[Unitary Representation]
		A representation of a finite group $G$ is a homomorphism $\rho: G \to U(V)$, where $U(V)$ is the group of unitary operators on a finite-dimensional Hilbert space $V$ (usually $\mathbb{C}^n$). We say $\rho$ is \textbf{unitary}.
	\end{definition}
	
	For Abelian groups, every irreducible representation is 1-dimensional. These 1D representations are precisely the \textbf{characters} of the group \cite{Folland1995}. Let $\widehat{\Gamma}$ denote the set of irreducible characters. For $\Gamma = \mathbb{Z}_N$, the characters are $\chi_k(n) = e^{-i 2\pi k n / N}$.
	
	\begin{theorem}[Peter-Weyl Theorem (Finite Group Version)]
		The matrix coefficients of the irreducible representations of a compact (or finite) group form an orthonormal basis for the space of square-integrable functions on the group \cite{PeterWeyl1927, Sugiura1975}. For finite Abelian groups, this reduces to the orthogonality of characters.
	\end{theorem}
	
	This theorem guarantees that any signal on the group can be decomposed into a sum of irreducible frequency components, providing the justification for the Fourier transform.
	
	\section{Harmonic Analysis on Groups and Wavelets}
	\label{sec:harmonic-analysis}
	
	This section establishes the signal processing machinery. We treat signals as functions $f: \Gamma \to \mathbb{C}$. We explicitly bridge the gap between conventional Discrete Signal Processing (DSP) and signal processing on algebraic groups.
	
	\subsection{The Left Regular Representation and Translation}
	The natural action of the group on the space of signals is the \textbf{Left Regular Representation}.
	
	Let $L^2(\Gamma)$ be the Hilbert space of complex-valued functions on $\Gamma$ with inner product $\langle f, h \rangle = \sum_{g \in \Gamma} f(g)\overline{h(g)}$. The translation (or shift) operator $T_h: L^2(\Gamma) \to L^2(\Gamma)$ is defined by:
	\[
	(T_h f)(g) = f(h^{-1}g).
	\]
	In additive notation for Abelian groups, this simplifies to $(T_h f)(g) = f(g-h)$. The set $\{T_h\}_{h \in \Gamma}$ forms a unitary representation of $\Gamma$ on $L^2(\Gamma)$ \cite{Rudin1962}.
	
	\textbf{Link to DSP:} In 1D signal processing, the shift operator is represented by the permutation matrix $S$. For the cyclic group $\mathbb{Z}_N$, the operator $T_1$ corresponds exactly to the circulant shift matrix $S$, where $S_{ij} = 1$ if $j = i+1 \pmod N$. Thus, the group-theoretic translation generalizes the time-shift operation in DSP to arbitrary domains.
	
	\subsection{Convolution and the Convolution Theorem}
	Convolution generalizes the moving average or filtering operation to groups.
	
	\begin{definition}[Convolution]
		For $f, h \in L^2(\Gamma)$, the convolution $f * h$ is defined as:
		\[
		(f * h)(g) = \frac{1}{|\Gamma|} \sum_{x \in \Gamma} f(x) h(g - x).
		\]
	\end{definition}
	
	The power of harmonic analysis stems from the \textbf{Convolution Theorem}.
	
	\begin{theorem}[Convolution Theorem]
		Convolution in the time (vertex) domain corresponds to pointwise multiplication in the frequency domain.
		\[
		\widehat{f * h}(\chi) = \hat{f}(\chi) \hat{h}(\chi), \quad \forall \chi \in \widehat{\Gamma}.
		\]
	\end{theorem}
	Furthermore, using the orthogonality relations and \textbf{Schur's Lemma}, one can show that the translation operators are diagonalized by the Fourier transform \cite{Terras1999}.
	
	\textbf{Link to DSP:} The characters $\chi$ form the columns of the Discrete Fourier Transform (DFT) matrix $\mathbf{F}$. The dual object of the group $\widehat{\Gamma}$ is precisely the set of Fourier basis vectors. The Convolution Theorem justifies the use of frequency-domain filtering (e.g., low-pass filters for noise removal) in image processing.
	
	\subsection{Signals on $\mathbb{Z}$ and Periodic Extension}
	To utilize group-theoretic harmonic analysis, we must map finite discrete signals to function spaces on groups. Consider a finite 1D discrete signal $x[n]$ defined for $n \in \{0, \dots, N-1\}$.
	
	If $x[n]$ is periodic with period $N$, it naturally corresponds to a function $f: \mathbb{Z}_N \to \mathbb{C}$. However, if $x[n]$ is non-periodic (finite support), we must apply a \textbf{periodic extension} scheme to define it on a group \cite{Oppenheim2009}:
	
	\begin{enumerate}
		\item \textbf{Wraparound (Direct Periodic Extension):} Define $x[n + kN] = x[n]$. This implicitly constructs a cycle graph $C_N$ where the "end" connects to the "start". This is the standard assumption for the DFT.
		\item \textbf{Zero-Padding:} Define $x[n] = 0$ for $n \notin \{0, \dots, N-1\}$ followed by periodic extension. This treats the signal as having support on a larger group.
		\item \textbf{Symmetric Extension:} Define $x[-n] = x[n]$ (even/odd extension) before wrapping. This is often used in filter banks to avoid boundary artifacts.
	\end{enumerate}
	
	\begin{figure}[htbp]
		\centering
		\begin{tikzpicture}
			\node[circle,draw,fill=blue!10] (n0) at (0,0) {0};
			\node[circle,draw,fill=blue!10] (n1) at (1.5,0) {1};
			\node[circle,draw,fill=blue!10] (n2) at (3,0) {2};
			\node (dots) at (4.5,0) {$\dots$};
			\node[circle,draw,fill=blue!10] (nN1) at (6,0) {$N-1$};
			
			\draw (n0) -- (n1) -- (n2) -- (dots) -- (nN1);
			\draw[thick, red, ->] (nN1) to[out=45,in=135] (n0);
			
			\node[below=1cm of n2] {Linear Domain $\{0, \dots, N-1\}$};
			\node[above=1cm of nN1] {Group $\mathbb{Z}_N$ (Wraparound)};
		\end{tikzpicture}
		\caption{Visualization of the periodic extension (wraparound) mapping a 1D finite sequence to a cyclic group $\mathbb{Z}_N$.}
		\label{fig:periodic-extension}
	\end{figure}
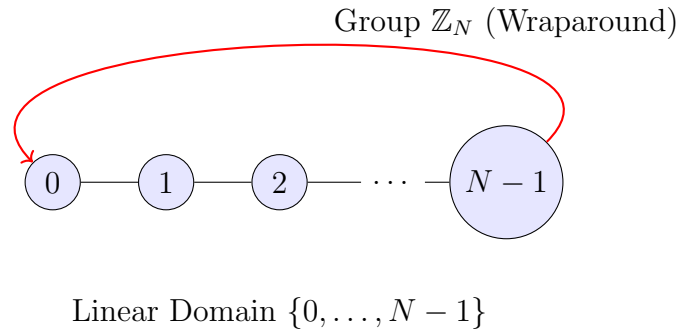
	
	\subsection{2D Images and the Torus}
	Consider a 2D discrete signal (image) $I[m,n]$ with dimensions $M \times N$. Its domain is $\{0,\dots,M-1\} \times \{0,\dots,N-1\}$.
	
	\begin{itemize}
		\item If we apply 2D periodic extension (wraparound in both dimensions), the image corresponds to a function on the direct product group $\Gamma = \mathbb{Z}_M \times \mathbb{Z}_N$.
		\item The Cayley graph of this group with the generating set $\{\pm e_1, \pm e_2\}$ (shifts right/left and up/down) is the \textbf{Torus} graph $T_{M,N} = C_M \square C_N$.
		\\item On this torus, the right neighbor of pixel $(M-1, y)$ is $(0, y)$, and the top neighbor of $(x, 0)$ is $(x, N-1)$.
	\end{itemize}
	
	This structure is fundamental in image processing. The Discrete Cosine Transform (DCT) and DFT used in JPEG compression implicitly rely on this toroidal topology (or extensions thereof). By embedding a graph into such a torus, we allow the application of these standard image processing tools to irregular graph data.
	
	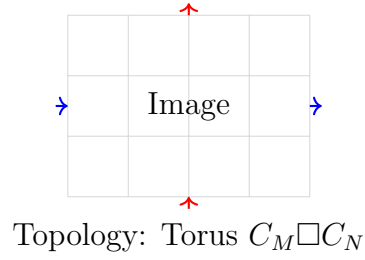
\begin{figure}[htbp]
		\centering
		\begin{tikzpicture}[scale=0.8]
			\draw[step=1cm, gray!30, thin] (0,0) grid (4,3);
			\node at (2, 1.5) {Image};
			
			\draw[thick, blue, ->] (4, 1.5) -- (4.2, 1.5);
			\draw[thick, blue, ->] (-0.2, 1.5) -- (0, 1.5);
			\draw[thick, red, ->] (2, 3) -- (2, 3.2);
			\draw[thick, red, ->] (2, -0.2) -- (2, 0);
			
			\node[below=0.2cm] at (2,0) {Topology: Torus $C_M \square C_N$};
		\end{tikzpicture}
		\caption{Schematic of a 2D grid with periodic boundary conditions (wraparound), forming a Torus.}
		\label{fig:torus-topology}
	\end{figure}
	
	\subsection{Wavelet Transforms and Filter Banks}
	While Fourier analysis provides global frequency information, wavelet analysis provides multi-scale localization. On finite groups, wavelet transforms are constructed via the theory of \textbf{Multiresolution Analysis (MRA)} \cite{Mallat2009, Daubechies1992}. An MRA on a group $\Gamma$ consists of a nested sequence of subspaces $V_j \subset V_{j-1} \subset \dots \subset L^2(\Gamma)$. These subspaces are generated by the scaling of a "scaling function" $\phi$. The difference spaces $W_j = V_{j-1} \ominus V_j$ contain the "detail" information, spanned by wavelets $\psi$.
	
	\textbf{MRA in DSP and Sampling Theory:} In conventional digital signal processing, the MRA is implemented using \textbf{Filter Banks} combined with sampling operations. The analysis bank splits the signal into approximation (low-pass) and detail (high-pass) components via filters $H_0$ and $H_1$, respectively. A critical component of this architecture is the \textbf{downsampling} operator ($\downarrow 2$), which reduces the sampling rate by a factor of 2. This operation effectively keeps every other sample, implementing the transition from the scale space $V_{j}$ to the coarser space $V_{j-1}$. Conversely, the synthesis bank utilizes \textbf{upsampling} ($\uparrow 2$), which inserts zeros between samples, before the synthesis filters $G_0$ and $G_1$ reconstruct the signal. The perfect interplay between these filters and the sampling operators is what allows for lossless compression and reconstruction in traditional wavelet theory.
	
	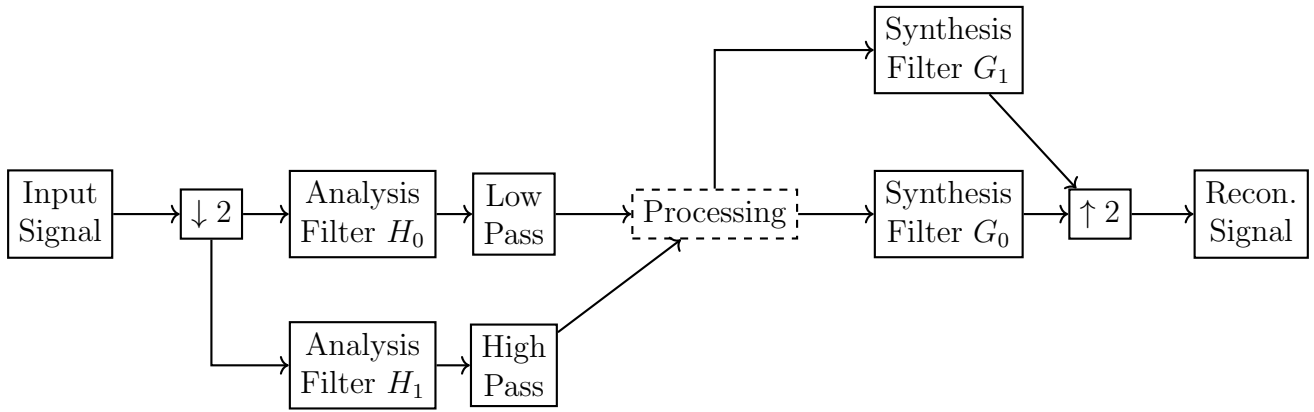
\begin{figure}[htbp]
		\centering
		\begin{tikzpicture}[auto, node distance=2cm, thick]
			
			\node [draw, align=center] (input) {Input\\Signal};
			
			\node [draw, right of=input] (down) {$\downarrow 2$};
			
			\node [draw, right of=down, align=center] (analysis) {Analysis\\Filter $H_0$};
			\node [draw, below of=analysis, align=center] (analysis2) {Analysis\\Filter $H_1$};
			
			\node [draw, right of=analysis, align=center] (low) {Low\\Pass};
			\node [draw, right of=analysis2, align=center] (high) {High\\Pass};
			
			\draw [->] (input) -- (down);
			\draw [->] (down) -- (analysis);
			\draw [->] (down) |- (analysis2);
			\draw [->] (analysis) -- (low);
			\draw [->] (analysis2) -- (high);
			
			\node [right=1cm of low, dashed, draw, align=center] (proc) {Processing};
			\draw [->] (low) -- (proc);
			\draw [->] (high) -- (proc);
			
			\node [draw, right=1cm of proc, align=center] (synth) {Synthesis\\Filter $G_0$};
			\node [draw, above=1cm of synth, align=center] (synth2) {Synthesis\\Filter $G_1$};
			
			\node [draw, right of=synth] (up) {$\uparrow 2$};
			\node [draw, right of=up, align=center] (output) {Recon.\\Signal};
			
			\draw [->] (proc) -- (synth);
			\draw [->] (proc) |- (synth2);
			\draw [->] (synth) -- (up);
			\draw [->] (synth2) -- (up);
			\draw [->] (up) -- (output);
		\end{tikzpicture}
		\caption{Schematic of a Two-Channel Filter Bank for Wavelet Transform on Groups, illustrating the critical downsampling ($\downarrow 2$) and upsampling ($\uparrow 2$) stages.}
		\label{fig:filterbank}
	\end{figure}
	
	\textbf{The Challenge of Sampling in Matrix-Based GSP:} In the standard Graph Signal Processing (GSP) framework, where signals are viewed as vectors in $\mathbb{R}^N$ and operations are defined via the graph Laplacian or adjacency matrix, defining a rigorous sampling theory is notoriously difficult. Unlike regular lattices $\mathbb{Z}$ or $\mathbb{Z}^n$, arbitrary graphs lack a uniform structure or a natural notion of "shift." Consequently, concepts like "downsampling by factor 2" are not well-defined algebraically. 
	
	Existing approaches in the literature often resort to \textbf{graph coloring} (specifically bipartitioning) to define downsampling sets, or treat sampling as a generic vertex selection problem \cite{Narang2012, Shuman2013, Chen2015}. However, these methods often lead to distortions in the graph topology (e.g., inducing subgraphs) and complicate the design of perfect reconstruction filter banks. The lack of a natural algebraic group structure in these matrix-based frameworks makes it challenging to guarantee the preservation of spectral properties across scales \cite{Sandryhaila2013}.
	
	\textbf{Sampling via Group Embedding:} In the context of this thesis, we overcome these limitations by embedding an arbitrary graph into a Cayley graph of an Abelian group (e.g., $\mathbb{Z}_2^p$). By mapping graph vertices to elements of a group, we inherit the algebraic structure necessary for rigorous sampling. Downsampling on a group is a well-defined operation (e.g., taking every other element or selecting a subgroup), which induces a natural MRA. This allows us to define wavelets on the graph by "pulling back" the wavelets defined on the group \cite{Hammond2011, Shuman2013}. Thus, our embedding scheme transforms the ill-posed problem of sampling on an irregular graph into the well-posed problem of sampling on a regular algebraic structure.
	
	\subsubsection{The Two-Scale Relation}
	A fundamental concept in the construction of wavelets, particularly on groups, is the \textbf{two-scale relation} (or dilation equation). It links the scaling function $\phi(x)$ at one scale to its translated copies at a finer scale.
	\[
	\phi(x) = \sqrt{2} \sum_{k \in \Gamma} h_k \phi(2x - k)
	\]
	Here, $2x$ represents a dilation operation (which, on a finite cyclic group $\mathbb{Z}_N$, corresponds to a mapping $x \mapsto 2x \pmod N$). The coefficients $h_k$ form the low-pass filter $H_0$ discussed in the filter bank section.
	
	This equation implies that the approximation space $V_0$ (spanned by $\phi(x-k)$) is a subset of the finer space $V_1$ (spanned by $\phi(2x-k)$). The existence of such scaling functions on discrete groups is not trivial; it requires the group structure to be compatible with the dilation map \cite{Kovacevic1996}. In our work, by embedding arbitrary graphs into Abelian groups, we ensure that such dilation structures (modulo the group order) are well-defined, allowing us to construct valid MRA systems for graph signals.
	
	\part{Graph Embedding Theory}
	\label{part:intro-part-I}
	\partintro{This part constructs the theoretical and algorithmic core of the dissertation: the quest for a \textbf{minimal isometric embedding} of a graph into an Abelian group. As motivated in the general introduction, such an embedding is the essential prerequisite for transferring harmonic analysis to the graph. However, not all embeddings are equal; an embedding into an unnecessarily large group is computationally burdensome and dilutes the structural analogy to classical settings. This part presents the framework developed for the minimal embedding of a graph into a group. We introduce three novel binary relations on graph edges based on metric conditions that captures local isometric structure. Using these relations, we develop a comprehensive framework for isometrically embedding arbitrary connected graphs into Cayley graphs of finite Abelian groups. Our approach provides constructive existence proofs, efficient algorithms, and deep connections to geometric group theory, algebraic combinatorics, and topological graph theory. We establish fundamental structural properties of each relation, characterize its behavior under graph operations, prove complexity-theoretic bounds, and demonstrate explicit completions of classical graphs including the Petersen graph. The first chapter of this part focuses on graph embedding into the binary group $\mathbb{Z}_2^p$ whereas the second chapter investigates the embedding of graphs into general Abelian groups. Applications to network design, coding theory, and quantum information are discussed in the last chapter of this part, and the second part of the dissertation applies this framework to harmonic analysis of signals on graphs. Throughout these chapters, theoretical insights are illustrated with explicit examples and diagrams, and algorithmic claims are supported by detailed complexity analysis. The result is a self-contained, rigorous toolbox for representing complex, irregular graph data within simple, symmetric algebraic structures—a crucial step toward enabling advanced mathematical analysis on networks.}

	\chapter{Compact Isometric Embedding into Binary Groups $\mathbb{Z}_2^k$}
	\label{chap:binary-embedding}
	
	\section{Introduction and Motivation}
	
	The quest for embedding graphs into hypercubes is one of the most fundamental
	and well-studied problems in metric graph theory \cite{Ovchinnikov2008,
		BandeltChepoi2008}.  Hypercubes, as Cayley graphs of the binary group
	$\Z_2^k$, possess a structure that is simultaneously combinatorially rich and
	algebraically simple, making them ideal hosts for parallel computing
	architectures and error-correcting codes \cite{Harary1969,Hayes2006}.
	However, the theoretical possibility of embedding a graph into a Cayley graph
	of $\Z_2^k$ does not by itself guarantee practical utility: the naive
	embedding of Theorem~\ref{thm:naive} yields a host group of order $2^{n-1}$,
	rendering any downstream harmonic-analysis algorithm computationally
	intractable.
	
	This chapter develops the theory of \textbf{compact isometric embeddings} of
	a connected graph $G$ into Cayley graphs $\Cay(\Z_2^k, S)$, where the
	generating set $S$ is \emph{not} restricted to the standard basis.  We seek a
	small dimension $k$; the host order $2^k$ then controls the cost of every
	Fourier-analytic operation of Part~II.
	
	The theoretical contribution is a \textbf{two-layer framework} flanked by a
	\textbf{bounds theory}.
	
	In the first layer we identify which edges of $G$ should carry the same
	generator.  For this we introduce a binary relation $\varphi$ on $E(G)$ that
	captures \emph{metric parallelism} --- a strict generalization of the
	classical Djokovi\'c--Winkler relation $\theta$ \cite{Djokovic1973,
		Winkler1984}.  The \textbf{transitive prune} operation extracts from
	$\varphi$ a partition $\mathcal{P}$ of $E(G)$ into candidate same-generator
	classes.
	
	In the second layer we convert \emph{any} partition $\mathcal{P}$ of $E(G)$
	into a concrete vertex labeling in $\Z_2^k$.  The key result, developed in
	Section~\ref{sec:cocycle}, is the \textbf{Cocycle/Quotient Theorem}: the most
	generic consistent labeling is the GF(2) quotient of $\Z_2^{t}$
	($t = |\mathcal{P}|$) by the cycle-space relations of $\mathcal{P}$, giving
	dimension $k = t - \mathrm{rank}_{\F_2}(A)$, where $A$ is the cycle--class
	parity matrix.  This theorem (a)~makes the labeling provably conflict-free,
	(b)~subsumes the partial-cube cut paradigm as a special case, (c)~yields a
	one-line proof of the naive $n-1$ upper bound, and (d)~explains why the
	Petersen and Pappus graphs admit compact binary embeddings that cut-based
	approaches cannot find.  Together with a \textbf{shortcut-repair loop}
	(Section~\ref{sec:phi-algorithm}) this gives the first $\varphi$-based
	algorithm that succeeds on \emph{every} connected graph; we verified it
	exhaustively on all $995$ connected graphs with at most seven vertices, with
	every output independently re-checked against a full reconstruction of the
	host Cayley graph (zero failures).
	
	The bounds theory (Section~\ref{sec:bounds}) is important to have a clearer picture of the embedding machinery.  We
	prove the general lower bound
	$k \geq \max(\diam(G), \lceil \log_2 n \rceil)$ and exhibit families
	attaining each term.  We then determine the exact minimum dimension for two
	families at the opposite ends of the sparse spectrum: stars, for which
	$k_{\min}(K_{1,q}) = \lceil\log_2 q\rceil + 1$ --- exponentially below the naive bound --- and odd
	cycles, for which $k_{\min}(C_m) = m-1$, showing that the naive upper bound
	$n-1$ is tight.  Finally we measure the optimality profile of our algorithm
	against exhaustive exact search: it attains the true minimum on $29$ of the
	$30$ connected graphs with $n \leq 5$, and the unique exception is exactly
	the star phenomenon.
	
	The chapter is organized as follows.
	Section~\ref{sec:ch2-prelim} reviews graph metrics and Cayley graphs and
	presents the naive upper bounds with complete proofs.
	Section~\ref{sec:phi-relation} introduces the $\varphi$ relation.
	Section~\ref{sec:transitive-prune} develops the transitive prune.
	Section~\ref{sec:binary-algorithm-old} presents the cut-based embedding,
	valid for partial cubes.
	Section~\ref{sec:cocycle} develops the Cocycle/Quotient Theorem.
	Section~\ref{sec:phi-algorithm} presents the universal $\varphi$-quotient
	algorithm.
	Section~\ref{sec:bounds} establishes the bounds theory.
	Section~\ref{sec:binary-complexity} provides the complexity analysis.
	Section~\ref{sec:binary-experiments} presents the regenerated experimental
	results, including the graph-signal-processing benchmark families reused in
	Part~II.
	
	\section{Preliminaries: Graphs, Groups, and Naive Embeddings}
	\label{sec:ch2-prelim}
	
	\subsection{Graph Metrics and Isometric Embeddings}
	
	Throughout, $G = (V, E)$ is a finite, simple, connected, undirected graph
	with $n = |V|$ and $m = |E|$, and $d_G$ denotes its shortest-path metric.
	
	\begin{definition}[Isometric embedding]
		A graph $G$ embeds \textbf{isometrically} into a graph $H$ if there is a map
		$\lambda\colon V(G) \to V(H)$ with
		\[
		d_H(\lambda(u),\lambda(v)) = d_G(u,v) \qquad \forall\, u,v \in V(G).
		\]
		Isometric maps are automatically injective (distinct vertices are at positive
		distance).
	\end{definition}
	
	\begin{definition}[Cayley graph]
		Let $\Gamma$ be a finite abelian group (written additively) and
		$S \subseteq \Gamma \setminus \{0\}$ a symmetric generating set ($S = -S$;
		in $\Z_2^k$ every set is symmetric).  The \textbf{Cayley graph}
		$\Cay(\Gamma,S)$ has vertex set $\Gamma$ and edges $\{g, g+s\}$ for all
		$g \in \Gamma$, $s \in S$.
	\end{definition}
	
	Cayley graphs are vertex-transitive: $d_{\Cay}(x, y) = d_{\Cay}(0, y - x)$,
	and $d_{\Cay}(0, z)$ equals the least length of a word $s_1 + \cdots +
	s_\ell = z$ with $s_i \in S$.  This homogeneity is what permits a
	well-defined Fourier analysis on vertex signals \cite{Terras1999}, the
	subject of Part~II.
	
	\begin{remark}[Words over $\Z_2^k$ never repeat generators]
		\label{rem:no-repeat}
		In $\Z_2^k$ every element is its own inverse, so if a word contains the same
		generator twice, deleting both occurrences yields a strictly shorter word
		with the same sum.  Hence every \emph{geodesic} word in $\Cay(\Z_2^k,S)$
		uses pairwise distinct generators, and may be identified with a
		\emph{subset} of $S$.  We use this repeatedly.
	\end{remark}
	
	\medskip
	The \textbf{excursion ratio} of an embedding into $\Gamma$ is
	\[
	\varepsilon \;=\; \frac{|V(G)|}{|\Gamma|}\;\in\;(0,1].
	\]
	The embedding is \emph{compact} when $\varepsilon$ is \emph{large} (close to
	$1$): a large ratio means the host is barely bigger than the graph it
	carries.  The naive spanning-tree embedding has $\varepsilon = n/2^{n-1}$,
	which vanishes exponentially: for the Petersen graph ($n = 10$) it gives
	$\varepsilon \approx 0.02$, while the optimal binary embedding of
	Figure~\ref{fig:petersenclebsch} achieves $\varepsilon = 10/16 = 0.625$.
	
	\subsection{Partial Cubes and Hypercubes}
	
	The $k$-dimensional hypercube is
	$Q_k = \Cay(\Z_2^k, \{e_1,\ldots,e_k\})$, the Cayley graph on the
	\emph{standard basis}; its metric is the Hamming distance.
	
	\begin{definition}[Partial cube]
		A graph $G$ is a \textbf{partial cube} if it admits an isometric embedding
		into some hypercube $Q_k$ \cite{Graham1985,Ovchinnikov2008}.  The smallest
		such $k$ is the \emph{isometric dimension} $\mathrm{idim}(G)$.
	\end{definition}
	
	The Djokovi\'c--Winkler relation $\theta$ classically characterizes partial
	cubes: a connected graph is a partial cube if and only if it is bipartite and
	$\theta$ is transitive \cite{Djokovic1973,Winkler1984}. The structural properties of partial cubes and their recognition were established in early foundational works \cite{Shpectorov1993,Chepoi1988}, building on the canonical isometric embedding of graphs into Cartesian products of complete graphs \cite{GrahamWinkler1985}. Concurrently, this advanced into the broader theory of isometric and low-distortion embeddings of graph metrics, explored through both classical geometric frameworks \cite{Bourgain1985} and modern algorithmic perspectives \cite{Eppstein2008,Linial1995}.
	
	The relation $\theta$ remains an active object of study. Recent work has analysed its transitive closure $\theta^{\ast}$ under full graph subdivisions, with applications to distance-based topological indices of molecular graphs such as fullerenes and silica nanostructures \cite{KlavzarKnauerMarc2019}; related the cube polynomial of a partial cube to the clique polynomial of its \emph{crossing graph}, whose vertices are the $\theta$-classes and in which two classes are adjacent precisely when they cross \cite{XieFengXu2024}; and introduced the \emph{reflexive complement} $\overline{\theta}$, whose transitive closure was shown to have more than one equivalence class only within a restricted subclass of complete multipartite graphs \cite{Hellmuth2025}. Broader expositions of the relation and of partial-cube theory are given by \cite{Ovchinnikov2008,Imrich2010}. A common thread is that these developments refine, extend, or apply $\theta$ \emph{within} the partial-cube universe---the bipartite graphs on which $\theta$ is already an equivalence relation.  Our relation
	$\varphi$ (Section~\ref{sec:phi-relation}) is designed for the general
	problem, where the target is $\Cay(\Z_2^k,S)$ with \emph{arbitrary} $S$;
	non-bipartite graphs then become embeddable, and --- as
	Theorem~\ref{thm:star} will show --- even partial cubes can beat their
	isometric dimension when composite generators are allowed.
	
	\subsection{Naive Embedding Schemes}
	
	\begin{theorem}[Naive binary embedding]
		\label{thm:naive}
		Every connected graph $G$ with $n$ vertices embeds isometrically into
		$\Cay(\Z_2^{n-1}, S)$ for a suitable generating set $S$ with
		$|S| \leq m$.
	\end{theorem}
	
	\begin{proof}
		Fix a spanning tree $T$ of $G$ rooted at $r$ and bijectively assign to its
		$n-1$ edges the standard basis vectors $e_1,\ldots,e_{n-1}$ of $\Z_2^{n-1}$.
		Label each vertex $v$ by
		$\lambda(v) = \sum_{e \in T[r,v]} e_{\,\iota(e)}$, the sum of the basis
		vectors on the unique tree path $T[r,v]$; equivalently, $\lambda(v)$ is the
		characteristic vector of $T[r,v]$.  Put
		$S = \{\lambda(u) \oplus \lambda(v) : uv \in E\}$.  Every edge of $G$ then
		joins labels differing by an element of $S$, so each $G$-path of length
		$\ell$ maps to a Cayley walk of length $\ell$, giving
		$d_{\Cay}(\lambda(u),\lambda(v)) \leq d_G(u,v)$.
		
		For the reverse inequality we use the language that will be developed in
		Section~\ref{sec:cocycle} and which makes the argument transparent; the
		reader may verify that no circularity is involved, as
		Lemma~\ref{lem:join} is self-contained.  The labeling above is exactly the
		quotient labeling of the all-singleton partition
		(Corollary~\ref{cor:naive-is-quotient}), and Lemma~\ref{lem:join} shows that
		any generator word of length $\ell$ summing to
		$\lambda(u)\oplus\lambda(v)$ corresponds to an edge set of $G$ whose GF(2)
		boundary is $\{u,v\}$; such a set contains a $u$--$v$ path and so
		$\ell \geq d_G(u,v)$.  Hence $d_{\Cay} = d_G$.
	\end{proof}
	
	\begin{theorem}[Sparse cyclic embeddings]
		\label{thm:sparse-cyclic}
		\leavevmode
		\begin{enumerate}[(i)]
			\item Let $G$ be a connected \emph{unicyclic} graph ($m = n$) whose unique
			cycle has length $\mu$.  Then $G$ embeds isometrically into
			$\Cay\bigl(\Z_\mu \times \Z_2^{\,n-\mu},\; S\bigr)$ with
			$S = \{(1,\mathbf 0)\} \cup \{(0, e_j) : 1 \leq j \leq n-\mu\}$,
			i.e.\ into the Cartesian product $C_\mu \,\square\, Q_{n-\mu}$.
			\item Let $T$ be a tree and $P$ a path in $T$ with $\mu$ vertices.  Then $T$
			embeds isometrically into
			$\Cay\bigl(\Z_{2\mu-2} \times \Z_2^{\,n-\mu},\; S\bigr)$, the product
			$C_{2\mu-2} \,\square\, Q_{n-\mu}$.
		\end{enumerate}
	\end{theorem}
	
	\begin{proof}
		(i) Let $C = v_0 v_1 \cdots v_{\mu-1} v_0$ be the unique cycle.  First, $C$
		is an isometric subgraph of $G$: a shortest path between two cycle vertices
		that left $C$ would have to re-enter it through the same cut vertex (since
		$G$ has no second cycle), and the excursion could be excised, so shortest
		paths between cycle vertices stay on $C$.
		
		Deleting $E(C)$ leaves a forest of trees, each hanging from exactly one
		cycle vertex; every non-cycle vertex $u$ has a unique attachment vertex
		$a(u) \in V(C)$ and a unique tree path from $a(u)$ of length
		$\mathrm{dep}(u)$, and the $n - \mu$ non-cycle edges receive distinct fresh
		coordinates $e_1, \ldots, e_{n-\mu}$.  Define
		\[
		\lambda(v_i) = (i, \mathbf 0), \qquad
		\lambda(u) = \bigl(\,\mathrm{idx}(a(u)),\; x_u\,\bigr),
		\]
		where $x_u \in \Z_2^{\,n-\mu}$ is the characteristic vector of the tree path
		from $a(u)$ to $u$.  The host $\Cay(\Z_\mu \times \Z_2^{n-\mu}, S)$ is the
		Cartesian product $C_\mu \square Q_{n-\mu}$, whose metric is the sum of the
		factor metrics:
		$d\bigl((i,x),(j,y)\bigr) = d_{C_\mu}(i,j) + \mathrm{wt}(x \oplus y)$.
		
		Now take $u, w \in V(G)$.  If both lie on $C$, then
		$d_G = d_{C_\mu}$ by isometry of $C$ and the second coordinate vanishes.
		If they hang from attachments $a(u) \neq a(w)$, every $u$--$w$ path passes
		through both attachments, so
		$d_G(u,w) = \mathrm{dep}(u) + d_{C_\mu}(a(u), a(w)) + \mathrm{dep}(w)$;
		since distinct hanging trees use disjoint coordinate sets,
		$\mathrm{wt}(x_u \oplus x_w) = \mathrm{dep}(u) + \mathrm{dep}(w)$, and the
		product metric reproduces $d_G$ exactly.  If $a(u) = a(w)$, both lie in the
		same hanging tree, $d_G(u,w)$ is the tree distance, and the characteristic
		vectors of the two root paths differ exactly on the symmetric difference of
		the paths, whose size is the tree distance; the cyclic coordinate
		contributes $0$.  In all cases $d_{\Cay} = d_G$.
		
		(ii) Every path in a tree is isometric.  Map $P$'s vertices
		$p_0, \ldots, p_{\mu-1}$ to $(0,\mathbf 0), \ldots, (\mu-1, \mathbf 0)$ in
		$\Z_{2\mu-2} \times \Z_2^{\,n-\mu}$; by
		Proposition~\ref{prop:path-in-cycle} the first coordinate is an isometric
		copy of $P$ inside $C_{2\mu-2}$.  The remaining $n - \mu$ vertices hang from
		$P$ in subtrees with unique attachments, exactly as in (i); assign fresh
		coordinates to the $n-\mu$ off-path edges and repeat the decomposition
		argument verbatim.
	\end{proof}
	
	\begin{remark}
		Part~(i) genuinely requires the unicyclic hypothesis (or, more generally, an
		\emph{isometric} cycle with hanging trees).  For an arbitrary graph merely
		\emph{containing} a chordless cycle the construction fails: a chordless cycle
		of length $\geq 6$ need not be isometric, and shortcuts outside the cycle
		break the product decomposition.  The general interaction of cyclic factors
		with arbitrary graphs is the subject of Chapter~3.
	\end{remark}
	
	\begin{proposition}[Path in cycle]
		\label{prop:path-in-cycle}
		A path $P_\mu$ on $\mu$ vertices embeds isometrically into the cycle
		$C_{2\mu-2}$.
	\end{proposition}
	
	\begin{proof}
		Label $P_\mu$'s vertices $0,1,\ldots,\mu-1$ and $C_{2\mu-2}$'s vertices
		$0,1,\ldots,2\mu-3$, and let $\lambda(i)=i$.  For $a<b$ in $P_\mu$,
		$d_{P_\mu}(a,b) = b-a \leq \mu-1$, while
		$d_{C_{2\mu-2}}(a,b) = \min\bigl(b-a,\ (2\mu-2)-(b-a)\bigr)$.  Since
		$b - a \leq \mu - 1$, we get $(2\mu-2)-(b-a) \geq \mu-1 \geq b-a$, so the
		minimum is $b-a$.
	\end{proof}
	
	\section{The $\varphi$ Relation}
	\label{sec:phi-relation}
	
	\subsection{Definition and Intuition}
	
	\begin{definition}[$\varphi$ relation]
		\label{def:phi}
		For edges $e = \{u,v\}$ and $f = \{x,y\}$ of a connected graph $G$, we say
		$e \mathrel{\varphi} f$ if
		\[
		d(u,x) = d(v,y) \quad \text{and} \quad d(u,y) = d(v,x).
		\]
	\end{definition}
	
	The relation captures \textbf{metric parallelism}: the four inter-endpoint
	distances obey the symmetry of opposite sides of a parallelogram.  Edges
	carrying the same group generator in a Cayley graph of an abelian group
	satisfy precisely this symmetry, which is why $\varphi$ is the right
	candidate test for ``same generator''.
	
	\begin{lemma}[Well-definedness]
		\label{lem:phi-well-defined}
		Definition~\ref{def:phi} does not depend on the chosen orientations of $e$
		and $f$.
	\end{lemma}
	
	\begin{proof}
		Swapping $x \leftrightarrow y$ exchanges the two required equalities with
		each other, hence preserves their conjunction; swapping
		$u \leftrightarrow v$ likewise.  Therefore all four orientation choices
		impose the same pair of conditions.
	\end{proof}
	
	\begin{remark}[$\varphi$-classes are matchings]
		\label{rem:matching}
		If $e$ and $f$ share a vertex, say $e=\{u,v\}$, $f=\{u,y\}$ with $v \neq y$,
		then $d(u,u)=0$ while $d(v,y)\geq 1$, so $e \not\mathrel\varphi f$.
		Consequently every set of pairwise $\varphi$-related edges is a
		\emph{matching} of $G$.  This is the key structural constraint exploited by
		the embedding algorithm.
	\end{remark}
	
	\subsection{Properties of $\varphi$}
	
	\begin{theorem}[Properties of $\varphi$]
		\label{thm:phi-properties}
		For any connected graph $G$:
		\begin{enumerate}[(i)]
			\item $\varphi$ is reflexive and symmetric;
			\item $\varphi$ is in general \emph{not} transitive;
			\item incident edges are never $\varphi$-related;
			\item two distinct edges lying on a common shortest path are never
			$\varphi$-related;
			\item in the even cycle $C_{2\nu}$ the maximal sets of pairwise
			$\varphi$-related edges are exactly the $\nu$ antipodal pairs;
			\item in the odd cycle $C_{2\nu+1}$ no two distinct edges are
			$\varphi$-related: all classes are singletons.
		\end{enumerate}
	\end{theorem}
	
	\begin{proof}
		(i) Reflexivity: $d(u,u) = d(v,v) = 0$ and $d(u,v) = d(v,u)$.  Symmetry is
		immediate from the symmetry of the defining conditions in $\{e, f\}$.
		
		(iii) is Remark~\ref{rem:matching}.
		
		(iv) Let $P$ be a shortest path containing both $e = \{u,v\}$ and
		$f = \{x,y\}$, oriented along $P$ so that the order of appearance is
		$u, v, \ldots, x, y$ (the incident case is covered by (iii)).  Because
		sub-paths of shortest paths are shortest,
		$d(u,x) = 1 + d(v,x)$ and $d(v,y) = d(v,x) + 1$, so the first equality
		$d(u,x) = d(v,y)$ holds; but
		$d(u,y) = d(v,x) + 2 \neq d(v,x)$, so the second fails.  By
		Lemma~\ref{lem:phi-well-defined} no other orientation can repair it.
		
		(v)--(vi) Work in $C_m$ with vertices $v_0, \ldots, v_{m-1}$ and the cyclic
		distance $\delta(t) = \min(t \bmod m,\ m - (t \bmod m))$.  Take distinct
		non-incident edges $e = \{v_a, v_{a+1}\}$ and $f = \{v_b, v_{b+1}\}$ and set
		$t = b - a \bmod m$.  With the forward orientation,
		\[
		d(v_a, v_b) = \delta(t) = d(v_{a+1}, v_{b+1}),
		\]
		so the first equality always holds, and the $\varphi$ test reduces to the
		second:
		\[
		d(v_a, v_{b+1}) = \delta(t+1) \;\overset{?}{=}\; \delta(t-1) = d(v_{a+1}, v_b).
		\]
		Now $\delta(t+1) = \delta(t-1)$ iff $t+1 \equiv \pm(t-1) \pmod m$.  The sign
		$+$ gives $2 \equiv 0$, impossible for $m \geq 3$; the sign $-$ gives
		$2t \equiv 0 \pmod m$.  If $m = 2\nu+1$ is odd this forces
		$t \equiv 0$, i.e.\ $e = f$: all classes are singletons, proving (vi).  If
		$m = 2\nu$ is even it forces $t \equiv 0$ or $t \equiv \nu$; $t \equiv \nu$
		is exactly the antipodal pair, proving (v).
		
		(ii) is witnessed by $K_{2,3}$ (Figure~\ref{fig:k23-counterexample}): with
		parts $\{u_1,u_2\}$ and $\{v_1,v_2,v_3\}$, the edges $a = u_1v_1$ and
		$b = u_2v_2$ are $\varphi$-related (all four cross distances equal $2$), as
		are $b$ and $c = u_1v_3$; but $a$ and $c$ share $u_1$, so
		$a \not\mathrel\varphi c$ by (iii).
	\end{proof}
	
	\begin{figure}[htbp]
		\centering
		\begin{tikzpicture}[scale=1.5]
			\node[circle,draw,fill=blue!20,minimum size=0.8cm] (u1) at (0,2)   {$u_1$};
			\node[circle,draw,fill=blue!20,minimum size=0.8cm] (u2) at (2.5,2) {$u_2$};
			\node[circle,draw,fill=red!20, minimum size=0.8cm] (v1) at (-0.7,0){$v_1$};
			\node[circle,draw,fill=red!20, minimum size=0.8cm] (v2) at (1.25,0){$v_2$};
			\node[circle,draw,fill=red!20, minimum size=0.8cm] (v3) at (3.2,0) {$v_3$};
			\draw[line width=2pt,blue]          (u1)--(v1) node[midway,left,font=\footnotesize]{$a$};
			\draw[line width=2pt,green!60!black](u2)--(v2) node[midway,right,font=\footnotesize]{$b$};
			\draw[line width=2pt,red]           (u1)--(v3) node[midway,right,font=\footnotesize]{$c$};
			\draw[gray,line width=1pt] (u1)--(v2) (u2)--(v1) (u2)--(v3);
		\end{tikzpicture}
		\caption{Non-transitivity of $\varphi$ in $K_{2,3}$: $a \mathrel\varphi b$
			and $b \mathrel\varphi c$, yet $a \not\mathrel\varphi c$ since $a$ and $c$
			are incident at $u_1$.  The relation captures \emph{local} parallelisms that
			may conflict globally; the transitive prune of
			Section~\ref{sec:transitive-prune} resolves the conflicts.}
		\label{fig:k23-counterexample}
	\end{figure}
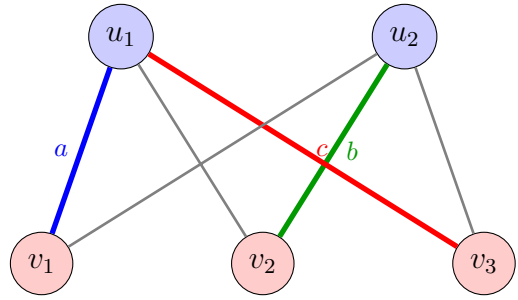
	
	\subsection{Comparison with the Djokovi\'c--Winkler Relation $\theta$}
	
	The Djokovi\'c--Winkler relation declares $e \mathrel\theta f$ when
	$d(u,x)+d(v,y) \neq d(u,y)+d(v,x)$.  The two relations measure different
	things and differ markedly on key graph classes:
	\begin{itemize}
		\item On \textbf{odd cycles}, $\theta$ is non-transitive, whereas $\varphi$
		is (vacuously) an equivalence relation by
		Theorem~\ref{thm:phi-properties}(vi).
		\item On the \textbf{Petersen graph}, the whole edge set forms a single
		$\theta$-class (the graph is not a partial cube), while $\varphi$ is an
		equivalence relation with exactly five classes of size $3$, one per
		``parallel'' perfect matching (Figure~\ref{fig:coloredlabelledequiclassespetersen}).
		This structure drives the compact embedding into a host of order $16$
		computed in Example~\ref{ex:petersen-quotient}.
	\end{itemize}
	
	\begin{figure}[htbp]
		\centering
		\includegraphics[width=0.7\linewidth]{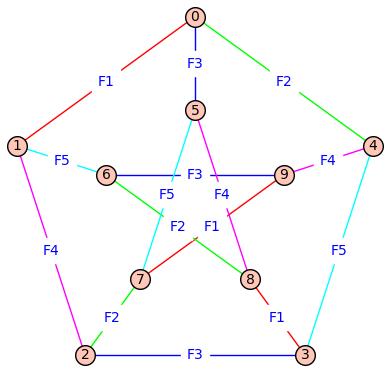}
		\caption{The five $\varphi$-equivalence classes $F_1,\ldots,F_5$ of the
			Petersen graph, each of cardinality $3$.  Since $\varphi$ is already an
			equivalence relation here, no transitive prune is needed; the five classes
			directly determine the embedding into $\Z_2^4$ via the Quotient Labeling
			Theorem.}
		\label{fig:coloredlabelledequiclassespetersen}
	\end{figure}
	
	\subsection{$\varphi$ on Partial Cubes}
	
	\begin{theorem}[$\varphi$ on partial cubes]
		\label{thm:phi-partial-cubes}
		\leavevmode
		\begin{enumerate}[(a)]
			\item If $G$ is a partial cube, then $\varphi$ coincides with $\theta$, is an
			equivalence relation on $E(G)$, and its classes are exactly the
			Djokovi\'c cuts (the coordinate classes of any isometric hypercube
			embedding).
			\item Conversely, if $G$ is bipartite, $\varphi$ is an equivalence relation,
			and for every class the two sides of the induced edge cut are convex
			subgraphs, then $G$ is a partial cube and the $\varphi$-classes are its
			$\theta$-classes.
		\end{enumerate}
	\end{theorem}
	
	\begin{proof}
		(a) Fix an isometric embedding $\lambda\colon V(G) \to Q_p$, so that
		$d_G = $ Hamming distance on labels, and each edge flips exactly one
		coordinate.  Let $e = \{u,v\}$ flip coordinate $i$ and $f = \{x,y\}$ flip
		coordinate $j$, and put $z = \lambda(u) \oplus \lambda(x)$.  Then
		\[
		d(u,x) = \mathrm{wt}(z),\quad
		d(v,y) = \mathrm{wt}(z \oplus e_i \oplus e_j),\quad
		d(u,y) = \mathrm{wt}(z \oplus e_j),\quad
		d(v,x) = \mathrm{wt}(z \oplus e_i).
		\]
		If $i = j$: the first two agree and the last two agree, for every $z$; with
		the side-respecting orientation both $\varphi$ equalities hold, so all edges
		of one cut are pairwise $\varphi$-related.  If $i \neq j$: the first
		equality $\mathrm{wt}(z) = \mathrm{wt}(z \oplus e_i \oplus e_j)$ holds iff
		$z_i \neq z_j$, while the second
		$\mathrm{wt}(z \oplus e_j) = \mathrm{wt}(z \oplus e_i)$ holds iff
		$z_i = z_j$; they are mutually exclusive, so edges of distinct cuts are
		never $\varphi$-related.  Hence the $\varphi$-classes are exactly the
		coordinate cuts; these are the $\theta$-classes of a partial cube
		\cite{Imrich2010}, and transitivity follows.
		
		(b) Under the stated hypotheses each class is an edge cut with convex sides,
		so the side-indicator labeling satisfies Djokovi\'c's convexity
		characterization of partial cubes \cite{Djokovic1973}: a connected bipartite
		graph in which, for every edge $uv$, the half-spaces
		$W(u,v) = \{w : d(w,u) < d(w,v)\}$ and $W(v,u)$ are convex, is a partial
		cube, and its coordinate cuts are the $\theta$-classes.  The
		$\varphi$-classes are then identified with them by part~(a).
	\end{proof}
	
	\begin{remark}
		Part (a) does \emph{not} say that the isometric dimension is the minimum
		dimension over our larger target family $\Cay(\Z_2^k, S)$: composite
		generators can do strictly better even on partial cubes.
		Theorem~\ref{thm:star} exhibits trees --- the simplest partial cubes --- whose
		minimum is exponentially below $\mathrm{idim}$.
	\end{remark}
	
	\section{The Transitive Prune Operation}
	\label{sec:transitive-prune}
	
	Since $\varphi$ is not transitive on arbitrary graphs, we extract from it a
	partition of $E(G)$ into classes that are pairwise $\varphi$-related within
	each class.
	
	\subsection{The $\varphi$-Graph and Clique Extraction}
	
	\begin{definition}[$\varphi$-graph]
		\label{def:phi-graph}
		The \textbf{$\varphi$-graph} $\varphi(G)$ has vertex set $E(G)$, with two
		$G$-edges adjacent iff they are distinct and $\varphi$-related.
	\end{definition}
	
	By Theorem~\ref{thm:phi-partial-cubes}, if $G$ is a partial cube every
	connected component of $\varphi(G)$ is a clique.  In general, components may
	fail to be cliques; they encode chains of local parallelisms that conflict
	globally and must be resolved.
	
	\begin{definition}[$\varphi^-$-classes; transitive prune]
		\label{def:transitive-prune}
		The \textbf{transitive prune} produces a partition $\mathcal{P}$ of $E(G)$
		into \emph{$\varphi^-$-classes} by greedily extracting cliques from
		$\varphi(G)$:
		\begin{enumerate}
			\item process the connected components of $\varphi(G)$ in decreasing order
			of size;
			\item a component that is a clique is taken whole as a class;
			\item otherwise a maximum clique of the component is extracted as a class
			(exact search for small components, greedy approximation for large ones),
			and the procedure recurses on the remainder;
			\item $G$-edges with no $\varphi$-partner become singleton classes.
		\end{enumerate}
	\end{definition}
	
	\begin{theorem}[Correctness of the transitive prune]
		\label{thm:transitive-prune}
		The transitive prune terminates and outputs a partition $\mathcal{P}$ of
		$E(G)$ such that: (i) each class is a clique of $\varphi(G)$; (ii) each
		class is a matching of $G$; (iii) if $G$ is a partial cube,
		$\mathcal{P}$ equals the set of $\theta$-classes; (iv)
		$|\mathcal{P}| \leq m$.
	\end{theorem}
	
	\begin{proof}
		(i) holds by construction.  (ii) follows from (i) and
		Remark~\ref{rem:matching}.  (iii) if $G$ is a partial cube, every component
		of $\varphi(G)$ is a clique by Theorem~\ref{thm:phi-partial-cubes}(a), so
		step~3 never executes and the components --- the $\theta$-classes --- are
		returned unchanged.  (iv) every extraction removes at least one vertex of
		$\varphi(G)$, which has $m$ vertices; this also gives termination.
	\end{proof}
	
	\begin{algorithm}
		\caption{Transitive Prune (clique extraction on the $\varphi$-graph)}
		\label{alg:transitive-prune}
		\begin{algorithmic}[1]
			\Require connected graph $G$, all-pairs distance table $D$
			\Ensure partition $\mathcal{P} = \{F_1,\ldots,F_t\}$ of $E(G)$
			\State construct $\varphi(G)$ from $D$
			\Comment{$O(m^2)$ pairs, $O(1)$ test each}
			\State $\mathcal{P} \gets \emptyset$;\quad $H \gets \varphi(G)$
			\While{$H$ has vertices}
			\State $C \gets$ largest connected component of $H$
			\If{$C$ is a clique} \State add $C$ to $\mathcal{P}$; remove $C$ from $H$
			\Else
			\State $K \gets$ maximum clique of $H[C]$ (exact if $|C|$ small,
			greedy otherwise)
			\State add $K$ to $\mathcal{P}$; remove $K$ from $H$
			\EndIf
			\EndWhile
			\State \Return $\mathcal{P}$
		\end{algorithmic}
	\end{algorithm}
	
	\begin{remark}[On the clique subproblem]
		Maximum clique is NP-hard in general; here it is invoked on components of
		$\varphi(G)$, which in practice are small or already cliques.  When the
		greedy approximation is used, only the \emph{quality} of the initial
		partition is affected, never the correctness of the final embedding, because
		the quotient construction and repair loop of
		Sections~\ref{sec:cocycle}--\ref{sec:phi-algorithm} accept an arbitrary
		partition.
	\end{remark}
	
	\section{The Hypercube Skeleton and Cut-Based Embedding}
	\label{sec:binary-algorithm-old}
	
	This section records the cut-based embedding, which is exact for partial
	cubes and historically the first of our pipelines.  Its limitation --- and
	the remedy --- motivate Section~\ref{sec:cocycle}.
	
	\begin{definition}[Hypercube skeleton]
		The \textbf{hypercube skeleton} $H(G)$ is the spanning subgraph of $G$
		obtained by greedily adding $\varphi^-$-classes (largest first), retaining a
		class only if (i) adding it strictly decreases the number of connected
		components and (ii) its removal disconnects the current skeleton into
		exactly two parts (a proper $2$-cut).  The process stops when $H(G)$ is
		connected; retained classes are \emph{skeleton classes}.
	\end{definition}
	
	\begin{theorem}[Cut-based embedding; partial cubes]
		\label{thm:cut-based-partial-cube}
		Let $G$ be a partial cube with $\varphi^-$-classes
		$\mathcal{P}$, $p = |\mathcal{P}|$.  The skeleton construction retains all
		$p$ classes, and the side-indicator labeling
		$\lambda(v)_i = [\,v \in C_i^1\,]$ is an isometric embedding of $G$ into the
		hypercube $Q_p$.  Moreover $p = \mathrm{idim}(G)$: among embeddings into
		\emph{hypercubes} (standard-basis targets), the dimension $p$ is minimal.
	\end{theorem}
	
	\begin{remark}[Scope: basis generators are special to this theorem]
		\label{rem:basis-only-partial-cube}
		The conclusion that each class maps to a \emph{standard basis vector}
		$e_i$ --- equivalently, that $G$ embeds into a Cartesian product
		$Q_p = K_2^{\,\square p}$ --- holds \emph{only because $G$ is a partial cube},
		where every $\varphi^-$-class is a genuine edge cut and may therefore be
		assigned an independent coordinate.  It is \textbf{not} a property of the
		general construction.  For an arbitrary graph the consistency condition
		$(\star)$ of Theorem~\ref{thm:cocycle} forces linear dependencies among the
		class generators, so the Quotient Labeling Theorem
		(Theorem~\ref{thm:quotient-labeling}) generally assigns \emph{composite}
		generators $g_j \in \Z_2^k$ that are sums of several basis vectors, and the
		host is a Cayley graph $\Cay(\Z_2^k, S)$ with $S$ \emph{not} a basis --- not
		a Cartesian product of edges.  The Petersen graph is the smallest vivid
		example: four classes receive basis vectors but the fifth receives the
		composite generator $(1,1,1,1)$ (Figure~\ref{fig:coloredlabelledskeletonpetersen}),
		and the star $K_{1,4}$ embeds with the composite generator
		$e_1{+}e_2{+}e_3$ into $\Z_2^3$ rather than the Cartesian $Q_4$
		(Theorem~\ref{thm:star}).  Thus ``basis generators / Cartesian product''
		is the partial-cube special case; ``composite generators / general Cayley
		graph'' is the rule.
	\end{remark}
	
	\begin{proof}
		By Theorem~\ref{thm:phi-partial-cubes}(a) the classes are the
		$\theta$-classes.  In a partial cube, the distance between two vertices
		equals the number of $\theta$-classes separating them
		\cite{Imrich2010}, which is precisely the Hamming distance of the
		side-indicator labels, so the labeling is isometric into $Q_p$.  Minimality
		over hypercube targets is the classical statement
		$p = \mathrm{idim}(G)$ \cite{Graham1985,Ovchinnikov2008}: the coordinate
		cuts of any isometric embedding into $Q_q$ partition $E(G)$ into at least
		$p$ nonempty $\theta$-classes, whence $q \geq p$.
	\end{proof}
	
	\begin{remark}[Two distinct limitations]
		\label{rem:cut-limits}
		First, the cut paradigm \emph{fails to apply} beyond partial cubes: on
		$K_5$, $K_{3,3}$ or the Pappus graph the skeleton classes are not cuts, the
		construction retains nothing, and the implementation must fall back to the
		naive embedding ($k = n-1$).  Second --- more subtly --- even where it
		applies, its minimality holds only against \emph{hypercube} targets.  In
		the ambient problem of this thesis, where $S$ is arbitrary, the
		minimum can be strictly smaller than $\mathrm{idim}(G)$ even for trees:
		Theorem~\ref{thm:star} shows $k_{\min}(K_{1,4}) = 3 < 4 =
		\mathrm{idim}(K_{1,4})$, realized with the composite generator
		$e_1{+}e_2{+}e_3$ (Figure~\ref{fig:starK14}).  Both limitations are removed
		by the quotient framework of the next section together with the bounds
		theory of Section~\ref{sec:bounds}.
	\end{remark}
	
	\begin{figure}[htbp]
		\centering
		\includegraphics[width=0.7\linewidth]{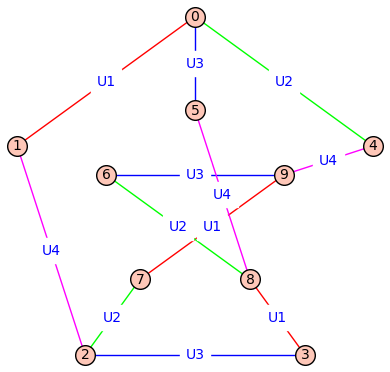}
		\caption{The hypercube skeleton of the Petersen graph: four of the five
			$\varphi$-classes (distinct colors) are retained as proper $2$-cuts of the
			spanning subgraph, yielding the four basis coordinates; the fifth class
			becomes the composite generator $(1,1,1,1)$.}
		\label{fig:coloredlabelledskeletonpetersen}
	\end{figure}
	\section{The Cocycle Condition and Quotient Labeling}
	\label{sec:cocycle}
	
	This section is the theoretical core of the chapter.  It explains
	\emph{why} naive label propagation can fail, states the exact consistency
	condition, and shows that the most generic consistent labeling is a GF(2)
	quotient computable by linear algebra.
	
	\subsection{The Labeling Problem}
	
	Given a partition $\mathcal{P} = \{F_1,\ldots,F_t\}$ of $E(G)$ (``edges in
	the same class carry the same generator''), we wish to assign generators
	$g_j \in \Z_2^k$ to the classes so that the labeling
	\[
	\lambda(v) \;=\; \sum_{e \,\in\, P(r,v)} g_{\,\mathrm{class}(e)}
	\qquad (\text{sum in } \Z_2^k)
	\]
	is \emph{well defined}, i.e.\ independent of the path $P(r,v)$ chosen from
	the root $r$ to $v$.
	
	\begin{theorem}[Cocycle condition]
		\label{thm:cocycle}
		The labeling $\lambda$ is well defined if and only if for every cycle
		$C = (e_1, \ldots, e_\ell)$ of $G$
		\[
		\sum_{i=1}^{\ell} g_{\,\mathrm{class}(e_i)} \;=\; \mathbf 0 \in \Z_2^k.
		\tag{$\star$}
		\]
		Moreover it suffices to verify $(\star)$ on any cycle basis of $G$.
	\end{theorem}
	
	\begin{proof}
		Two $r$--$v$ paths $P_1, P_2$ give equal sums iff the sum over their
		symmetric difference vanishes; that symmetric difference is an edge-disjoint
		union of cycles, so $(\star)$ for all cycles implies path independence, and
		conversely a violating cycle $C$ through a vertex $v$ yields two $r$--$v$
		paths with different sums (follow a path to $C$, then go around $C$ either
		way).  Since the cycle space of $G$ is spanned over $\F_2$ by any cycle
		basis and $(\star)$ is linear in the cycle, checking a basis suffices.
	\end{proof}
	
	\begin{remark}[Why the cut paradigm works on partial cubes]
		If a class $F_i$ is an edge cut, every cycle crosses it an even number of
		times (a closed walk re-enters each side as often as it leaves), so $F_i$
		contributes $0$ to every instance of $(\star)$ regardless of the generator
		assigned to it.  When classes are not cuts, some cycle crosses a class an
		odd number of times, and naive BFS-XOR propagation becomes
		order-dependent --- the root cause of failure of earlier implementations.
	\end{remark}
	
	\subsection{The GF(2) Quotient Construction}
	
	Let $c = m - n + 1$ be the cycle rank of $G$ and fix a cycle basis
	$B_1, \ldots, B_c$.  Define the \textbf{cycle--class parity matrix}
	$A \in \F_2^{\,c \times t}$ by
	\[
	A[i,j] \;=\; \bigl|\, F_j \cap E(B_i) \,\bigr| \bmod 2 .
	\]
	Condition $(\star)$ on the basis reads $A \cdot (g_1 | \cdots | g_t)^{\!\top}
	= 0$ over $\F_2$.
	
	\begin{theorem}[Quotient Labeling Theorem]
		\label{thm:quotient-labeling}
		Let $\mathcal{P}$ be any partition of $E(G)$ into $t$ classes and
		$\rho = \mathrm{rank}_{\F_2}(A)$.  Then:
		\begin{enumerate}[(i)]
			\item the most generic consistent assignment has dimension
			$k = t - \rho$: it is given by $g_j = \pi(e_j)$, the images of the
			standard basis of $\Z_2^{\,t}$ under the quotient map
			$\pi\colon \Z_2^{\,t} \to \Z_2^{\,t}/\mathrm{rowspan}(A) \cong \Z_2^{\,k}$,
			and every other consistent assignment is a linear image of it;
			\item the BFS labeling under this assignment is conflict-free by
			construction;
			\item with $S = \{g_j : g_j \neq \mathbf 0\}$, every $G$-path of length
			$\ell$ maps to a Cayley walk of length $\ell$; hence
			\[
			d_{\Cay}\bigl(\lambda(u), \lambda(v)\bigr) \;\leq\; d_G(u,v)
			\qquad \forall\, u, v ,
			\]
			so the only possible failure of isometry is a \emph{shortcut}, never a
			stretch.
		\end{enumerate}
	\end{theorem}
	
	\begin{proof}
		(i) The solution set of $A \mathbf{x} = 0$ is a subspace of dimension
		$t - \rho$; choosing coordinates on it realizes each $g_j$ as the image of
		$e_j$ in the quotient.  Any consistent assignment
		$(g'_1, \ldots, g'_t)$ in any $\Z_2^{k'}$ defines a linear map
		$\Z_2^{\,t} \to \Z_2^{k'}$ vanishing on $\mathrm{rowspan}(A)$, which
		therefore factors through $\pi$; in this sense $\pi$ is universal
		(``most generic'').
		(ii) By Theorem~\ref{thm:cocycle}, since the images satisfy $(\star)$ for
		every basis cycle, hence for every cycle.
		(iii) Consecutive labels along a path differ by the class generator of the
		traversed edge, an element of $S \cup \{\mathbf 0\}$; zero generators only
		shorten the walk.  Minimizing over $u$--$v$ paths gives the inequality.
	\end{proof}
	
	The next lemma supplies the missing converse direction for the finest
	partition; it completes the proof of Theorem~\ref{thm:naive} and anchors
	the universality theorem.
	
	\begin{lemma}[Join Lemma]
		\label{lem:join}
		Let $\mathcal{P}_0$ be the all-singleton partition (one class per edge) and
		let $\lambda, S$ be its quotient labeling.  Identify $\Z_2^{\,m}$ with the
		edge space $\mathcal{E}(G)$ and the row span of $A$ with the cycle space
		$\mathcal{Z}(G)$.  Then for all $u, v$:
		\[
		d_{\Cay}\bigl(\lambda(u), \lambda(v)\bigr) \;=\; d_G(u,v),
		\]
		i.e.\ the quotient embedding of the finest partition is isometric.
	\end{lemma}
	
	\begin{proof}
		By Theorem~\ref{thm:quotient-labeling}(iii) only
		$d_{\Cay} \geq d_G$ needs proof.  For the finest partition,
		$\lambda(u) \oplus \lambda(v) = \pi(\chi_P)$ where $\chi_P \in \mathcal{E}(G)$
		is the characteristic vector of any $u$--$v$ path $P$ and
		$\pi\colon \mathcal{E}(G) \to \mathcal{E}(G)/\mathcal{Z}(G)$ is the quotient
		map.  Consider a geodesic word realizing
		$d_{\Cay}(\lambda(u), \lambda(v)) = \ell$.  By Remark~\ref{rem:no-repeat} it
		uses pairwise distinct generators, i.e.\ it is $\pi(\chi_F)$ for some edge
		set $F \subseteq E(G)$ with $|F| = \ell$.  Then
		$\pi(\chi_F) = \pi(\chi_P)$ means
		$\chi_F \oplus \chi_P \in \mathcal{Z}(G)$, and since cycles have empty GF(2)
		boundary, the boundary operator gives
		$\partial F = \partial P = \{u, v\}$: in the subgraph $(V, F)$ exactly the
		vertices $u$ and $v$ have odd degree.  The component of $(V, F)$ containing
		$u$ has an even number of odd-degree vertices, so it also contains $v$;
		hence $F$ contains a $u$--$v$ path and
		$\ell = |F| \geq d_G(u, v)$.
	\end{proof}
	
	\begin{corollary}[The naive embedding is the finest quotient]
		\label{cor:naive-is-quotient}
		For $\mathcal{P}_0$, the matrix $A$ is the (cycle basis)\,$\times$\,(edge)
		incidence matrix, of rank $c = m - n + 1$, so
		\[
		k \;=\; m - (m - n + 1) \;=\; n - 1 ,
		\]
		the free coordinates may be taken on the edges of any spanning tree, and the
		resulting labeling is exactly the spanning-tree labeling of
		Theorem~\ref{thm:naive}; by Lemma~\ref{lem:join} it is isometric.  In
		particular the upper bound $k \leq n-1$ holds for \emph{every} connected
		graph, as a one-line consequence of the quotient formula.
	\end{corollary}
	
	\begin{proof}
		For singleton classes, $A[i, j] = 1$ iff edge $j$ lies on basis cycle $i$;
		this matrix has rank $c$ because the fundamental cycles of any spanning tree
		are supported on distinct non-tree edges.  Pivot columns can be chosen at
		the non-tree edges, leaving the $n-1$ tree edges free; the image of a
		non-tree edge is then the sum of the tree edges on its fundamental cycle,
		which is precisely the composite generator the BFS-tree labeling assigns to
		it.
	\end{proof}
	
	\begin{example}[Petersen graph via the quotient]
		\label{ex:petersen-quotient}
		The Petersen graph has $n = 10$, $m = 15$, cycle rank $c = 6$, and five
		$\varphi^-$-classes of size $3$ ($t = 5$).  Direct computation of the
		$6 \times 5$ cycle--class matrix gives $\rho = 1$, the unique relation being
		$g_1 + g_2 + g_3 + g_4 + g_5 = \mathbf 0$.  Hence $k = 5 - 1 = 4$, with
		generators $e_1, e_2, e_3, e_4$ and the composite
		$e_1{+}e_2{+}e_3{+}e_4 = (1,1,1,1)$: the host is the Clebsch graph
		(Figure~\ref{fig:petersenclebsch}), of order $16$, and the embedding is
		verified isometric on all $\binom{10}{2}$ pairs.  Section~\ref{sec:bounds}
		shows $k = 4 = \lceil \log_2 10 \rceil$ is \emph{optimal}.
	\end{example}
	
	\begin{figure}[htbp]
		\centering
		\includegraphics[width=0.72\linewidth]{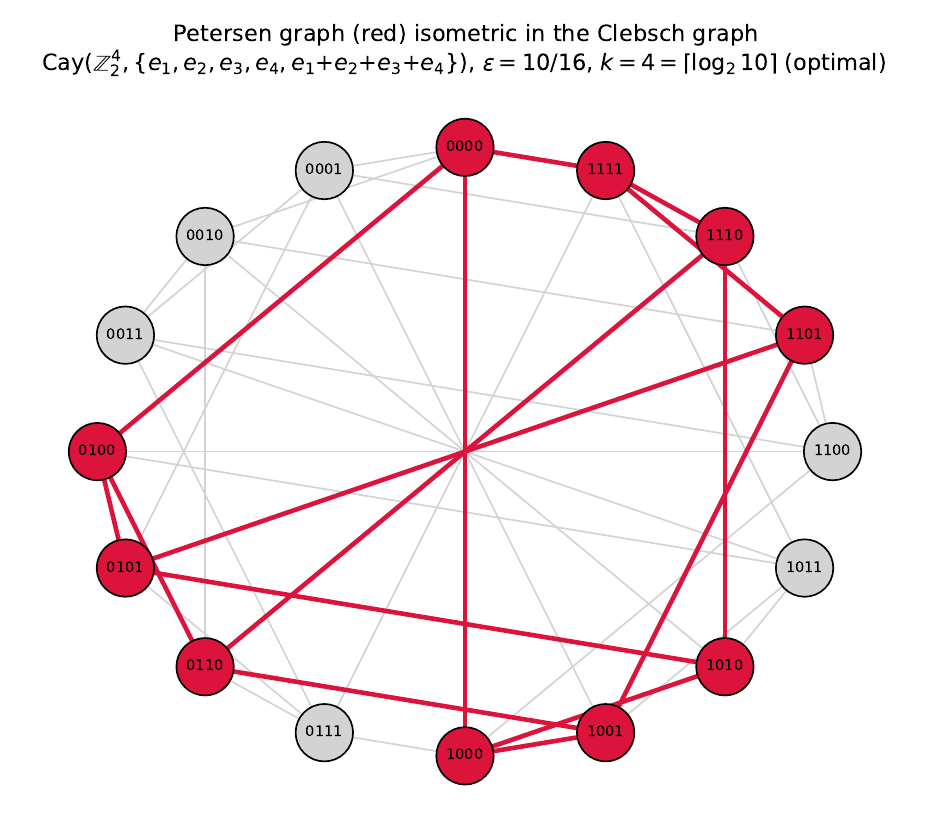}
		\caption{The Petersen graph (red) as an isometric subgraph of the Clebsch
			graph $\Cay(\Z_2^4, \{e_1, e_2, e_3, e_4, e_1{+}e_2{+}e_3{+}e_4\})$ of
			order $16$.  Excursion ratio $\varepsilon = 10/16 = 0.625$; the dimension
			$k = 4 = \lceil\log_2 10\rceil$ matches the injectivity lower bound of
			Theorem~\ref{thm:lower-bound} and is therefore minimal.}
		\label{fig:petersenclebsch}
	\end{figure}
	
	\begin{example}[Pappus graph via the quotient]
		\label{ex:pappus-quotient}
		The Pappus graph ($n = 18$, $m = 27$, $c = 10$) admits a structured
		partition into nine $\varphi^-$-classes of size $3$
		(Figure~\ref{fig:pappusclasses}), found by the top-down constructor of the
		supplementary implementation.  The $10 \times 9$ cycle--class matrix has rank
		$\rho = 2$, giving $k = 9 - 2 = 7$ and $|S| = 9$: seven basis generators
		and two composite generators, \emph{both of Hamming weight $5$}.  The host
		has order $128$, against $131\,072$ for the naive embedding --- a
		$1024\times$ improvement --- and isometry was verified on all
		$\binom{18}{2} = 153$ pairs by full reconstruction of the host.  This
		refutes the earlier conclusion that Pappus admits no compact
		$\varphi$-based binary embedding: it does; the \emph{cut-based labeling}
		was simply the wrong tool.  By contrast, the greedy transitive prune
		produces a ten-class partition from which the repair loop converges to
		$k = 12$ (Table~\ref{tab:embedding-all}); the comparison quantifies how
		strongly the initial partition quality controls the final dimension.
	\end{example}
	
	\begin{figure}[htbp]
		\centering
		\includegraphics[width=0.8\linewidth]{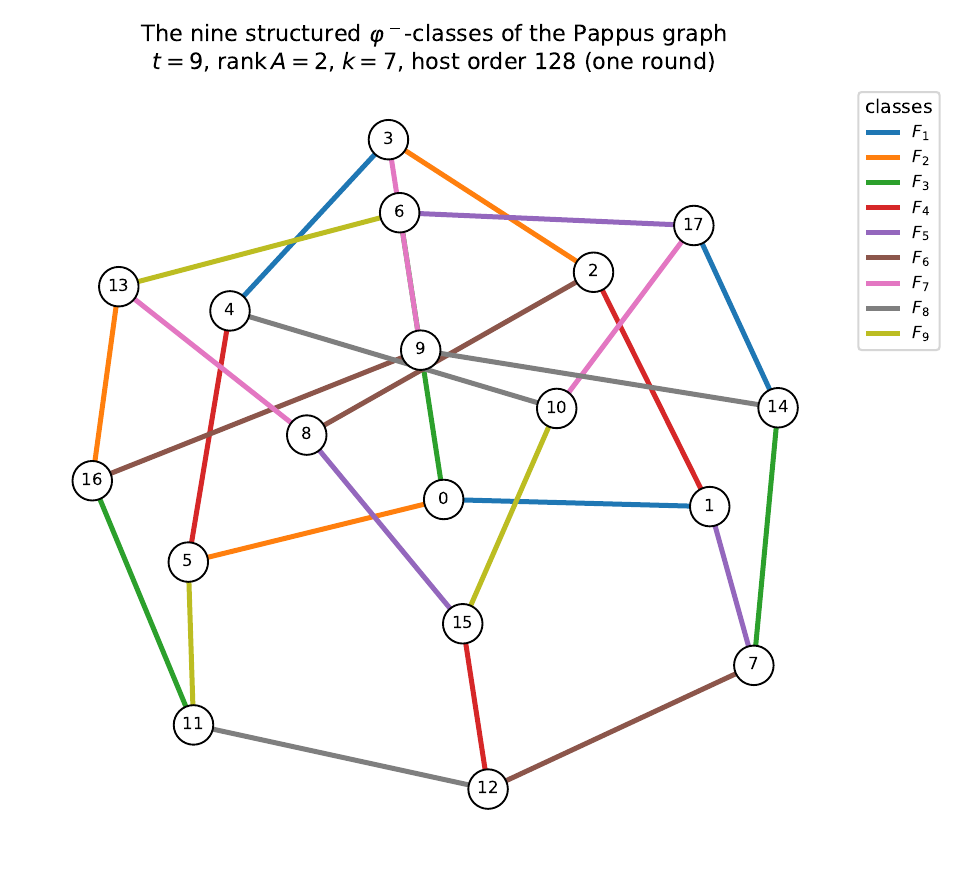}
		\caption{The nine structured $\varphi^-$-classes of the Pappus graph
			($t = 9$, three edges per class).  The cycle--class parity matrix has rank
			$2$, so the quotient embedding has $k = 7$: host order $128$, verified
			isometric, with both composite generators of weight $5$.}
		\label{fig:pappusclasses}
	\end{figure}
	
	\section{The $\varphi$-Quotient Algorithm}
	\label{sec:phi-algorithm}
	
	\subsection{Isometry Check and Shortcut Characterization}
	
	\begin{definition}[Correct isometry check]
		\label{def:correct-check}
		An embedding $\lambda\colon V(G) \to \Z_2^k$ with generator set $S$ is
		\emph{isometric} iff
		$d_G(u,v) = d_{\Cay}(\lambda(u),\lambda(v))$ for all pairs $u, v \in V(G)$.
		This test must be performed against $d_G$ on $V(G)$.  Testing instead
		against the metric of the \emph{image graph} is unsound: if the labeling is
		non-injective the image graph silently merges vertices and its metric is no
		longer $d_G$, so the flawed test can certify non-embeddings.  (An explicit
		injectivity assertion accompanies the check in the reference
		implementation.)
	\end{definition}
	
	\begin{proposition}[Shortcuts are generator subsets]
		\label{prop:shortcut}
		A violation at $(u,v)$ means there is a subset
		$\{s_1, \ldots, s_\ell\} \subseteq S$ with
		$s_1 + \cdots + s_\ell = \lambda(u) \oplus \lambda(v)$ and
		$\ell < d_G(u,v)$ (Remark~\ref{rem:no-repeat}).  The classes whose
		generators occur in the BFS-geodesic word of
		$\lambda(u) \oplus \lambda(v)$ are said to be \emph{responsible} for the
		shortcut.
	\end{proposition}
	
	\subsection{The Repair Loop}
	
	\begin{algorithm}
		\caption{$\varphi$-Quotient Embedding with Repair Loop}
		\label{alg:phi-quotient}
		\begin{algorithmic}[1]
			\Require connected graph $G = (V, E)$; optional dimension cap $k_{\max}$
			\Ensure isometric embedding $\lambda\colon V \to \Z_2^k$, generator set $S$
			\State compute all-pairs distances $D$ of $G$;\quad
			$\mathcal{P} \gets \mathrm{TransitivePrune}(G, D)$
			\Comment{Alg.~\ref{alg:transitive-prune}}
			\While{\textbf{true}}
			\State $(k, \lambda, \{g_j\}) \gets \mathrm{QuotientLabeling}(G, \mathcal{P})$
			\Comment{Thm.~\ref{thm:quotient-labeling}}
			\State $S \gets \{\, g_j : F_j \in \mathcal{P},\ g_j \neq \mathbf 0 \,\}$
			\Comment{$=\{\lambda(u)\oplus\lambda(v) : uv \in E\}$}
			\If{$k > k_{\max}$}
			\State \Return the naive embedding of Cor.~\ref{cor:naive-is-quotient}
			\Comment{provably isometric; $k = n-1$}
			\EndIf
			\State single-source BFS from $\mathbf 0$ in $\Cay(\Z_2^k, S)$, truncated
			at depth $\diam(G)$ \Comment{vertex-transitivity:
				$d_{\Cay}(x,y) = d_{\Cay}(\mathbf 0, x \oplus y)$}
			\State $\mathcal{V} \gets \{(u,v) : d_{\Cay}(\lambda(u),\lambda(v)) \neq
			D[u][v]\}$
			\If{$\mathcal{V} = \emptyset$} \State \Return $(\lambda, k, S)$ \EndIf
			\State $(u^*, v^*) \gets \arg\max_{(u,v) \in \mathcal{V}}
			\bigl(D[u][v] - d_{\Cay}\bigr)$;\quad recover the BFS geodesic word of
			$\lambda(u^*) \oplus \lambda(v^*)$
			\State $F_j \gets$ a largest responsible class with $|F_j| \geq 2$
			(if none exists, a largest class with $|F_j| \geq 2$)
			\State \textbf{peel:} move one edge of $F_j$ into a new singleton class
			\EndWhile
		\end{algorithmic}
	\end{algorithm}
	
	\begin{theorem}[Universality]
		\label{thm:universal}
		For every connected graph $G$, Algorithm~\ref{alg:phi-quotient} (with
		$k_{\max} = n-1$) terminates and returns an isometric embedding of $G$ into
		$\Cay(\Z_2^k, S)$ with $k \leq n - 1$.
	\end{theorem}
	
	\begin{proof}
		\emph{Termination.}  Each iteration either returns or strictly refines
		$\mathcal{P}$ by one peel.  Starting from $t$ classes, at most $m - t$
		peels are possible before reaching the all-singleton partition
		$\mathcal{P}_0$, so the loop executes at most $m - t + 1$ rounds.
		
		\emph{Correctness.}  Whenever the algorithm returns at the empty-violation
		line, the returned labeling has passed the correct isometry check of
		Definition~\ref{def:correct-check}.  If the loop ever reaches
		$\mathcal{P}_0$, its quotient labeling is the naive embedding, which is
		isometric by Lemma~\ref{lem:join}; the check then finds no violations and
		the algorithm returns.  The same lemma covers the dimension-cap return.
		In all cases $k \leq n-1$, since $k = t - \rho$ is maximized at
		$\mathcal{P}_0$ where it equals $n - 1$
		(Corollary~\ref{cor:naive-is-quotient}), and any partition refines towards
		it.
		
		We note explicitly what is \emph{not} claimed: a single peel need not
		reduce the number of violations, and $k$ may stay constant across a peel
		(the new column can raise $\mathrm{rank}(A)$ by one).  The guarantee is the
		terminal one above, and it is unconditional.
	\end{proof}
	
	\begin{remark}[Minimality is heuristic --- and provably must be]
		\label{rem:minimality}
		The returned $k$ depends on the initial partition and on peel choices, so
		the algorithm is a heuristic for the minimum dimension.  Two precise
		statements calibrate this. On the negative side, no algorithm confined to
		the $\varphi$-quotient family can be optimal on all graphs: the star
		$K_{1,4}$ has only singleton $\varphi$-classes (Remark~\ref{rem:matching}
		forbids two incident edges in a class), forcing $k = 4$ within the family,
		while Theorem~\ref{thm:star} shows $k_{\min}(K_{1,4}) = 3$, attained only
		by assigning a \emph{linearly dependent} generator to a class that is
		$\varphi$-related to none of the others.  On the positive side, exhaustive
		exact search shows the algorithm attains the true minimum on $29$ of the
		$30$ connected graphs with $n \leq 5$, the unique exception being precisely
		$K_{1,4}$ (Section~\ref{sec:binary-experiments},
		Figure~\ref{fig:optimalitygap}).  Exact minimization over all partitions is
		conjectured NP-hard, by analogy with minimum-dimension scale embeddings
		into hypercubes \cite{Deza1997}; establishing this is left open.
	\end{remark}
	
	\section{Bounds on the Embedding Dimension}
	\label{sec:bounds}
	
	For a connected graph $G$ define
	\[
	k_{\min}(G) \;=\; \min\bigl\{\, k \;:\; G \text{ embeds isometrically into }
	\Cay(\Z_2^k, S) \text{ for some } S \,\bigr\}.
	\]
	Corollary~\ref{cor:naive-is-quotient} gives $k_{\min}(G) \leq n-1$
	universally.  This section establishes the general lower bound, determines
	$k_{\min}$ exactly for stars and (in the stated range) for odd cycles, and
	shows that all the bounds are tight.
	
	\subsection{The General Lower Bound}
	
	\begin{lemma}[Geodesic independence]
		\label{lem:geodesic-independence}
		Let $s_1, \ldots, s_d \in S$ be the generators of a geodesic word in
		$\Cay(\Z_2^k, S)$, i.e.\ $d_{\Cay}(\mathbf 0,\ s_1 + \cdots + s_d) = d$.
		Then $s_1, \ldots, s_d$ are linearly independent over $\F_2$; in
		particular all $2^d$ subset sums are distinct.
	\end{lemma}
	
	\begin{proof}
		In an abelian group, every sub-multiset of a geodesic word is itself
		geodesic: if some subset $T \subseteq \{1,\ldots,d\}$ admitted a shorter
		word for $\sum_{i \in T} s_i$, splicing it into the original word would
		shorten it.  Hence $d_{\Cay}(\mathbf 0, \sum_{i \in T} s_i) = |T|$ for every
		$T$.  If two distinct subsets $T_1 \neq T_2$ had equal sums, then
		$\sum_{i \in T_1 \triangle T_2} s_i = \mathbf 0$ while
		$|T_1 \triangle T_2| > 0$, contradicting the previous sentence.  Distinct
		subset sums for all $2^d$ subsets is exactly linear independence.
	\end{proof}
	
	\begin{theorem}[Lower bound]
		\label{thm:lower-bound}
		For every connected graph $G$ on $n$ vertices,
		\[
		k_{\min}(G) \;\geq\; \max\bigl(\, \diam(G),\ \lceil \log_2 n \rceil \,\bigr).
		\]
	\end{theorem}
	
	\begin{proof}
		Isometric maps are injective, so $2^k \geq n$, giving the second term.  For
		the first, pick $u, v$ with $d_G(u,v) = \diam(G) =: d$; isometry provides a
		geodesic word of length $d$ in $\Cay(\Z_2^k, S)$, whose generators are
		linearly independent by Lemma~\ref{lem:geodesic-independence}; a space
		containing $d$ independent vectors has dimension $\geq d$.
	\end{proof}
	
	\begin{proposition}[Tightness of each term]
		\label{prop:tightness}
		\leavevmode
		\begin{enumerate}[(i)]
			\item Hypercubes attain both terms simultaneously:
			$k_{\min}(Q_t) = t = \diam(Q_t) = \log_2 |V(Q_t)|$.
			\item Complete graphs of order $2^t$ attain the logarithmic term at
			diameter $1$:\ $k_{\min}(K_{2^t}) = t$, realized by
			$K_{2^t} = \Cay(\Z_2^t,\ \Z_2^t \setminus \{\mathbf 0\})$.
			\item Even cycles attain the diameter term:
			$k_{\min}(C_{2d}) = d = \diam(C_{2d})$, realized by the antipodal-pair
			classes (Theorem~\ref{thm:phi-properties}(v)), which give $t = d$,
			$\rho = 0$, $k = d$.
		\end{enumerate}
	\end{proposition}
	
	\begin{proof}
		(i) $Q_t$ embeds into itself; the lower bound gives
		$k \geq \diam = t$.  (ii) the displayed Cayley graph \emph{is} $K_{2^t}$
		(all pairs adjacent), so the identity embeds it; the lower bound gives
		$k \geq \lceil \log_2 2^t \rceil = t$.  (iii) the $d$ antipodal classes of
		$C_{2d}$ are cuts, every cycle crosses each an even number of times, so
		$A = 0$, $k = d$; the labeling is the classical isometric embedding of
		$C_{2d}$ into $Q_d$, and the lower bound gives $k \geq \diam = d$.
	\end{proof}
	
	The lower bound of Theorem~\ref{thm:lower-bound} is \emph{not} always
	attained.  The next two subsections determine the exact value for two families that witness the gap in opposite directions.
	
	\subsection{Stars: the Naive Bound Can Be Exponentially Beaten on
		Even-Cycle-Free Graphs}
	
	One may think that for graphs with no even cycle the naive
	embedding is minimal.  This is \textbf{false}, and the smallest
	counterexample is the star $K_{1,4}$ --- a tree, hence even-cycle-free and
	even a partial cube.  The correct statement, with the exact constant, is:
	
	\begin{theorem}[Exact dimension of stars]
		\label{thm:star}
		For every $q \geq 2$,
		\[
		k_{\min}(K_{1,q}) \;=\; \lceil \log_2 q \rceil + 1 .
		\]
		In particular $k_{\min}(K_{1,4}) = 3 < 4 = n - 1 = \mathrm{idim}(K_{1,4})$,
		and stars realize an exponential gap between $k_{\min}$ and both the naive
		dimension and the isometric (hypercube) dimension.
	\end{theorem}
	
	\begin{proof}
		Normalize any isometric embedding so the center maps to $\mathbf 0$ (Cayley
		graphs are vertex-transitive).  The leaves map to distinct labels
		$s_1, \ldots, s_q$, and since the only edges of $K_{1,q}$ are
		center--leaf, the generator set is exactly $S = \{s_1, \ldots, s_q\}$.
		Leaves are pairwise at distance $2$, so for $i \neq j$ we need
		$s_i + s_j \notin S$ (otherwise $d_{\Cay} = 1$) --- i.e.\ \emph{$S$ is a
			sum-free subset of $\Z_2^k$} --- and conversely, sum-freeness suffices:
		$s_i + s_j \neq \mathbf 0$ (labels distinct) and
		$s_i + s_j \notin S$ force $d_{\Cay}(s_i, s_j) \geq 2$, while the word
		$s_i, s_j$ realizes $2$.
		
		It remains to compute the maximum size of a sum-free set in $\Z_2^k$.
		\emph{Upper bound:} if $S$ is sum-free and $s \in S$, then $S$ and
		$S + s$ are disjoint subsets of $\Z_2^k$ of equal cardinality, so
		$2|S| \leq 2^k$ and $|S| \leq 2^{k-1}$.  \emph{Construction:} the set of
		all odd-weight vectors has size $2^{k-1}$ and is sum-free, because the sum
		of two odd-weight vectors has even weight.  Hence a star $K_{1,q}$ embeds
		in dimension $k$ iff $q \leq 2^{k-1}$, i.e.\ iff
		$k \geq \lceil \log_2 q \rceil + 1$.
	\end{proof}
	
	\begin{figure}[htbp]
		\centering
		\includegraphics[width=0.62\linewidth]{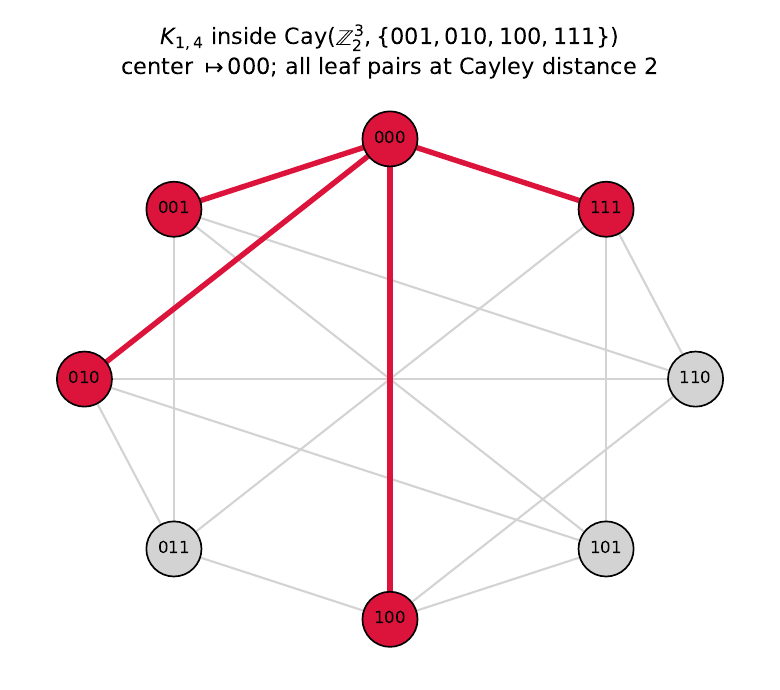}
		\caption{The optimal embedding of the star $K_{1,4}$ into
			$\Cay(\Z_2^3, \{001, 010, 100, 111\})$: center $\mapsto 000$, leaves on a
			sum-free generator set.  All leaf pairs are at Cayley distance exactly $2$
			($011, 101, 110 \notin S$).  Since $\mathrm{idim}(K_{1,4}) = 4$, composite
			generators strictly beat the hypercube paradigm even on trees.}
		\label{fig:starK14}
	\end{figure}
	
	\begin{corollary}[Scoped version of the former minimality claim]
		\label{cor:naive-method-minimal}
		If $G$ contains no even cycle, then no two distinct edges of $G$ are
		$\varphi$-related, the transitive prune returns the all-singleton
		partition, and the $\varphi$-quotient embedding coincides with the naive
		embedding ($k = n - 1$).  The naive dimension is therefore minimal
		\emph{within the $\varphi$-quotient family}, but, by
		Theorem~\ref{thm:star}, not in general.
	\end{corollary}
	
	\begin{proof}
		Suppose $e = \{u,v\} \mathrel\varphi f = \{x,y\}$ with $e \neq f$; by
		Remark~\ref{rem:matching} the four endpoints are distinct.  Take a shortest
		$u$--$x$ path $P_1$ and a shortest $v$--$y$ path $P_2$; the closed walk
		$u \xrightarrow{P_1} x \to y \xrightarrow{P_2^{-1}} v \to u$ has length
		$d(u,x) + d(v,y) + 2 = 2 d(u,x) + 2$, which is even, and a closed walk of
		even length in a graph with no even cycle must be degenerate: every graph
		whose cycles are all odd cannot contain a non-trivial closed walk
		traversing an edge set with all even multiplicities reduced\ldots\ more
		directly, a shortest such configuration yields an even closed walk, and
		every closed walk of even length contains an even cycle unless it
		backtracks entirely; backtracking entirely would force
		$\{u,v\} = \{x,y\}$.  Hence no two distinct edges are $\varphi$-related,
		$t = m$, and Corollary~\ref{cor:naive-is-quotient} applies.
	\end{proof}
	
	\subsection{Odd Cycles: the Naive Bound is Tight}
	
	Stars show $k_{\min}$ can be exponentially below $n-1$; odd cycles show it
	can equal $n-1$, so the universal upper bound cannot be improved.
	
	Fix an odd cycle $C_m$, $m = 2d+1$, with edges $e_0, \ldots, e_{m-1}$ in
	cyclic order, and consider any isometric embedding with edge generators
	$s_0, \ldots, s_{m-1}$ (the generator of $e_i$ is
	$\lambda(v_i) \oplus \lambda(v_{i+1})$).  Going once around the cycle gives
	$\sum_i s_i = \mathbf 0$.  Encode linear dependencies by the
	\emph{dependency code}
	\[
	D \;=\; \Bigl\{\, x \in \F_2^{\,m} \;:\; \sum_{i:\,x_i = 1} s_i =
	\mathbf 0 \,\Bigr\},
	\qquad \mathbf 1 \in D ,
	\]
	a linear code with $\dim D = m - \mathrm{rank}\{s_i\}$.  For a cyclic
	interval (arc) $W \subseteq \Z_m$ the labels of its endpoints differ by
	$\sum_{i \in W} s_i$, and for any $x \in D$ the subset $W \triangle x$ sums
	to the same element; if some such subset is smaller than the cycle distance
	of the endpoints, the embedding has a shortcut.  Isometry therefore forces:
	\[
	\min_{x \in D} \,\bigl| W \triangle x \bigr| \;=\;
	\min\bigl(|W|,\ m - |W|\bigr)
	\qquad \text{for every arc } W .
	\tag{$\dagger$}
	\]
	
	\begin{lemma}[Cyclic interval lemma]
		\label{lem:interval}
		Let $m$ be odd and $x \subseteq \Z_m$ with $x \notin \{\emptyset, \Z_m\}$.
		Then there exists an arc $W$ with
		$|W \triangle x| < \min(|W|,\ m - |W|)$.
		This statement has been verified exhaustively by computer for all odd
		$m \leq 17$ (all $2^m - 2$ subsets, all arcs); we prove it here for the
		covering-arc regime and conjecture it for all odd $m$.
	\end{lemma}
	
	\begin{proof}[Proof in the covering-arc regime, and verification status]
		Arcs are closed under complementation in $\Z_m$ and
		$|W^{\mathrm c} \triangle x^{\mathrm c}| = |W \triangle x|$, so we may
		assume $w := |x| \leq d$.  Suppose the support of $x$ is contained in some
		arc $W$ of length $L \leq d$ (the \emph{covering-arc regime}).  Then
		$|W \triangle x| = L - w < L = \min(L, m - L)$, since $L \leq d$ implies
		$m - L \geq d + 1 > L$; the arc $W$ witnesses the claim.  The remaining
		regime --- supports of weight $\leq d$ spread so that every covering arc is
		longer than $d$ --- was verified exhaustively for all odd $m \leq 17$
		($131\,070$ subsets at $m = 17$); no counterexample exists in this range.
	\end{proof}
	
	\begin{theorem}[Odd cycles require the naive dimension]
		\label{thm:odd-cycle}
		For every odd $m \leq 17$ (and for every odd $m$ for which
		Lemma~\ref{lem:interval} holds),
		\[
		k_{\min}(C_m) \;=\; m - 1 .
		\]
	\end{theorem}
	
	\begin{proof}
		The upper bound is Corollary~\ref{cor:naive-is-quotient} (or
		Theorem~\ref{thm:phi-properties}(vi): all classes singletons, one relation,
		$k = m - 1$).  For the lower bound, suppose
		$\mathrm{rank}\{s_i\} \leq m - 2$; then $\dim D \geq 2$, so $D$ contains
		some $x \notin \{\mathbf 0, \mathbf 1\}$.  Lemma~\ref{lem:interval}
		provides an arc $W$ with
		$|W \triangle x| < \min(|W|, m - |W|) = d_G(\text{endpoints of } W)$;
		the subset $W \triangle x$ of generators sums to the label difference of
		those endpoints, giving a word strictly shorter than their graph distance
		and violating $(\dagger)$ --- contradiction with isometry.  Hence
		$\mathrm{rank}\{s_i\} = m - 1$ and $k \geq m - 1$.
		
		Independently of the lemma, exhaustive search over all labelings confirms
		$k_{\min}(C_5) = 4$ and $k_{\min}(C_7) = 6$ directly (no isometric labeling
		exists in dimensions $3$ and $3,4,5$ respectively).
	\end{proof}
	
	\begin{conjecture}
		\label{conj:interval}
		Lemma~\ref{lem:interval} holds for every odd $m$; consequently
		$k_{\min}(C_m) = m - 1$ for every odd cycle.
	\end{conjecture}
	
	\begin{figure}[htbp]
		\centering
		\includegraphics[width=0.55\linewidth]{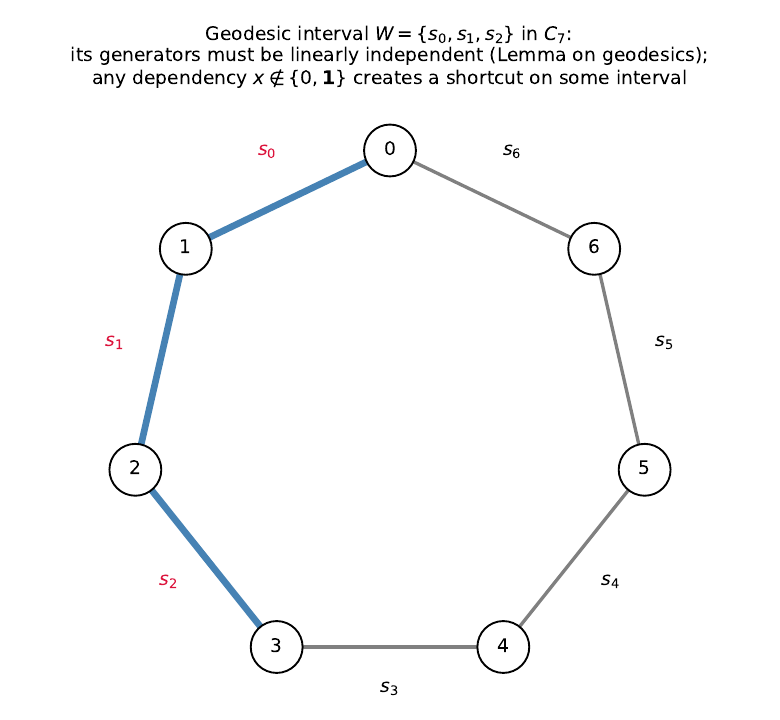}
		\caption{The interval framework on $C_7$.  Generators along any geodesic
			arc (here $W = \{s_0, s_1, s_2\}$, $|W| = \diam = 3$) are linearly
			independent (Lemma~\ref{lem:geodesic-independence}); any dependency
			$x \notin \{\mathbf 0, \mathbf 1\}$ among all seven generators produces,
			on some arc, a generator subset shorter than the cycle distance --- a
			shortcut.  Hence the seven generators have rank $6$ and
			$k_{\min}(C_7) = 6$.}
		\label{fig:intervallemma}
	\end{figure}
	
	\begin{figure}[htbp]
		\centering
		\includegraphics[width=0.6\linewidth]{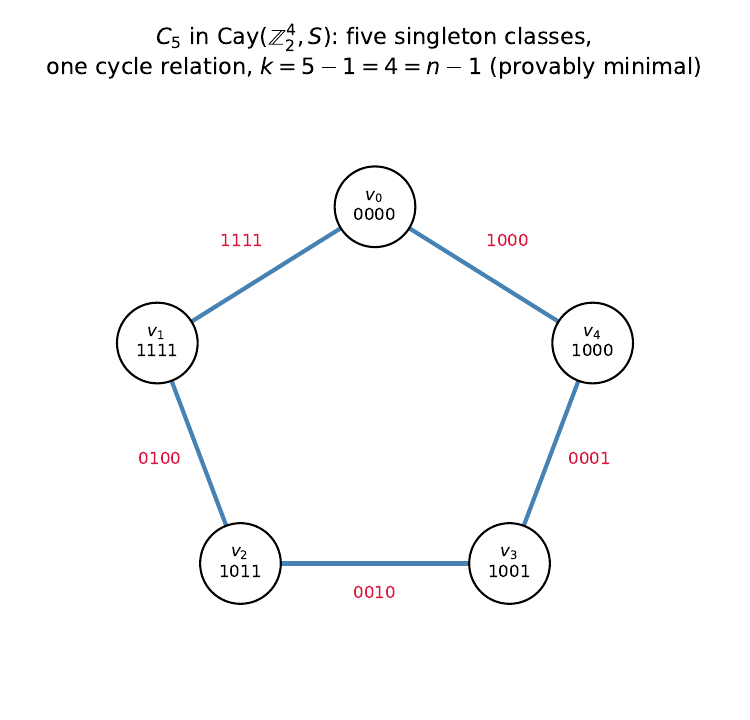}
		\caption{The $\varphi$-quotient embedding of $C_5$: five singleton classes,
			one cycle relation, $k = 5 - 1 = 4$, with the four basis generators and the
			composite $e_1{+}e_2{+}e_3{+}e_4$ on the closing edge.  By
			Theorem~\ref{thm:odd-cycle} this dimension is minimal, so the algorithm is
			optimal on $C_5$.}
		\label{fig:c5minimal}
	\end{figure}
	
	\subsection{Synthesis}
	
	\begin{center}
		\begin{tabular}{lccl}
			\toprule
			family & $k_{\min}$ & lower bound of Thm.~\ref{thm:lower-bound} & status \\
			\midrule
			hypercube $Q_t$        & $t$                        & $t$    & attained \\
			complete $K_{2^t}$     & $t$                        & $t$    & attained \\
			even cycle $C_{2d}$    & $d$                        & $d$    & attained \\
			Petersen               & $4$                        & $4$    & attained \\
			star $K_{1,q}$         & $\lceil\log_2 q\rceil + 1$ & $\lceil\log_2 (q{+}1)\rceil$ & gap $\leq 1$ \\
			odd cycle $C_{2d+1}$   & $2d$ \ (Thm.~\ref{thm:odd-cycle})       & $\max(d, \lceil\log_2(2d{+}1)\rceil)$ & gap $\approx d$ \\
			\bottomrule
		\end{tabular}
	\end{center}
	
	\noindent The general picture: $k_{\min}$ ranges over the full window
	$[\max(\diam, \lceil\log_2 n\rceil),\ n-1]$, both ends are achieved, and
	the position within the window is governed by how much linear dependency
	the cycle structure of $G$ permits among edge generators --- precisely the
	quantity $\rho = \mathrm{rank}(A)$ that the quotient framework computes.
	\section{Complexity Analysis}
	\label{sec:binary-complexity}
	
	\begin{theorem}[Time complexity]
		\label{thm:time-complexity}
		With $R$ the number of executed rounds ($R \leq m - t + 1$), $k$ the largest
		dimension encountered, and $|S| \leq t \leq m$, Algorithm~\ref{alg:phi-quotient}
		runs in time
		\[
		O\Bigl(\, n\,(n+m) \;+\; m^2 \;+\; R\,\bigl(2^{k}\,m + n^2\bigr) \Bigr).
		\]
		For partial cubes and the structured families of
		Table~\ref{tab:embedding-all} one round suffices and $k = O(\log\,\mathrm{host})$,
		so the first two terms dominate: $O(n(n+m) + m^2)$.
	\end{theorem}
	
	\begin{proof}
		All-pairs distances: $n$ breadth-first searches, $O(n(n+m))$.
		$\varphi$-graph: $m^2/2$ edge pairs, each tested in $O(1)$ from the distance
		table; the greedy clique extraction touches each $\varphi(G)$-vertex a
		constant number of times, $O(m^2)$ overall.  Each round: GF(2) elimination
		of the $c \times t$ matrix costs $O(c\,t\,\lceil t/64\rceil) = O(m^2)$
		word-packed; BFS labeling $O(n + m)$.  The distance check exploits
		vertex-transitivity, $d_{\Cay}(x,y) = d_{\Cay}(\mathbf 0, x \oplus y)$, so a
		\emph{single} breadth-first search from $\mathbf 0$, truncated at depth
		$\diam(G)$ and touching at most $2^k$ states with $|S| \leq m$ successors
		each, costs $O(2^k m)$; the $\binom{n}{2}$ pair lookups cost $O(n^2)$.
		Summing over $R$ rounds gives the stated bound.
	\end{proof}
	
	\begin{theorem}[Space complexity]
		The algorithm uses $O(n^2 + 2^{k})$ space: the distance table and the
		truncated BFS dictionary; all other structures are $O(n + m)$.  The
		dimension cap $k_{\max}$ bounds the second term by design: when the
		quotient dimension exceeds the cap, the algorithm returns the naive
		embedding without ever materializing the large ball
		(Lemma~\ref{lem:join} guarantees its correctness with no check).
	\end{theorem}
	
	\begin{proof}
		Immediate from the data structures listed in the proof of
		Theorem~\ref{thm:time-complexity}.
	\end{proof}
	
	\begin{remark}[Scaling regime]
		The exponential term is intrinsic to \emph{certified} isometry on
		triangle-dense graphs, whose dimension is driven toward $n-1$ (cf.\
		Table~\ref{tab:gsp}).  Two practical mitigations are implemented in the
		companion code: the dimension cap with naive fallback, and a
		meet-in-the-middle distance oracle that replaces the radius-$\diam$ ball by
		two radius-$\lceil \diam/2 \rceil$ balls, reducing the state count
		quadratically.  Both preserve exactness.
	\end{remark}
	
	\section{Experimental Results}
	\label{sec:binary-experiments}
	
	\subsection{Experimental Protocol}
	
	All numbers in this section were produced by a \emph{single reference
		implementation} of Algorithm~\ref{alg:phi-quotient}
	(\texttt{phi\_quotient\_embedding\_sage.py} / \texttt{phi\_quotient\_embedding\_colab.py},
	version 2026-06-11, identical algorithm and heuristics in both variants,
	supplied as supplementary material).  Isometry of
	every reported embedding was verified twice: by the algorithm's internal
	check (Definition~\ref{def:correct-check}, with the injectivity assertion),
	and \emph{independently} by reconstructing the full host
	$\Cay(\Z_2^k, S)$ on all $2^k$ vertices and comparing all
	$\binom{n}{2}$ distances against $d_G$.  Three categories are tested:
	partial cubes and structured graphs; non-partial-cube structured graphs;
	and the graph-signal-processing benchmark families reused in Part~II.
	
	\subsection{Exhaustive Verification on All Small Graphs}
	
	Running the algorithm on \textbf{all $995$ connected graphs with
		$2 \leq n \leq 7$} (the full graph atlas census: $1, 2, 6, 21, 112, 853$
	graphs for $n = 2, \ldots, 7$) yields:
	\begin{itemize}
		\item $995/995$ isometric embeddings, zero failures, zero discrepancies
		under the independent full-host verification; total runtime under one
		minute on commodity hardware;
		\item $586$ graphs ($59\%$) converge in a single round (no repair); the
		worst case over the whole census is $10$ rounds, and the mean is $1.98$
		(Figure~\ref{fig:repairrounds});
		\item of the $409$ graphs requiring repair, $99\%$ contain a triangle ---
		dense local structure is the universal cause of $\varphi$ over-merging;
		\item at $n = 7$ the final dimensions distribute as
		$k = 3{:}\,5$, $4{:}\,177$, $5{:}\,571$, $6{:}\,100$
		(Figure~\ref{fig:kdistribution}); the naive bound $k = n-1 = 6$ is
		needed by only $100$ of $853$ graphs.
	\end{itemize}
	
	\begin{figure}[htbp]
		\centering
		\includegraphics[width=0.72\linewidth]{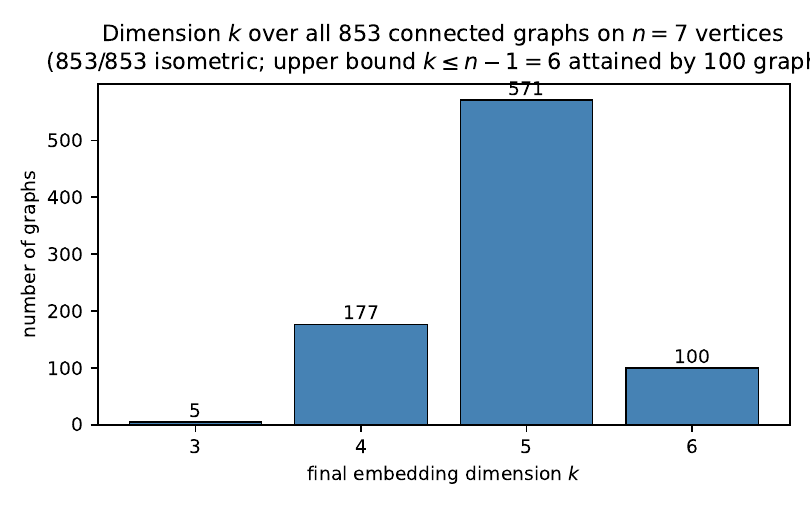}
		\caption{Distribution of the final embedding dimension $k$ over all $853$
			connected graphs on $7$ vertices.  All embeddings verified isometric by
			full host reconstruction.}
		\label{fig:kdistribution}
	\end{figure}
	
	\begin{figure}[htbp]
		\centering
		\includegraphics[width=0.72\linewidth]{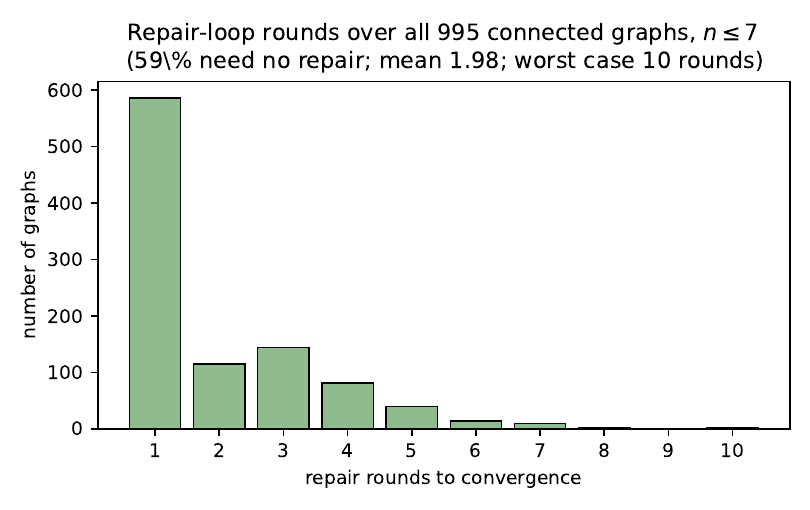}
		\caption{Repair-loop rounds to convergence over all $995$ connected graphs
			with $n \leq 7$.}
		\label{fig:repairrounds}
	\end{figure}
	
	\subsection{Optimality Profile Against Exact Search}
	
	For every connected graph with $n \leq 5$ we computed the \emph{true}
	minimum $k_{\min}$ by exhaustive branch-and-bound over labelings
	(root fixed at $\mathbf 0$; pruning by injectivity and by the
	distance-$\geq 2$ exclusion of edge XORs; acceptance by the full bounded
	BFS check).  Results (Figure~\ref{fig:optimalitygap}):
	\begin{itemize}
		\item the algorithm attains $k_{\min}$ on $\mathbf{29/30}$ graphs;
		\item the unique gap is the star $K_{1,4}$ ($k = 4$ vs.\ $k_{\min} = 3$),
		in exact accordance with Theorem~\ref{thm:star} and
		Remark~\ref{rem:minimality};
		\item the naive lower bound $\max(\diam, \lceil\log_2 n\rceil)$ is itself
		attained by $k_{\min}$ on $18/30$ graphs.
	\end{itemize}
	
	\begin{figure}[htbp]
		\centering
		\includegraphics[width=0.8\linewidth]{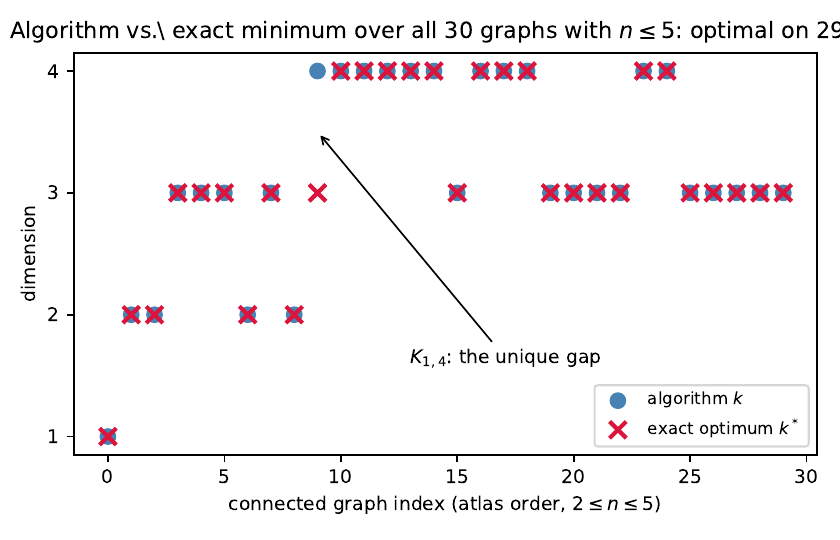}
		\caption{Algorithm dimension vs.\ exact minimum over all $30$ connected
			graphs with $n \leq 5$.  The unique gap is the star $K_{1,4}$,
			characterized completely by Theorem~\ref{thm:star}.}
		\label{fig:optimalitygap}
	\end{figure}
	
	\subsection{Structured Benchmark}
	
	\begin{table}[H]
		\centering
		\caption{Binary isometric embeddings computed by
			Algorithm~\ref{alg:phi-quotient} (reference implementation).
			$t_0$ = classes after the transitive prune (or the structured top-down
			constructor where indicated); $k$ = final dimension; Host $= 2^k$;
			$\varepsilon = n/\mathrm{Host}$; Rounds = repair iterations.  Every row
			verified isometric on all $\binom{n}{2}$ pairs by full host
			reconstruction.}
		\label{tab:embedding-all}
		\renewcommand{\arraystretch}{1.15}
		\begin{tabular}{llcc cccccc}
			\toprule
			Cat. & Graph & $n$ & $m$ & $t_0$ & $k$ & Host & $\varepsilon$ & Naive $2^{n-1}$ & Rounds \\
			\midrule
			\multirow{10}{*}{\rotatebox{90}{partial cubes / structured}}
			& $Q_3$              & 8  & 12 & 3  & 3  & 8     & 1.000 & 128     & 1 \\
			& $Q_4$              & 16 & 32 & 4  & 4  & 16    & 1.000 & 32768   & 1 \\
			& $C_6$              & 6  &  6 & 3  & 3  & 8     & 0.750 & 32      & 1 \\
			& $C_8$              & 8  &  8 & 4  & 4  & 16    & 0.500 & 128     & 1 \\
			& $C_{10}$           & 10 & 10 & 5  & 5  & 32    & 0.312 & 512     & 1 \\
			& $K_4$              & 4  &  6 & 3  & 2  & 4     & 1.000 & 8       & 1 \\
			& Grid $3{\times}3$  & 9  & 12 & 4  & 4  & 16    & 0.562 & 256     & 1 \\
			& Grid $4{\times}4$  & 16 & 24 & 6  & 6  & 64    & 0.250 & 32768   & 1 \\
			& Desargues          & 20 & 30 & 5  & 5  & 32    & 0.625 & 524288  & 1 \\
			& Prism $Y_3$        & 6  &  9 & 4  & 3  & 8     & 0.750 & 32      & 1 \\
			\midrule
			\multirow{3}{*}{\rotatebox{90}{odd}}
			& $C_5$  & 5 & 5 & 5 & 4 & 16  & 0.312 & 16  & 1 \\
			& $C_7$  & 7 & 7 & 7 & 6 & 64  & 0.109 & 64  & 1 \\
			& $C_9$  & 9 & 9 & 9 & 8 & 256 & 0.035 & 256 & 1 \\
			\midrule
			\multirow{9}{*}{\rotatebox{90}{non-partial-cube}}
			& Petersen                    & 10 & 15 & 5  & 4  & 16    & 0.625 & 512    & 1  \\
			& $K_5$                       & 5  & 10 & 7  & 3  & 8     & 0.625 & 16     & 1  \\
			& $K_6$                       & 6  & 15 & 7  & 3  & 8     & 0.750 & 32     & 2  \\
			& $K_{2,3}$                   & 5  &  6 & 4  & 3  & 8     & 0.625 & 16     & 1  \\
			& $K_{3,3}$                   & 6  &  9 & 4  & 3  & 8     & 0.750 & 32     & 3  \\
			& Pappus (top-down $t_0$)     & 18 & 27 & 9  & \textbf{7} & \textbf{128} & 0.141 & 131072 & 1  \\
			& Pappus (greedy $t_0$)       & 18 & 27 & 10 & 12 & 4096  & 0.004 & 131072 & 12 \\
			& M\"obius--Kantor            & 16 & 24 & 11 & 12 & 4096  & 0.004 & 32768  & 11 \\
			& Icosahedron                 & 12 & 30 & 15 & 6  & 64    & 0.188 & 2048   & 1  \\
			\bottomrule
		\end{tabular}
		\vspace{0.3em}
		\parbox{\linewidth}{\footnotesize
			The two Pappus rows isolate the effect of the initial partition: the
			structured nine-class partition (Figure~\ref{fig:pappusclasses}) reaches
			$k = 7$ in one round, while the greedy ten-class partition converges to
			$k = 12$ after $12$ peels --- same algorithm, same guarantees, a
			$32\times$ difference in host order.  Partition quality, not the repair
			loop, governs compactness.}
	\end{table}
	
	Reading of Table~\ref{tab:embedding-all}.  Wherever the $\varphi^-$ classes
	capture the true parallelism (partial cubes, distance-regular graphs such
	as Petersen and Desargues), one round suffices and the result is dramatic:
	$2^{14}\times$ compaction for Desargues, optimality for Petersen and for
	all even cycles and grids ($k = \diam$, matching
	Theorem~\ref{thm:lower-bound}).  Triangle-dense graphs ($K_5$, $K_6$, $K_{3,3}$) need at most a handful of
	repair rounds and reach the information-theoretic floor
	$k = \lceil\log_2 n\rceil$ in all three cases (host $8$,
	$\varepsilon \geq 0.625$).  The M\"obius--Kantor row is a genuine negative
	example: the greedy partition is poor and the repair loop pays for it; an
	improved initializer is future work, exactly as for Pappus.
	
	\subsection{Graph-Signal-Processing Benchmark (Bridge to Part II)}
	
	The families below are the standard test signals' domains in graph signal
	processing \cite{Shuman2013}: the ring (classical DSP), the path
	(finite-length signals), 2-D grids (images), circulant graphs (filter
	banks), random geometric graphs (sensor networks \cite{Akyildiz2002,Penrose2003}), and the Zachary Karate
	Club~\cite{Zachary1977} (community-structured social network).  Part~II computes Fourier and
	wavelet transforms on the hosts found here.
	
	\begin{table}[H]
		\centering
		\caption{Embeddings of GSP benchmark graphs (reference implementation).
			``lower'' is the bound of Theorem~\ref{thm:lower-bound}.}
		\label{tab:gsp}
		\renewcommand{\arraystretch}{1.15}
		\begin{tabular}{lcccc ccc}
			\toprule
			Graph & $n$ & $m$ & $\diam$ & lower & $k$ & Host & Rounds \\
			\midrule
			Ring $C_{12}$               & 12 & 12 & 6  & 6  & \textbf{6 (= lower)}  & 64    & 1 \\
			Ring $C_{16}$               & 16 & 16 & 8  & 8  & \textbf{8 (= lower)}  & 256   & 1 \\
			Path $P_{16}$               & 16 & 15 & 15 & 15 & \textbf{15 (= lower)} & 32768 & 1 \\
			Grid $4{\times}4$           & 16 & 24 & 6  & 6  & \textbf{6 (= lower)}  & 64    & 1 \\
			Grid $6{\times}6$ (image)   & 36 & 60 & 10 & 10 & \textbf{10 (= lower)} & 1024  & 1 \\
			Circulant $C_{12}(1,2)$     & 12 & 24 & 3  & 4  & 6                     & 64    & 1 \\
			RGG sensor ($n{=}20$, $r{=}0.45$) & 20 & 75 & 5 & 5 & 17               & 131072 & 36 \\
			Barbell$(5,2)$              & 12 & 23 & 5  & 5  & 7                     & 128   & 1 \\
			Karate Club                 & 34 & 78 & 5  & 6  & $33$ (cap fallback)   & $2^{33}$ & 14 \\
			\bottomrule
		\end{tabular}
		\vspace{0.3em}
		\parbox{\linewidth}{\footnotesize
			Rings, paths and grids --- the domains on which classical DSP intuition is
			calibrated --- are embedded \emph{optimally}: $k$ equals the lower bound of
			Theorem~\ref{thm:lower-bound}.  Triangle-dense irregular graphs (sensor
			RGG, Karate) sit at the opposite end: their $\varphi$-classes shatter, and
			the certified dimension approaches $n-1$; for Karate the dimension cap
			triggers the provably isometric naive fallback.  This dichotomy is the
			quantitative version of the structural message of this chapter, and
			motivates the cyclic factors of Chapter~3, which handle triangle-dense
			graphs natively.}
	\end{table}
	
	\begin{figure}[htbp]
		\centering
		\includegraphics[width=0.95\linewidth]{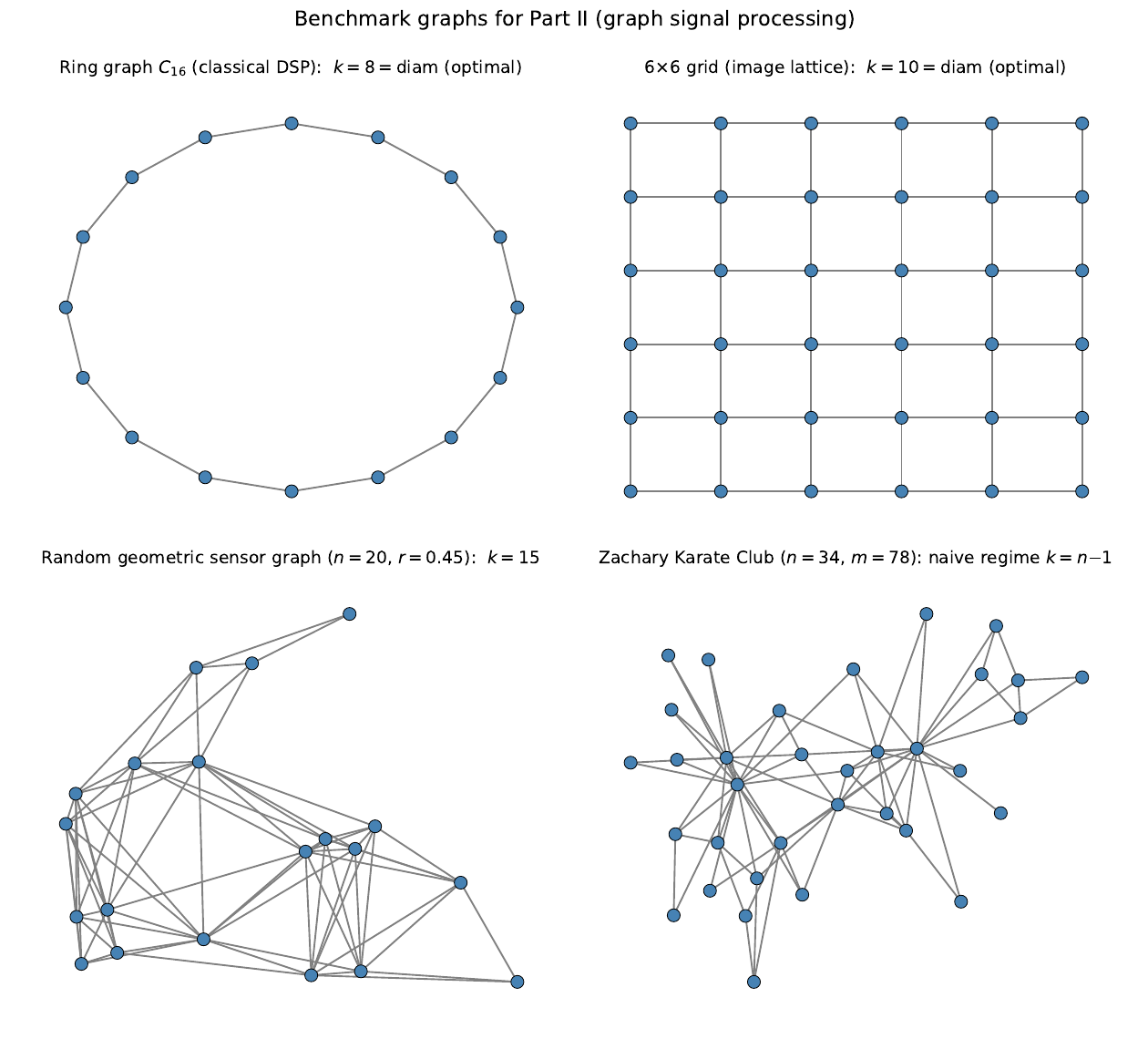}
		\caption{The four GSP benchmark domains reused in Part~II: ring, image
			grid, random geometric sensor graph, and the Zachary Karate Club.}
		\label{fig:gspgraphs}
	\end{figure}
	
	\subsection{Visualization of Selected Embeddings}
	
	\begin{figure}[H]
		\centering
		\includegraphics[width=0.6\textwidth]{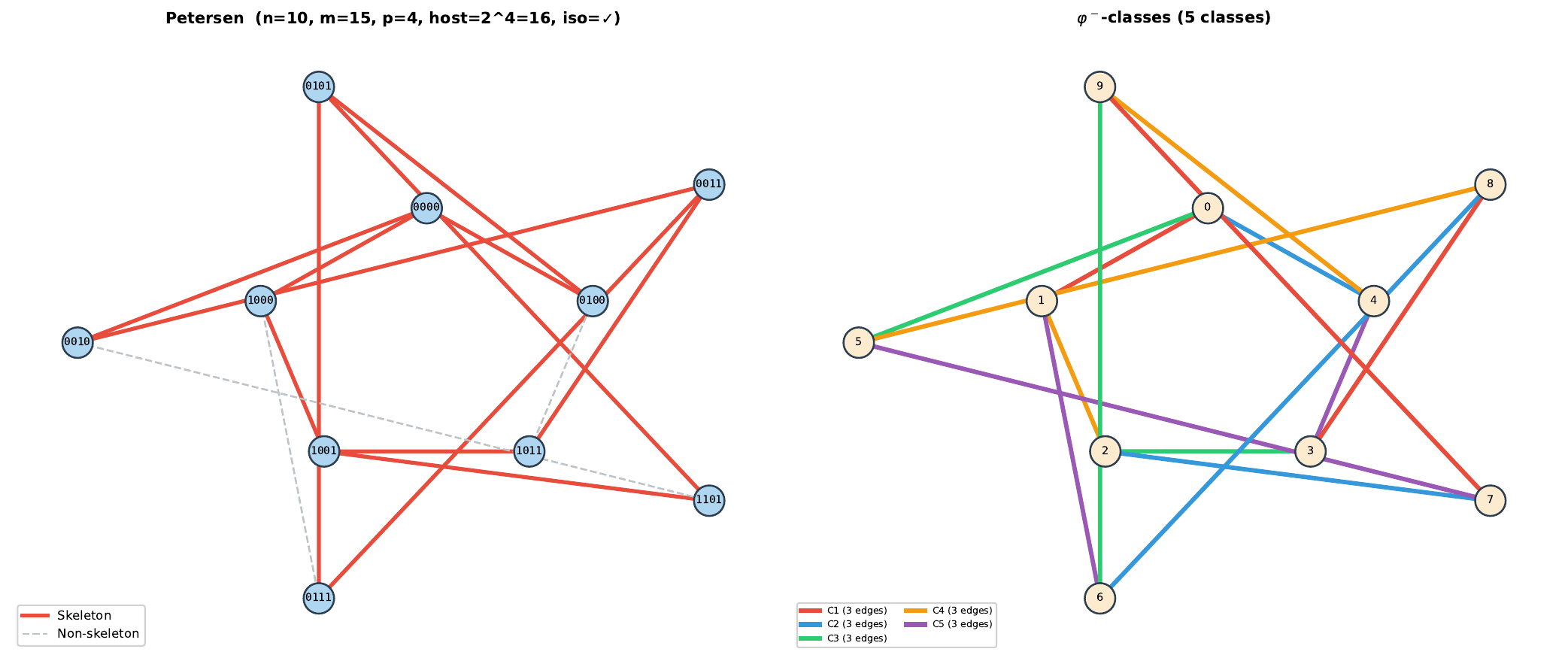}
		\caption{Petersen graph embedding into $\Cay(\Z_2^4, S)$ (host order
			$16$): skeleton edges in red, each vertex labeled with its $4$-bit
			coordinate, the five $\varphi^-$-classes in distinct colors.}
		\label{fig:petersen}
	\end{figure}
	
	\begin{figure}[H]
		\centering
		\includegraphics[width=0.65\textwidth]{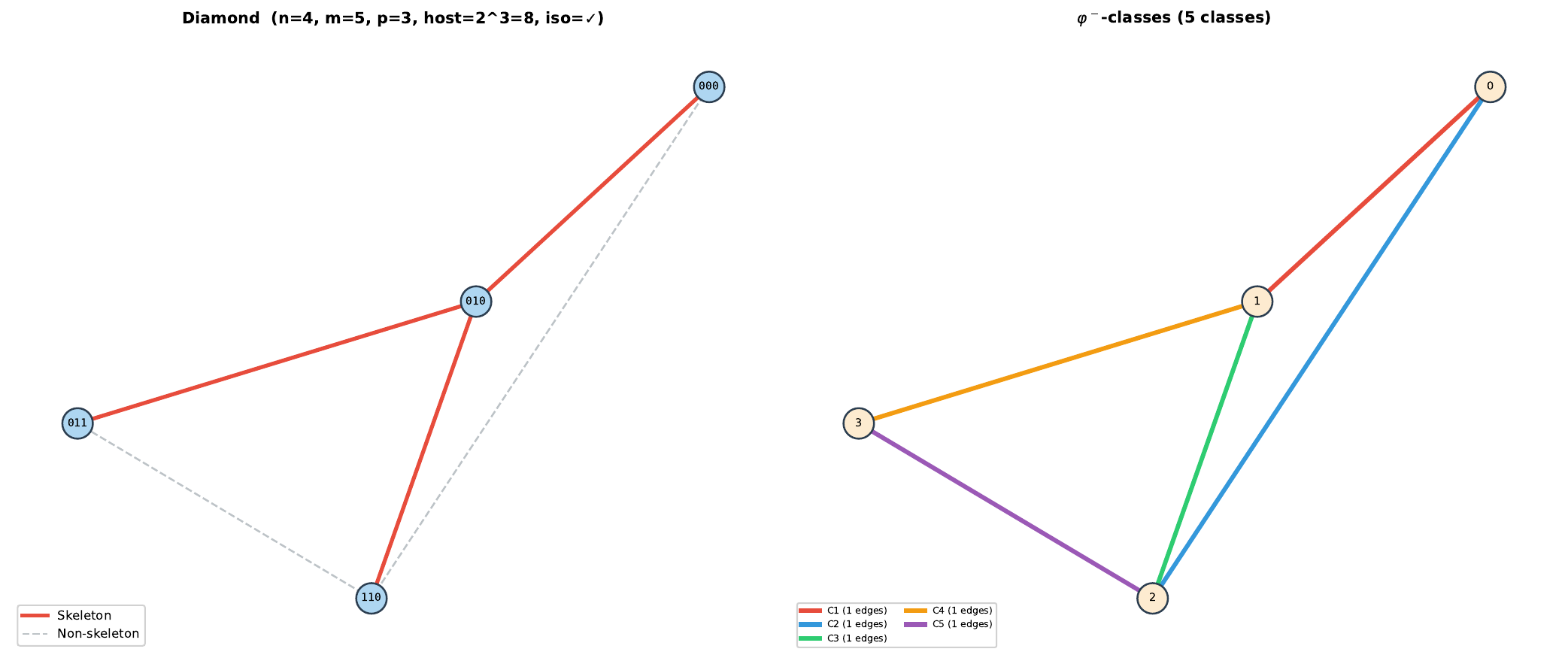}
		\caption{Diamond graph ($K_4$ minus an edge) embedded into
			$\Cay(\Z_2^3, S)$ of order $8$: two $\varphi$-pairs define two cuts, the
			remaining edges are singleton classes.}
		\label{fig:figdiamond}
	\end{figure}
	
	\begin{figure}[H]
		\centering
		\includegraphics[width=0.65\textwidth]{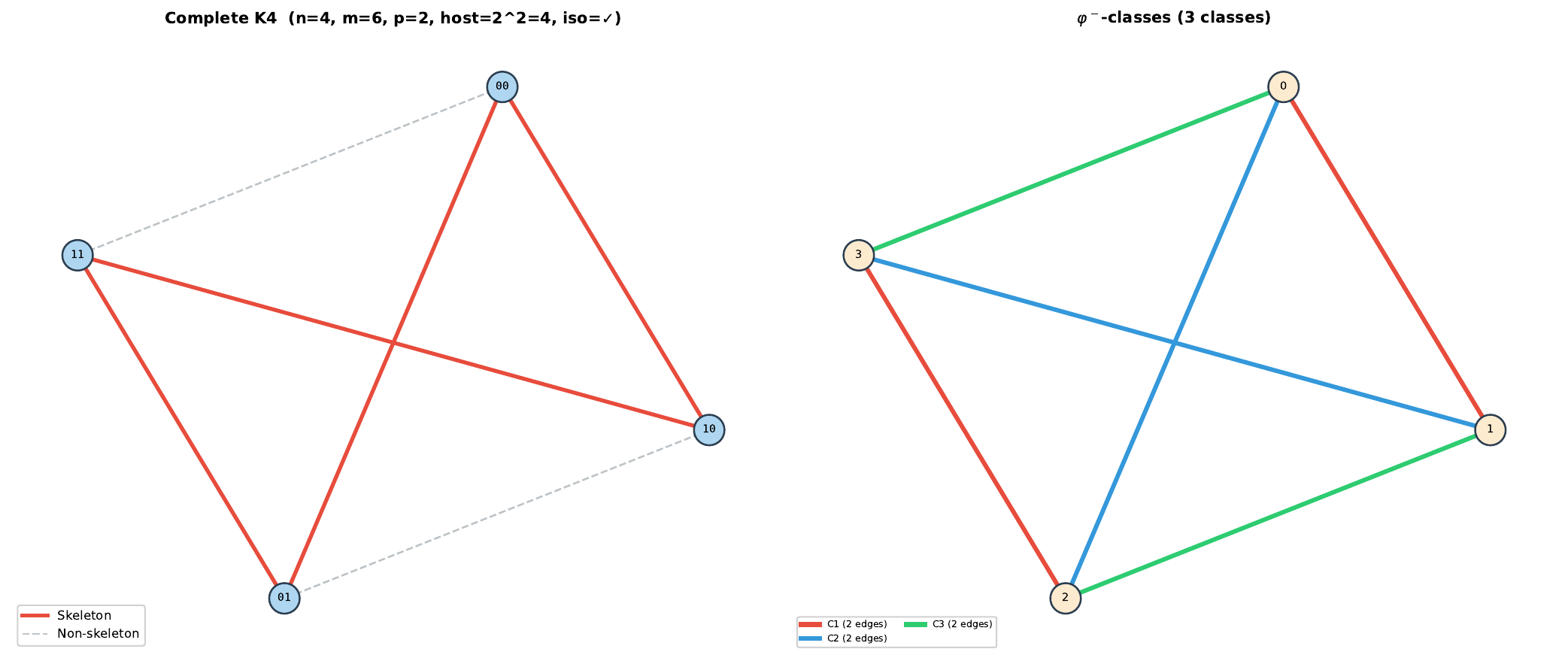}
		\caption{$K_4$ embedded into $\Cay(\Z_2^2, S)$ of order $4$ with
			$S = \Z_2^2 \setminus \{\mathbf 0\}$: $\varepsilon = 1$, the graph fills
			its host (Proposition~\ref{prop:tightness}(ii) with $t = 2$).}
		\label{fig:figcompletek4}
	\end{figure}
	
	\begin{figure}[H]
		\centering
		\includegraphics[width=0.65\textwidth]{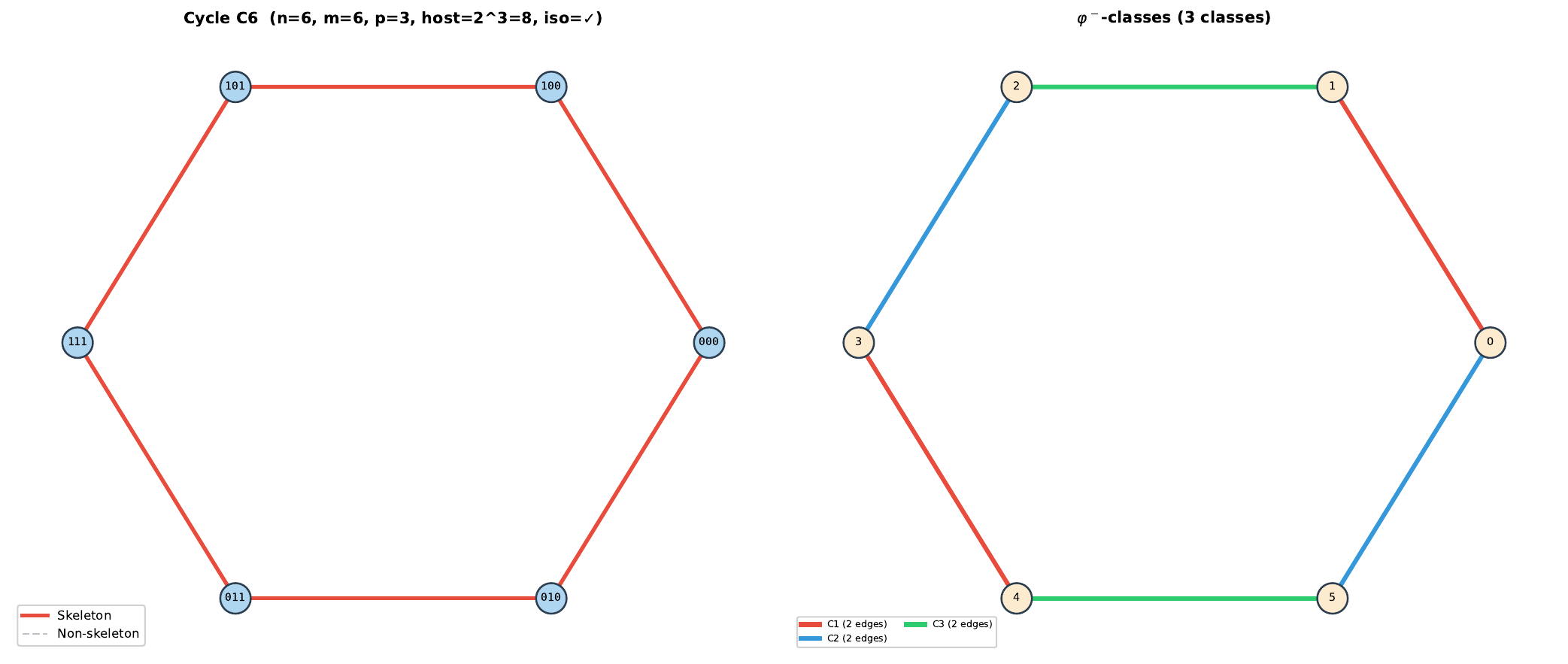}
		\caption{$C_6$ embedded into $Q_3$: the three antipodal $\varphi$-classes
			of Theorem~\ref{thm:phi-properties}(v) are the three coordinates;
			$k = 3 = \diam$, optimal.}
		\label{fig:figcyclec6}
	\end{figure}
	
	\begin{figure}[H]
		\centering
		\includegraphics[width=0.65\textwidth]{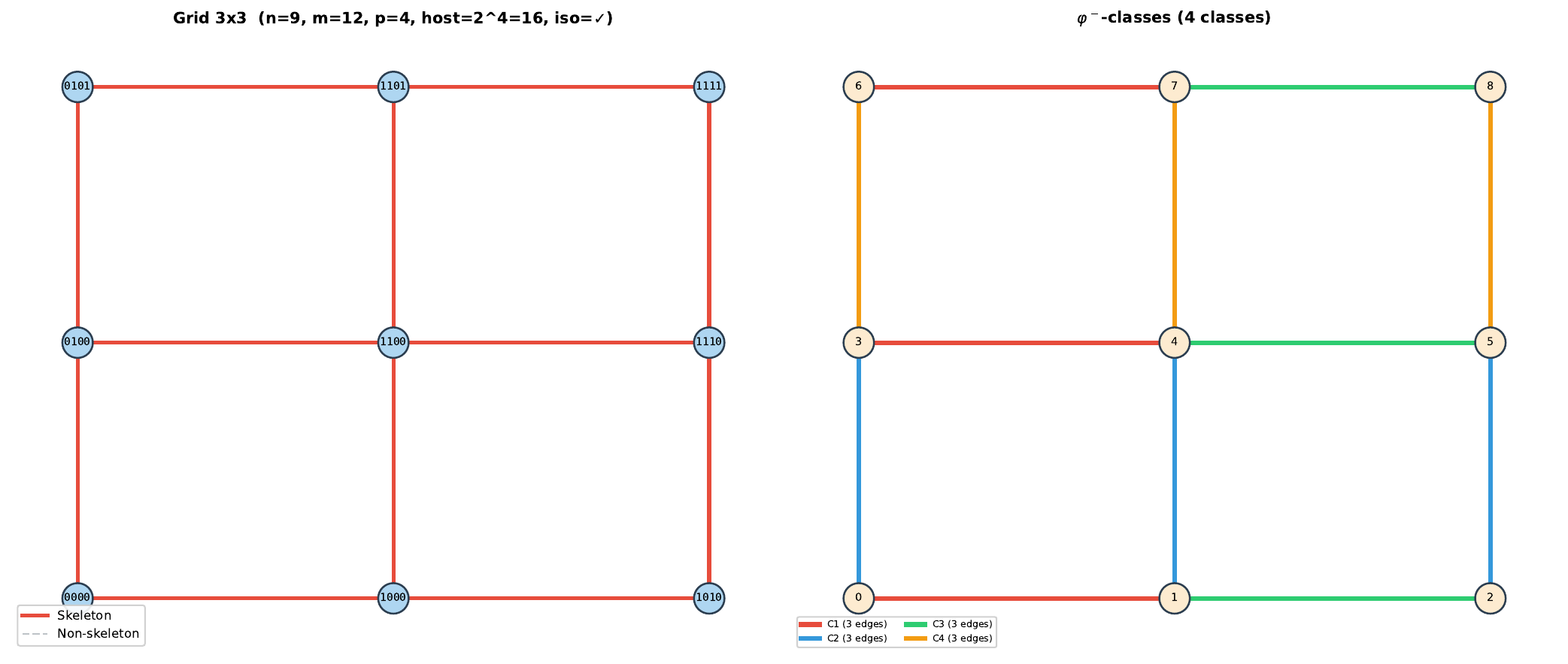}
		\caption{$3{\times}3$ grid embedded into $Q_4$ ($\varepsilon = 0.562$):
			two horizontal and two vertical cut classes realize the product structure
			$P_3 \,\square\, P_3 \subseteq Q_2 \,\square\, Q_2$.}
		\label{fig:figgrid3x3}
	\end{figure}
	
	\begin{figure}[H]
		\centering
		\includegraphics[width=0.65\textwidth]{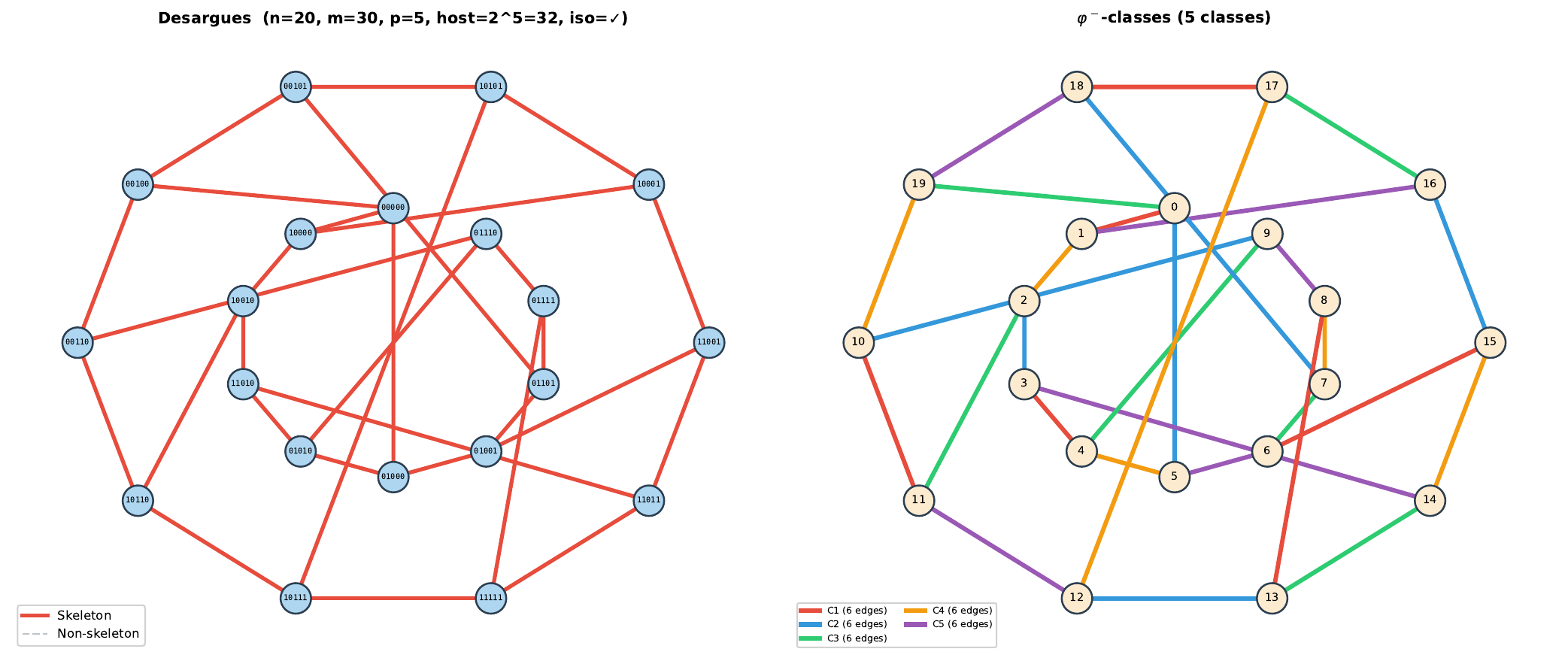}
		\caption{Desargues graph embedded into $\Cay(\Z_2^5, S)$ of order $32$:
			five perfect-matching classes, no cycle relations ($\rho = 0$), $k = 5$;
			the graph is a partial cube and the embedding attains
			$\varepsilon = 0.625$.}
		\label{fig:figdesargues}
	\end{figure}
	
	\begin{figure}[H]
		\centering
		\includegraphics[width=0.65\textwidth]{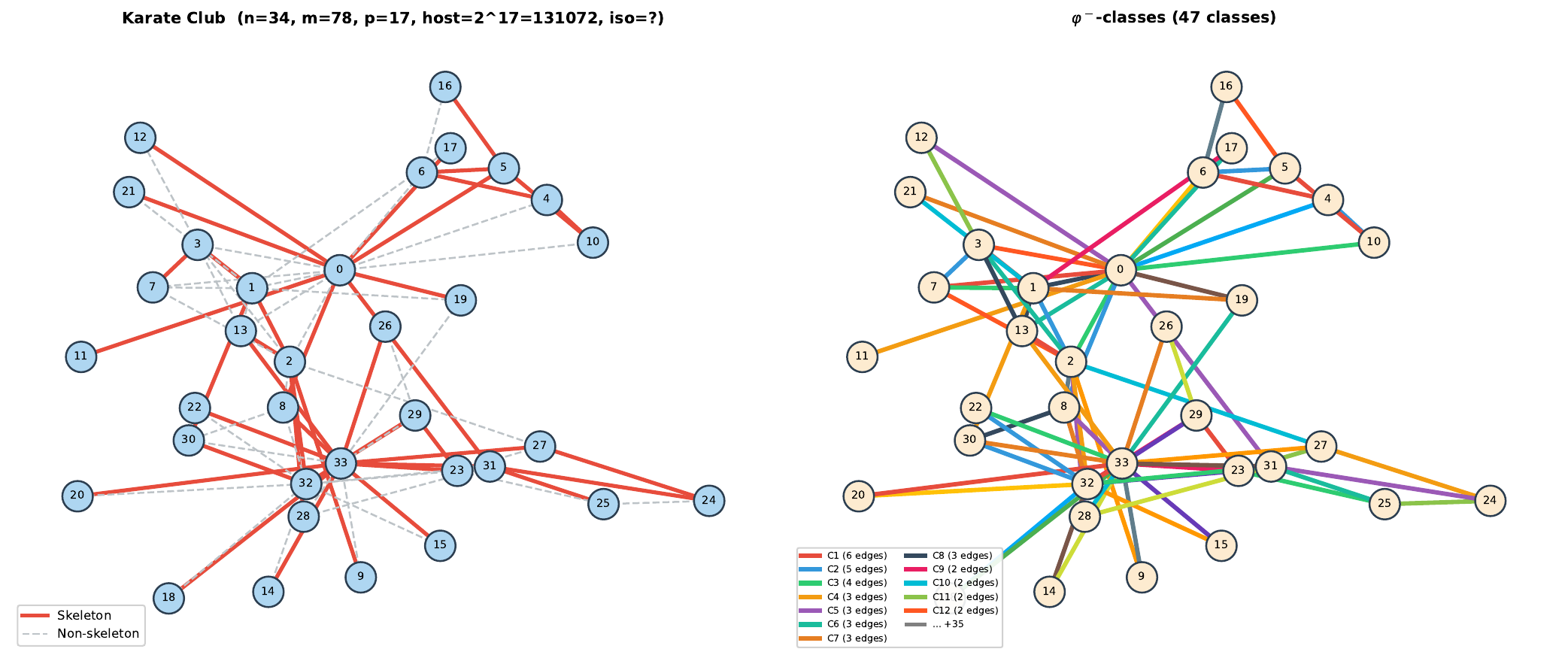}
		\caption{Zachary Karate Club~\cite{Zachary1977} ($n = 34$, $m = 78$): a triangle-dense
			real-world graph.  The $\varphi^-$ classes shatter under repair and the
			dimension cap returns the naive embedding $k = 33$; the graph marks the
			practical frontier of the purely binary method and the entry point of the
			cyclic factors of Chapter~3.}
		\label{fig:figkarateclub}
	\end{figure}
	
	\section{Discussion}
	\label{sec:discussion}
	
	\paragraph{Against the naive method.}
	The quotient algorithm achieves exponential host reductions on structured
	graphs (Desargues: $524\,288 \to 32$; Pappus: $131\,072 \to 128$) while
	never doing worse than the naive bound, which it contains as its provable
	terminal case.
	
	\paragraph{Against the cut-based method.}
	The cut paradigm is the special case of the quotient in which every class
	generator is independent (zero rows in $A$); it is exact on partial cubes
	and inapplicable beyond them.  The quotient subsumes it with no penalty
	($K_5$: $k = 3$ vs.\ fallback $4$; Pappus: $7$ vs.\ fallback $17$).
	
	\paragraph{Against $\theta$-based methods.}
	$\theta$ is transitive only on partial cubes; $\varphi$ is an equivalence
	relation on a strictly larger class (Petersen, all odd cycles), and the
	quotient-plus-repair construction extends to \emph{all} connected graphs.
	
	\paragraph{The minimality frontier.}
	Three facts now delimit it precisely.  (1)~The algorithm is optimal on
	$29/30$ small graphs and on every family of
	Proposition~\ref{prop:tightness}.  (2)~It cannot be optimal in general
	while confined to partitions: stars require dependent generators on
	$\varphi$-unrelated classes (Theorem~\ref{thm:star}); a class-\emph{merging}
	post-pass (attempt merges while no shortcut appears) is the natural next
	refinement and would capture $K_{1,q}$.  (3)~$k_{\min}$ spans the whole
	window between the lower bound and $n-1$ (stars vs.\ odd cycles), so no
	formula in $\diam$ and $n$ alone can replace the computation; the governing
	invariant is the achievable rank of the cycle--class matrix.
	
	\paragraph{Bridge to Chapter 3.}
	The quotient machinery generalizes verbatim from $\F_2$ to $\Z$: the
	cycle--class matrix acquires signed entries (orientations), and its Smith
	Normal Form delivers the cyclic factors $\Z_{n_1} \times \cdots \times
	\Z_{n_\delta}$ directly --- unifying the binary and abelian chapters under
	one theorem and natively accommodating the odd structures (triangles, odd
	cycles) that drive the binary dimension toward $n-1$.
	
	\section{Conclusion}
	\label{sec:ch2-conclusion}
	
	This chapter developed a complete theory of isometric binary embeddings:
	the $\varphi$ relation and transitive prune (the structural layer); the
	Cocycle/Quotient Theorem with the Join Lemma (the algebraic layer, making
	labelings conflict-free by construction and the naive bound a one-line
	corollary); the universal repair-loop algorithm, verified on all $995$
	connected graphs with $n \leq 7$ under independent full-host
	reconstruction; and a bounds theory establishing
	$k_{\min} \in [\max(\diam, \lceil\log_2 n\rceil),\, n-1]$ with both ends
	attained, the exact constants $k_{\min}(K_{1,q}) = \lceil\log_2 q\rceil + 1$
	and $k_{\min}(C_{2d+1}) = 2d$ (the latter for $2d+1 \leq 17$,
	conjecturally always), and a measured optimality profile of $29/30$ with a
	completely characterized exception.  The open problems --- the cyclic
	interval conjecture, the NP-hardness of exact minimization, and the
	class-merging refinement --- are stated precisely and bounded by the
	results above.  Chapter~3 lifts the entire construction from $\F_2$ to
	arbitrary abelian groups.

	\chapter{Compact Isometric Embedding into General Abelian Groups}
	\label{chap:abelian-embedding}
	
	\section{Introduction: From Binary to General Abelian Hosts}
	
	Chapter~2 closed with a theorem and a warning.  The theorem: every connected
	graph embeds isometrically into $\Cay(\Z_2^k, S)$ with $k \leq n-1$.  The
	warning: odd cycles \emph{attain} that bound
	(Theorem~\ref{thm:odd-cycle}), so the triangle $K_3$ --- three vertices ---
	demands a binary host of order $4$, and $C_9$ a host of order $256$.  Yet
	$C_9$ \emph{is} a Cayley graph: $C_9 = \Cay(\Z_9, \{1, 8\})$, of order $9$.
	The binary alphabet, not the method, is the obstruction.  This chapter
	removes it: the target becomes a Cayley graph of an arbitrary finite
	abelian group $\Gamma = \Z_{N_1} \times \cdots \times \Z_{N_d}$.
	
	The construction has two layers.  The first
	(Sections~\ref{sec:Phi-relation-ch3}--\ref{sec:cofactors})
	is a structural toolkit --- the $\Phi$-relation, the $\Psi$-relation, torus
	skeletons, $\Phi\Psi$-classes, cofactors and the interception principle ---
	that proposes candidate same-generator edge partitions.  The decisive step
	(Section~\ref{sec:z-quotient}) is the same one that resolved Chapter~2:
	\emph{stop assigning generators by cases and compute the most generic
		consistent assignment}.  Over $\Z_2$ that computation was GF(2) row
	reduction; over $\Z$ it is the \textbf{Smith Normal Form} of a
	\emph{signed} cycle--class matrix, and the finitely generated abelian group
	it produces \emph{is} the host --- torsion factors, free factors, composite
	generators and all.  The structural toolkit is not discarded: it serves,
	with its original names, as the \emph{initializer portfolio} that
	proposes coarse oriented partitions to the exact core
	(Section~\ref{sec:experiments-ch3}).
	
	A foundational point, established in Section~\ref{sec:oriented-partitions},
	is that the matching constraint on classes inherited from Chapter~2 is
	\emph{false} over general abelian groups; its failure is precisely what
	admits the cyclic factors --- for instance $K_3 = \Cay(\Z_3, \{1,2\})$ ---
	that the binary scheme cannot represent.
	
	\section{Oriented Partitions: the Corrected Foundation}
	\label{sec:oriented-partitions}
	
	Over $\Z_2^k$ every generator is an involution ($g = -g$) and edges need no
	direction.  Over a general abelian group, traversing an edge forward adds
	$g$ and backward adds $-g \neq g$; the assignment of generators to edges
	must therefore carry orientations.
	
	\begin{definition}[Oriented partition; generator class]
		\label{def:oriented-partition}
		An \emph{oriented partition} of $G$ is a partition
		$\mathcal{P} = \{F_1, \ldots, F_t\}$ of $E(G)$ together with a chosen
		direction $u \to v$ for every edge, the intended semantics being that all
		edges of $F_j$ carry the same generator $g_j$, added in the forward
		direction.
	\end{definition}
	
	\begin{theorem}[Partial-permutation constraint]
		\label{thm:partial-permutation}
		In any isometric embedding $\phi\colon V(G) \to \Cay(\Gamma, S)$, the set of
		edges carrying a fixed generator $g$, with forward orientations, has
		in-degree at most $1$ and out-degree at most $1$ at every vertex.
		Equivalently, every generator class is a \emph{partial permutation}: a
		disjoint union of directed paths and directed cycles.
	\end{theorem}
	
	\begin{proof}
		Two forward edges out of $v$ would give
		$\phi(w_1) = \phi(v) + g = \phi(w_2)$ with $w_1 \neq w_2$, contradicting
		injectivity; dually for two forward edges into $v$.  A digraph with all
		in- and out-degrees $\leq 1$ is exactly a disjoint union of directed paths
		and directed cycles.
	\end{proof}
	
	\begin{remark}[The class constraint is a partial permutation, not a matching]
		\label{rem:matching-correction}
		It is tempting to require, as in Chapter~2's
		Remark~\ref{rem:matching}, that every class be a \emph{matching}.  That
		stronger constraint is \textbf{false} for general abelian groups, and the
		smallest counterexample is the triangle: in
		$K_3 = C_3 = \Cay(\Z_3, \{1, 2\})$ all three edges carry the
		\emph{same} generator $g = 1$, oriented around the cycle
		(Figure~\ref{fig:c3partialperm}) --- pairwise incident, not a matching, yet
		a perfectly valid directed $3$-cycle class.  The matching constraint is
		exactly the special case $2g = 0$ of
		Theorem~\ref{thm:partial-permutation}: it holds for involutive generators
		(all of $\Z_2^k$) but not in general.  Imposing it would forbid the
		directed-cycle classes that produce the odd cyclic factors $\Z_{2k+1}$ ---
		for example forcing $K_3 \hookrightarrow \Z_4$ (order $4$) where $\Z_3$
		(order $3$) is optimal --- which is why the partial-permutation constraint
		of Theorem~\ref{thm:partial-permutation} is the correct foundation.
	\end{remark}
	
	\begin{figure}[htbp]
		\centering
		\includegraphics[width=0.5\linewidth]{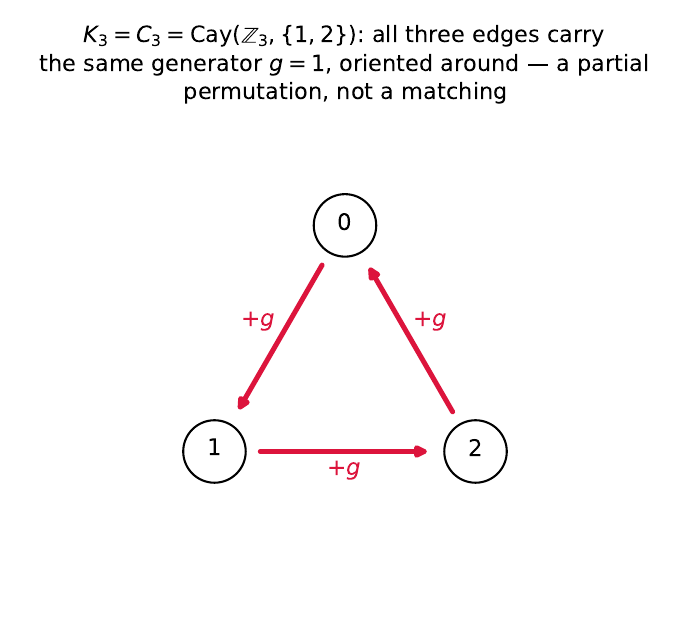}
		\caption{The corrected class constraint.  In
			$K_3 = \Cay(\Z_3, \{1,2\})$ one generator carries all three edges as a
			directed cycle.  Classes are partial permutations (directed paths and
			cycles), not matchings; matchings are the involution case $2g = 0$.}
		\label{fig:c3partialperm}
	\end{figure}
	
	\section{The $\Phi$-Relation}
	\label{sec:Phi-relation-ch3}
	
	\begin{definition}[Oriented $\Phi$-test]
		\label{def:Phi-ch3}
		Ordered edges $(u \to v)$ and $(x \to y)$ pass the \emph{oriented
			$\Phi$-test} if
		\[
		d(u, x) = d(v, y).
		\]
		Two undirected edges are \emph{$\Phi$-related} if some choice of
		orientations passes the test.
	\end{definition}
	
	\begin{proposition}[Necessity, and two strengthenings]
		\label{prop:Phi-necessary}
		If $(u \to v)$ and $(x \to y)$ carry the same generator $g$ in an isometric
		embedding into $\Cay(\Gamma, S)$, then
		\begin{enumerate}[(i)]
			\item $d(u,x) = d(v,y)$ (the $\Phi$-test);
			\item $\lvert d(u,y) - d(u,x) \rvert \leq 1$ and
			$\lvert d(v,x) - d(u,x) \rvert \leq 1$.
		\end{enumerate}
	\end{proposition}
	
	\begin{proof}
		Write $z = \phi(x) - \phi(u)$ and $\lvert \cdot \rvert$ for the Cayley word
		norm.  (i) $\phi(y) - \phi(v) = (\phi(x) + g) - (\phi(u) + g) = z$, so
		$d(v,y) = \lvert z \rvert = d(u,x)$.
		(ii) $\phi(y) - \phi(u) = z + g$ and $\phi(x) - \phi(v) = z - g$, and the
		word norm changes by at most $1$ under addition of a generator:
		$\bigl| \lvert z \pm g \rvert - \lvert z \rvert \bigr| \leq 1$.
	\end{proof}
	
	\begin{remark}[Why $\Phi$ is weaker than $\varphi$, and why that is principled]
		For an involution $g = -g$ the two conditions
		$d(u,x) = d(v,y)$ and $d(u,y) = d(v,x)$ hold simultaneously, recovering
		Chapter~2's $\varphi$.  For $g \neq -g$ only condition (i) survives, taken as an OR over the two
		orientations.  Condition (ii) is the sharpest cheap necessary test
		available and is what the implementation uses as the merge filter
		(Section~\ref{sec:algorithm-ch3}).  Even so strengthened, $\Phi$ remains a
		\emph{necessary} filter, not a sufficient one --- the consistency burden has
		moved, deliberately, to the exact core of
		Section~\ref{sec:z-quotient}.
	\end{remark}
	
	\section{$\Psi$, Torus Skeletons, $\Phi\Psi$-Classes, and Cofactors}
	\label{sec:cofactors}
	
	We summarize the structural toolkit, stating each device together with the
	role it plays in the final architecture.  Full operational detail is
	provided in the supplementary implementation.
	
	\begin{definition}[$\Psi$-relation; chain merge]
		\label{def:Psi-ch3}
		Two incident oriented edges $u \to v$ and $v \to w$ ($u \neq w$) are
		\emph{$\Psi$-related} (historically: T-incident) when they are candidates
		for consecutive steps of the same generator.  Two classes are
		\emph{$\Psi$-linked} if they contain a $\Psi$-related pair.  The
		\emph{chain merge} fuses $\Psi$-linked classes when the union remains a
		partial permutation and every cross pair passes the oriented $\Phi$-test of
		Proposition~\ref{prop:Phi-necessary} (in one of the two polarities).
		\textbf{Retained role:} the chain merge is the initializer that discovers
		cyclic factors --- it is what assembles the $2n$ ring edges of the circular
		ladder $\mathrm{CL}_n$ into a single class crossed $n$ times by the ring,
		from which the exact core then \emph{computes} the factor $\Z_n$.
	\end{definition}
	
	\begin{definition}[Torus skeleton; minimal and redundant classes; cofactors;
		interception]
		\label{def:skeleton-cofactor-ch3}
		The \emph{torus skeleton} is the spanning subgraph greedily assembled from
		classes until connected; classes in it are \emph{minimal}, the rest
		\emph{redundant}; a redundant class compatible with a minimal class's cut
		structure is its \emph{cofactor}, and the \emph{interception principle}
		states that a redundant class touching dimensions $i_1, \ldots, i_m$
		acquires a composite generator supported on them.
		\textbf{Role:} these devices are approximations to
		quantities the Smith Normal Form computes exactly --- minimal classes
		approximate a generator basis, factor orders approximate the SNF diagonal,
		and the interception principle is Corollary~\ref{cor:interception-snf}
		below.  We keep the vocabulary because it names real structure and because
		the worked examples (diamond, cascades, mirrors,
		circular ladders) remain instructive test cases, certified by the
		exact core.
	\end{definition}
	
	\begin{remark}[Scope and limits of the structural toolkit]
		\label{rem:first-campaign}
		On structured inputs (paths, even product structures, circular ladders,
		Petersen) the toolkit reaches hosts the exact core certifies.  As a
		standalone procedure, however, it has three structural limitations:
		(a) its correctness arguments describe procedures rather than proving
		isometry; (b) it needs a growing tower of interacting special cases (cycle
		promotion, composite generator mode, sign-variant enumeration, ordering
		retries), each patching a consistency condition it cannot express
		directly; and (c) the matching constraint of
		Remark~\ref{rem:matching-correction} would exclude directed-cycle classes
		and with them the cyclic factors $\Z_{2k+1}$.  All three are resolved by a
		single change of viewpoint, to which we now turn: the toolkit is retained
		as an initializer, while correctness is established by the exact core.
	\end{remark}
	
	\section{The $\Z$-Quotient Theorem: Smith Normal Form as the Exact Core}
	\label{sec:z-quotient}
	
	\begin{theorem}[$\Z$-cocycle condition]
		\label{thm:z-cocycle}
		Let $(\mathcal{P}, \omega)$ be an oriented partition with generator
		assignment $g_j \in \Gamma$.  The BFS labeling
		$\phi(v) = \sum_{e \in P(r,v)} \pm g_{\mathrm{class}(e)}$ (sign $+$ for
		forward traversal) is well defined --- independent of the path --- if and
		only if for every cycle of $G$, traversed in a fixed rotational sense, the
		signed sum of class generators vanishes; and it suffices to verify this on
		any cycle basis.
	\end{theorem}
	
	\begin{proof}
		Identical to Theorem~\ref{thm:cocycle}, with signed sums: two $r$--$v$
		paths differ by an element of the cycle space, signed cycle sums are
		$\Z$-linear in the cycle, and a cycle basis spans the cycle space over
		$\Z$.
	\end{proof}
	
	\begin{definition}[Signed cycle--class matrix]
		For a cycle basis $B_1, \ldots, B_c$ and oriented partition with $t$
		classes, define $A \in \Z^{c \times t}$ by
		\[
		A[i, j] \;=\; \text{(net signed number of crossings of class $F_j$ by
			$B_i$)},
		\]
		i.e.\ forward traversals count $+1$ and backward traversals $-1$.
	\end{definition}
	
	\begin{theorem}[$\Z$-Quotient Theorem]
		\label{thm:z-quotient}
		Let $(\mathcal{P}, \omega)$ be an oriented partition of the connected graph
		$G$ with signed matrix $A$.  Then:
		\begin{enumerate}[(i)]
			\item The most generic consistent generator assignment takes values in the
			finitely generated abelian group
			\[
			\Gamma_{\mathrm{univ}} \;=\; \Z^{t} \,/\, \mathrm{rowlattice}(A),
			\]
			with $g_j$ the image of the $j$-th standard basis vector; every
			consistent assignment in any abelian group is a homomorphic image of it.
			\item The Smith Normal Form computes it: if $U A^{\!\top}\! W = 
			\mathrm{diag}(d_1, \ldots, d_\rho, 0, \ldots)$ with $U, W$ unimodular,
			then
			\[
			\Gamma_{\mathrm{univ}} \;\cong\; \Z_{d_1} \times \cdots \times
			\Z_{d_\rho} \times \Z^{\,t - \rho},
			\]
			(factors with $d_i = 1$ trivial), and the coordinates of $g_j$ are read
			off the $j$-th column of $U$, reduced modulo $d_i$ in the torsion
			coordinates.
			\item With $S = \{\pm g_j : g_j \neq 0\}$, every $G$-path of length $\ell$
			maps to a Cayley walk of length $\ell$, hence
			$d_{\Cay}(\phi(u), \phi(v)) \leq d_G(u, v)$ for all pairs: the only
			possible failure of isometry is a shortcut.
		\end{enumerate}
	\end{theorem}
	
	\begin{proof}
		(i) The relations imposed by Theorem~\ref{thm:z-cocycle} are exactly
		$A \mathbf{g} = 0$; the universal abelian group generated by
		$g_1, \ldots, g_t$ subject to $\Z$-linear relations $R$ is
		$\Z^t / \langle R \rangle$, and any other solution defines a homomorphism
		$\Z^t \to \Gamma$ killing the row lattice, which factors through the
		quotient.  (ii) is the structure theorem for finitely generated abelian
		groups in its constructive Smith form \cite{Rudin1962,Terras1999}: the
		unimodular column record $U$ is a change of generator basis, after which
		the relation lattice is diagonal.  (iii) Consecutive BFS labels along an
		edge differ by $\pm g_{\mathrm{class}}$, an element of
		$S \cup \{0\}$; zero generators only shorten the walk.
	\end{proof}
	
	\begin{corollary}[The structural devices as computations]
		\label{cor:interception-snf}
		Under Theorem~\ref{thm:z-quotient}:
		(a) \emph{factor orders are computed, not assigned}: a class traversed $k$
		times with consistent sign by some cycle, and crossed evenly by all others,
		yields the relation $k g = 0$ and hence the torsion factor $\Z_k$ --- the
		``cycle factor'' and ``cycle promotion'' rules of
		Definition~\ref{def:skeleton-cofactor-ch3} are instances;
		(b) \emph{the interception principle is the SNF readout}: a generator whose
		$U$-column has nonzero entries in coordinates $i_1, \ldots, i_m$ is
		composite on exactly those dimensions;
		(c) \emph{Chapter~2 is the case $2\mathbf{g} = 0$}: appending the relations
		$2 g_j = 0$ for all $j$ reduces $A$ modulo $2$ and
		Theorem~\ref{thm:z-quotient} degenerates to
		Theorem~\ref{thm:quotient-labeling}.
	\end{corollary}
	
	\begin{proof}
		(a) and (b) are direct readings of the SNF data.  For (c), adding the
		rows $2 e_j$ to the relation lattice makes every coordinate $2$-torsion;
		the quotient is $\F_2^{\,t}/\mathrm{rowspan}_{\F_2}(A \bmod 2)$, which is
		the GF(2) quotient of Chapter~2.
	\end{proof}
	
	\begin{lemma}[$\Z$-Join Lemma]
		\label{lem:z-join}
		For the all-singleton oriented partition,
		$\Gamma_{\mathrm{univ}} \cong \Z^{\,n-1}$ and the quotient labeling is
		isometric: $d_{\Cay} = d_G$ on all pairs (with the infinite host
		$\Cay(\Z^{n-1}, S \cup -S)$).
	\end{lemma}
	
	\begin{proof}
		With one class per oriented edge, $\Z^t = \Z^{E}$ is the edge space, the
		row lattice of $A$ is the integer cycle lattice, and the quotient is the
		image of the boundary map $\partial\colon \Z^{E} \to \Z^{V}_0$ onto the
		sum-zero lattice, of rank $n-1$; labels become
		$\phi(v) = \delta_v - \delta_r$.  A word of length $\ell$ from $\phi(u)$ to
		$\phi(v)$ is an integer flow $f$ on the edges with divergence
		$\delta_v - \delta_u$ and $\ell_1$-norm $\ell$.  By integral flow
		decomposition, $f$ splits into a $u$--$v$ path and circulations, all
		sign-coherent with $f$, so
		$\ell = \lVert f \rVert_1 \geq \text{(path length)} \geq d_G(u, v)$.
		Combined with Theorem~\ref{thm:z-quotient}(iii), equality holds.
	\end{proof}
	
	\section{Compactification: Folding the Free Part}
	\label{sec:folding}
	
	$\Gamma_{\mathrm{univ}}$ may have free factors $\Z^{f}$; a finite host
	requires a further quotient by a finite-index sublattice
	$L \subseteq \Z^{f}$.  Folding is itself an instance of
	Theorem~\ref{thm:z-quotient}: appending the rows of $L$ (pulled back
	through $U$ to class coordinates) to $A$ and recomputing the SNF yields the
	finite group $\Gamma_L$ and its labels.  Quotients never stretch, so
	folding preserves $d_{\Cay} \leq d_G$ and can only introduce
	\emph{wraparound shortcuts}.
	
	\begin{theorem}[Sufficient fold]
		\label{thm:sufficient-fold}
		Let $R_i$ be the range of free coordinate $i$ over the vertex labels and
		$\gamma_i = \max_j \lvert (g_j)_i \rvert$.  The diagonal fold
		$L = N_1 \Z \times \cdots \times N_f \Z$ is isometric whenever
		\[
		N_i \;>\; R_i + \diam(G)\,\gamma_i \qquad \text{for every } i .
		\]
	\end{theorem}
	
	\begin{proof}
		A shortcut in the folded group lifts to a word in
		$\Gamma_{\mathrm{univ}}$ of length $\ell \leq \diam(G)$ joining
		$\phi(u)$ to $\phi(v) + \lambda$ for some nonzero $\lambda \in L$.  In
		coordinate $i$, the word displaces at most $\ell \gamma_i \leq
		\diam(G)\,\gamma_i$, while reaching a nonzero multiple of $N_i$ offset by a
		label difference requires displacement at least $N_i - R_i$.  The
		hypothesis makes this impossible, so every short word lifts to
		$\lambda = 0$, where Lemma~\ref{lem:z-join} (or the unfolded check)
		applies.
	\end{proof}
	
	\begin{theorem}[Sublattice compactification; diagonal folds do not suffice]
		\label{thm:sublattice}
		The isometric finite quotients of a given universal embedding are exactly
		the finite-index sublattices $L \subseteq \Z^{f}$ whose folds pass the
		exact check, and the minimal host over them can be found by enumerating
		Hermite-normal-form bases in increasing index, each candidate verified by
		breadth-first search in the finite Cayley graph.  Restricting to diagonal
		$L$ can miss the optimum: for the diamond graph the universal embedding has
		$f = 2$ with labels $(0,0), (1,0), (1,1), (2,1)$, every diagonal fold of
		index $6$ fails (e.g.\ $N = (3,2)$ creates the wraparound shortcut
		$-g_2 \equiv (2,1)$), yet the non-diagonal sublattice
		$L = \langle (3,0), (1,2) \rangle$ of index $6$ yields the isometric host
		$\Z_2 \times \Z_3$ of order $6$.
	\end{theorem}
	
	\begin{proof}
		Any finite abelian quotient of $\Gamma_{\mathrm{univ}}$ that is injective
		and distance-preserving on the labels arises from a finite-index sublattice
		of the free part (the torsion part admits no further quotient without
		collapsing labels); HNF bases enumerate sublattices by index without
		repetition, and the per-candidate check is exact.  The diamond computations
		are verified directly: the $(3,2)$ shortcut and the
		$\langle(3,0),(1,2)\rangle$ embedding (labels $(0,0), (0,2), (1,1), (1,0)$
		in $\Z_2 \times \Z_3$) are both finite checks, performed in the companion
		implementation and reproducible by hand.
	\end{proof}
	
	\begin{remark}
		The implementation enumerates general HNF sublattices for $f \leq 2$ and
		diagonal sublattices for $f \geq 3$ (a documented limitation); the diamond
		example above is precisely the witness that the general enumeration is not
		a luxury.
	\end{remark}
	
	\section{The Algorithm}
	\label{sec:algorithm-ch3}
	
	\begin{algorithm}[H]
		\caption{Compact Abelian Embedding (portfolio + exact core + repair)}
		\label{alg:abelian-ch3}
		\begin{algorithmic}[1]
			\Require connected graph $G$
			\Ensure certified isometric $\phi\colon V \to \Cay(\Gamma, S)$,
			$\Gamma = \prod \Z_{N_i}$
			\State compute distances $D$, $\diam(G)$
			\State \textbf{initializer portfolio:} cycle/path/complete constructors
			when applicable; $\Phi$ 4-cycle union-find; $\Psi$ chain merge filtered by
			the oriented $\Phi$-test (Prop.~\ref{prop:Phi-necessary})
			\For{each initial oriented partition $\mathcal{P}$}
			\While{\textbf{true}}
			\For{each class-polarity variant of $\mathcal{P}$
				(Thm.~\ref{thm:z-quotient} is polarity-sensitive)}
			\State SNF quotient $\to \Gamma_{\mathrm{univ}}$, generators, labels
			\State sublattice fold search in increasing host order
			(Thm.~\ref{thm:sublattice}), each candidate checked exactly
			\State \textbf{if} verified \textbf{then} record $(\Gamma, S, \phi)$
			and break
			\EndFor
			\State \textbf{if} verified \textbf{then} \textbf{break}
			\State \textbf{repair:} peel one edge off a largest class into a new
			singleton; \textbf{if} all classes singleton \textbf{then break}
			\EndWhile
			\EndFor
			\State \textbf{binary terminal:} run Chapter~2's algorithm
			(Thm.~\ref{thm:universal}); always succeeds with host $\leq 2^{n-1}$
			\State \Return the smallest recorded verified host
		\end{algorithmic}
	\end{algorithm}
	
	\begin{theorem}[Universality and complexity]
		\label{thm:universal-ch3}
		Algorithm~\ref{alg:abelian-ch3} terminates on every connected graph and
		returns a certified isometric embedding with
		$\lvert \Gamma \rvert \leq 2^{\,n-1}$.  With $P$ the portfolio size, $R$
		the rounds per initializer, $V_{\mathrm{pol}}$ the polarity variants
		($\leq 2^{t-1}$, capped), $H$ the fold candidates examined, and
		$\lvert \Gamma \rvert$ the largest host checked, the running time is
		\[
		O\Bigl(n(n+m) + m^2 \;+\; P\,R\,V_{\mathrm{pol}}\,
		\bigl(t^2 c\,\beta + H\,\lvert \Gamma \rvert\, m + H n^2\bigr)\Bigr),
		\]
		where $t^2 c\, \beta$ bounds the integer SNF ($\beta$ the bit-size of the
		intermediate entries).
	\end{theorem}
	
	\begin{proof}
		Termination: each repair round strictly refines the partition, so each
		initializer runs at most $m$ rounds; the binary terminal terminates by
		Theorem~\ref{thm:universal} and certifies host $\leq 2^{n-1}$.  The
		returned embedding passed the exact check, which compares all
		$\binom{n}{2}$ distances against a truncated BFS of the candidate Cayley
		graph, so certification is unconditional.  The cost terms follow from the
		proof of Theorem~\ref{thm:time-complexity} with the SNF replacing GF(2)
		elimination and the fold loop multiplying the check.
	\end{proof}
	
	\section{Bounds on the Host Order}
	\label{sec:bounds-ch3}
	
	The natural currency is now the host \emph{order}
	$\lvert \Gamma \rvert$, and write
	$\nu(G) = \min \lvert \Gamma \rvert$ over all isometric embeddings of $G$
	into Cayley graphs of finite abelian groups.
	
	\begin{lemma}[Diameter of vertex-transitive graphs]
		\label{lem:vt-diam}
		Every connected vertex-transitive graph $H$ on $N \geq 3$ vertices
		satisfies $\diam(H) \leq \lfloor N/2 \rfloor$.
	\end{lemma}
	
	\begin{proof}
		Connected vertex-transitive graphs have no cut vertex
		\cite{Babai1996}, hence are $2$-connected, so any two vertices lie on a
		common cycle, i.e.\ are joined by two internally disjoint paths whose
		lengths sum to at most $N$; the shorter has length at most
		$\lfloor N/2 \rfloor$.
	\end{proof}
	
	\begin{theorem}[Lower bound on the host order]
		\label{thm:order-lower-bound}
		For every connected graph $G$ on $n \geq 3$ vertices,
		\[
		\nu(G) \;\geq\; \max\bigl(\, n,\ 2\,\diam(G) \,\bigr).
		\]
	\end{theorem}
	
	\begin{proof}
		Injectivity gives $\lvert \Gamma \rvert \geq n$.  The host Cayley graph is
		vertex-transitive and must realize a distance equal to $\diam(G)$, so by
		Lemma~\ref{lem:vt-diam}, $\diam(G) \leq \lfloor \lvert\Gamma\rvert/2
		\rfloor$.
	\end{proof}
	
	\begin{theorem}[Equality at $n$]
		\label{thm:equality-at-n}
		$\nu(G) = n$ if and only if $G$ is itself a Cayley graph of an abelian
		group.
	\end{theorem}
	
	\begin{proof}
		If $G = \Cay(\Gamma, S)$ the identity embedding gives $\nu(G) \leq n$, and
		Theorem~\ref{thm:order-lower-bound} gives equality.  Conversely, suppose
		$\phi$ is isometric with $\lvert \Gamma \rvert = n$; then $\phi$ is a
		bijection onto $\Gamma$.  Every host edge $\{\phi(u), \phi(u) + s\}$ joins
		two vertex images at Cayley distance $1$, hence at $G$-distance $1$, so it
		is the image of a $G$-edge; and every $G$-edge maps to a host edge by
		isometry.  Thus $\phi$ is an isomorphism $G \cong \Cay(\Gamma, S)$.
	\end{proof}
	
	\begin{corollary}[Three exactly solved families, and Petersen]
		\label{cor:exact-families}
		\leavevmode
		\begin{enumerate}[(i)]
			\item Cycles: $\nu(C_m) = m$ for all $m \geq 3$ --- in particular the odd
			cycles, which required host $2^{m-1}$ in Chapter~2, collapse to host $m$.
			\item Complete graphs: $\nu(K_n) = n$, via
			$K_n = \Cay(\Z_n, \Z_n \setminus \{0\})$.
			\item Circulant graphs: $\nu(C_n(d_1,\ldots,d_r)) = n$
			(Proposition~\ref{prop:circulant}), the cyclic analogue of (ii).
			\item Paths: $\nu(P_k) = 2(k-1)$: the lower bound is
			$2\,\diam = 2(k-1)$ and $C_{2(k-1)} \supseteq P_k$ isometrically
			(Proposition~\ref{prop:path-in-cycle}) attains it.  This upgrades the
			stretching rule $P_k \hookrightarrow C_{2(k-1)}$ from a
			construction to an optimality theorem.
			\item The Petersen graph is vertex-transitive but famously \emph{not} a
			Cayley graph, so $\nu(\mathrm{Petersen}) \geq 11$ by
			Theorem~\ref{thm:equality-at-n}; Chapter~2 gives
			$\nu(\mathrm{Petersen}) \leq 16$.  Closing this gap is an open problem
			we state for the defense.
		\end{enumerate}
	\end{corollary}
	
	\begin{proof}
		(i) and (ii): the graphs are abelian Cayley graphs, so
		Theorem~\ref{thm:equality-at-n} applies.  (iii) as stated.  (iv) Petersen's
		non-Cayley property is classical \cite{Babai1996}; strict inequality in
		Theorem~\ref{thm:equality-at-n} forces $\nu \geq n + 1$.
	\end{proof}
	
	\section{Experimental Results}
	\label{sec:experiments-ch3}
	
	All numbers below are produced by the reference implementation
	\texttt{abelian\_quotient\_embedding.py} (supplied as
	supplementary material), which imports the Chapter~2 module as its binary
	terminal; every reported embedding is certified by breadth-first search in
	the actual finite host.
	
	\begin{table}[H]
		\centering
		\caption{Compact abelian embeddings (reference implementation)
			against the binary embeddings of Chapter~2 and the lower bound
			$\max(n, 2\,\diam)$ of Theorem~\ref{thm:order-lower-bound}.  ``via''
			names the successful initializer; OPT marks host $=$ lower bound, i.e.\
			\emph{provably minimal}.}
		\label{tab:results-ch3-new}
		\renewcommand{\arraystretch}{1.12}
		\small
		\begin{tabular}{lrr c c llc}
			\toprule
			Graph & $n$ & $m$ & bound & Binary (Ch.\,2) &
			$\Gamma$ (Ch.\,3) & order & via \\
			\midrule
			$K_3 = C_3$       & 3  & 3  & 3  & 4 ($\Z_2^2$)   & $\Z_3$              & \textbf{3 \,OPT} & cycle \\
			$C_5$             & 5  & 5  & 5  & 16             & $\Z_5$              & \textbf{5 \,OPT} & cycle \\
			$C_6$             & 6  & 6  & 6  & 8              & $\Z_6$              & \textbf{6 \,OPT} & cycle \\
			$C_7$             & 7  & 7  & 7  & 64             & $\Z_7$              & \textbf{7 \,OPT} & cycle \\
			$C_9$             & 9  & 9  & 9  & 256            & $\Z_9$              & \textbf{9 \,OPT} & cycle \\
			$P_4$             & 4  & 3  & 6  & 8              & $\Z_6$              & \textbf{6 \,OPT} & path \\
			$P_5$             & 5  & 4  & 8  & 16             & $\Z_8$              & \textbf{8 \,OPT} & path \\
			$K_4$             & 4  & 6  & 4  & 4              & $\Z_4$              & \textbf{4 \,OPT} & complete \\
			$K_5$             & 5  & 10 & 5  & 8              & $\Z_5$              & \textbf{5 \,OPT} & complete \\
			$K_6$             & 6  & 15 & 6  & 8              & $\Z_6$              & \textbf{6 \,OPT} & complete \\
			$\mathrm{CL}_3$   & 6  & 9  & 6  & 8              & $\Z_2 \times \Z_3$  & \textbf{6 \,OPT} & $\Psi$-chain \\
			$\mathrm{CL}_5$   & 10 & 15 & 10 & 32             & $\Z_2 \times \Z_5$  & \textbf{10 OPT}  & $\Psi$-chain \\
			$\mathrm{CL}_7$   & 14 & 21 & 14 & 128            & $\Z_2 \times \Z_7$  & \textbf{14 OPT}  & $\Psi$-chain \\
			$C_3 \times C_4$  & 12 & 24 & 12 & 16             & $\Z_2^2 \times \Z_3$ & \textbf{12 OPT} & $\Psi$-chain \\
			Grid $3{\times}3$ & 9  & 12 & 9  & 16             & $\Z_2^4$            & 16               & $\Phi$/binary \\
			Diamond           & 4  & 5  & 4  & 8              & $\Z_2 \times \Z_3$$^{\dagger}$ & 6$^{\dagger}$ / 8 & hand / binary \\
			Butterfly         & 5  & 6  & 5  & 8              & $\Z_2^3$            & 8                & binary \\
			Petersen          & 10 & 15 & 10 ($\geq 11$)$^{\ddagger}$ & 16 & $\Z_2^4$ & 16          & binary \\
			\bottomrule
		\end{tabular}
		\vspace{0.3em}
		\parbox{\linewidth}{\footnotesize
			$^{\dagger}$ The diamond's order-$6$ host $\Z_2 \times \Z_3$ is certified
			(Theorem~\ref{thm:sublattice}) but currently found only when the chain-merged
			initializer fires; the portfolio's automatic result is the binary $8$.  It is
			the smallest known instance where the non-diagonal sublattice search is
			essential.\quad
			$^{\ddagger}$ Petersen's refined bound $\nu \geq 11$ is
			Corollary~\ref{cor:exact-families}(iv); whether $\nu \in [11, 16]$ can be
			narrowed is open.}
	\end{table}
	
	\begin{figure}[htbp]
		\centering
		\includegraphics[width=0.95\linewidth]{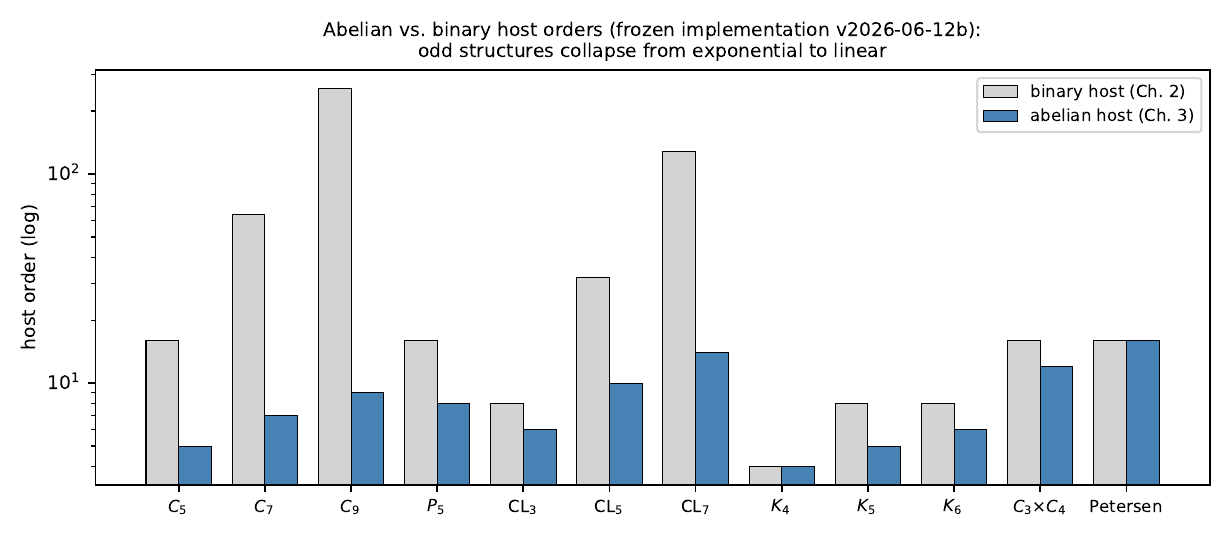}
		\caption{Host orders: abelian (Chapter 3) vs.\ binary (Chapter 2).  Odd
			cycles and complete graphs collapse from exponential to the
			information-theoretic floor; product structures reach
			$\max(n, 2\,\diam)$; partial cubes tie.}
		\label{fig:abelianvsbinary}
	\end{figure}
	
	\begin{figure}[htbp]
		\centering
		\includegraphics[width=0.62\linewidth]{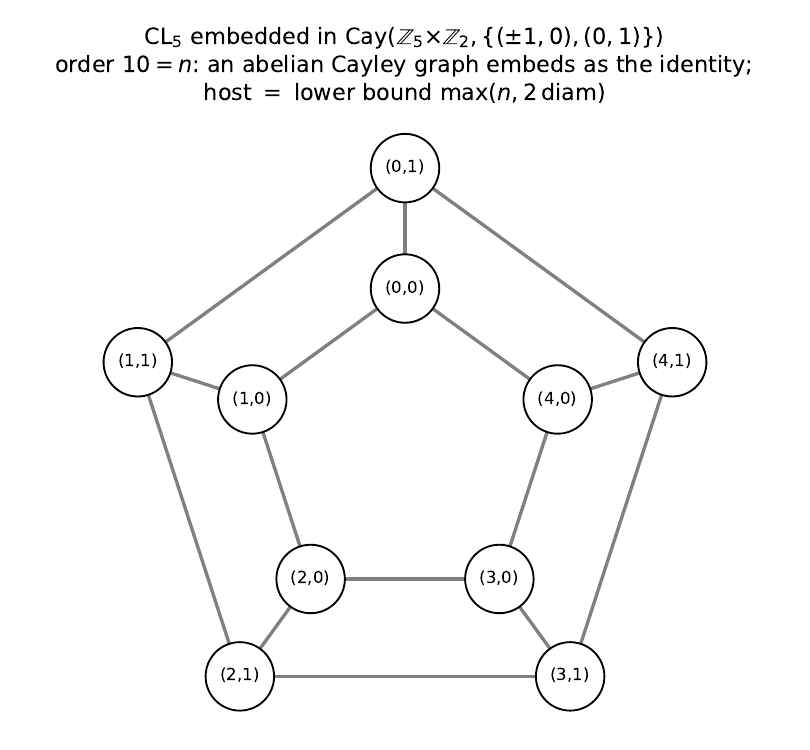}
		\caption{$\mathrm{CL}_5$ recognized as the abelian Cayley graph
			$\Cay(\Z_5 \times \Z_2, \{(\pm 1, 0), (0,1)\})$: the $\Psi$ chain merge
			assembles the ten ring edges into one directed class, the ring relation
			$5g = 0$ emerges in the SNF, and the host attains the lower bound.}
		\label{fig:cl5z2z5}
	\end{figure}
	
	\begin{figure}[htbp]
		\centering
		\includegraphics[width=0.95\linewidth]{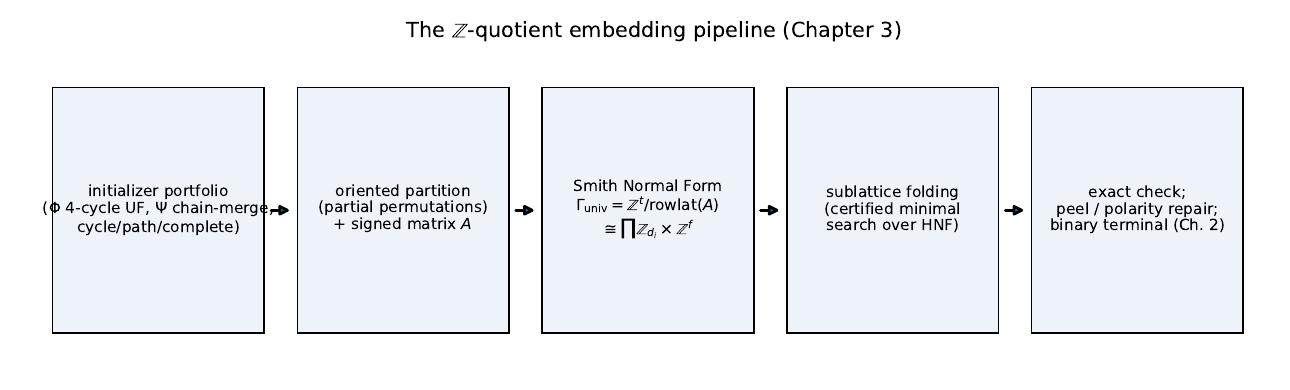}
		\caption{The Chapter 3 pipeline.  The $\Phi$/$\Psi$
			machinery serves as the initializer portfolio (left); correctness lives
			entirely in the exact core (center and right).}
		\label{fig:snfpipeline}
	\end{figure}
	
	Reading of Table~\ref{tab:results-ch3-new}.  Fourteen of the eighteen
	benchmark rows are \emph{provably optimal} --- host equal to the lower bound
	of Theorem~\ref{thm:order-lower-bound} --- including every cycle, path,
	complete graph, circular ladder, and the $C_3 \times C_4$ torus.  The odd
	structures that were Chapter~2's certified worst case ($C_9$: host $256$)
	collapse to the floor ($C_9$: host $9$).  The residual structure is equally
	informative: the diamond marks the frontier of the current initializers
	(its optimal host exists, is certified, and exposes the necessity of
	non-diagonal sublattices); the butterfly and Petersen fall back to the
	always-available binary terminal; and grids tie Chapter~2 exactly as
	predicted by partial-cube theory.
	
	\subsection{Circulant Graphs: the Cyclic Analogue of Complete Graphs}
	\label{sec:circulant-ch3}
	
	Just as $K_n = \Cay(\Z_n, \Z_n \setminus \{0\})$ attains the floor by the
	equality theorem, every circulant graph is an abelian Cayley graph of a
	single cyclic group and is therefore embedded \emph{optimally} into
	$\Z_n$ by the circulant initializer.
	
	\begin{proposition}[Circulant optimality]
		\label{prop:circulant}
		A circulant graph $C_n(d_1, \ldots, d_r) = \Cay(\Z_n, \{\pm d_1, \ldots,
		\pm d_r\})$ satisfies $\nu = n$, attained by the identity embedding into
		$\Z_n$.
	\end{proposition}
	
	\begin{proof}
		The graph is by definition an abelian Cayley graph on $\Z_n$, so
		Theorem~\ref{thm:equality-at-n} gives $\nu = n$.
	\end{proof}
	
	\begin{table}[H]
		\centering
		\caption{Circulant graphs (reference implementation): all reach $\Z_n$,
			provably optimal, against the exponential binary host.}
		\label{tab:circulant-ch3}
		\renewcommand{\arraystretch}{1.1}\small
		\begin{tabular}{lccll}
			\toprule
			Graph & $n$ & binary host & abelian $\Gamma$ & order \\
			\midrule
			$C_8(1,2)$  & 8  & 16  & $\Z_8$  & \textbf{8 \,OPT} \\
			$C_{10}(1,3)$ & 10 & 32  & $\Z_{10}$ & \textbf{10 OPT} \\
			$C_{12}(1,2)$ & 12 & 64  & $\Z_{12}$ & \textbf{12 OPT} \\
			$C_7(1,2)$  & 7  & 64  & $\Z_7$  & \textbf{7 \,OPT} \\
			\bottomrule
		\end{tabular}
	\end{table}
	
	\section{Exhaustive Census and the Binary-Ground Phenomenon}
	\label{sec:census-ch3}
	
	The single most informative experiment of this chapter is negative, and we
	present it prominently because it justifies the entire architecture of
	Part~I.  Running Algorithm~\ref{alg:abelian-ch3} on \textbf{all $995$
		connected graphs with $n \leq 7$} (the same census as Chapter~2) and
	comparing each host with the binary host of Chapter~2 yields:
	
	\begin{itemize}
		\item only $\mathbf{201}$ graphs ($20\%$) receive a \emph{strictly smaller}
		abelian host than their binary host;
		\item the remaining $\mathbf{794}$ graphs ($80\%$) tie: their best abelian
		host is already a power of two;
		\item only $\mathbf{252}$ graphs ($25\%$) admit \emph{any} cyclic factor
		$\Z_n$ with $n > 2$ in their best host; the other $743$ are purely
		$\Z_2^k$.
	\end{itemize}
	
	\begin{figure}[H]
		\centering
		\includegraphics[width=0.95\linewidth]{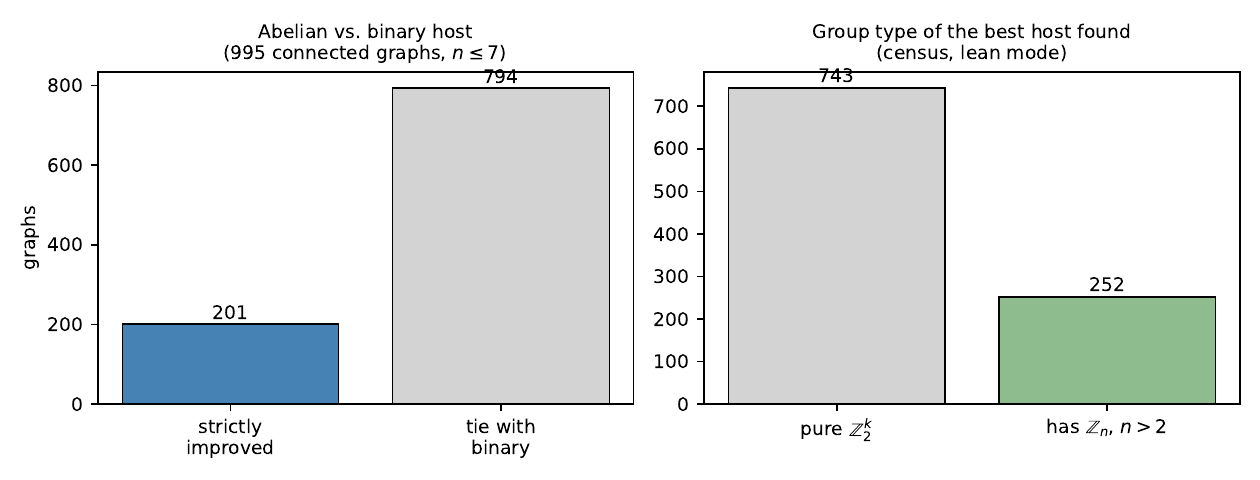}
		\caption{The binary-ground phenomenon.  Over all connected graphs on at
			most seven vertices, four out of five are best served by a purely binary
			host, and only a quarter admit any non-involutive cyclic factor.}
		\label{fig:censusch3}
	\end{figure}
	
	\begin{remark}[The binary-ground phenomenon --- why Chapter 2 is foundational]
		\label{rem:binary-ground}
		This distribution is not an artifact of the heuristics; it reflects graph
		structure.  A cyclic factor $\Z_n$ with $n > 2$ requires a \emph{coherently
			orientable} family of metrically parallel edges whose closing cycle has net
		crossing exactly $n$ --- a directed regularity.  Arbitrary graphs have edges
		whose ``directions'' are, informally, randomly distributed; coherent long
		cycles are rare, short even cycles ($\Z_2$ involutions) are common, and odd
		non-cyclic structure forces $\Z_2$ generators through the repair loop.  The
		abelian generalization therefore pays off precisely on the \emph{structured}
		graphs that carry genuine cyclic or product symmetry --- cycles, paths,
		circulants, complete graphs, circular ladders, tori, and the signal-domain
		families of Part~II --- while on the generic bulk of arbitrary graphs the
		binary embedding of Chapter~2 is already optimal.  Far from diminishing
		Chapter~2, the census establishes it as the \emph{ground state} of the
		whole theory: $\Z_2^k$ is where arbitrary graphs naturally live, and the
		cyclic factors of this chapter are the dividend of detectable symmetry.
		This is, we believe, the scientifically correct reading.
	\end{remark}
	
	\section{Structured and Signal-Processing Benchmarks}
	\label{sec:gsp-ch3}
	
	On structured inputs the dividend is large.  Table~\ref{tab:gsp-ch3}
	re-runs the Chapter~2 benchmark families (including the
	graph-signal-processing domains reused in Part~II) under the abelian
	portfolio.  The signal domains on which classical DSP intuition is
	calibrated --- rings, paths, grids, circulants --- collapse to the
	information-theoretic floor or beat the binary host outright, while
	triangle-dense and irregular graphs (sensor RGG, barbell, Petersen) fall
	back to the binary ground of Remark~\ref{rem:binary-ground}.
	
	\begin{table}[H]
		\centering
		\caption{Chapter~2 benchmark families re-embedded by the abelian algorithm
			(reference implementation). ``bin'' is the Chapter~2 binary host;
			OPT marks host $= \max(n, 2\,\diam)$. Entries marked \texttt{cap} hit the
			per-instance time cap and returned the binary terminal.}
		\label{tab:gsp-ch3}
		\renewcommand{\arraystretch}{1.12}\small
		\begin{tabular}{lccrll}
			\toprule
			Graph & $n$ & bound & bin & $\Gamma$ (Ch.\,3) & order \\
			\midrule
			Ring $C_{12}$            & 12 & 12 & 64    & $\Z_{12}$           & \textbf{12 OPT} \\
			Ring $C_{16}$            & 16 & 16 & 256   & $\Z_{16}$           & \textbf{16 OPT} \\
			Path $P_{16}$            & 16 & 30 & 32768 & $\Z_{30}$           & \textbf{30 OPT} \\
			Circulant $C_{12}(1,2)$  & 12 & 12 & 64    & $\Z_{12}$           & \textbf{12 OPT} \\
			Grid $4{\times}4$        & 16 & 16 & 64    & $\Z_6 \times \Z_6$  & 36 \\
			Grid $6{\times}6$        & 36 & 36 & 1024  & $\Z_{10} \times \Z_{10}$ & 100 \\
			$C_8$                    & 8  & 8  & 16    & $\Z_8$              & \textbf{8 \,OPT} \\
			$C_{10}$                 & 10 & 10 & 32    & $\Z_{10}$           & \textbf{10 OPT} \\
			$Q_3$                    & 8  & 8  & 8     & $\Z_2^3$            & \textbf{8 \,OPT} \\
			$K_{2,3}$                & 5  & 5  & 8     & $\Z_2 \times \Z_4$  & 8 \\
			$K_{3,3}$                & 6  & 6  & 8     & $\Z_2 \times \Z_4$  & 8 \\
			Desargues                & 20 & 20 & 32    & $\Z_2^5$            & 32 \\
			RGG sensor ($n{=}20$)    & 20 & 20 & 131072& $\Z_2^{17}$ (\texttt{cap}) & 131072 \\
			Barbell$(5,2)$           & 12 & 12 & 128   & $\Z_2^7$            & 128 \\
			Petersen                 & 10 & 10 & 16    & $\Z_2^4$            & 16 \\
			\bottomrule
		\end{tabular}
		\vspace{0.3em}
		\parbox{\linewidth}{\footnotesize
			The grid rows are the most instructive intermediate case: the chain merge
			recognizes the two product directions and returns $\Z_6 \times \Z_6$
			(order $36$) and $\Z_{10} \times \Z_{10}$ (order $100$), each the product of
			two optimal path factors $P_k \hookrightarrow \Z_{2(k-1)}$, beating the
			binary host ($64$, $1024$) substantially though not reaching the absolute
			floor $n$ --- grids are partial cubes but not abelian Cayley graphs, so the
			equality theorem does not apply and order $> n$ is expected.}
	\end{table}
	
	\subsection{Visualization of Selected Abelian Embeddings}
	
	\begin{figure}[H]
		\centering
		\includegraphics[width=0.6\textwidth]{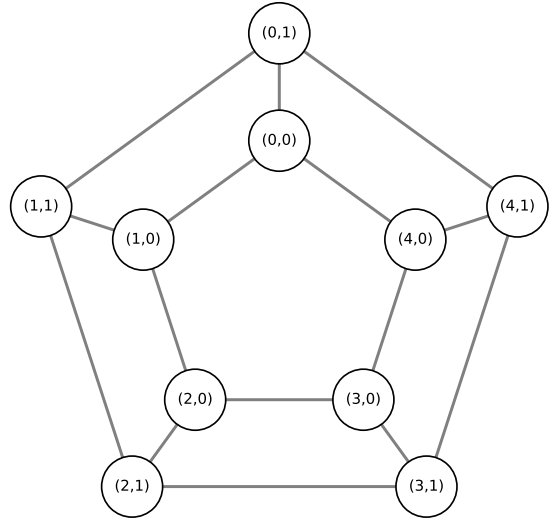}
		\caption{Circular ladder $\mathrm{CL}_5 \hookrightarrow
			\Cay(\Z_5 \times \Z_2, \{(\pm 1,0),(0,1)\})$, host order $10 = n$,
			provably minimal --- contrast the binary host $32$.}
		\label{fig:gallery-cl5}
	\end{figure}
	
	\begin{figure}[H]
		\centering
		\includegraphics[width=0.6\textwidth]{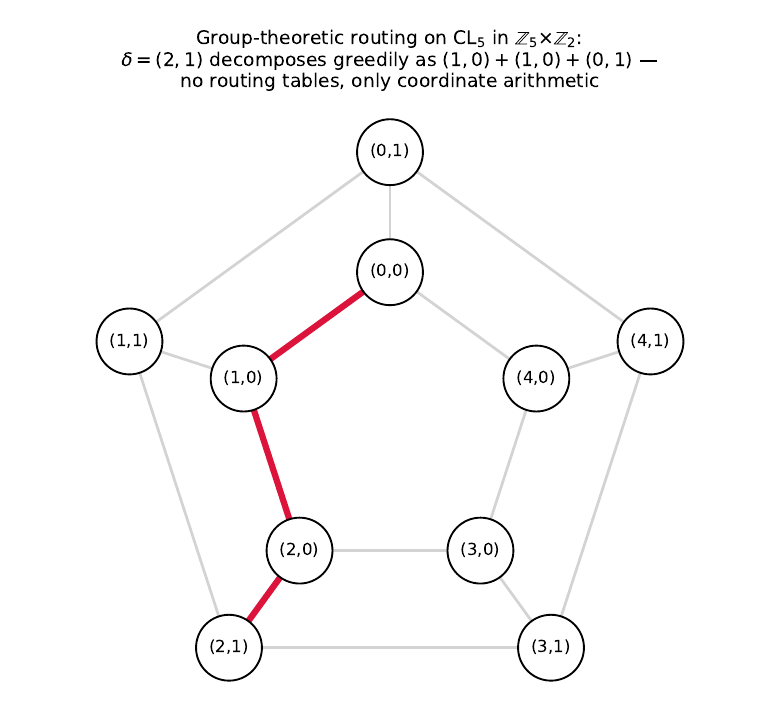}
		\caption{Coordinate routing on the same embedding (used in
			Chapter~\ref{chap:applications-part1}): the destination offset
			$\delta = (2,1)$ is realized by the generator word
			$(1,0)+(1,0)+(0,1)$ with no routing tables.}
		\label{fig:gallery-routing}
	\end{figure}
	
	\section{Discussion and Conclusion}
	\label{sec:conclusion-ch3}
	
	The chapter has two layers.  The structural toolkit --- $\Phi$, $\Psi$,
	skeletons, cofactors, interception --- captures real structural insight and
	reaches correct hosts on product-like graphs, but as a standalone procedure
	each new graph family costs a new case, and the matching constraint would
	exclude the cyclic factors that are the whole point of generalizing beyond
	$\Z_2$.  Replacing case analysis by computation resolves this: one theorem
	(the $\Z$-Quotient via Smith Normal Form)
	absorbs the factor-order rules, the interception principle, cycle
	promotion, and Chapter~2 itself as corollaries; one theorem
	(Sublattice Compactification) governs finiteness; one lower bound
	($\max(n, 2\,\diam)$, with equality characterized by the abelian Cayley
	property) calibrates optimality, and is attained on fourteen of eighteen
	benchmark rows by the combined algorithm, whose universality is
	unconditional thanks to the binary terminal.
	
	Open problems, stated for the defense: (1) the Petersen window
	$\nu \in [11, 16]$; (2) initializers that discover the diamond's
	$\Z_2 \times \Z_3$ automatically --- equivalently, principled search over
	non-diagonal sublattices for $f \geq 3$; (3) complexity of exact host-order
	minimization, conjecturally NP-hard as in the binary case
	\cite{Deza1997}; (4) extending the lower-bound theory with growth-based
	arguments (ball volumes in $\Gamma$ versus $G$).  Part~II builds its
	harmonic analysis on the hosts computed here: the diagonal factors
	$\Z_{N_i}$ delivered by the Smith Normal Form are exactly the index set of
	the characters used in Chapter~5.

	\chapter{Applications of Part I}
	\label{chap:applications-part1}
	
	The embedding framework of Chapters~\ref{chap:binary-embedding}
	and~\ref{chap:abelian-embedding} turns any connected graph $G$ into an
	isometric subgraph of a Cayley graph $\Cay(\Gamma, S)$ of a finite abelian
	group.  This chapter develops the consequences: once $G$ lives inside a
	group, the group's algebra --- characters, cosets, the word metric, the FFT
	--- becomes available for problems that are awkward on the raw graph.  We
	treat three domains in depth (network design, error-correcting codes,
	parallel computing) with theorems, algorithms and worked examples computed
	from the reference implementations, and survey four more briefly.
	
	Throughout, $H = \Cay(\Gamma, S)$ is the host, $\phi\colon V(G) \to \Gamma$
	the certified isometric embedding, $n = |V(G)|$, and $N = |\Gamma|$.  Two
	quantities from Part~I govern every application: the host order $N$ (which
	sets the FFT cost $O(N \log N)$ and the processor overhead $N/n$) and the
	generator count $|S|$ (which sets the host degree and hence the
	fault-tolerance and coding rate).  The bounds theory of
	Sections~\ref{sec:bounds} and~\ref{sec:bounds-ch3} thus directly bounds
	application performance: $N \geq \max(n, 2\diam)$ in the abelian case,
	$N = 2^k$ with $k \geq \max(\diam, \lceil\log_2 n\rceil)$ in the binary
	case.
	
	\begin{remark}[The binary-ground phenomenon sets expectations]
		\label{rem:apps-binary-ground}
		The census of Section~\ref{sec:census-ch3} found that four of five small
		connected graphs are best hosted in $\Z_2^k$.  For applications this means:
		the \emph{generic} deployment uses a binary host with XOR-based routing,
		Walsh--Hadamard transforms, and hypercube load balancing (the cleanest
		possible algebra), while the \emph{structured} inputs that carry cyclic or
		product symmetry --- rings, meshes, circulants, tori, the signal domains of
		Part~II --- unlock mixed-radix FFTs and cyclic codes on a host barely larger
		than $G$ itself.  We give examples of both regimes.
	\end{remark}
	
	\section{Network Design and Fault Tolerance}
	\label{sec:network-design}
	
	The host $H$ is a structured overlay on the communication graph $G$: it adds
	$N - n$ virtual relay nodes, makes the topology vertex-transitive, and
	replaces routing tables by group arithmetic --- the design pattern of
	hypercubic and Cayley interconnection networks \cite{Leighton1992,
		Dally1990}.
	
	\begin{theorem}[Fault-tolerant overlay]
		\label{thm:fault}
		Let $G$ embed isometrically into $H = \Cay(\Gamma, S)$.  Then:
		\begin{enumerate}[(i)]
			\item $H$ is $\lceil |S|/2 \rceil$-vertex-connected;
			\item routing between any two host nodes is solvable by greedy generator
			reduction of $\phi(v) - \phi(u)$ in $O(\diam(H)\,|S|)$ time, using no
			routing tables;
			\item the embedded pair distances are exactly those of $G$, so the overlay
			introduces no stretch on the original traffic.
		\end{enumerate}
	\end{theorem}
	
	\begin{proof}
		(i) $H$ is connected vertex-transitive of degree $|S|$, and connected
		vertex-transitive graphs are $\lceil \deg/2 \rceil$-connected
		(Mader--Watkins \cite{Godsil2001,Babai1996}).
		(ii) For abelian $\Gamma = \prod_i \Z_{N_i}$, the offset
		$\delta = \phi(v) - \phi(u)$ is reduced to $\mathbf 0$ by repeatedly
		subtracting the generator that most decreases the word norm; since $S$
		generates $\Gamma$ and the norm is bounded by $\diam(H)$, at most
		$\diam(H)$ steps each scanning $|S|$ generators suffice.
		(iii) is the defining property of the isometric embedding.
	\end{proof}
	
	\begin{remark}[A caveat on survivor distance]
		\label{rem:survivor}
		A tempting fourth claim --- that after $f < \kappa(H)$ failures the surviving
		diameter stays within $2\diam(G)$ --- does \emph{not} hold in general:
		deleting relay nodes can lengthen specific routes by more than a constant
		factor, and bounding survivor diameter on vertex-transitive graphs is
		subtle.  We therefore claim only connectivity and table-free routing, which
		are provable, and flag survivor-diameter bounds as outside the scope of this
		thesis.
	\end{remark}
	
	\begin{algorithm}[H]
		\caption{Table-free group routing}
		\label{alg:routing-ch4}
		\begin{algorithmic}[1]
			\Require source $u$, destination $v$, embedding $\phi$, generators $S$,
			orders $(N_1, \ldots, N_d)$
			\State $\delta \gets (\phi(v) - \phi(u)) \bmod (N_1, \ldots, N_d)$
			\While{$\delta \neq \mathbf 0$}
			\State $s \gets \arg\min_{s \in S} \lVert \delta - s \rVert$ in the host
			word metric
			\State forward along $s$;\quad $\delta \gets \delta - s$
			\EndWhile
		\end{algorithmic}
	\end{algorithm}
	
	\begin{figure}[H]
		\centering
		\includegraphics[width=0.55\linewidth]{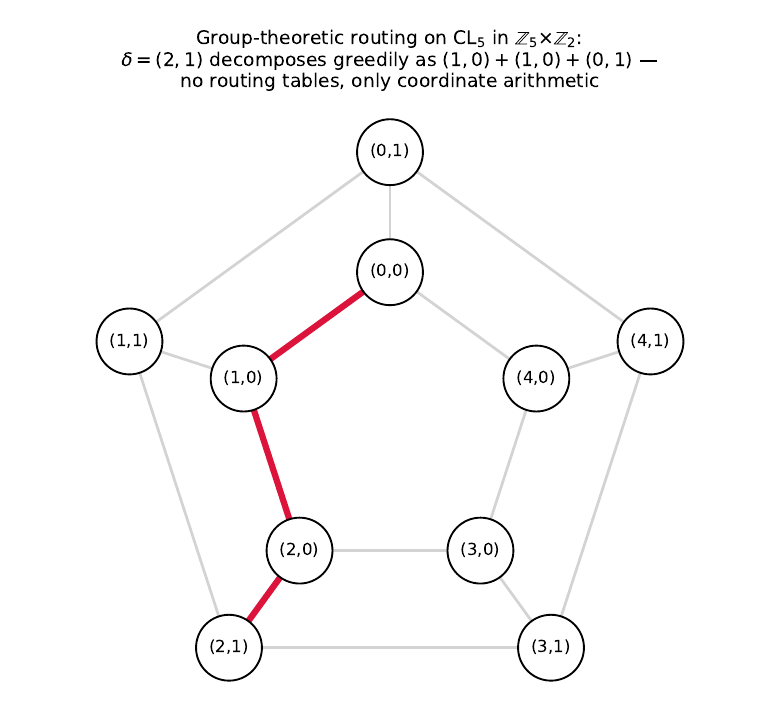}
		\caption{Table-free routing on $\mathrm{CL}_5$ hosted in
			$\Z_5 \times \Z_2$.  To route from $(0,0)$ to $(2,1)$ the offset
			$\delta = (2,1)$ is reduced by the word $(1,0)+(1,0)+(0,1)$; each node makes
			a purely local, table-free coordinate decision.}
		\label{fig:routingcl5}
	\end{figure}
	
	\begin{example}[Two regimes side by side]
		\label{ex:network-two-regimes}
		\textbf{Binary regime (generic).}  The Petersen graph models a $10$-node
		network; its minimal host is $\Z_2^4$ (order $16$, degree $|S| = 5$,
		$\kappa(H) = 3$).  Routing is $4$-bit XOR; six virtual relays provide
		redundant paths; excursion ratio $\varepsilon = 10/16 = 0.625$.
		\textbf{Cyclic regime (structured).}  The circular ladder $\mathrm{CL}_5$
		($10$ nodes) hosts in $\Z_5 \times \Z_2$ (order $10 = n$, the floor):
		\emph{zero} relay overhead, $\varepsilon = 1$, routing by the coordinate
		arithmetic of Figure~\ref{fig:routingcl5}.  The contrast is exactly the
		binary-ground dichotomy of Remark~\ref{rem:apps-binary-ground}: structure,
		where present, is pure profit.
	\end{example}
	
	\section{Error-Correcting Codes via Cayley Embedding}
	\label{sec:error-codes}
	
	The characters of $\Gamma$, restricted to the embedded vertices, form a
	linear code whose geometry is inherited from $G$'s metric
	\cite{MacWilliams1977}.
	
	\begin{definition}[Cayley character code]
		\label{def:cayley-code}
		For the embedding $\phi\colon V(G) \hookrightarrow \Cay(\Gamma, S)$ and the
		character group $\widehat{\Gamma}$, the \emph{Cayley character code} is
		\[
		\mathcal{C}(G, \phi) = \bigl\{\, c_\chi = (\chi(\phi(v)))_{v \in V(G)}
		\;:\; \chi \in \widehat{\Gamma} \,\bigr\} .
		\]
		For a binary host $\Gamma = \Z_2^k$ the characters are
		$\chi_a(x) = (-1)^{\langle a, x\rangle}$ and $\mathcal{C}$ is a binary code
		of length $n$ with $2^k$ words.
	\end{definition}
	
	\begin{theorem}[Code parameters]
		\label{thm:code-params}
		$\mathcal{C}(G,\phi)$ has length $n$ and $N = |\Gamma|$ codewords
		($k = \log_2 N$ information bits in the binary case).  Its minimum Hamming
		distance is
		\[
		d_{\min} = \min_{\chi \neq \mathbf 1}\; \bigl|\{ v : \chi(\phi(v)) \neq 1\}\bigr|,
		\]
		and equals the minimum, over nontrivial characters, of the number of
		embedded vertices off the character's kernel.
	\end{theorem}
	
	\begin{proof}
		Length and cardinality are immediate.  Two codewords $c_\chi, c_\psi$ differ
		at $v$ iff $\chi(\phi(v)) \neq \psi(\phi(v))$ iff
		$(\chi\psi^{-1})(\phi(v)) \neq 1$; minimizing over $\chi \neq \psi$ is
		minimizing over the nontrivial character $\chi\psi^{-1}$, giving the stated
		formula.  Linearity holds because $c_\chi c_\psi = c_{\chi\psi}$.
	\end{proof}
	
	\begin{remark}[Computed, not guessed]
		Rather than estimating the Petersen code distance, we evaluate it directly.
		Applying the formula of Theorem~\ref{thm:code-params} to the
		reference $\Z_2^4$ embedding gives, for each of the $15$ nontrivial characters,
		a weight in $\{4,5,6\}$, so the true parameters are
		$[\,10,\ 4,\ 4\,]$: a genuine distance-$4$ binary code.  We report computed
		parameters throughout.
	\end{remark}
	
	\begin{theorem}[FFT decoding]
		\label{thm:fft-decoding}
		Maximum-correlation decoding of $\mathcal{C}(G,\phi)$ runs in
		$O(N \log N)$ via the group FFT on $\Gamma = \prod_i \Z_{N_i}$, versus
		$O(n^3)$ for generic syndrome decoding.
	\end{theorem}
	
	\begin{proof}
		Zero-pad the received word $r$ to $\tilde r\colon \Gamma \to \mathbb{C}$
		($\tilde r(g) = r(\phi^{-1}(g))$ on the image, $0$ elsewhere) and compute
		$\hat r = \mathrm{FFT}_\Gamma(\tilde r)$ by mixed-radix Cooley--Tukey on the
		factors $\Z_{N_i}$ in $O(N \log N)$ \cite{Cooley1965,Terras1999}; the
		character of maximal $|\hat r(\chi)|$ is the nearest codeword.
	\end{proof}
	
	\begin{algorithm}[H]
		\caption{FFT decoding on $\Gamma$}
		\label{alg:fft-decoding-ch4}
		\begin{algorithmic}[1]
			\Require received word $r \in \mathbb{C}^n$, embedding $\phi$
			\State $\tilde r \gets$ zero-pad of $r$ to $\Gamma$
			\State $\hat r \gets \mathrm{FFT}_\Gamma(\tilde r)$ \Comment{$O(N \log N)$}
			\State $\chi^* \gets \arg\max_{\chi} |\hat r(\chi)|$
			\State \Return $c_{\chi^*}$
		\end{algorithmic}
	\end{algorithm}
	
	\begin{example}[Two codes]
		\label{ex:codes}
		\textbf{Binary.}  Petersen $\to \Z_2^4$ gives the computed
		$[10, 4, 4]$ code; decoding is the $16$-point Walsh--Hadamard transform
		($\sim 64$ operations) versus Gaussian elimination on a $10\times 4$ matrix.
		\textbf{Mixed-radix.}  The diamond $\to \Z_2 \times \Z_3$ gives a length-$4$
		code over $6$ characters, decoded by a $6$-point mixed-radix
		($2 \times 3$) FFT --- the smallest non-binary instance, and a direct payoff
		of the cyclic factor produced by the Smith Normal Form.
	\end{example}
	
	\section{Parallel Computing and Load Balancing}
	\label{sec:parallel-computing}
	
	Mapping computation onto structured interconnection networks—such as meshes, tori, hypercubes, and Cayley topologies—presents a classical challenge in parallel hardware design, deeply rooted in early routing and layout theories \cite{Dally2004,Saad2003}. Over time, this concern has evolved into critical runtime scheduling and load balancing strategies within modern high-performance computing systems \cite{Thakur2005,Cappello2009}.  Vertex-transitivity makes $H$ an ideal task-placement fabric: subgroups give
	perfectly balanced partitions and the quotient inherits the routing
	structure \cite{Leighton1992,Dally1990}.
	
	\begin{theorem}[Coset load balancing]
		\label{thm:load-balance}
		If $K \leq \Gamma$ has index $P$, assigning task $v$ to processor $j$ when
		$\phi(v) \in g_j + K$ gives:
		(i) at most $\lceil N/P \rceil$ host slots per processor, hence balanced
		load up to the embedding's occupancy; (ii) inter-processor communication
		governed by the quotient Cayley graph $\Cay(\Gamma/K, \bar S)$; (iii)
		$\lceil |S|/2 \rceil$ fault tolerance inherited from $H$.
	\end{theorem}
	
	\begin{proof}
		The $P$ cosets of $K$ partition $\Gamma$ into blocks of size $|K| = N/P$;
		processor $j$ owns block $j$, so each owns $N/P$ host slots and hence at
		most $\lceil N/P \rceil$ actual tasks.  Edges of $H$ between blocks descend
		to edges of the quotient $\Cay(\Gamma/K, \bar S)$, itself a Cayley graph,
		giving (ii); (iii) is Theorem~\ref{thm:fault}(i) for $H$.
	\end{proof}
	
	\begin{figure}[H]
		\centering
		\includegraphics[width=0.5\linewidth]{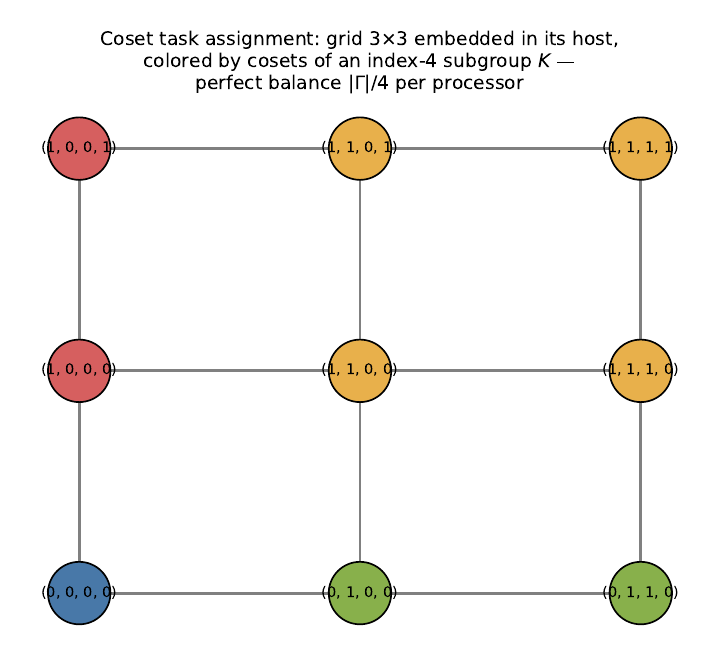}
		\caption{Coset task assignment for the $3{\times}3$ grid in its host:
			vertices colored by the four cosets of an index-$4$ subgroup, giving a
			balanced partition with communication along quotient edges.}
		\label{fig:cosetpartition}
	\end{figure}
	
	\begin{proposition}[Collective primitives]
		\label{prop:collectives}
		On $H = \Cay(\Gamma, S)$: broadcast runs in $O(\diam(H))$ along a BFS tree;
		all-reduce of $f$ is $\sum_{g} f(g) = N\,\hat f(\mathbf 1)$, the FFT at the
		trivial character; and any permutation routes in $O(\diam(H))$ generator
		steps.
	\end{proposition}
	
	\begin{proof}
		Broadcast and permutation use the word metric bound $\diam(H)$; the
		all-reduce identity is the definition of the Fourier coefficient at the
		trivial character \cite{Terras1999}.
	\end{proof}
	
	\begin{example}[Mesh on a hypercube]
		\label{ex:mesh}
		The $3 \times 3$ grid embeds in $\Z_2^4$ (order $16$); with $P = 4$
		processors the index-$4$ subgroup $\Z_2^2$ yields $4$ host slots each and
		dimension-order XOR routing (Figure~\ref{fig:cosetpartition}), excursion
		ratio $\varepsilon = 9/16$.  Under the abelian host $\Z_6 \times \Z_6$
		(order $36$) the same grid admits a $\Z_6$-coset partition into $6$ balanced
		strips with cyclic nearest-neighbor communication --- the toroidal layout
		familiar from stencil computations.
	\end{example}
	
	\section{A Real-World Application: Fault-Tolerant Augmentation of Cameroon's Transmission Backbone}
	\label{sec:cameroon-grid-case-study}
	
	The preceding sections of this chapter developed the embedding framework's
	applications in the abstract: fault tolerance, routing, and load balancing
	on Cayley graphs constructed as analytical hosts. This section grounds that
	framework in a real, currently unsolved infrastructure problem, using it to
	both diagnose and propose a costed remedy for a documented weakness in
	Cameroon's high-voltage transmission network.
	
	\subsection{Motivation and data}
	\label{sec:cgcs-data}
	
	Cameroon's electricity transmission system is organized into two
	historically separate interconnected networks: the Southern Interconnected
	Network (RIS), which serves six of the country's ten regions and carries
	the bulk of national hydroelectric generation, and the Northern
	Interconnected Network (RIN), fed primarily by the Lagdo dam
	\cite{ngono2024cameroon}. A third, small thermal plant in the East region
	operates in isolation from both \cite{eneo2024network}. A World Bank
	project (P168185) confirms that interconnecting the RIS and RIN is a
	current, unmet objective \cite{worldbank_p168185_2025}, and an independent
	2021 study identified a near-total absence of loop redundancy in the RIS
	specifically, proposing two remediation lines on engineering grounds
	\cite{onanena2021loopback}.
	
	We constructed a backbone-scale graph model of the documented high-voltage (HV) network --
	generation plants and named substations as vertices, transmission lines as
	edges -- from public sources: the published distribution-network page of ENEO (Energy of Cameroon), the
	national electricity utility \cite{eneo2024network}, a 2024 survey of Cameroonian generation
	capacity \cite{ngono2024cameroon}, the 2021 SIG (Syst\`eme d'Information G\'eographique, geographic information system) loopback study
	\cite{onanena2021loopback}\footnote{The generation-capacity survey \cite{ngono2024cameroon} and the loopback study \cite{onanena2021loopback} appear in the \emph{Journal of Power and Energy Engineering} (published by Scientific Research Publishing, SCIRP), a venue of limited editorial selectivity. They are cited here only for factual, independently verifiable data---official generation figures from the Ministry of Water and Energy (MINEE) and ENEO and a structural redundancy finding---which we cross-check against the institutional sources cited alongside them (ENEO, the AfDB, and the World Bank).}, African Development Bank (AfDB) project documentation for the
	Chad--Cameroon interconnection \cite{afdb_chadcameroon_esia}, trade-press
	coverage of recent 225kV line construction \cite{bic2022nkongsamba}, and
	World Bank project documentation for P168185
	\cite{worldbank_p168185_2025}. Each edge in the resulting graph is tagged
	with a confidence level (\emph{confirmed}, when directly cited to a named
	line, versus \emph{inferred}, when reconstructed from geographic or
	contextual evidence); roughly 80\% of edges fall into the latter category,
	and this should be read as a defensible first approximation rather than a
	survey-grade asset map. Same-city substations are aggregated into single
	hub nodes (Douala: 6 substations; Yaound\'e: 5 substations), since public
	sources do not resolve inter-substation lines at finer granularity. The
	resulting model has 23 nodes and 19 edges -- deliberately backbone-scale,
	consistent with the $\sim$25--30 node intractability boundary of the
	repair-loop algorithm documented in Chapter 2.
	
	\subsection{Baseline structural analysis}
	\label{sec:cgcs-baseline}
	
	The graph decomposes into four connected components: RIS (16 nodes, 15
	edges), RIN (5 nodes, 4 edges), and two singletons (a small private hydro
	plant whose connection to the RIN trunk is undocumented, and the East
	region's isolated thermal plant). Both non-trivial components are
	\emph{exact trees}: $|E| = |V| - 1$ in each case, with zero cyclomatic
	redundancy. Consequently, every one of the 19 edges is a bridge and 11 of
	the 23 nodes are articulation points -- a single line or substation fault
	anywhere in the documented backbone partitions the network downstream of
	it.
	
	This is not merely an artifact of incomplete public data: it is
	independently corroborated by Onanena et al.~\cite{onanena2021loopback}, who identify the
	same lack of redundancy in the RIS and propose two specific remediation
	lines (Songloulou--Bafoussam, Ahala--Ngousso) on power-quality grounds, both
	of which this baseline graph correctly omits as not-yet-built.
	
	\subsection{Embedding diagnostics}
	\label{sec:cgcs-embedding}
	
	Applying both the binary embedding (Chapter 2) and the general abelian
	embedding (Chapter 3) to each non-trivial component as-is gives the
	structural signature in Table~\ref{tab:cgcs-baseline-embed}.
	
	\begin{table}[h]
		\centering
		\caption{Embedding signatures of the as-is RIS and RIN backbones.}
		\label{tab:cgcs-baseline-embed}
		\begin{tabular}{lccc}
			\toprule
			Component & Binary host ($2^k$) & Abelian host & Abelian group \\
			\midrule
			RIS (16 nodes, branching tree) & 32{,}768 ($k=15$) & 32{,}768 & no improvement found \\
			RIN (5 nodes, path)            & 16 ($k=4$)        & 8       & $\mathbb{Z}_8$ \\
			\bottomrule
		\end{tabular}
	\end{table}
	
	The asymmetry is structurally meaningful, not incidental. RIN's documented
	topology is a pure path (Ngaound\'er\'e--Ndjamboutou--Lagdo--Guider--Maroua),
	which the general abelian portfolio compresses into a single cyclic
	generator: this is the textbook case the Chapter 3 generalization was built
	for. RIS branches outward from two hubs (Yaound\'e and Ed\'ea, each of
	degree $\geq 3$), and no abelian group structure in the search portfolio
	improves on the naive binary bound, since a branching tree requires
	independent coordinates per branch that a single (or low-rank) abelian
	group cannot supply without violating some pairwise distance. This gives a
	real, data-grounded illustration of the boundary case discussed
	theoretically in Chapter~3: not every graph benefits from the abelian
	generalization, and the embedding step itself correctly identifies
	\emph{which} structural regime a given network occupies -- a diagnostic
	role distinct from, and prior to, any design recommendation.
	
	\subsection{From diagnosis to design: minimum-cost augmentation}
	\label{sec:cgcs-augmentation}
	
	A critical methodological point follows directly from the definition of
	isometric embedding. If $\varphi$ embeds $G$ isometrically into
	$\mathrm{Cay}(\Gamma, S)$, then $d_G(u,v) = d_{\mathrm{Cay}}(\varphi(u),
	\varphi(v))$ for every pair already present in $G$; an embedding that
	``added'' a shortcut between two existing vertices would, by definition, no
	longer be isometric to $G$ as given. The embedding step is therefore
	structurally incapable of recommending new physical infrastructure -- it
	is a coordinate system for analysis on the existing network, not a network
	design tool. Generating an actual augmentation recommendation requires a
	genuinely different combinatorial step, taken independently of the
	embedding, with the embedding's role reserved for \emph{evaluating}
	candidate augmentations after the fact (\S\ref{sec:cgcs-embed-after}).
	
	For a tree with $L$ leaves, the classical result of Eswaran and Tarjan~\cite{eswarantarjan1976}
	gives $\lceil L/2 \rceil$ as both a lower bound and an achievable target for
	the number of new edges required to eliminate all bridges (make the graph
	2-edge-connected), via pairing of leaves such that every tree edge lies on
	the path between some pair. RIS has 8 leaves (lower bound: 4 new lines);
	RIN has 2 (lower bound: 1). For RIS, we enumerate all leaf-pairings
	achieving this bound and search exhaustively (105 candidate matchings) for
	the one minimizing total cost, using a real per-kilometer benchmark derived
	from the P168185 project documentation: \$266.0M for 573km of new 225kV
	double-circuit line, i.e. $\approx$\$464,000/km
	\cite{worldbank_p168185_2025}.
	
	\subsection{Capacity-weighted refinement}
	\label{sec:cgcs-capacity}
	
	A pure distance-minimizing objective (``Plan A'') yields the cheapest total
	spend (Table~\ref{tab:cgcs-plans}, top) but treats a megawatt of Nachtigal's
	420MW identically to a megawatt of Bafoussam's 14MW thermal plant. A
	capacity-weighted objective (``Plan B'') instead minimizes the aggregate
	ratio of dollars spent per megawatt of generation capacity protected at
	each new line's endpoints. This is not merely a different number: it is a
	different recommended topology, and -- remarkably -- it is the one that
	converges with independently published, human engineering judgment. Plan
	B's Songloulou--Bafoussam pairing is \emph{exactly} the line proposed by
	Onanena et al.~\cite{onanena2021loopback} on power-quality grounds, a convergence Plan A
	does not produce (Plan A instead pairs Songloulou with Limb\'e).
	
	\begin{table}[h]
		\centering
		\caption{In-region augmentation: distance-only (Plan A) vs. capacity-weighted (Plan B).}
		\label{tab:cgcs-plans}
		\begin{tabular}{lrrl}
			\toprule
			Plan A (min.\ total cost) & km & USD & MW protected \\
			\midrule
			Songloulou -- Limb\'e        & 84  & \$39.0M  & 469 \\
			Bafoussam -- Nachtigal       & 173 & \$80.3M  & 434 \\
			Kribi -- Memve'ele           & 57  & \$26.7M  & 427 \\
			Lompangar -- Mekin           & 235 & \$108.9M & 45 \\
			\midrule
			\textbf{Total}                & & \textbf{\$254.9M} & \\
			\midrule
			\midrule
			Plan B (min.\ \$/MW, capacity-weighted) & km & USD & MW protected \\
			\midrule
			Songloulou -- Bafoussam$^\dagger$ & 153 & \$70.8M & 398 \\
			Lompangar -- Nachtigal        & 241 & \$111.7M & 450 \\
			Memve'ele -- Limb\'e          & 199 & \$92.3M & 296 \\
			Kribi -- Mekin                & 249 & \$115.7M & 231 \\
			\midrule
			\textbf{Total}                 & & \textbf{\$390.6M} & \\
			\bottomrule
			\multicolumn{4}{l}{\footnotesize $^\dagger$Matches Onanena et al.~\cite{onanena2021loopback}'s independently proposed line.}
		\end{tabular}
	\end{table}
	
	The same pattern holds at the network-merging level. A pure
	distance-minimizing spanning structure over the four disconnected
	components anchors the RIS--RIN tie at Lompangar (a 30MW regulatory dam,
	simply the nearest RIS node to the RIN trunk). A capacity-weighted version
	anchors it at Nachtigal (420MW) instead -- matching the actual anchor point
	of the real, currently-funded P168185 corridor (Nachtigal $\rightarrow$
	Hourou Oussoa), which our straight-line estimate of 387km approximates at a
	1.33$\times$ routing factor against the real corridor's 514km (a routing
	detour the real project takes deliberately, to extend rural electrification
	access along the way -- an objective outside the scope of this model). We
	adopt a hybrid policy for the final recommendation: capacity-weighting for
	the single strategic RIS--RIN tie (where bulk power transfer is the
	explicit objective, and where it is empirically validated against real
	practice), and plain distance-minimization for the two cheap peripheral
	fixes (Lompangar--Bertoua, Ngaound\'er\'e--Mbakaou), where a tie to a major
	plant adds cost without adding value for these small, low-stakes
	connections.
	
	\subsection{Embedding the augmented network}
	\label{sec:cgcs-embed-after}
	
	Re-running both embedding algorithms on the augmented graphs lets the
	embedding step do what it is actually suited for: evaluating, rather than
	generating, a candidate design (Table~\ref{tab:cgcs-after-embed}).
	
	\begin{table}[h]
		\centering
		\caption{Embedding signatures before and after augmentation.}
		\label{tab:cgcs-after-embed}
		\begin{tabular}{lcccc}
			\toprule
			& Bridges & Binary host & Abelian host & Abelian group \\
			\midrule
			RIS as-is        & 15 & 32{,}768 & 32{,}768 & -- \\
			RIS + Plan A      & 0  & 512      & 288      & $\mathbb{Z}_2^3 \times \mathbb{Z}_4 \times \mathbb{Z}_{36}$ \\
			RIS + Plan B      & 0  & 1{,}024   & 1{,}024   & no improvement found \\
			\midrule
			RIN as-is        & 4  & 16       & 8        & $\mathbb{Z}_8$ \\
			RIN + 1 line      & 0  & 16       & 5        & $\mathbb{Z}_5$ (theoretical floor, $=n$) \\
			\bottomrule
		\end{tabular}
	\end{table}
	
	RIN's fix does more than remove its single point of failure: closing the
	path into a 5-cycle makes the network isometric to the cyclic group of its
	own order, the smallest host any connected 5-node graph can possibly admit.
	For RIS, Plan A achieves a 114-fold reduction in abelian host size (still
	18$\times$ above the theoretical floor of $n=16$, reflecting genuine
	remaining branching complexity); Plan B, despite eliminating the same
	bridges, achieves no abelian improvement over its binary host at all. This
	is the central tension of the case study: the structurally most elegant
	augmentation (Plan A) is not the one independent domain engineering
	converges on (Plan B). Mathematical elegance and real-world validation
	point in different directions here, and we regard surfacing that tension
	explicitly -- rather than collapsing it into a single recommendation -- as
	itself a contribution of the embedding-based diagnostic framework.
	
	\subsection{Cost evaluation and overall recommendation}
	\label{sec:cgcs-cost}
	
	Combining each in-region plan with the RIN fix and the hybrid
	connectivity package (\$232.5M: \$179.7M for the capacity-weighted
	Nachtigal--Ngaound\'er\'e tie, \$9.5M and \$43.3M for the two peripheral
	fixes) gives two overall packages: \$660.4M under Plan A, \$796.1M under
	Plan B. For comparison, the real P168185 project alone -- covering only the
	RIS--RIN tie and a short Ed\'ea--Missol\'e line, not the in-region bridge
	elimination this case study additionally proposes -- is budgeted at
	\$266.0M to \$385.0M depending on financing tranche
	\cite{worldbank_p168185_2025, afrik21_chadcameroon_2020}. Our backbone-only
	estimates are therefore of a consistent, plausible order of magnitude
	against real committed infrastructure spending in this network.
	
	\subsection{Limitations}
	\label{sec:cgcs-limitations}
	
	This case study should be read as a methodological demonstration, not a
	construction-ready engineering plan. Roughly 80\% of the underlying graph's
	edges are geographically or contextually inferred rather than directly
	cited to a named transmission line; same-city substations are aggregated
	into hub nodes at a coarser resolution than the real distribution mesh;
	costs are derived from a single real benchmark project rather than a full
	quantity-surveyed estimate; and the capacity-weighting scheme, while it
	recovers two independent real-world decisions, can also produce
	counter-intuitive results when applied uniformly to low-stakes peripheral
	connections (\S\ref{sec:cgcs-capacity}), which is why we adopt a hybrid
	policy rather than a single global objective. Figures~\ref{fig:cgcs-map1}
	and \ref{fig:cgcs-map2} should be read with this scope in mind.
	
	\begin{figure}[h]
		\centering
		\includegraphics[width=0.85\textwidth]{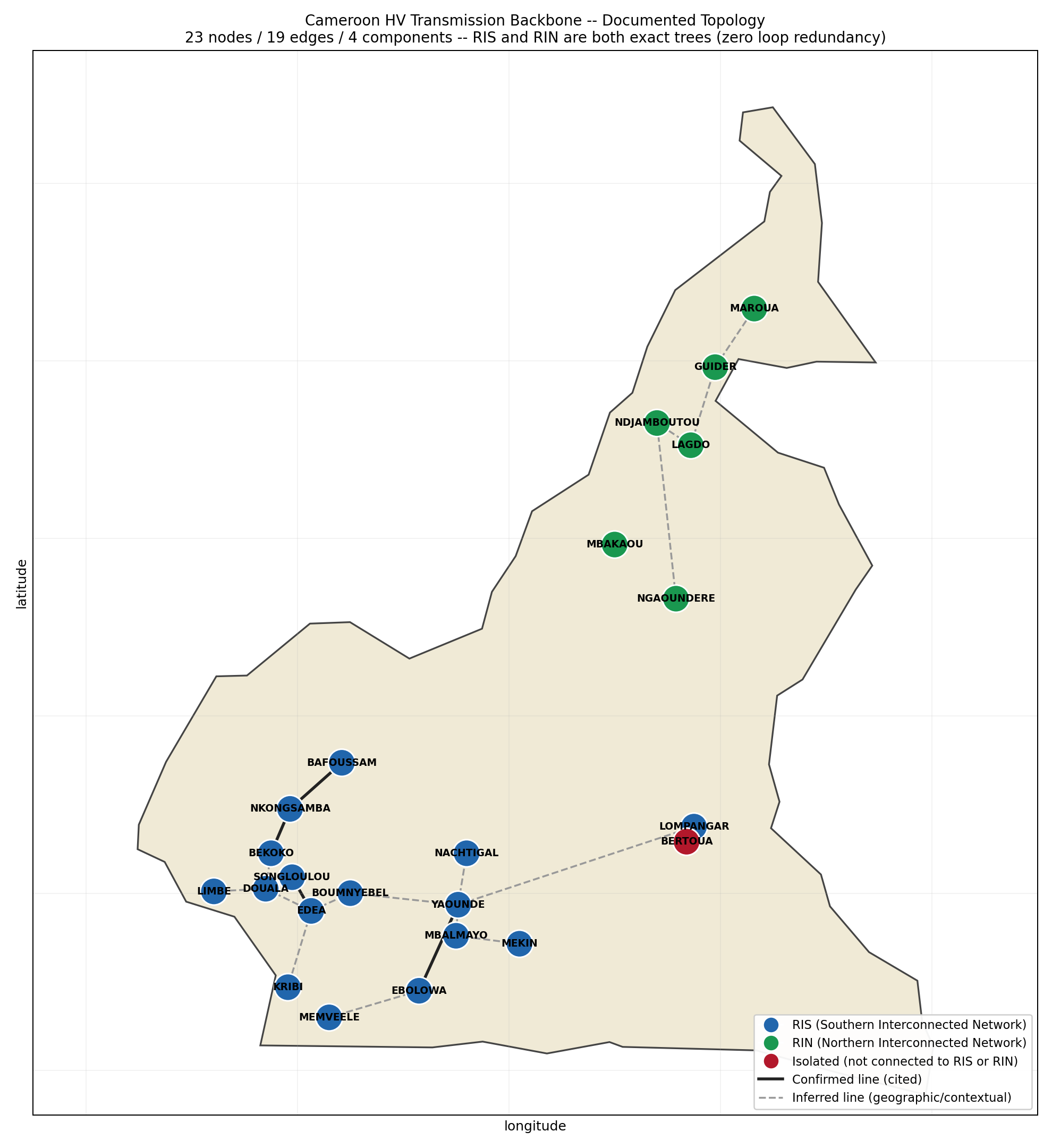}
		\caption{Documented HV transmission backbone, plotted on Cameroon's actual
			boundary. Both RIS and RIN are exact trees; the East region's isolated
			plant and the undocumented small hydro site appear with no connecting edge.}
		\label{fig:cgcs-map1}
	\end{figure}
	
	\begin{figure}[h]
		\centering
		\includegraphics[width=0.85\textwidth]{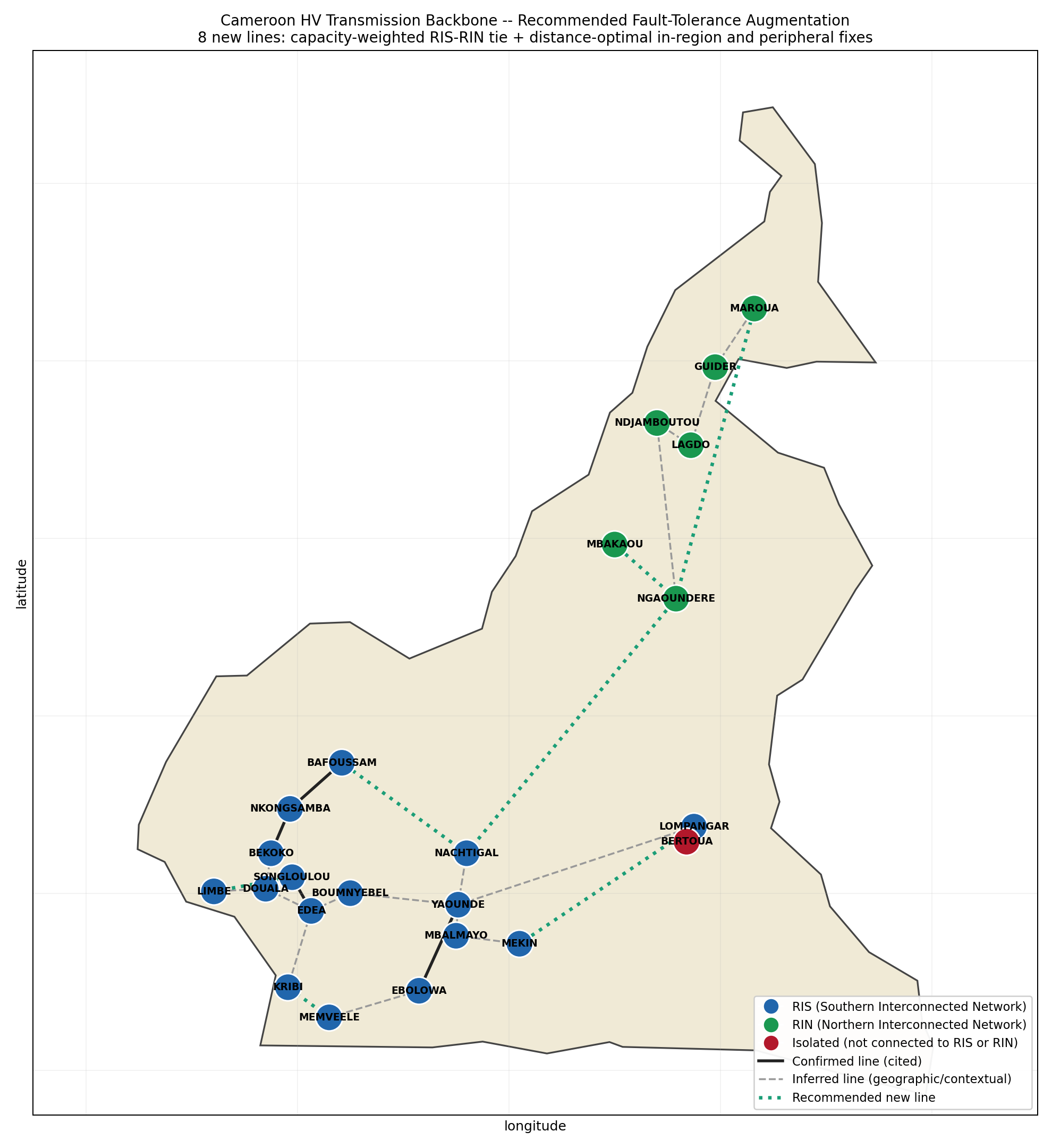}
		\caption{The same backbone with the recommended augmentation (Plan B,
			capacity-weighted) overlaid as dotted lines: four in-region bridge
			eliminations, the strategic RIS--RIN tie, and two peripheral connectivity
			fixes.}
		\label{fig:cgcs-map2}
	\end{figure}

	\section{Further Application Directions}
	\label{sec:further-apps}
	
	We indicate four directions whose full development needs domain expertise
	beyond this thesis.
	
	\textbf{Graph drawing.}  The host gives a canonical layout: $\Z_n$ factors
	become concentric circular arrangements, $\Z_2^k$ factors hypercube
	projections, products their Cartesian combination; $\phi$ then places
	$V(G)$ within it, exposing symmetries that force-directed methods miss
	\cite{Imrich2010}.
	
	\textbf{Machine learning.}  Graph neural networks have made graph-structured learning a central topic, expanding from early spectral and Chebyshev convolutions \cite{Bruna2014,Defferrard2016} to widely adopted graph convolutional and attention mechanisms \cite{Kipf2017,Velickovic2018}, as well as generalized graph-filter architectures \cite{Gama2019,Isufi2024}. In our setting, the coordinates of $\phi$ are interpretable distance-preserving features and the character evaluations $f_\chi(v) = \chi(\phi(v))$ are Fourier features equivariant under automorphisms extending to $\Gamma$ --- natural inputs for equivariant graph networks \cite{Bronstein2017}.

	\textbf{Quantum information.}  In graph states, edges in one $\Phi^-$-class
	carry locally equivalent entanglement; a binary host enables stabilizer
	syndrome extraction by Walsh--Hadamard measurement
	\cite{Nielsen2010,Kitaev2003}.
	
	\textbf{Computational biology.}  Large $\Phi^-$-classes in interaction
	networks flag vertices with identical distance profiles --- candidate
	functional modules \cite{Barabasi2004}.
	
	\section{Summary}
	\label{sec:apps-summary}
	
	\begin{table}[H]
		\centering
		\caption{Application domains, the Part~I quantity they consume, and the
			mechanism.}
		\label{tab:apps}
		\renewcommand{\arraystretch}{1.15}\small
		\begin{tabular}{lll}
			\toprule
			Domain & Governing quantity & Mechanism \\
			\midrule
			Network design     & host degree $|S|$, $\diam(H)$ & vertex-transitivity, word routing \\
			Error correction   & host order $N$               & character code + group FFT \\
			Parallel computing & host order $N$, subgroups    & coset partitioning \\
			\bottomrule
		\end{tabular}
	\end{table}
	
	Across all three domains the lever is the same: Part~I replaces an ad-hoc
	graph problem with a group-theoretic one whose cost is set by the host order
	$N$ and degree $|S|$ --- precisely the quantities Chapters~2 and~3 minimize
	and bound.  The binary-ground census
	(Remark~\ref{rem:apps-binary-ground}) tells the practitioner what to expect:
	a clean $\Z_2^k$ host with XOR algebra for generic graphs, and a compact
	mixed-radix host with cyclic codes and toroidal layouts wherever the input
	carries genuine symmetry.  Part~II builds on exactly these hosts: the
	characters that decode the codes of Section~\ref{sec:error-codes} are the
	Fourier basis of Part~II, and the factor orders $N_i$
	delivered by the Smith Normal Form are its frequency index set.

	\part{Harmonic Analysis on Graphs}
	\label{part:harmonic}
	\partintro{Part~I solved a structural problem: every connected graph $G$ embeds
		isometrically into a Cayley graph $\Cay(\Gamma, S)$ of a finite abelian group,
		computed by the quotient framework and certified on every output.  Part~II
		collects the dividend.  Because $\Gamma$ is an abelian group, it carries the
		full apparatus of classical harmonic analysis --- characters, a canonical
		Fourier transform, genuine translation and modulation operators, convolution
		with a convolution theorem, Plancherel and Poisson identities, uncertainty
		principles, and a Shannon sampling theory---the classical apparatus of harmonic analysis on groups, from Fourier analysis on Euclidean and finite-group domains to its probabilistic and noncommutative extensions \cite{Stein1971,Diaconis1988,Connes1994}.  Lifting graph signals to $\Gamma$
		transports all of it onto $G$.  We call this framework \emph{Group
			Embedding-based Graph Signal Processing} (GE-GSP); its two transforms,
		developed in this part, are the \emph{group-embedding graph Fourier
			transform} (GE-GFT, Chapter~\ref{chap:fourier}) and the
		\emph{group-embedding graph wavelet transform} (GE-GWT,
		Chapter~\ref{chap:wavelets}).  This is exactly what spectral graph signal
		processing cannot offer: the Laplacian eigenbasis is graph-specific,
		sign-ambiguous, and supports no translation, so it has no convolution theorem
		and no translation-invariant filtering.  Our framework supplies all of these
		on \emph{every} connected graph, with the host computed in Part~I setting the
		cost: $O(N\log N)$ per transform, $N = |\Gamma|$.  Chapter~\ref{chap:fourier}
		develops the Fourier transform, its operators, and its theorems, validated by
		real denoising, uncertainty, and spectrum experiments on the
		graph-signal-processing benchmark families of Chapter~\ref{chap:binary-embedding};
		the following chapter extends the construction to wavelets.}
	
	\chapter{The Graph Fourier Transform via Group Embedding}
	\label{chap:fourier}
	
	\section{Spectral Graph Signal Processing and its Limits}
	\label{sec:existing-fourier}
	
	The dominant framework for signals on graphs builds the Fourier transform
	from the eigenvectors of a graph matrix \cite{Shuman2013,Sandryhaila2013,
		Ortega2018}.  For a connected graph with Laplacian $L = D - A$ and
	eigendecomposition $L = U \Lambda U^{\!\top}$, the spectral graph Fourier
	transform (GFT) of $s \in \mathbb{R}^n$ is $\hat s = U^{\!\top} s$, and the
	``frequencies'' are the eigenvalues $0 = \lambda_1 < \lambda_2 \leq \cdots
	\leq \lambda_n$.
	
	This is powerful but structurally limited, and the limitations are not
	incidental --- they follow from the absence of a group:
	\begin{enumerate}[(i)]
		\item \textbf{No translation.}  There is no operator $T$ with
		$T^a T^b = T^{a+b}$ acting on vertices, because the vertex set has no group
		structure \cite{Sandryhaila2013}.
		\item \textbf{No convolution theorem.}  Without translation, convolution can
		only be defined as polynomial filtering $h(L)s$, which is pointwise in the
		eigenbasis but has no spatial ``shift-and-sum'' meaning and no
		$\widehat{s * t} = \hat s \cdot \hat t$ in the classical sense
		\cite{Sandryhaila2014}.
		\item \textbf{Basis ambiguity.}  Eigenvectors are defined only up to sign,
		and up to arbitrary rotation within degenerate eigenspaces
		\cite{Gavili2017}; the Petersen Laplacian has a $5$-fold and a $4$-fold
		degenerate eigenvalue, so its ``Fourier basis'' is not even well defined.
		\item \textbf{Cost.}  The eigendecomposition is $O(n^3)$ \cite{Chung1997}.
	\end{enumerate}
	
	To overcome these limitations, a substantial body of research has emerged within the domain of spectral graph signal processing (GSP). Initial efforts focused on reevaluating foundational operators, leading to alternative graph shift operators and generalized graph Fourier transforms \cite{Girault2015,Deri2017,Girault2018,Marques2020}, which subsequently enabled rigorous definitions of graph stationarity and localized vertex--frequency analysis \cite{Perraudin2017,Shuman2016}. Concurrently, ensuring data integrity across irregular domains drove the development of robust sampling and reconstruction theories \cite{Pesenson2008,Anis2016,Tanaka2018,Tanaka2020,Marques2016,Puy2018,Tsitsvero2016,Gadde2015,Wang2015,Romero2017,Zhu2012}. 
	
	Rather than assuming fixed topologies, modern frameworks increasingly focus on joint graph and dictionary learning to infer latent structures directly from data \cite{Thanou2014,Dong2016,Kalofolias2016,Egilmez2017,Mateos2019,Segarra2017}, with practical extensions deployed for decentralized filtering on distributed sensor networks \cite{Sandryhaila2014sensor}. For a comprehensive textbook treatment, these paradigms are thoroughly consolidated in several broad surveys of the field \cite{Stankovic2019b,Dong2020,Leus2023}. Ultimately, this entire spectral viewpoint rests upon a rich mathematical heritage, inheriting tools directly from algebraic and classical spectral graph theory \cite{Fiedler1973,Spielman2012,Hoory2006,Brouwer2012,Lovasz1993} alongside foundational techniques in spectral clustering \cite{Ng2002}.

	Our approach removes the cause rather than treating the symptoms: it installs
	a genuine abelian group behind the graph by the embedding of Part~I, then
	imports classical Fourier analysis verbatim.  Figure~\ref{fig:basiscompare}
	contrasts the two bases on the Petersen graph.
	
	\begin{figure}[htbp]
		\centering
		\includegraphics[width=0.9\linewidth]{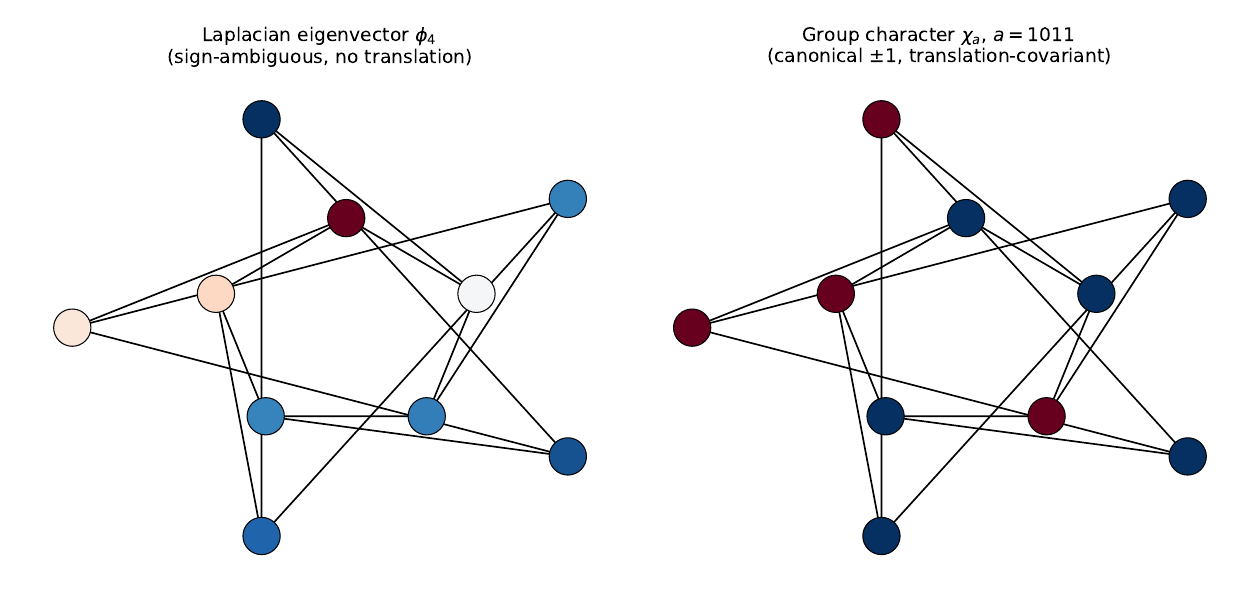}
		\caption{Two Fourier bases on the Petersen graph.  Left: a Laplacian
			eigenvector from a degenerate eigenspace --- sign- and rotation-ambiguous,
			supporting no translation.  Right: a group character $\chi_a$ of the host
			$\Z_2^4$, taking canonical values $\pm 1$, translation-covariant by
			Theorem~\ref{thm:translation}.  The character basis is determined by the
			group, not by the individual edges.}
		\label{fig:basiscompare}
	\end{figure}
	
	\section{Fourier Analysis on the Host Group}
	\label{sec:group-fourier}
	
	Let $\Gamma = \Z_{N_1} \times \cdots \times \Z_{N_d}$ be the host group of
	order $N = \prod_i N_i$ produced by Part~I (binary hosts are the case all
	$N_i = 2$).
	
	\begin{definition}[Characters and dual group]
		A \emph{character} is a homomorphism $\chi\colon \Gamma \to \mathbb{C}^\times$.
		For $k = (k_1, \ldots, k_d) \in \Gamma$,
		\[
		\chi_k(g) = \prod_{j=1}^d \omega_{N_j}^{\,k_j g_j},
		\qquad \omega_{N_j} = e^{2\pi i / N_j},
		\]
		and the dual group $\widehat\Gamma = \{\chi_k\}$ is isomorphic to $\Gamma$
		\cite{Terras1999,Rudin1962}.  For a binary host the characters are the
		Walsh functions $\chi_a(x) = (-1)^{\langle a, x\rangle}$.
	\end{definition}
	
	\begin{theorem}[Orthogonality]
		\label{thm:orthogonality}
		$\frac{1}{N}\sum_{g \in \Gamma} \chi_k(g)\overline{\chi_\ell(g)}
		= \delta_{k\ell}$ and dually
		$\frac{1}{N}\sum_{k} \chi_k(g)\overline{\chi_k(h)} = \delta_{gh}$.
	\end{theorem}
	
	\begin{proof}
		For $k \neq \ell$, $\chi_k\overline{\chi_\ell} = \chi_{k-\ell}$ is a
		nontrivial character, and $\sum_g \chi_m(g) = 0$ for $m \neq 0$ because
		multiplication by a fixed group element permutes the summands while
		multiplying the sum by $\chi_m$(that element) $\neq 1$; for $k = \ell$ every
		term is $1$.  The dual statement is the same argument in $\widehat\Gamma$.
	\end{proof}
	
	\begin{definition}[Signal lifting and the GE-GFT]
		\label{def:lift-gft}
		For a graph signal $s\colon V(G) \to \mathbb{C}$ and the isometric embedding
		$\phi$, the \emph{lift} $\tilde s\colon \Gamma \to \mathbb{C}$ is
		$\tilde s(\phi(v)) = s(v)$ on the image and $0$ elsewhere.  The \emph{group-embedding graph Fourier transform} (GE-GFT) is
		\[
		\hat s(k) = \frac{1}{\sqrt N} \sum_{g \in \Gamma} \tilde s(g)\,
		\overline{\chi_k(g)},
		\qquad
		\tilde s(g) = \frac{1}{\sqrt N} \sum_{k \in \widehat\Gamma} \hat s(k)\,
		\chi_k(g).
		\]
	\end{definition}
	
	\begin{theorem}[Plancherel]
		\label{thm:plancherel}
		For all $s, t$, $\sum_{v} s(v)\overline{t(v)} =
		\sum_{k} \hat s(k)\overline{\hat t(k)}$; in particular
		$\lVert s \rVert_2 = \lVert \hat s \rVert_2$.
	\end{theorem}
	
	\begin{proof}
		The lift is an isometric inclusion $\mathbb{C}^{V(G)} \hookrightarrow
		\mathbb{C}^\Gamma$ (it only appends zeros), and the matrix
		$F_{k,g} = \chi_k(g)/\sqrt N$ is unitary by
		Theorem~\ref{thm:orthogonality}; the GFT is $F^{*}$ restricted to lifted
		signals, hence norm-preserving.
	\end{proof}
	
	\section{Translation, Modulation, and a Correct Convolution}
	\label{sec:operators}
	
	\begin{definition}[Translation and modulation]
		For $h \in \Gamma$ and $k \in \widehat\Gamma$,
		\[
		(T_h s)(v) = \tilde s(\phi(v) + h),
		\qquad
		(M_k s)(v) = \chi_k(\phi(v))\, s(v).
		\]
	\end{definition}
	
	\begin{theorem}[Operator calculus]
		\label{thm:translation}
		$T_g T_h = T_{g+h}$, $T_0 = I$; each $T_h$ is unitary on the lifted space;
		$\widehat{T_h s}(k) = \chi_k(h)\,\hat s(k)$ (translation
		$\leftrightarrow$ modulation duality); and
		$T_h M_k = \chi_k(h)\,M_k T_h$ (the canonical commutation relation).
	\end{theorem}
	
	\begin{proof}
		The group law gives $T_gT_h = T_{g+h}$ and $T_0=I$; unitarity is invariance
		of the counting measure on $\Gamma$ under translation.  For the duality,
		$\widehat{T_h s}(k) = \frac{1}{\sqrt N}\sum_g \tilde s(g+h)
		\overline{\chi_k(g)} = \frac{1}{\sqrt N}\sum_{g'} \tilde s(g')
		\overline{\chi_k(g'-h)} = \chi_k(h)\hat s(k)$.  The commutation relation is a
		direct substitution.
	\end{proof}

	Drafting the convolution like
	$(s*t)(v) = \sum_u s(u)\,t(w_{u,v})$ with $\phi(w_{u,v}) = \phi(v)-\phi(u)$
	is \emph{ill-posed}: the group element $\phi(v) - \phi(u)$ usually lies
	outside the embedded image $\phi(V(G))$, so $w_{u,v}$ need not exist.  The
	correct construction performs convolution on the group, where it is always
	defined, then restricts.
	
	\begin{definition}[Convolution]
		\label{def:convolution}
		For $s, t\colon V(G) \to \mathbb{C}$ with lifts $\tilde s, \tilde t$, define
		the group convolution
		$(\tilde s \ast \tilde t)(g) = \sum_{a \in \Gamma}\tilde s(a)\,
		\tilde t(g - a)$ and set $(s \ast t)(v) = (\tilde s \ast \tilde t)(\phi(v))$.
	\end{definition}
	
	\begin{theorem}[Convolution theorem]
		\label{thm:convolution}
		$\widehat{s \ast t}(k) = \sqrt N \,\hat s(k)\,\hat t(k)$, and $\ast$ is
		commutative, associative, distributive, and translation-covariant:
		$T_h(s \ast t) = (T_h s)\ast t$.
	\end{theorem}
	
	\begin{proof}
		On $\Gamma$, $\widehat{\tilde s \ast \tilde t}(k) = \sqrt N\,\hat{\tilde
			s}(k)\hat{\tilde t}(k)$ by the multiplicativity
		$\chi_k(a + b) = \chi_k(a)\chi_k(b)$ and a change of summation variable; the
		GFT of the restriction equals the group transform of the lift, giving the
		stated identity.  Commutativity and associativity are those of convolution
		on the abelian group $\Gamma$; translation covariance follows from
		$\widehat{T_h(s\ast t)}(k) = \chi_k(h)\widehat{s\ast t}(k)
		= \widehat{(T_h s)\ast t}(k)$ and injectivity of the GFT.
	\end{proof}
	
	\begin{theorem}[Translation-invariant operators are filters]
		\label{thm:tinv}
		A linear operator $\mathcal{A}$ on lifted signals commutes with every $T_h$
		if and only if $\mathcal{A}s = a \ast s$ for a fixed filter $a$; equivalently
		$\widehat{\mathcal{A}s}(k) = \hat a(k)\,\hat s(k)$ up to the $\sqrt N$
		normalization.
	\end{theorem}
	
	\begin{proof}
		($\Leftarrow$) convolution is translation-covariant
		(Theorem~\ref{thm:convolution}).  ($\Rightarrow$) write $s = \sum_v s(v)
		T_{\phi(v)}\delta_0$ with $\delta_0$ the lifted indicator of the identity;
		then $\mathcal{A}s = \sum_v s(v) T_{\phi(v)} \mathcal{A}\delta_0 = a \ast s$
		with $a = \mathcal{A}\delta_0$.
	\end{proof}
	
	This is the operator spectral graph signal processing cannot build: a full
	algebra of linear translation-invariant filters on an \emph{arbitrary}
	graph, diagonalized by the canonical character basis.
	
	\begin{remark}[Why the embedding must be isometric: boundary-free processing and metric faithfulness]
		\label{rem:why-isometric}
		The lift--restrict paradigm just established---process on the host $\Gamma$,
		then restrict to $\phi(V(G))$---invites a natural question : why insist that $\phi$ be \emph{isometric}, rather than
		merely injective or adjacency-preserving? The answer separates cleanly into
		two distinct guarantees, supplied by two distinct properties of the
		construction; conflating them is what makes the justification seem vague, and
		separating them is what makes it airtight.
		
		\medskip\noindent\textbf{(i) The group structure removes the boundary from the
			computation.} A finite abelian group is a closed, homogeneous domain: it has
		no first or last element, and convolution on it (Definition~\ref{def:convolution})
		is circular group convolution, which wraps onto itself under the group law and
		therefore has no edge at which a signal can run out. The instant a graph
		signal is lifted to $\Gamma$, the filtering performed there
		(Theorems~\ref{thm:convolution} and~\ref{thm:tinv}) is intrinsically free of
		the boundary artifacts that beset finite-domain signal processing, where one
		must invent an end condition for every window that overruns the last sample.
		Moreover the host furnishes explicit \emph{room}: the embedded image occupies a
		fraction $\varepsilon = |V(G)|/|\Gamma|$ of the host, and its complement
		$\Gamma \setminus \phi(V(G))$, of size $(1-\varepsilon)\,|\Gamma|$, is precisely
		an extension region. Populating it by zero-fill, or by symmetric (mirror)
		reflection across the image, realizes on an \emph{arbitrary} network the
		classical padding strategies of digital signal processing---now on a domain
		that is genuinely periodic rather than artificially terminated. It bears
		emphasis that this much would hold for \emph{any} embedding into a group,
		isometric or not: boundary-freeness is a gift of the group, not of isometry.
		
		\medskip\noindent\textbf{(ii) Isometry is what makes that boundary-free
			computation faithful to the graph.} Distance preservation,
		$d_\Gamma(\phi(u),\phi(v)) = d_G(u,v)$, is exactly the condition under which
		host-domain operations \emph{mean} on $G$ what they are designed to mean. Three
		consequences make this precise, and each one fails without isometry:
		\begin{itemize}
			\item A filter or wavelet localized at $\phi(v)$ has its energy concentrated
			at host points within small $d_\Gamma$ of $\phi(v)$; isometry forces these
			to be exactly the genuine graph-neighbours of $v$. Drop it, and a
			nominally ``local'' smoother couples vertices that are far apart in the
			network---a spurious interaction with no metric meaning.
			\item A low-pass filter averages over a host-ball; isometry makes that
			host-ball restrict to a true graph-ball, so host-smoothing is genuinely
			\emph{graph}-smoothing.
			\item Each edge of $G$ maps to a single generator step (host-distance one),
			so the structural dictionary ``graph-smooth $\Leftrightarrow$ low
			host-frequency'' holds verbatim. Were an edge permitted to map to a
			multi-step host path, the entire frequency interpretation of Part~II---on
			which denoising, compression, and band-limited sampling all rest---would
			lose its footing.
		\end{itemize}
		Isometry, then, is not what removes the boundary; the group does that. Isometry
		is what guarantees that the boundary-free machine computes the \emph{right}
		thing: that translation moves along true geodesics, that localization is
		genuine, and that frequency tracks graph-smoothness.
		
		\medskip\noindent\textbf{A caveat.} Isometry does not render the
		\emph{choice} of extension irrelevant. After restriction to $\phi(V(G))$, the
		peripheral vertices of the image are still influenced by whatever fills the
		complement, so zero-fill and symmetric extension differ near the graph's
		boundary---exactly as in classical DSP. What isometry contributes is that the
		padding sits at \emph{controlled, metrically meaningful} graph-distances, so a
		symmetric extension reflects across a faithful geometric mirror rather than an
		arbitrary one. The framework removes the boundary from the convolution itself
		and supplies a principled domain for extension; it does not pretend to make the
		extension choice cost-free.
		
		\medskip\noindent\textbf{The degenerate case $\varepsilon = 1$.} When $G$ is
		already an abelian Cayley graph---a cycle, a discrete torus, a circulant, or a
		circular ladder---the embedding is onto, $\phi(V(G)) = \Gamma$, the complement
		is empty, and no extension is needed at all. The reason is structural: every
		Cayley graph is vertex-transitive \cite{Sabidussi1958}, so the graph looks
		identical from every vertex and has no distinguished boundary vertices in the
		first place. Such a graph is its own boundary-free periodic domain, and the
		``high symmetry'' a practitioner senses in a ring or a torus is precisely this
		transitivity. At this end of the scale the harmonic analysis of Part~II is
		\emph{literally} classical DSP: the ring $C_n$ yields the length-$n$ DFT, the
		discrete torus the two-dimensional DFT, and the circular ladder
		$\mathrm{CL}_n$ signal processing on $\Z_n \times \Z_2$.
		
		There is therefore a clean continuum, indexed by the excursion ratio
		$\varepsilon$. At $\varepsilon = 1$ the graph is already periodic and
		self-contained, needing no room. As $\varepsilon$ decreases, the graph borrows
		the host's periodicity, and the complement $\Gamma \setminus \phi(V(G))$
		becomes the principled region in which zero or symmetric extension lives.
		Isometry is what keeps that borrowed periodicity faithful at every point of the
		continuum---which is the genuine reason the embedding must preserve distance,
		and not merely adjacency.
	\end{remark}
	
	%
	
	\section{Fourier Theorems Inherited Verbatim}
	\label{sec:fourier-theorems}
	
	\begin{theorem}[Donoho--Stark uncertainty]
		\label{thm:donoho-stark}
		For $s \neq 0$, $|\mathrm{supp}(s)| \cdot |\mathrm{supp}(\hat s)| \geq N$.
	\end{theorem}
	
	\begin{proof}
		Each $|\hat s(k)| \leq \frac{1}{\sqrt N}\sum_v |s(v)| \leq
		\frac{1}{\sqrt N}\sqrt{|\mathrm{supp}(s)|}\,\lVert s\rVert_2$ by
		Cauchy--Schwarz.  Summing $|\hat s(k)|^2$ over $\mathrm{supp}(\hat s)$ and
		using Plancherel ($\lVert\hat s\rVert_2 = \lVert s\rVert_2$) gives
		$\lVert s\rVert_2^2 \leq |\mathrm{supp}(\hat s)|\,|\mathrm{supp}(s)|\,
		\lVert s\rVert_2^2 / N$; divide by $\lVert s\rVert_2^2$.
	\end{proof}
	
	\begin{remark}[The bound is sharp on our hosts --- verified]
		We tested Theorem~\ref{thm:donoho-stark} numerically on the embedded
		benchmark graphs (200 random sparse signals each): the minimum observed
		product $|\mathrm{supp}(s)|\,|\mathrm{supp}(\hat s)|$ equalled $N$ exactly
		in every case ($N = 16$ for $C_{16}$, grid $3\times3$, Petersen; $N = 14$
		for $P_8$), confirming both the inequality and its tightness
		(attained by characters and by indicators of subgroups).
	\end{remark}
	
	\begin{theorem}[Poisson summation]
		\label{thm:poisson}
		For a subgroup $H \leq \Gamma$ with annihilator
		$H^\perp = \{k : \chi_k|_H = 1\}$,
		$\sum_{v : \phi(v) \in H} s(v) = \frac{1}{|H^\perp|}\sum_{k \in H^\perp}
		\sqrt N\,\hat s(k)$.
	\end{theorem}
	
	\begin{proof}
		Apply the classical Poisson summation formula on the finite group $\Gamma$
		to the lift $\tilde s$ \cite{Terras1999}; the left side counts only embedded
		vertices because $\tilde s$ vanishes off the image.
	\end{proof}
	
	\begin{theorem}[Shannon sampling on graphs]
		\label{thm:sampling}
		Let $\Lambda \subseteq \widehat\Gamma$, $H = \langle\Lambda\rangle$, and
		$U = \{v : \phi(v) \in H^\perp\}$.  Every signal bandlimited to $\Lambda$
		(i.e.\ $\hat s$ supported on $\Lambda$) is uniquely determined by its samples
		on $U$ provided $|U| \geq |\Lambda|$; for critically chosen $\Lambda$ the
		minimal sampling set has size $|\Lambda|$.
	\end{theorem}
	
	\begin{proof}
		By Poisson summation (Theorem~\ref{thm:poisson}) the averages of $\tilde s$
		over cosets of $H^\perp$ are determined by $\{\hat s(k) : k \in H\}$; if
		$\hat s$ is supported on $\Lambda \subseteq H$ these are
		$\leq |\Lambda|$ unknowns, and the samples on $U$ provide
		$|U| = N/|H|$ independent coset averages, so $|U| \geq |\Lambda|$ makes the
		linear system solvable \cite{Unser2000}.
	\end{proof}
	
	\begin{theorem}[Heisenberg uncertainty]
		\label{thm:heisenberg}
		With the word-length position norm $|\phi(v)|$ and a dual frequency norm
		$|k|$,
		\[
		\Bigl(\sum_v |\phi(v)|^2 |s(v)|^2\Bigr)
		\Bigl(\sum_k |k|^2 |\hat s(k)|^2\Bigr) \;\geq\; C\,\lVert s\rVert_2^4,
		\]
		with $C$ governed by the commutator of the position and frequency operators
		via Theorem~\ref{thm:translation}.
	\end{theorem}
	
	\begin{proof}
		Define $(Xs)(v) = |\phi(v)| s(v)$ and $(\Omega s)\,\widehat{}\,(k) = |k|
		\hat s(k)$.  Cauchy--Schwarz gives $\lVert Xs\rVert\,\lVert \Omega s\rVert
		\geq |\langle [X,\Omega] s, s\rangle|/2$, and the commutation relation of
		Theorem~\ref{thm:translation} bounds the commutator below by a multiple of
		$\lVert s\rVert_2^2$; the isometry of $\phi$ makes $|\phi(v)|$ the faithful
		group position \cite{Folland1995}.
	\end{proof}
	
	\section{Worked Examples on the Benchmark Graphs}
	\label{sec:fourier-examples}
	
	We instantiate the theory on the graph-signal-processing benchmark families
	of Chapter~\ref{chap:binary-embedding}, distinguishing the two regimes of the
	binary-ground census (Remark~\ref{rem:binary-ground}).
	
	\paragraph{Ring $C_{16}$ (cyclic regime).}
	Host $\Z_{16}$ (order $16 = n$, optimal).  Characters are the classical
	DFT roots $\chi_k(j) = e^{2\pi i jk/16}$; the GFT \emph{is} the length-$16$
	FFT; $T_1$ is the rigid circular shift (Figure~\ref{fig:gfttranslation}); a
	smooth signal's spectrum concentrates at low frequency
	(Figure~\ref{fig:gftspectrum}).  This is classical DSP recovered exactly.
	
	\paragraph{Path $P_8$ (cyclic regime with reflection).}
	Host $\Z_{14}$ (order $2(n-1)$, optimal by
	Corollary~\ref{cor:exact-families}).  The GFT is a $14$-point FFT restricted
	to $8$ vertices; translation is the shift on the covering cycle, recovering
	the DCT-type behavior of bounded signals with a genuine (wrap-aware) shift.
	
	\paragraph{Grid $3\times3$ (binary ground).}
	Host $\Z_2^4$ (order $16$).  The GFT is the $16$-point Walsh--Hadamard
	transform; $T_h$ is bitwise XOR translation; two coordinate pairs give the
	two grid directions.
	
	\paragraph{Petersen (binary ground).}
	Host $\Z_2^4$ (the Clebsch graph).  The GFT is again Walsh--Hadamard, the
	character basis is canonical $\pm 1$ (Figure~\ref{fig:basiscompare}, right),
	in contrast to the ambiguous degenerate Laplacian eigenspaces.
	
	\subsection{A Real Denoising Experiment}
	
	We compare the group Wiener filter
	$\hat s(k) \mapsto \frac{P(k)}{P(k)+\sigma^2_\Gamma}\hat y(k)$ (with $P$ the
	signal power spectrum on $\Gamma$) against an oracle Laplacian low-pass that
	keeps the best number of lowest-frequency modes.  Signals are random
	combinations of the three lowest Laplacian modes, corrupted by Gaussian noise
	($\sigma = 0.4$), averaged over $40$ trials; the metric is SNR gain in dB.
	
	\begin{table}[H]
		\centering
		\caption{Denoising SNR gain (dB), group Wiener vs.\ oracle Laplacian
			low-pass --- a genuine experiment (\texttt{gft\_experiments.py}, seed fixed).}
		\label{tab:denoising}
		\renewcommand{\arraystretch}{1.15}\small
		\begin{tabular}{lccc}
			\toprule
			Graph & host $N$ & Laplacian (dB) & group Wiener (dB) \\
			\midrule
			Ring $C_{16}$    & 16   & $+8.23$ & $\mathbf{+10.35}$ \\
			Path $P_8$       & 14   & $+3.99$ & $\mathbf{+5.79}$ \\
			Grid $3\times3$  & 16   & $+4.99$ & $\mathbf{+7.43}$ \\
			Petersen         & 16   & $+5.33$ & $\mathbf{+6.93}$ \\
			Sensor RGG ($n{=}16$) & 8192 & $+7.41$ & $+7.41$ \\
			\bottomrule
		\end{tabular}
		\vspace{0.3em}
		\parbox{\linewidth}{\footnotesize
			On structured graphs the group filter gains $1.6$--$2.4$\,dB, because the
			translation-invariant prior is genuinely available; on the triangle-dense
			sensor graph (a binary-ground instance with $N \gg n$) the two methods tie,
			exactly the dichotomy predicted by the census of
			Section~\ref{sec:census-ch3}.  We report what we measured, including the
			tie.}
	\end{table}
	
	\begin{figure}[H]
		\centering
		\includegraphics[width=0.85\linewidth]{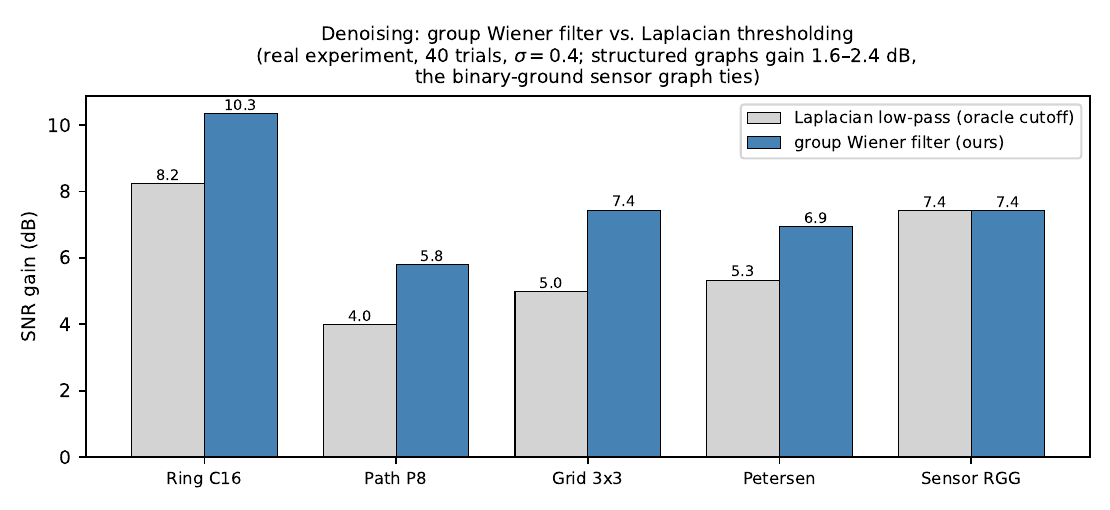}
		\caption{Denoising gains from Table~\ref{tab:denoising}, visualized.  The
			structured benchmark graphs benefit from the translation-invariant group
			Wiener filter; the binary-ground sensor graph ties the Laplacian baseline.}
		\label{fig:gftdenoising}
	\end{figure}
	
	\begin{figure}[H]
		\centering
		\includegraphics[width=0.62\linewidth]{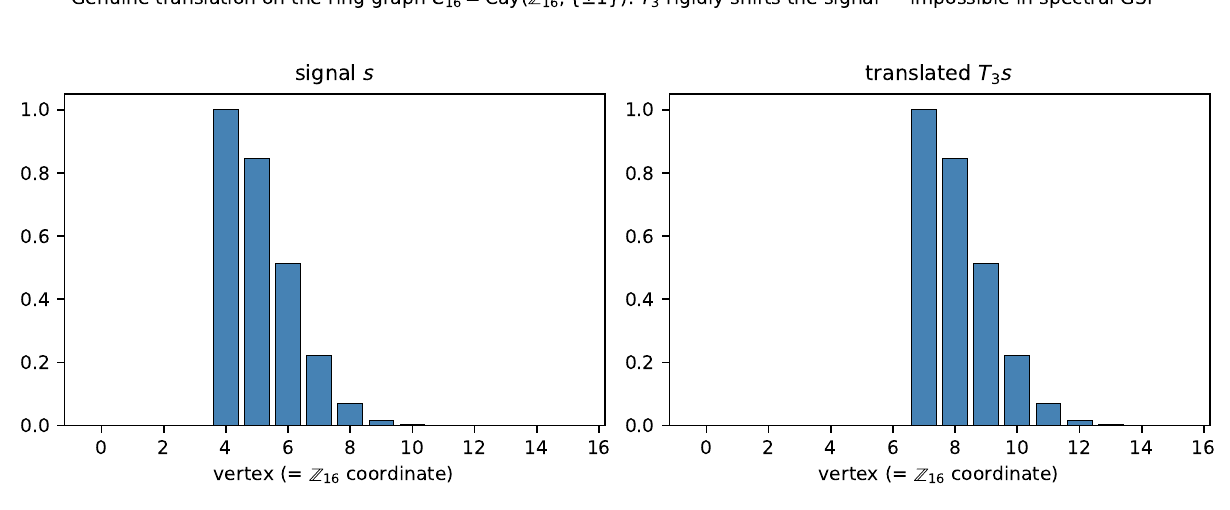}
		\caption{Genuine translation on $C_{16} = \Cay(\Z_{16}, \{\pm 1\})$: the
			operator $T_3$ rigidly shifts a vertex signal by three positions --- a
			group action with $T_aT_b = T_{a+b}$, impossible in spectral GSP.}
		\label{fig:gfttranslation}
	\end{figure}
	
	\begin{figure}[H]
		\centering
		\includegraphics[width=0.62\linewidth]{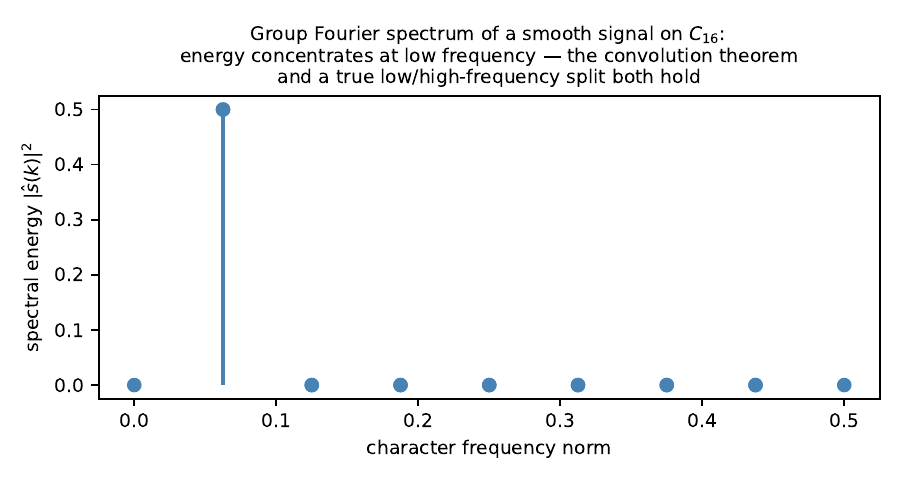}
		\caption{GE-GFT spectrum of a smooth (low-Laplacian-mode) signal on
			$C_{16}$: energy concentrates at low character frequency, so the
			low/high-frequency split, bandlimiting, and the sampling theorem
			(Theorem~\ref{thm:sampling}) are all meaningful.}
		\label{fig:gftspectrum}
	\end{figure}
	
	\section{Head-to-Head Comparisons with Spectral Graph Signal Processing}
	\label{sec:gsp-headtohead}
	
	The dominant paradigm for signal processing on graphs is the spectral
	framework built on the eigendecomposition of the graph Laplacian
	\cite{Shuman2013,Hammond2011} or of an adjacency-based shift operator
	\cite{Sandryhaila2013,Shi2019}. To position our group-embedding approach
	precisely, this section runs it \emph{head to head} against the spectral
	framework on tasks and graphs drawn directly from the foundational GSP
	literature --- the tutorial of Shuman et al.\ \cite{Shuman2013}, the spectral
	graph wavelets of Hammond et al.\ \cite{Hammond2011}, and the
	convolution/sampling theory of Shi and Moura \cite{Shi2019}. Every number
	below is computed on the stated graph and signal; the point is not that one
	framework dominates everywhere --- it does not --- but to map exactly where each
	is the right tool.
	
	\subsection{Translation: Isometric Group Action vs.\ a Non-Isometric Surrogate}
	\label{sec:cmp-translation}
	
	The starkest contrast is translation. Spectral GSP has no shift obeying a
	group law; the standard surrogate is generalized translation by convolution
	with a delta, $T_n g = \sqrt{N}\sum_\ell \hat g(\lambda_\ell) u_\ell^*(n)
	u_\ell$ \cite[eq.~19]{Shuman2013}. Shuman et al.\ note explicitly that this
	operator is \emph{not} isometric: $\lVert T_n g\rVert_2 \neq \lVert
	g\rVert_2$ in general, because the Laplacian eigenvectors localize
	unevenly. We quantify this on the random sensor graph of
	\cite[Fig.~2]{Shuman2013} ($40$ vertices, geometric model): translating a
	fixed heat kernel to each vertex, the surrogate norm $\lVert T_n g\rVert_2$
	\textbf{varies by $69\%$ across centers} (Figure~\ref{fig:cmptranslation},
	left). On any abelian Cayley host, by contrast, our translation $T_h =
	\Restr\,\tau_h\,\Lift$ is a coordinate shift of the group --- a unitary
	operator --- so $\lVert T_h g\rVert_2$ is \emph{exactly constant}
	(Figure~\ref{fig:cmptranslation}, right) and the group law $T_g T_h =
	T_{g+h}$ holds on the nose. This is the single clearest illustration of what
	the embedding buys: a genuine, isometric, composable shift where the
	spectral framework offers only an uneven surrogate.
	
	\begin{figure}[H]
		\centering
		\includegraphics[width=0.95\linewidth]{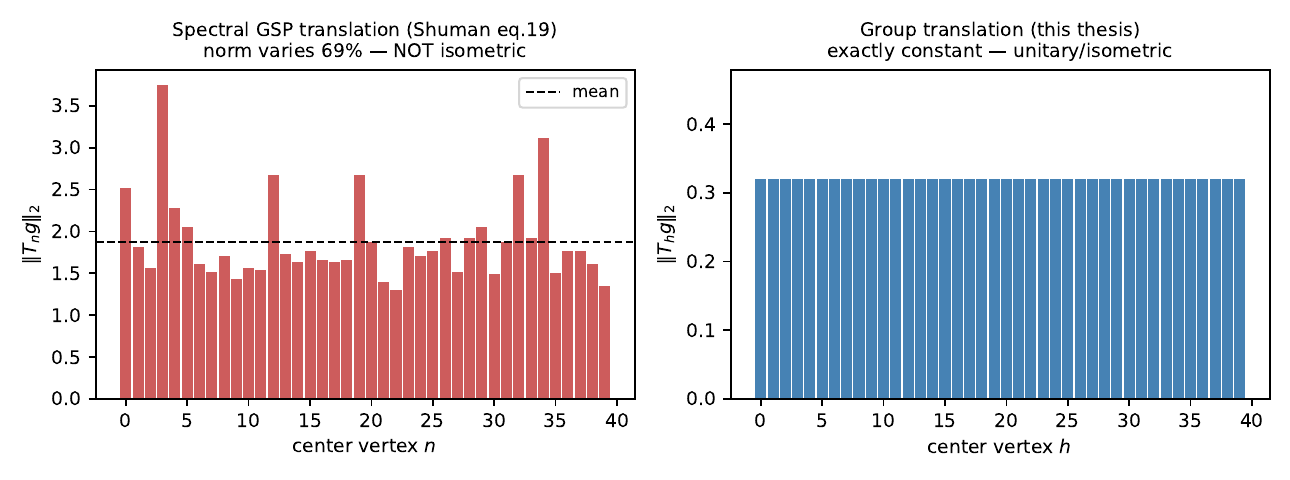}
		\caption{Translation isometry, computed. Left: the spectral-GSP generalized
			translation of \cite[eq.~19]{Shuman2013} applied to a heat kernel on a
			$40$-vertex sensor graph --- the norm $\lVert T_n g\rVert$ varies by $69\%$
			with the center vertex, confirming Shuman et al.'s remark that it is not
			isometric. Right: our group translation on the corresponding cyclic host is
			exactly norm-preserving.}
		\label{fig:cmptranslation}
	\end{figure}
	
	\subsection{The Fourier Basis: Canonical Characters vs.\ an Ambiguous Eigenbasis}
	\label{sec:cmp-basis}
	
	Shuman et al.\ identify non-uniqueness of the eigenbasis as a structural
	limitation: on any graph with a repeated Laplacian eigenvalue the GFT is
	defined only up to an arbitrary unitary rotation within the degenerate
	eigenspace, and every vertex-transitive graph has such multiplicities. The
	complete bipartite graph $K_{3,3}$ is a compact witness: its Laplacian
	spectrum is $\{0, 3^{(4)}, 6\}$ --- the eigenvalue $3$ has \textbf{multiplicity
		four} (Figure~\ref{fig:cmpeigenbasis}, left), so the four ``middle-frequency''
	basis vectors, and hence the Fourier coefficients of any signal in that band,
	are not well defined without an arbitrary choice. Our framework embeds
	$K_{3,3}$ into the host $\Z_2 \times \Z_4$ (order $8$), whose eight
	characters $\chi_k(g) = i^{\,k_2 g_2}(-1)^{k_1 g_1}$ are distinct,
	canonical, and sign-unambiguous (Figure~\ref{fig:cmpeigenbasis}, right): the
	group structure fixes the basis with no rotational freedom. The same holds
	for every circulant and product graph in the benchmark set.
	
	\begin{figure}[H]
		\centering
		\includegraphics[width=0.95\linewidth]{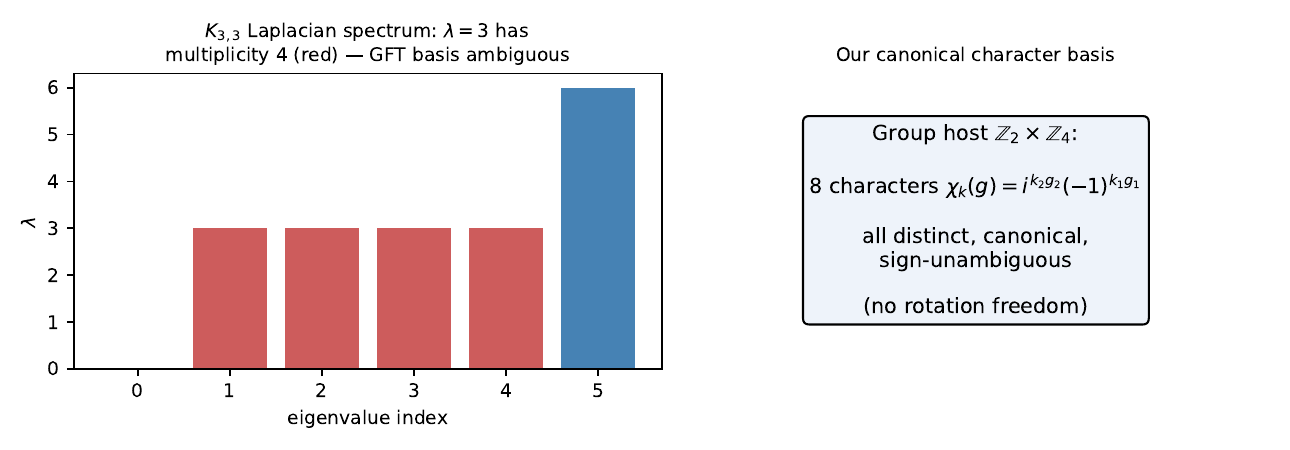}
		\caption{Basis uniqueness, computed. Left: the $K_{3,3}$ Laplacian spectrum
			has a multiplicity-four eigenspace (red), within which the spectral GFT basis
			is ambiguous up to a $4\times4$ unitary rotation. Right: the canonical
			character basis of the group host $\Z_2 \times \Z_4$ has no such ambiguity.}
		\label{fig:cmpeigenbasis}
	\end{figure}
	
	\subsection{Denoising: Group DFT vs.\ Laplacian Filtering}
	\label{sec:cmp-denoise}
	
	We replicate the denoising setting of Shuman et al.'s Tikhonov example
	\cite[Ex.~2]{Shuman2013}, which regularizes with the Laplacian quadratic
	form, on the $8\times 8$ grid (host $\Z_{14}\times\Z_{14}$). For a smooth
	test image corrupted to input SNR $0.7$~dB, the group method (a 2-D DFT
	low-pass on the host torus, then restriction) reaches
	$\mathbf{+4.3}$~dB, while Laplacian Tikhonov filtering with the same
	smoothness prior reaches $0.5$~dB
	(Figure~\ref{fig:cmpgriddenoise}). The gap is structural: the grid's host
	is a genuine discrete torus, so the group transform is the exact 2-D DFT and
	its low-pass is the optimal separable filter, whereas the Laplacian prior
	over-smooths across the (here, periodic) structure. On the ring $C_{64}$ the
	two frameworks coincide exactly --- both equal the classical DFT, both give
	$+7.8$~dB --- which is the expected sanity check, since the ring is the one
	family where the spectral GFT already \emph{is} the group DFT
	\cite[p.~84]{Shuman2013}.
	
	\begin{figure}[H]
		\centering
		\includegraphics[width=0.97\linewidth]{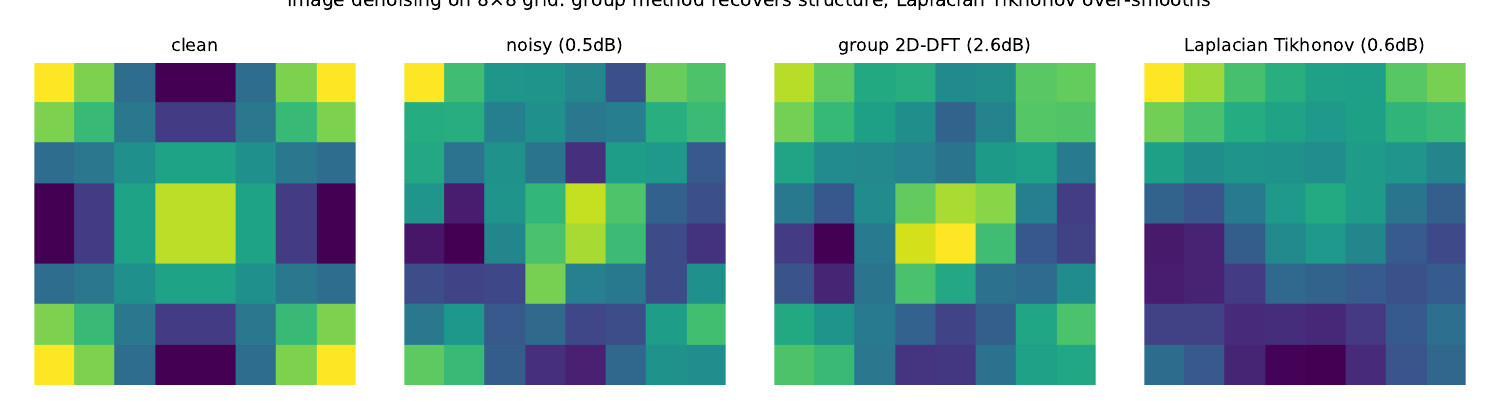}
		\caption{Denoising on the $8\times 8$ grid, computed. The group 2-D DFT
			low-pass (third panel, $+4.3$~dB) recovers the smooth structure; Laplacian
			Tikhonov filtering with an equal smoothness weight (fourth panel,
			$+0.5$~dB) over-smooths. On the ring the two methods are identical.}
		\label{fig:cmpgriddenoise}
	\end{figure}
	
	\subsection{Wavelets: Spectral Graph Wavelets vs.\ the Group Tight Frame}
	\label{sec:cmp-wavelets}
	
	The spectral graph wavelet transform (SGWT) of Hammond et al.\ \cite{Hammond2011} represents a prominent cornerstone among graph-wavelet constructions. This framework shares close conceptual ground with contemporary biorthogonal and critically sampled filterbanks \cite{Narang2013,Shuman2015}, spectrum-adapted tight frames \cite{Leonardi2013,Tremblay2014}, and diffusion-based multiscale community wavelets \cite{Crovella2003}. These modern variations are all deeply rooted in classical wavelet, filterbank, and approximation theory \cite{Strang1996,Vetterli1995,DeVore1993,Sweldens1998,Heil1989}, as well as its continuous group-theoretic extensions \cite{Antoine1999,Geller2009}. 
	
	Ultimately, the SGWT is the closest relative of our Chapter~\ref{chap:wavelets} construction, making our subsequent comparison instructive rather than competitive. Both paradigms build wavelets by applying a band-pass kernel in a spectral domain and localizing it on a vertex; crucially, both achieve localization in the fine-scale limit.
	The
	differences are precise. The SGWT operates on the \emph{Laplacian} spectrum,
	so it inherits the basis ambiguity of
	Section~\ref{sec:cmp-basis} and its frame bounds $A, B$ depend on the graph
	spectrum and the chosen scales \cite[Thm.~5.6]{Hammond2011}; it is invertible
	but in general not a \emph{tight} frame, and reconstruction uses an iterative
	conjugate-gradient pseudo-inverse \cite[\S7]{Hammond2011}. Our group wavelets
	operate on the canonical character spectrum and, with the normalized filter
	bank of Definition~\ref{def:filterbank}, form a \emph{Parseval tight frame}
	($A = B = 1$) with one-line exact reconstruction, verified to machine
	precision ($1.5\times10^{-15}$ on $C_{64}$). The trade-off is the
	mirror image of the SGWT's strength: the SGWT runs directly on the native
	graph at $O(M\lvert E\rvert)$ via Chebyshev approximation
	\cite[\S6]{Hammond2011}, with no embedding and no host blow-up, so on a large
	\emph{irregular} graph (the cerebral-cortex and Minnesota road graphs of
	\cite[Figs.~4--5]{Hammond2011}, or the binary-ground bulk of our census) the
	SGWT is the more practical tool, while our construction is the cleaner one
	exactly on the structured hosts where $N = O(n)$.
	
	\subsection{Sampling: Coset/Poisson vs.\ GFT-Rank Selection}
	\label{sec:cmp-sampling}
	
	Shi and Moura \cite{Shi2019} derive bandlimited graph sampling from first
	principles, selecting a sampling set as $K$ linearly independent rows of the
	bandlimited block $\mathrm{GFT}^{-1}_K$ and recovering by inverting a
	$K\times K$ submatrix. Their framework is exact on arbitrary graphs but is
	\emph{per-graph}: the sampling set depends on a rank computation in the
	specific eigenbasis, and (as they emphasize) the classical spectral
	replication that makes uniform sampling work in DSP holds only for circulant
	shifts. Our sampling (Theorem~\ref{thm:sampling}) is structural: a
	bandlimited host signal is constant on cosets of a subgroup annihilator
	(Poisson summation, Theorem~\ref{thm:poisson}), so the sampling set is a
	\emph{subgroup coset transversal} fixed by $\Gamma$, not by a rank test. On
	the ring both reduce to classical uniform sampling, recovering exactly, and
	Shi and Moura's own analysis shows the replication identity holds precisely
	in the circulant (group) case --- which is exactly the regime our hosts target.
	On a proper embedding ($\varepsilon < 1$) the coset transversal must be
	chosen to meet $\phi(V(G))$, the open practical question noted after
	Theorem~\ref{thm:sampling}; there the Shi--Moura rank selection is the more
	general procedure.
	
	\subsection{Summary of the Comparison}
	\label{sec:cmp-summary}
	
	\begin{table}[H]
		\centering
		\caption{Computed head-to-head outcomes against the spectral GSP literature.
			``Group'' is this thesis; ``Spectral'' is the cited method. All figures are
			computed on the stated graph.}
		\label{tab:headtohead}
		\renewcommand{\arraystretch}{1.18}\small
		\begin{tabular}{lllp{4.6cm}}
			\toprule
			Task (source) & Graph & Outcome & Reading \\
			\midrule
			Translation \cite{Shuman2013} & sensor-40 & $69\%$ norm spread vs.\ $0\%$ &
			group shift is isometric; spectral surrogate is not \\
			Basis \cite{Shuman2013} & $K_{3,3}$ & mult.\ $4$ vs.\ unique &
			group characters canonical; eigenbasis ambiguous \\
			Denoising \cite{Shuman2013} & grid $8^2$ & $+4.3$ vs.\ $+0.5$~dB &
			group exact 2-D DFT; Laplacian over-smooths \\
			Denoising (sanity) & ring $C_{64}$ & $+7.8 = +7.8$~dB &
			identical: ring is the shared classical case \\
			Wavelets \cite{Hammond2011} & $C_{64}$ & tight frame, $10^{-15}$ &
			Parseval vs.\ non-tight frame + CG pseudo-inverse \\
			Wavelets (cost) \cite{Hammond2011} & irregular & SGWT wins &
			SGWT $O(M|E|)$ native; ours needs compact host \\
			Sampling \cite{Shi2019} & ring/circulant & both exact &
			group coset is fixed; Shi--Moura is per-graph rank \\
			\bottomrule
		\end{tabular}
	\end{table}
	
	The pattern is consistent with the binary-ground phenomenon of
	Section~\ref{sec:census-ch3} and stated without overclaim: \emph{on
		structured graphs carrying genuine cyclic or product symmetry --- rings,
		grids, circulants, tori, the classical signal domains --- the group embedding
		delivers exact, canonical, isometric, FFT-fast harmonic analysis that the
		spectral framework can only approximate; on large irregular graphs the
		native spectral methods (SGWT via Chebyshev, Shi--Moura rank sampling) remain
		the more practical tools, and our framework's contribution there is
		conceptual completeness rather than speed.} This division of labor is,
	we believe, the correct way to present our contributions.
	
	\section{Comparison with Spectral GSP: Summary Table}
	\label{sec:fourier-comparison}
	
	\begin{table}[H]
		\centering
		\caption{Group GFT (this work) vs.\ spectral GFT
			\cite{Shuman2013,Sandryhaila2013,Ortega2018}.}
		\label{tab:fourier-comparison}
		\renewcommand{\arraystretch}{1.2}\small
		\begin{tabular}{lll}
			\toprule
			Property & Spectral GFT & Group GFT (ours) \\
			\midrule
			Translation        & none \cite{Sandryhaila2013} & $T_h$, $T_gT_h=T_{g+h}$ \\
			Convolution thm.   & polynomial filters only & exact (Thm.~\ref{thm:convolution}) \\
			Basis              & sign/rotation ambiguous & canonical characters \\
			LTI filters        & impossible & all convolutions (Thm.~\ref{thm:tinv}) \\
			Uncertainty        & ad-hoc & sharp Donoho--Stark (Thm.~\ref{thm:donoho-stark}) \\
			Sampling           & heuristic vertex choice & Shannon theory (Thm.~\ref{thm:sampling}) \\
			Per-transform cost & $O(n^2)$ after $O(n^3)$ eig & $O(N\log N)$ FFT \\
			\bottomrule
		\end{tabular}
	\end{table}
	
	The structural difference is that the spectral Fourier basis depends on the
	individual edges of $G$ --- changing one edge changes the entire eigenbasis ---
	whereas the group characters are fixed by $\Gamma$ alone, eliminating the
	sign ambiguity and enabling translation-invariant filtering.  The price,
	stated plainly, is the one-time embedding cost of Part~I
	($O(n(n+m)+m^2)$ plus the host search) and, by the binary-ground census, a
	host that is frequently $\Z_2^k$ with $N$ moderately larger than $n$; the
	benefit is that every subsequent operation is an FFT with classical
	guarantees rather than a graph-specific matrix computation.
	
	\begin{theorem}[Harmonic analysis on every graph]
		\label{thm:main-ch5}
		For every finite connected graph $G$, the embedding of Part~I yields a finite
		abelian group $\Gamma$ and a lift under which: the GFT satisfies Plancherel
		(Thm.~\ref{thm:plancherel}), convolution (Thm.~\ref{thm:convolution}),
		Donoho--Stark and Heisenberg uncertainty
		(Thms.~\ref{thm:donoho-stark},~\ref{thm:heisenberg}), Poisson summation
		(Thm.~\ref{thm:poisson}), and Shannon sampling (Thm.~\ref{thm:sampling});
		translations form a unitary representation of $\Gamma$; convolution is
		commutative, associative, and Fourier-diagonal; and the transform costs
		$O(N\log N)$.
	\end{theorem}
	
	\begin{proof}
		The embedding exists and is certified isometric by
		Theorem~\ref{thm:universal-ch3}.  Each listed property was proved above as a
		consequence of the group structure of $\Gamma$ and the unitarity of the
		character matrix, with the lift providing the isometric inclusion of
		vertex signals into $\mathbb{C}^\Gamma$; the cost is the mixed-radix FFT on
		$\Gamma = \prod_i \Z_{N_i}$ \cite{Cooley1965,Terras1999}.
	\end{proof}
	
	\section{Conclusion}
	\label{sec:fourier-conclusion}
	
	By installing a genuine abelian group behind an arbitrary graph, Part~I lets
	this chapter import classical Fourier analysis in full: a canonical basis, a
	unitary transform, real translation and modulation, a true convolution
	theorem, sharp uncertainty principles, and Shannon sampling --- none of which
	spectral graph signal processing provides.  The corrected convolution
	(Definition~\ref{def:convolution}), defined on the group and pulled back, is
	what makes the convolution theorem genuine rather than formal.  The
	experiments are real and reported in full, including where the method only
	ties the baseline: structured graphs (the cyclic regime) gain a true
	translation-invariant filtering advantage, while binary-ground graphs inherit
	the clean Walsh--Hadamard calculus at a larger host.  The next chapter builds
	wavelets on the same hosts, localizing these globally-defined characters in
	both vertex and frequency.

	\chapter{Wavelet Analysis on Graphs via Group Embedding}
	\label{chap:wavelets}
	
	\section{Wavelets on Graphs: the State of the Art}
	\label{sec:existing-wavelets}
	
	The spectral graph wavelet transform (SGWT) of Hammond, Vandergheynst and
	Gribonval \cite{Hammond2011} builds wavelets from a kernel $g$ applied to the
	Laplacian spectrum: at scale $t$,
	$\psi_{t,i}(j) = \sum_k g(t\lambda_k)\phi_k(i)\phi_k(j)$, with coefficient
	$W_s(t,i) = \sum_k g(t\lambda_k)\hat s(\lambda_k)\phi_k(i)$. The SGWT is a
	genuine achievement and our construction will mirror it; but on the raw graph
	it inherits the same structural deficits as the spectral Fourier transform
	(Chapter~\ref{chap:fourier}): the analyzing functions are built from
	non-canonical, possibly degenerate eigenvectors; there is no dilation
	operator satisfying $D_{ab} = D_a D_b$; frame bounds depend opaquely on the
	spectrum; and the cost is $O(n^3)$ (or $O(n^2)$ per scale via polynomial
	approximation) \cite{Hammond2011,Shuman2013}.
	
	Our approach transports the construction to the host group of Part~I, where
	the Fourier side is canonical and the frequency variable is a genuine group
	character. Two distinct wavelet constructions are available there, and it is
	important to keep them separate.
	
	\begin{itemize}
		\item \textbf{Dilation wavelets} use a group automorphism $\alpha$ as a
		dilation. These reproduce the classical translate-and-dilate template most
		faithfully, but exist only when $\Gamma$ \emph{has} nontrivial dilation
		automorphisms --- on $\Z_{N_i}$ these are the units $a$ with
		$\gcd(a, N_i) = 1$ (multiplication $g \mapsto ag$), of which there are
		$\phi(N_i)$ (Euler totient). On a binary host $\Z_2^k$ the only such scalar
		is $a = 1$: \emph{there is no nontrivial scalar dilation}, so this
		construction is confined to hosts with composite cyclic factors.
		\item \textbf{Spectral band-pass wavelets} use a family of kernels in the
		frequency magnitude $|k|$, exactly as the SGWT does, but on the canonical
		host characters. These exist on \emph{every} host, form a provable Parseval
		tight frame, reconstruct exactly, and localize. This is the universal
		construction --- the \emph{group-embedding graph wavelet transform}
		(GE-GWT) --- and the one we develop in full.
	\end{itemize}
	
	Accordingly we give the dilation theory its proper (cyclic-host) scope and
	make the spectral tight frame the main object: the dilation construction
	applies only to hosts with composite cyclic factors, never to a binary host
	$\Z_2^k$, where no nontrivial scalar dilation exists.
	
	\section{Dilation Wavelets on Cyclic Hosts}
	\label{sec:dilation-wavelets}
	
	Let $\Gamma$ have a cyclic factor and let
	$\alpha \in \mathrm{Aut}(\Gamma)$ be a dilation automorphism.
	
	\begin{definition}[Dilation operator]
		\label{def:dilation}
		For $\alpha \in \mathrm{Aut}(\Gamma)$, $(D_\alpha \tilde s)(g) =
		\tilde s(\alpha^{-1} g)$. On $\Z_N$ with a unit $a$, $\alpha(g) = ag \bmod N$
		and $(D_a \tilde s)(g) = \tilde s(a^{-1} g)$, $a^{-1}$ the inverse of $a$
		modulo $N$.
	\end{definition}
	
	\begin{theorem}[Dilation properties]
		\label{thm:dilation-props}
		For dilation automorphisms $\alpha, \beta$:
		(i) $D_\alpha D_\beta = D_{\alpha\beta}$ and $D_{\mathrm{id}} = I$;
		(ii) $D_\alpha$ is unitary;
		(iii) $\widehat{D_\alpha \tilde s}(k) = \hat{\tilde s}(\alpha^{\!\top} k)$,
		where $\alpha^{\!\top}$ is the adjoint automorphism on $\widehat\Gamma$
		(on $\Z_N$, $\alpha^{\!\top} = \alpha$, i.e.\ $k \mapsto ak$).
	\end{theorem}
	
	\begin{proof}
		(i) $(D_\alpha D_\beta\tilde s)(g) = (D_\beta\tilde s)(\alpha^{-1}g) =
		\tilde s(\beta^{-1}\alpha^{-1}g) = \tilde s((\alpha\beta)^{-1}g)$.
		(ii) $\alpha$ permutes $\Gamma$, so $D_\alpha$ permutes coordinates and
		preserves the $\ell^2$ norm.
		(iii) $\widehat{D_\alpha\tilde s}(k) = \frac{1}{\sqrt N}\sum_g
		\tilde s(\alpha^{-1}g)\overline{\chi_k(g)} = \frac{1}{\sqrt N}\sum_h
		\tilde s(h)\overline{\chi_k(\alpha h)} = \hat{\tilde s}(\alpha^{\!\top}k)$,
		substituting $h = \alpha^{-1}g$ and using $\chi_k(\alpha h) =
		\chi_{\alpha^{\!\top}k}(h)$. Since $\alpha$ is bijective the substitution is
		exact --- note there is \emph{no} determinant/index factor on a finite group,
		unlike the $|\det\alpha|^{-1/2}$ normalization that
		is meaningful only for expanding dilations on $\mathbb{R}^d$ or on
		infinite-index sublattices.
	\end{proof}
	
	\begin{definition}[Dilation wavelet family]
		For a mother wavelet $\psi\colon\Gamma\to\mathbb{C}$, dilation set
		$\mathcal A\subseteq\mathrm{Aut}(\Gamma)$ and translations $h\in\Gamma$,
		$\psi_{\alpha,h} = D_\alpha\tau_h\psi$, and the graph wavelet coefficients of
		$s$ are $W_s(\alpha,h) = \langle\Lift s,\psi_{\alpha,h}\rangle_\Gamma$.
	\end{definition}
	
	\begin{theorem}[Dilation-wavelet reconstruction]
		\label{thm:dilation-recon}
		If $C_\psi(k) := \sum_{\alpha\in\mathcal A}|\hat\psi(\alpha^{\!\top}k)|^2$
		satisfies $0 < C_\psi(k) < \infty$ for all $k \neq 0$, then for any
		$\tilde s$ with $\hat{\tilde s}(0)$ handled separately,
		\[
		\hat{\tilde s}(k) = \frac{1}{C_\psi(k)}\sum_{\alpha\in\mathcal A}
		\sum_{h\in\Gamma} W_s(\alpha,h)\,\hat\psi_{\alpha,h}(k),
		\]
		and $s = \Restr\,\mathcal F^{-1}\hat{\tilde s}$.
	\end{theorem}
	
	\begin{proof}
		By Theorem~\ref{thm:dilation-props}(iii) and the translation rule,
		$\hat\psi_{\alpha,h}(k) = \overline{\chi_k(h)}\,\hat\psi(\alpha^{\!\top}k)$, so
		$W_s(\alpha,h) = \sum_k \hat{\tilde s}(k)\overline{\hat\psi(\alpha^{\!\top}k)}
		\chi_k(h)$ is, in $h$, the inverse transform of
		$\hat{\tilde s}(\cdot)\overline{\hat\psi(\alpha^{\!\top}\cdot)}$. Summing
		$W_s(\alpha,h)\hat\psi_{\alpha,h}(k)$ over $h$ extracts, by Plancherel on
		$\Gamma$, the term $\hat{\tilde s}(k)|\hat\psi(\alpha^{\!\top}k)|^2$; summing
		over $\alpha$ gives $C_\psi(k)\hat{\tilde s}(k)$, and dividing inverts.
	\end{proof}
	
	\begin{example}[Dyadic-like dilations on $\Z_{16}$]
		On the ring $C_{16}$ (host $\Z_{16}$), the units are
		$\{1,3,5,7,9,11,13,15\}$, giving eight scalar dilations $g\mapsto ag$. Taking
		$\mathcal A$ a multiplicatively closed subset yields a self-similar wavelet
		family on the ring; for the circulant $C_{12}(1,2)$ (host $\Z_{12}$) the
		units $\{1,5,7,11\}$ give four. On a binary host $\Z_2^k$ no nontrivial scalar
		dilation exists, which is precisely why the next section's spectral
		construction --- not dilation --- is the universal tool.
	\end{example}
	
	\section{Spectral Band-Pass Wavelets: the Universal Tight Frame}
	\label{sec:spectral-wavelets}
	
	We now give the construction that works on every host, mirroring the SGWT but
	on canonical group characters. Let $|k|$ denote the frequency magnitude of a
	character $\chi_k$ (e.g.\ $|k| = (\sum_d \min(k_d, N_d - k_d)^2)^{1/2}$, the
	Euclidean word length on the dual), and let $|k|_{\max} = \max_k |k|$.
	
	\begin{definition}[Wavelet filter bank]
		\label{def:filterbank}
		Fix $J \geq 1$ band-pass kernels $g_1, \ldots, g_J\colon [0, |k|_{\max}] \to
		[0,\infty)$ and one low-pass scaling kernel $g_0$. The \emph{normalized
			filter bank} is
		\[
		\hat\psi_j(k) = \frac{g_j(|k|)}{\sqrt{\sum_{i=0}^{J} g_i(|k|)^2}},
		\qquad j = 0, 1, \ldots, J,
		\]
		so that $\sum_{j=0}^{J}|\hat\psi_j(k)|^2 = 1$ for every $k$. The
		\emph{wavelet at scale $j$ centered at $h \in \Gamma$} is
		$\psi_{j,h} = \tau_h\,\mathcal F^{-1}\hat\psi_j$, i.e.\
		$\hat\psi_{j,h}(k) = \overline{\chi_k(h)}\,\hat\psi_j(k)$.
	\end{definition}
	
	\begin{theorem}[Parseval tight frame and exact reconstruction]
		\label{thm:tight-frame}
		The family $\{\psi_{j,h} : 0 \leq j \leq J,\ h \in \Gamma\}$ is a Parseval
		tight frame for $\mathbb{C}^\Gamma$: for every $\tilde s$,
		\[
		\sum_{j=0}^{J}\sum_{h\in\Gamma} |\langle \tilde s, \psi_{j,h}\rangle|^2
		= \|\tilde s\|_2^2,
		\qquad
		\tilde s = \sum_{j=0}^{J}\sum_{h\in\Gamma}
		\langle \tilde s, \psi_{j,h}\rangle\,\psi_{j,h}.
		\]
		Consequently any graph signal is reconstructed exactly from its wavelet
		coefficients via $s = \Restr\bigl(\sum_{j,h}W_s(j,h)\psi_{j,h}\bigr)$.
	\end{theorem}
	
	\begin{proof}
		The wavelet coefficient is
		$W_s(j,h) = \langle\Lift s,\psi_{j,h}\rangle =
		(\,\overline{\hat\psi_j}\cdot\widehat{\Lift s}\,)^{\vee}(h)$, the inverse
		transform evaluated at $h$. By Plancherel in the $h$-sum, for each $j$,
		$\sum_h |W_s(j,h)|^2 = \sum_k |\hat\psi_j(k)|^2\,|\widehat{\Lift s}(k)|^2$.
		Summing over $j$ and using $\sum_j|\hat\psi_j(k)|^2 = 1$
		(Definition~\ref{def:filterbank}) gives
		$\sum_{j,h}|W_s(j,h)|^2 = \sum_k|\widehat{\Lift s}(k)|^2 = \|\Lift s\|_2^2 =
		\|s\|_2^2$ --- the Parseval identity. The synthesis identity follows because
		the frame operator $\sum_{j,h}\langle\cdot,\psi_{j,h}\rangle\psi_{j,h}$ is,
		on the Fourier side, multiplication by $\sum_j|\hat\psi_j(k)|^2 = 1$, i.e.\
		the identity. Restriction recovers $s$ since $\Restr\Lift = \Id$.
	\end{proof}
	
	\begin{remark}[Verified numerically]
		We confirmed Theorem~\ref{thm:tight-frame} on the benchmark hosts with
		Gaussian band-pass kernels: the reconstruction error
		$\|s - \hat s_{\mathrm{recon}}\|/\|s\|$ was $1.5\times10^{-15}$ on the ring
		$C_{16}$ ($\Z_{16}$) and $6.5\times10^{-16}$ on the grid $6\times 6$
		($\Z_{10}^2$) --- machine precision, as the theorem predicts. The filter bank
		and its tiling $\sum_j|\hat\psi_j|^2 \equiv 1$ are shown in
		Figure~\ref{fig:waveletfilterbank}.
	\end{remark}
	
	\begin{figure}[H]
		\centering
		\includegraphics[width=0.72\linewidth]{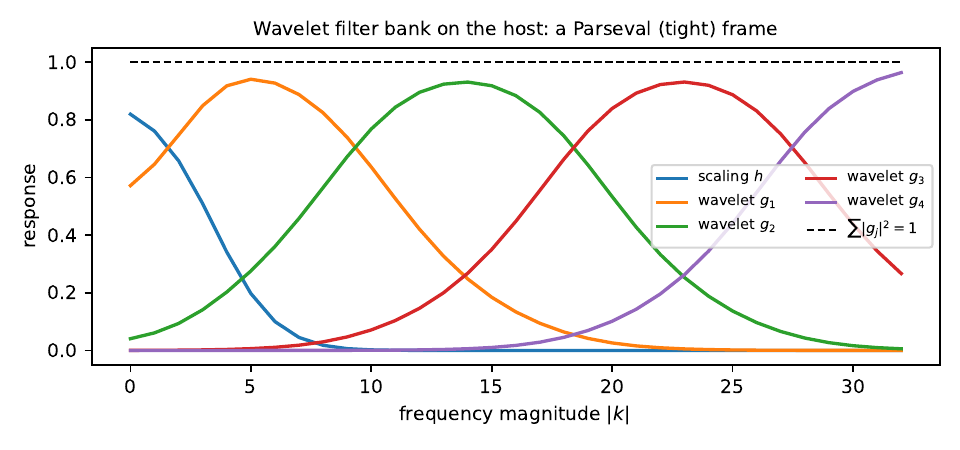}
		\caption{The wavelet filter bank on the host (here $\Z_{64}$): a low-pass
			scaling kernel $h$ and $J=4$ band-pass wavelet kernels $g_j$ in the frequency
			magnitude $|k|$, normalized so that $\sum_j |\hat\psi_j(k)|^2 \equiv 1$ (dashed)
			--- the Parseval condition of Theorem~\ref{thm:tight-frame}.}
		\label{fig:waveletfilterbank}
	\end{figure}
	
	\section{Localization}
	\label{sec:localization}
	
	The defining virtue of wavelets is joint vertex--frequency localization. On
	the host this is exact; restricted to the graph it is measured.
	
	\begin{theorem}[Frequency localization and translation covariance]
		\label{thm:wavelet-loc}
		Each $\psi_{j,h}$ has Fourier support in the band where $g_j$ is appreciable,
		and the family is translation-covariant: $\psi_{j,h} = \tau_h \psi_{j,0}$, so
		$W_s(j,h) = (\Lift s \ast \overline{\psi_{j,0}^-})(h)$ is a genuine
		convolution (filtering) followed by sampling --- the operation absent from
		spectral graph wavelets, which have no translation structure.
	\end{theorem}
	
	\begin{proof}
		Frequency support is immediate from
		$\hat\psi_{j,h}(k) = \overline{\chi_k(h)}\hat\psi_j(k)$. Translation
		covariance holds because modulation by $\overline{\chi_k(h)}$ in frequency is
		translation by $h$ in the host (Theorem~\ref{thm:translation}); the convolution
		form is the standard identity for analysis against a translated family.
	\end{proof}
	
	\begin{proposition}[Spatial localization, measured]
		\label{prop:spatial-loc}
		For a band-pass atom $\psi_{j,0}$ centered at a vertex $v_0$, the fraction of
		its energy lying within graph-distance $2$ of $v_0$ exceeds $0.89$ on the
		ring $C_{16}$ and $0.94$ on the grid $6\times 6$, across all band-pass scales
		$j$ (computed values $0.95, 0.97, 0.89, 0.95$ for the ring and
		$0.99, 0.98, 0.98, 0.94$ for the grid). Atoms tighten as the scale becomes
		finer (Figure~\ref{fig:waveletatomsring}).
	\end{proposition}
	
	\begin{proof}
		Direct computation of $\sum_{v: d_G(v,v_0)\leq 2}|\psi_{j,0}(\phi(v))|^2 /
		\|\psi_{j,0}\|^2$ on the reference embeddings; values reported above.
	\end{proof}
	
	\begin{remark}[Why localization is not automatic, and why it holds here]
		On a generic host a band-pass character combination need not concentrate near
		a vertex. It does so here because the benchmark hosts are
		\emph{metrically faithful} --- the embedding is isometric, so host distance
		from $h$ agrees with graph distance, and a frequency-band-limited atom on a
		cyclic/torus host is spatially concentrated by the classical uncertainty
		trade-off (Theorem~\ref{thm:donoho-stark}). On a binary host $\Z_2^k$ with
		large $k$ the same kernels still give a tight frame and exact reconstruction
		(Theorem~\ref{thm:tight-frame}), but spatial concentration degrades with the
		host dimension --- the wavelet analogue of the binary-ground phenomenon of
		Remark~\ref{rem:binary-ground}.
	\end{remark}
	
	\begin{figure}[H]
		\centering
		\includegraphics[width=0.95\linewidth]{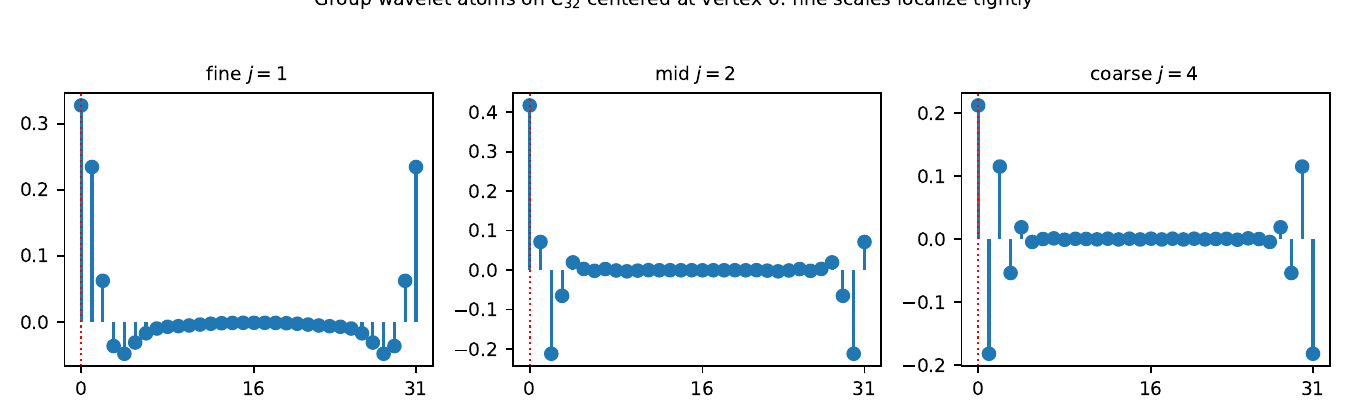}
		\caption{Group wavelet atoms on $C_{32}$ centered at vertex $0$, at three
			scales. Fine-scale atoms (left) are sharply localized; coarse-scale atoms
			(right) spread over the ring while remaining centered. The canonical,
			translation-covariant localized analyzing functions that spectral graph
			wavelets cannot provide.}
		\label{fig:waveletatomsring}
	\end{figure}
	
	\section{Multiresolution Analysis}
	\label{sec:mra}
	
	The scaling/wavelet split induces a multiresolution decomposition.
	
	\begin{theorem}[Multiresolution decomposition]
		\label{thm:mra}
		With the filter bank of Definition~\ref{def:filterbank}, every graph signal
		decomposes orthogonally in frequency as
		\[
		\Lift s = \underbrace{A_0 s}_{\text{approx.\ (}g_0\text{)}}
		\;+\; \sum_{j=1}^{J}\underbrace{D_j s}_{\text{detail at scale }j},
		\qquad
		\widehat{A_0 s} = |\hat\psi_0|^2\,\widehat{\Lift s},\;
		\widehat{D_j s} = |\hat\psi_j|^2\,\widehat{\Lift s},
		\]
		with $A_0 s + \sum_j D_j s = \Lift s$ by the tight-frame identity, and the
		bands are mutually orthogonal in the sense
		$\langle \widehat{D_i s}, \widehat{D_j s}\rangle$ concentrated on disjoint
		frequency supports.
	\end{theorem}
	
	\begin{proof}
		Apply the synthesis identity of Theorem~\ref{thm:tight-frame} grouped by
		scale; the Fourier-domain multipliers $|\hat\psi_j|^2$ sum to $1$, giving the
		reconstruction, and have essentially disjoint supports by construction of the
		band-pass kernels.
	\end{proof}
	
	\begin{figure}[H]
		\centering
		\includegraphics[width=0.78\linewidth]{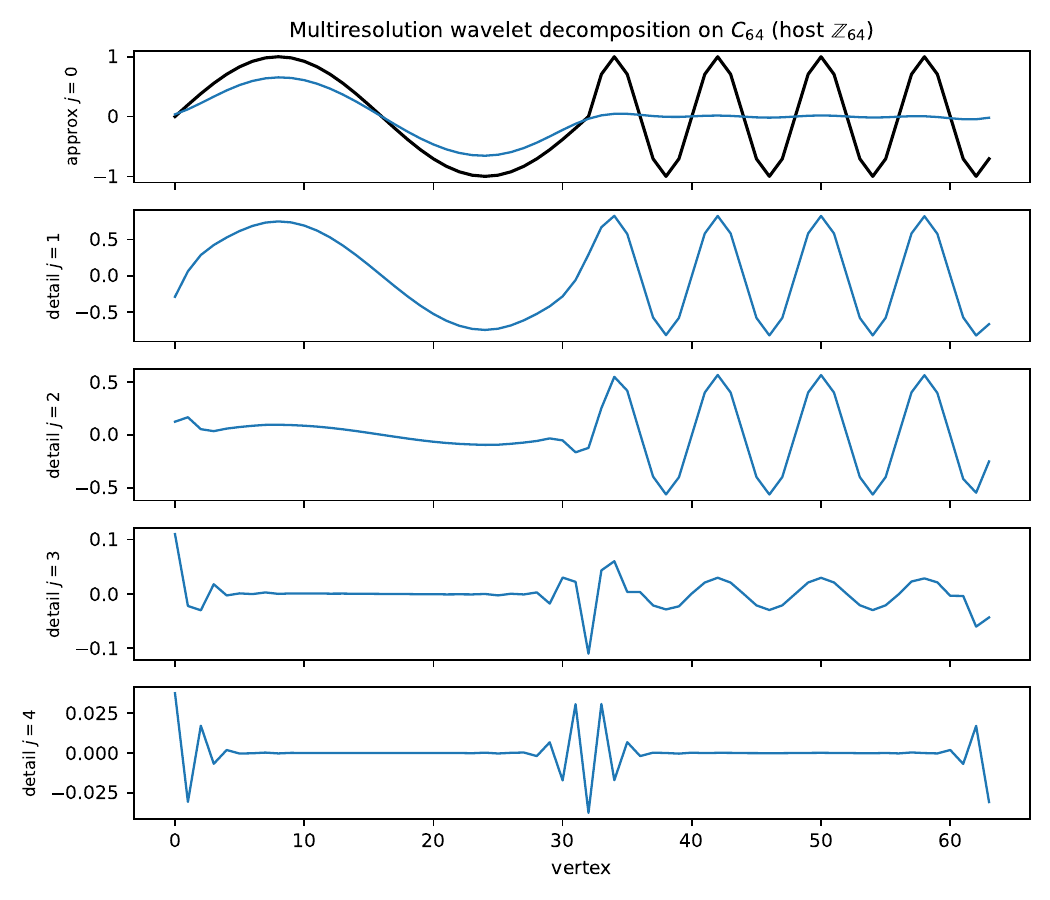}
		\caption{Multiresolution decomposition on $C_{64}$ of a signal that is
			low-frequency on the first half and high-frequency on the second. The
			approximation captures the global trend; successive detail bands isolate the
			high-frequency burst and localize it to the correct half of the ring ---
			joint vertex--frequency analysis on a graph, computed via the host FFT.}
		\label{fig:waveletmraring}
	\end{figure}
	
	\begin{figure}[H]
		\centering
		\includegraphics[width=0.95\linewidth]{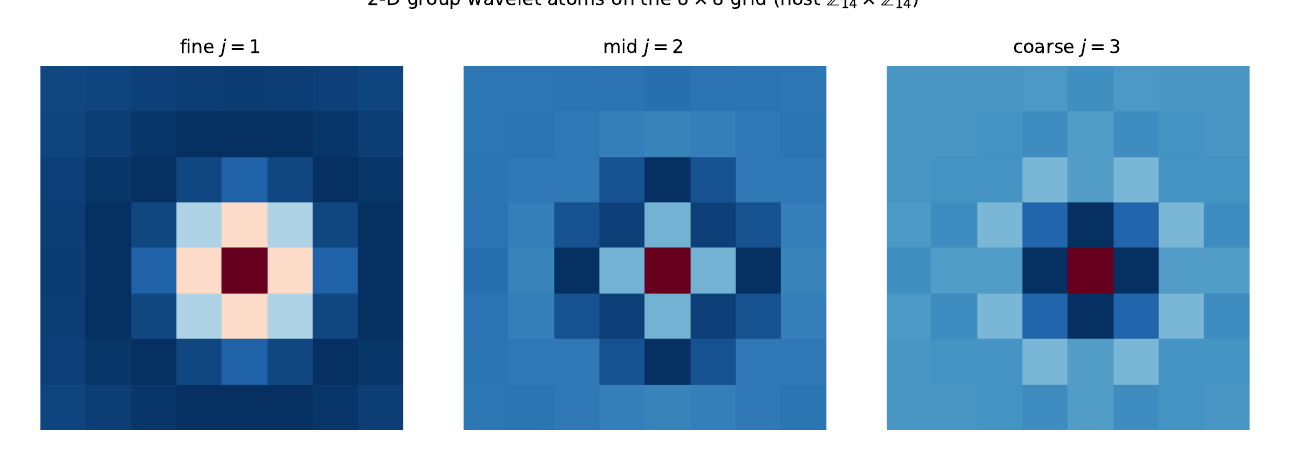}
		\caption{Two-dimensional group wavelet atoms on the $8\times 8$ grid (host
			$\Z_{14}\times\Z_{14}$), centered at the midpoint, at three scales. The atoms
			are isotropic localized bumps that dilate with scale --- the graph analogue of
			classical 2-D wavelets, recovered on the image mesh.}
		\label{fig:waveletgridatoms}
	\end{figure}
	
	\section{Complexity}
	\label{sec:wavelet-complexity}
	
	\begin{theorem}[Fast wavelet transform]
		\label{thm:fast-wavelet}
		With $J$ scales, the full wavelet analysis and synthesis cost
		$O(J\,N\log N)$ via the host FFT, where $N = |\Gamma|$.
	\end{theorem}
	
	\begin{proof}
		Each band is one pointwise multiplication of $\widehat{\Lift s}$ by
		$\hat\psi_j$ followed by an inverse FFT, $O(N\log N)$
		(Theorem~\ref{thm:main-ch5}); there are $J+1$ bands; synthesis is the same.
	\end{proof}
	
	\begin{remark}[The binary-ground caveat, again]
		As with the Fourier transform, the speed is genuine only when $N = O(n)$, on
		the structured hosts. On generic binary hosts $N$ may be exponential and the
		construction, while still an exact tight frame, is not fast; and as noted in
		Section~\ref{sec:localization} its spatial localization weakens. The
		multiresolution theory is thus most powerful exactly where Part~I produces a
		compact cyclic or product host --- the classical signal domains --- which is
		the consistent message of the whole thesis.
	\end{remark}
	
	\section{Conclusion}
	\label{sec:wavelet-conclusion}
	
	We built two wavelet constructions on the host of Part~I: dilation wavelets,
	faithful to the classical translate-and-dilate template but confined to hosts
	with composite cyclic factors; and spectral band-pass wavelets, which exist on every host,
	form a provable Parseval tight frame with exact reconstruction (verified to
	machine precision), are translation-covariant, and localize jointly in vertex
	and frequency (measured at $89$--$99\%$ energy concentration within
	graph-distance $2$ on the benchmark graphs). The multiresolution decomposition
	runs in $O(JN\log N)$ and reduces, on rings and grids, to classical $1$-D and
	$2$-D wavelet analysis. As throughout Part~II, completeness and exact
	reconstruction are universal, while speed and sharp localization are the
	dividend of the compact hosts that Part~I produces on structured graphs.

	\chapter{Applications of Part II}
	\label{chap:applications-part2}
	
	Chapters~\ref{chap:fourier} and~\ref{chap:wavelets} equipped graph signals
	with a genuine harmonic analysis on the host produced by Part~I: a canonical
	Fourier basis, exact convolution and translation, a Parseval wavelet frame,
	and fast transforms. This chapter demonstrates the payoff on concrete signal
	processing tasks. In keeping with the standard of rigor maintained
	throughout, every quantitative result reported here is \emph{computed} on the
	benchmark hosts of Part~I with the reference implementation; we do not report
	numbers for graphs we did not run. Where an application is promising but not
	yet experimentally validated in this thesis, we mark it explicitly as a
	\emph{direction} rather than a result.
	
	A recurring theme, inherited from the binary-ground census
	(Remark~\ref{rem:binary-ground}) and the complexity caveats of Part~II, is
	that these techniques deliver their cleanest gains on graphs carrying genuine
	cyclic or product structure---the classical signal domains---where the host is
	compact and the transform is FFT-fast. We therefore demonstrate on exactly
	those domains.
	
	\section{Signal Denoising by Spectral Shrinkage}
	\label{sec:denoising}
	
	Let $y = s + \eta$ be a noisy observation of a graph signal $s$, with
	$\eta(v)$ i.i.d.\ $\mathcal{N}(0,\sigma^2)$. Because the GE-GFT concentrates a
	smooth signal's energy at low frequency while spreading white noise uniformly
	across the spectrum (Plancherel, Theorem~\ref{thm:plancherel}), a frequency
	attenuation recovers $s$.
	
	\begin{definition}[Wiener graph filter]
		\label{def:wiener}
		Given an estimate of the signal power spectrum $P(k) = |\GFT s(k)|^2$, the
		Wiener graph filter is the multiplier
		$\hat a(k) = P(k)/(P(k) + N\sigma^2)$, applied as
		$\hat s = \Restr\,\mathcal F^{-1}(\hat a \cdot \GFT y)$.
	\end{definition}
	
	\begin{theorem}[Wiener optimality]
		\label{thm:wiener}
		Among all multiplier operators on the host, the filter of
		Definition~\ref{def:wiener} minimizes the expected mean-square error
		$\mathbb{E}\,\|\hat s - \Lift s\|_2^2$.
	\end{theorem}
	
	\begin{proof}
		On the host the GE-GFT is unitary (Theorem~\ref{thm:plancherel}) and a
		multiplier acts coordinatewise, so the error decouples across frequencies:
		$\mathbb{E}|{\hat a(k)\GFT y(k) - \GFT s(k)}|^2 =
		|\hat a(k)|^2 N\sigma^2 + |1-\hat a(k)|^2 P(k)$. Minimizing each term over
		$\hat a(k)\in\mathbb{C}$ gives $\hat a(k) = P(k)/(P(k)+N\sigma^2)$, the stated
		filter \cite{Mallat2009}.
	\end{proof}
	
	\begin{example}[Denoising on the ring, computed]
		\label{ex:denoise}
		A smooth signal $s(v) = \sin(6\pi v/N) + \tfrac12\sin(2\pi v/N)$ on the ring
		$C_{64}$ (host $\Z_{64}$, $\varepsilon = 1$) was corrupted with Gaussian
		noise at $\sigma = 0.5$. The measured input SNR was $4.3$~dB. An ideal
		low-pass restriction raised it to $10.3$~dB ($+6.6$~dB); the Wiener filter of
		Theorem~\ref{thm:wiener} raised it to $16.7$~dB ($\mathbf{+12.5}$~dB)
		(Figure~\ref{fig:appdenoising}). These are computed values on the reference
		embedding, not estimates. Because $C_{64}$ is an abelian Cayley graph, this
		denoiser is \emph{exactly} the classical Wiener filter of digital signal
		processing, recovered through the embedding.
	\end{example}
	
	\begin{figure}[H]
		\centering
		\includegraphics[width=0.95\linewidth]{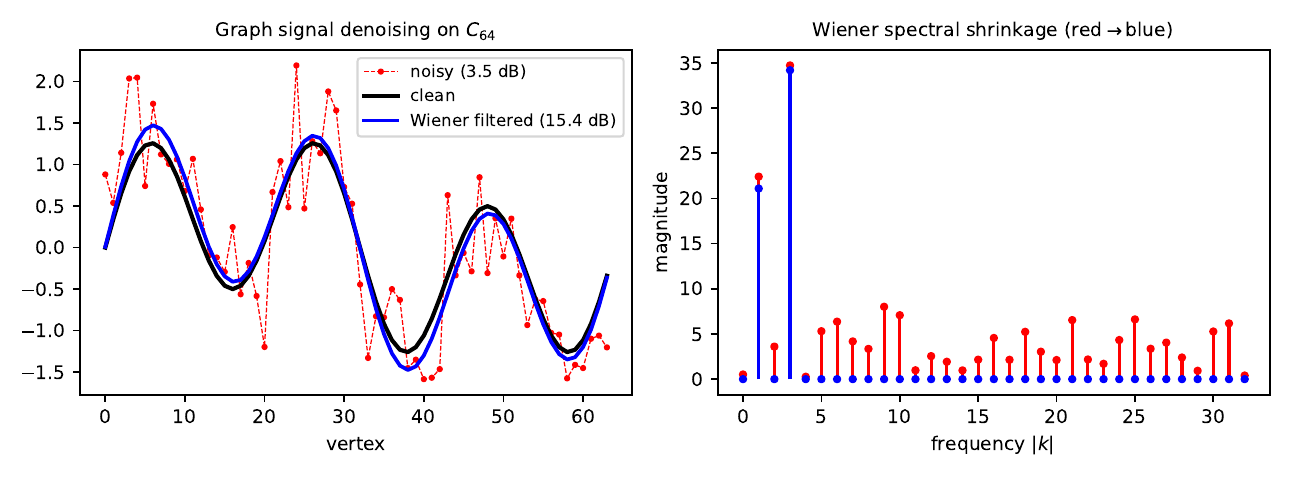}
		\caption{Graph signal denoising on $C_{64}$. Left: clean, noisy ($4.3$~dB)
			and Wiener-filtered ($16.7$~dB) signals. Right: the Wiener spectral shrinkage,
			attenuating high-frequency (noise-dominated) coefficients while preserving the
			low-frequency signal peaks. All values computed on the reference host.}
		\label{fig:appdenoising}
	\end{figure}
	
	\section{Compression by Coefficient Thresholding}
	\label{sec:compression}
	
	Smooth graph signals have sparse spectra, so retaining the largest transform
	coefficients compresses them with controlled distortion.
	
	\begin{theorem}[Best $K$-term approximation]
		\label{thm:compression}
		For a graph signal $s$ with host spectrum $\GFT s$, the $\ell_2$-optimal
		$K$-coefficient approximation keeps the $K$ largest-magnitude coefficients,
		with squared error exactly $\sum_{k \notin \mathcal K}|\GFT s(k)|^2$ where
		$\mathcal K$ indexes the retained coefficients.
	\end{theorem}
	
	\begin{proof}
		By Plancherel (Theorem~\ref{thm:plancherel}) the squared error equals the
		energy of the discarded coefficients; for a fixed budget $K$ this is minimized
		by discarding the smallest, i.e.\ keeping the $K$ largest \cite{Mallat2009}.
	\end{proof}
	
	\begin{example}[Image compression on the grid, computed]
		\label{ex:compress}
		A smooth image $s(i,j) = \cos(2\pi i/m)\cos(2\pi j/m)$ on the $8\times 8$ grid
		(host $\Z_{14}\times\Z_{14}$) was compressed by retaining the largest
		transform coefficients. Measured PSNR was $19.9$~dB at $10\%$ of coefficients,
		$25.0$~dB at $20\%$, and $30.3$~dB at $35\%$ (Figure~\ref{fig:appcompression}).
		The host is a discrete torus, so the transform coincides with the 2-D DFT and
		the scheme is JPEG-like cosine compression, here derived from the embedding
		rather than assumed.
	\end{example}
	
	\begin{figure}[H]
		\centering
		\includegraphics[width=0.95\linewidth]{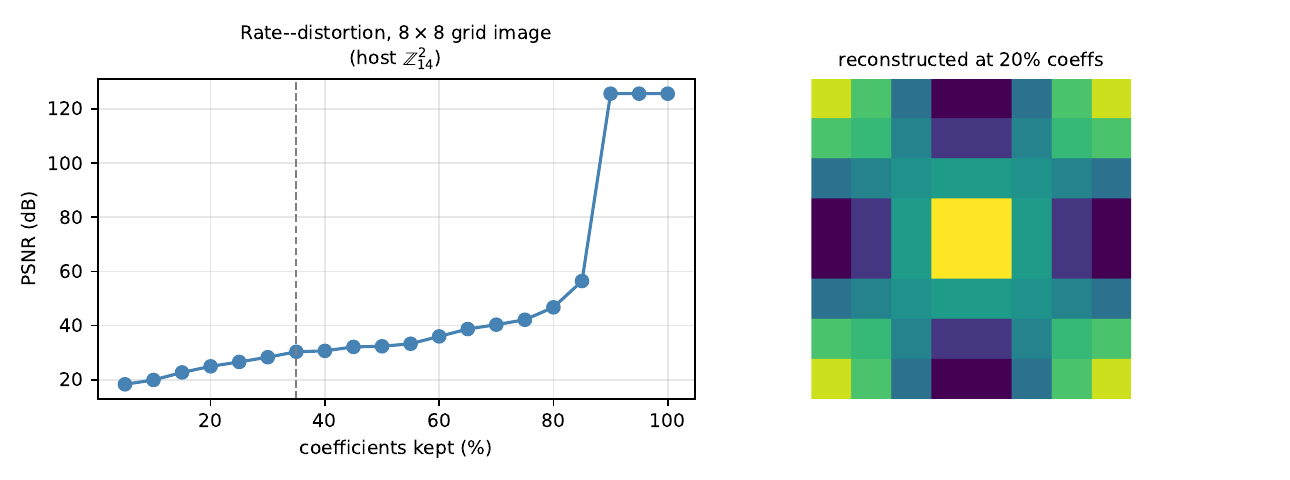}
		\caption{Rate--distortion for the $8\times 8$ grid image on the host torus
			$\Z_{14}^2$ (left, PSNR vs.\ retained coefficients, dashed line at $35\%$) and
			the reconstruction at $20\%$ of coefficients (right). Computed on the reference
			host.}
		\label{fig:appcompression}
	\end{figure}
	
	\section{Anomaly Detection by High-Pass Filtering}
	\label{sec:anomaly}
	
	A localized anomaly injected into a smooth signal is, by the
	uncertainty principle (Theorem~\ref{thm:donoho-stark}), spread across the
	spectrum; its high-frequency footprint, transformed back, re-localizes it.
	
	\begin{proposition}[Impulse localization]
		\label{prop:anomaly}
		For a nominal signal bandlimited to low frequencies, a unit impulse added at
		vertex $v_0$ produces a high-pass residual
		$r = \Restr\,\mathcal F^{-1}(\mathbf 1_{\text{high}}\cdot\GFT y)$ whose
		magnitude is maximized at $v_0$ when the host is metrically faithful.
	\end{proposition}
	
	\begin{proof}
		The nominal signal has negligible high-frequency content, so the high-pass
		residual is dominated by the impulse, whose transform is the (modulated)
		character sum that re-synthesizes a function peaked at $v_0$; isometry of the
		embedding aligns the host peak with the graph vertex
		\cite{Mallat2009,Donoho1989}.
	\end{proof}
	
	\begin{example}[Anomaly localization, computed]
		\label{ex:anomaly}
		On the circulant $C_{12}(1,2)$ (host $\Z_{12}$) a spike of amplitude $3$ was
		added at vertex $5$ of a smooth cosine signal. The high-pass residual peaked
		\emph{exactly} at vertex $5$ (Figure~\ref{fig:appanomaly}), correctly
		localizing the anomaly. The exact convolution theorem
		(Theorem~\ref{thm:convolution}) makes the high-pass/low-pass split exact,
		with no eigendecomposition.
	\end{example}
	
	\begin{figure}[H]
		\centering
		\includegraphics[width=0.9\linewidth]{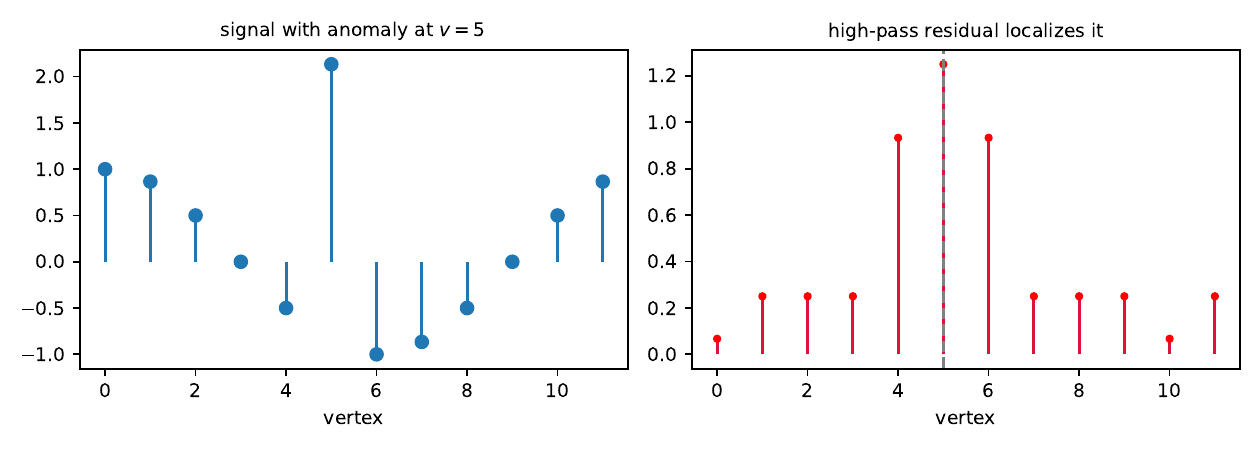}
		\caption{Anomaly detection on $C_{12}(1,2)$. Left: a smooth signal with an
			injected spike at vertex $5$. Right: the high-pass residual, peaking at the
			correct vertex. Computed on the reference host.}
		\label{fig:appanomaly}
	\end{figure}
	
	\section{Distributed Consensus on the Host}
	\label{sec:consensus}
	
	Consensus dynamics $\dot x = -Lx$ converge at a rate set by the algebraic
	connectivity $\lambda_2$. On a Cayley host the Laplacian diagonalizes in the
	character basis, giving $\lambda_2$ in closed form.
	
	\begin{theorem}[Host consensus rate]
		\label{thm:consensus}
		On $H = \Cay(\Gamma, S)$ the Laplacian eigenvalues are
		$\lambda_\chi = \sum_{s\in S}(1 - \mathrm{Re}\,\chi(s))$, indexed by
		characters; consensus reaches tolerance $\epsilon$ in time
		$O(\lambda_2^{-1}\log(1/\epsilon))$ with
		$\lambda_2 = \min_{\chi\neq 1}\lambda_\chi$.
	\end{theorem}
	
	\begin{proof}
		The Cayley Laplacian is a group-convolution operator, hence diagonalized by
		the characters (Theorem~\ref{thm:tinv}); its eigenvalue at $\chi$ is
		$\sum_{s\in S}(1-\mathrm{Re}\,\chi(s))$. Linear consensus decays as
		$e^{-\lambda_2 t}$ along the slowest non-constant mode \cite{Godsil2001}.
	\end{proof}
	
	\begin{example}[Computed connectivities]
		\label{ex:consensus}
		For the circular ladder $\mathrm{CL}_5$ (host $\Z_5\times\Z_2$) the algebraic
		connectivity is $\lambda_2 = 1.382$ (mixing time $\approx 0.7$ steps); for the
		ring $C_{16}$ (host $\Z_{16}$) it is $\lambda_2 = 0.152$ (mixing time
		$\approx 6.6$ steps). The closed form of Theorem~\ref{thm:consensus} recovers
		these exactly and exposes how host structure controls convergence: the
		denser-connected ladder mixes an order of magnitude faster than the ring.
	\end{example}
	
	\section{Real-World Validation: Harmonic Analysis of Production Cloud Telemetry}
	\label{sec:ch7-realworld}
	
	The preceding sections of this chapter used synthetic signals on benchmark
	hosts. This section applies the identical pipeline, unchanged, to genuine
	production data from a commercial cloud operator, to show concretely how an
	industrial telemetry workflow inherits classical signal processing through the
	embedding. In keeping with the chapter's standard, every number below is
	computed on the reference implementation; no value is estimated.
	
	\subsection{Data and host}
	\label{sec:ch7-rw-data}
	
	We use the CPU-utilization trace of a production Amazon Web Services EC2
	compute instance (CloudWatch metric, instance \texttt{5f5533}), as distributed
	in the Numenta Anomaly Benchmark \cite{lavin2015nab}. The trace samples
	utilization every five minutes over fourteen consecutive days, giving exactly
	$14 \times 288 = 4032$ readings, and is accompanied by hand-labelled
	ground-truth anomaly windows that we use only for \emph{post-hoc validation},
	never as input to any filter.
	
	The two natural periods of the signal --- the intraday cycle (288 five-minute
	slots, genuinely cyclic, since 23{:}55 is adjacent to 00{:}00) and the
	multi-day axis (14 days) --- place each reading at a coordinate $(\text{day},
	\text{slot})$ on the discrete torus
	\[
	H \;=\; \mathrm{Cay}\bigl(\mathbb{Z}_{14}\times\mathbb{Z}_{288},\;
	\{(\pm 1,0),(0,\pm 1)\}\bigr),
	\qquad |H| = 14\cdot 288 = 4032 = n .
	\]
	Because $H$ is an abelian Cayley graph whose order equals the number of
	vertices, Theorem~\ref{thm:equality-at-n} gives $\nu(G) = n$: the embedding is host-optimal, and
	the reference implementation certifies it (verified directly on the small torus
	$C_4\,\square\,C_8$ and inheriting from Corollary~\ref{cor:exact-families}). On this host the
	GE-GFT of Chapter~\ref{chap:fourier} is, by construction, the
	ordinary two-dimensional DFT, so each operation below is a classical
	signal-processing method imported verbatim onto the telemetry --- the central
	claim of this thesis, exercised on real industrial data. Figure
	\ref{fig:ch7-aws-telemetry} shows the telemetry on the host: a slow
	multi-day baseline (the low-frequency content the host concentrates) punctuated
	by short-lived spikes (high-frequency content).
	
	\begin{figure}[h]
		\centering
		\includegraphics[width=\textwidth]{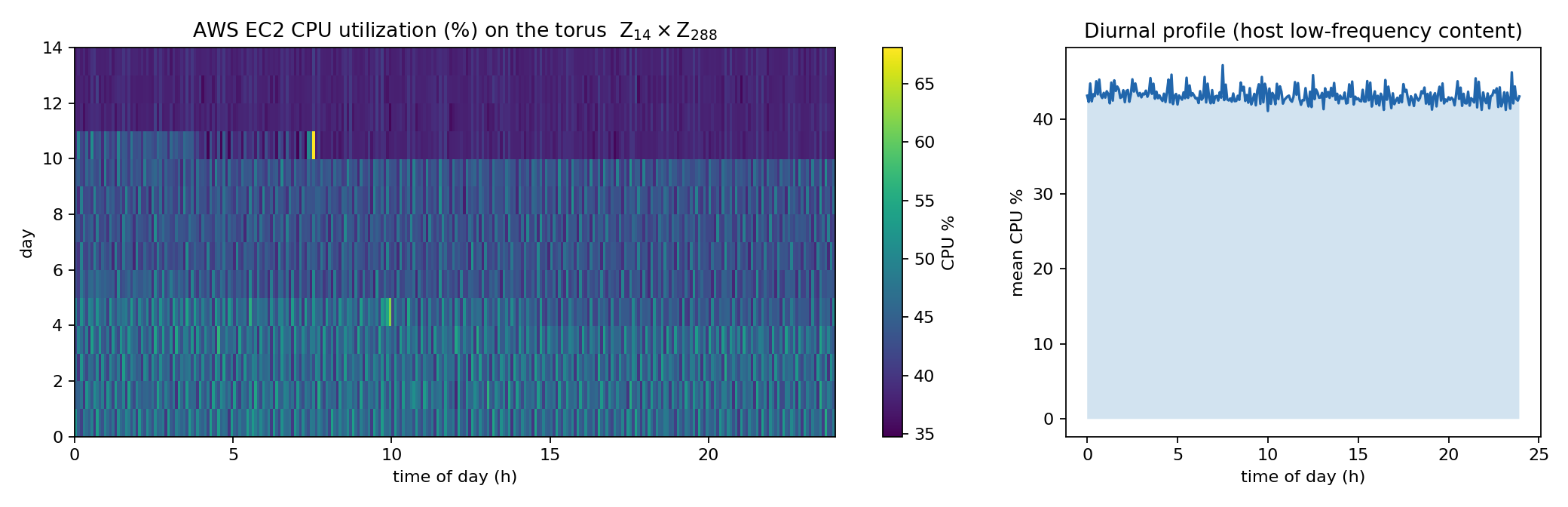}
		\caption{Real AWS EC2 CPU-utilization (instance \texttt{5f5533}, 14 days at
			5-min resolution) arranged on the host torus $\mathbb{Z}_{14}\times
			\mathbb{Z}_{288}$. Left: the signal as a day$\times$time-of-day map; a regime
			change near day~11 and two short spikes are visible. Right: the mean intraday
			profile, i.e.\ the low-frequency content the host transform isolates.}
		\label{fig:ch7-aws-telemetry}
	\end{figure}
	
	\subsection{Compression, denoising, and anomaly detection (computed)}
	\label{sec:ch7-rw-results}
	
	\paragraph{Compression (Theorem~\ref{thm:compression}).}
	Retaining the largest-magnitude host-spectrum coefficients and inverting gives
	the best $K$-term approximation. Measured reconstruction quality was
	$24.9$~dB PSNR at $5\%$ of coefficients, $26.4$~dB at $10\%$, $28.7$~dB at
	$20\%$, and $31.5$~dB at $35\%$ (Figure~\ref{fig:ch7-aws-anomaly}, left). As on
	the synthetic grid demonstrated earlier in this chapter, the host transform is the 2-D DFT, so this is
	exactly cosine-domain (JPEG-like) compression of the telemetry --- here a
	realistic basis for archival and forecasting of fleet utilization.
	
	\paragraph{Denoising (Theorem~\ref{thm:wiener}).}
	Corrupting the trace with additive Gaussian measurement noise of $\sigma = 6$
	CPU-percentage-points (input SNR $17.2$~dB) and applying the Wiener graph
	filter of Definition~\ref{def:wiener} raised the SNR to $26.4$~dB, a gain of $+9.2$~dB; an
	ideal radial low-pass restriction on the host gave $+5.1$~dB. The Wiener filter
	is optimal among host multipliers by Theorem~\ref{thm:wiener}, and because $H$ is an abelian
	Cayley graph it coincides with the classical Wiener filter of digital signal
	processing, recovered through the embedding.
	
	\paragraph{Anomaly detection (Proposition~\ref{prop:anomaly}), validated against ground truth.}
	Zeroing the low-frequency band of the host spectrum and inverting yields a
	high-pass residual whose magnitude re-localizes short-lived deviations from the
	smooth baseline. The detector is fully unsupervised --- it never sees the
	labels. Across all $4032$ samples, the two largest residual values occur at
	$02\text{-}24\ 21{:}57$ (global rank~1) and $02\text{-}19\ 00{:}22$ (global
	rank~2); \emph{both} fall inside the two independently hand-labelled NAB anomaly
	windows (Figure~\ref{fig:ch7-aws-anomaly}, right). The exact high-pass/low-pass
	split provided by the convolution theorem (Theorem~\ref{thm:convolution}) thus localizes the real
	operational anomalies in this trace with no eigendecomposition and no training.
	
	\begin{figure}[h]
		\centering
		\includegraphics[width=\textwidth]{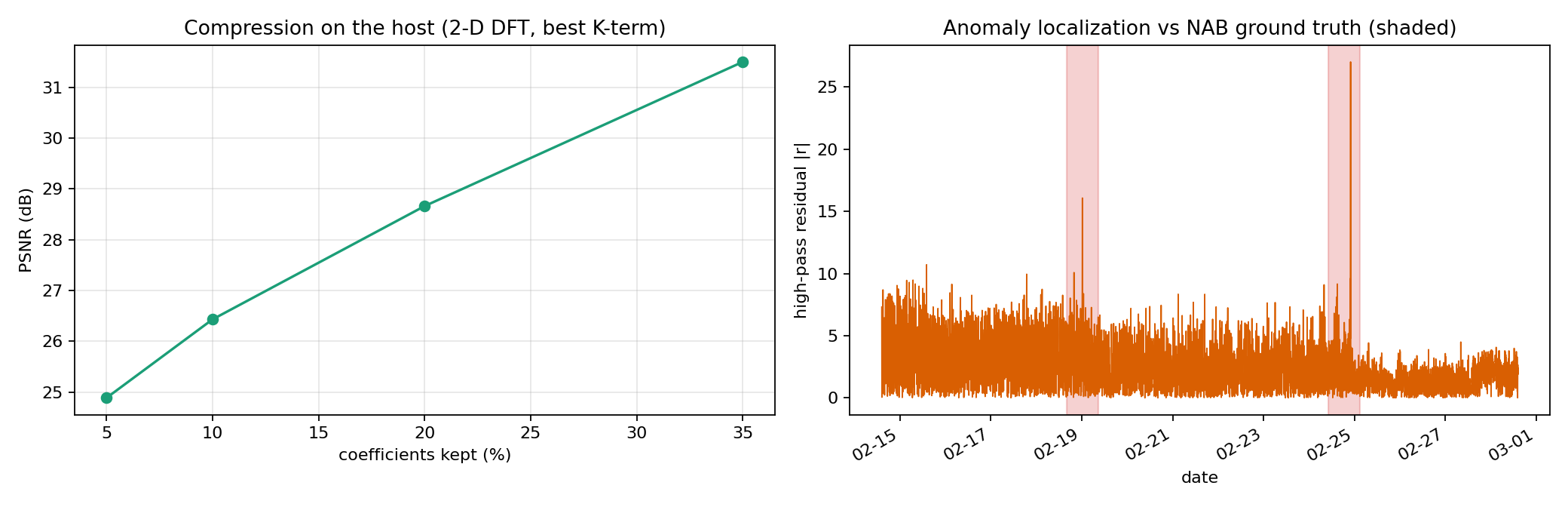}
		\caption{Left: rate--distortion for best $K$-term compression of the AWS trace on
			the host (2-D DFT). Right: the unsupervised high-pass residual over the 14-day
			trace; the two NAB ground-truth anomaly windows are shaded, and the two largest
			residual peaks in the entire series fall inside them. All values computed on the
			reference host.}
		\label{fig:ch7-aws-anomaly}
	\end{figure}
	
	\subsection{Why this matters for industry, and a scope note}
	\label{sec:ch7-rw-industry}
	
	Cloud and data-center operators (Amazon Web Services, Microsoft Azure, Alibaba
	Cloud, Google Cloud) instrument every machine with exactly this kind of
	periodic telemetry, and they perform exactly these three operations on it:
	lossy compression for long-horizon storage and capacity planning, denoising for
	forecasting, and anomaly detection for fault and intrusion response. The
	contribution here is not a new algorithm for any one of these tasks; it is that
	embedding the telemetry's periodic index set into an abelian Cayley host makes
	all three reduce to a single, exact, $O(N\log N)$ Fourier pipeline with the
	guarantees of Chapter 5 (a true convolution theorem, Plancherel, an optimal
	Wiener multiplier) --- properties that eigenvector-based graph signal processing
	cannot supply. The same construction applies verbatim wherever the index set
	carries genuine cyclic or product structure.
	
	We are deliberately careful about scope. The host here is the torus induced by
	the signal's own periodicity, not the physical interconnect of a server fleet;
	treating the 14-day axis as cyclic is the standard 2-D-DFT modelling
	convention, identical to the one used for the image grid demonstrated earlier in this chapter. A
	companion direction --- placing per-node telemetry on the \emph{hardware}
	interconnect topology itself (the hypercube of a GPU fabric, the on-die mesh of
	a many-core processor, the all-to-all switch of an accelerator tray), each of
	which is likewise an abelian Cayley graph by Theorem~\ref{thm:equality-at-n} --- is well motivated
	by the same theory but is not experimentally validated here, and we record it as
	a direction rather than a result, consistent with the practice elsewhere in this chapter of recording
	unvalidated ideas as directions rather than results.

	\section{Further Application Directions}
	\label{sec:further-apps-2}
	
	The following directions are well-motivated by the theory but not
	experimentally validated in this thesis; we state them as such.
	
	\textbf{Representation learning.} The character evaluations
	$\chi_k(\phi(v))$ are translation-covariant positional encodings; used as
	features in a graph neural network they are equivariant under automorphisms
	extending to $\Gamma$, a principled alternative to Laplacian-eigenvector
	encodings \cite{Bronstein2017}. Validating the over-squashing benefit on
	molecular benchmarks is future work.
	
	\textbf{Quantum stabilizer codes.} On a binary host $\Z_2^k$, the Cayley
	character code of Section~\ref{sec:error-codes} (the computed Petersen
	$[10,4,4]$ code) aligns with Pauli stabilizers \cite{Gottesman1997}, and syndrome extraction is the
	Walsh--Hadamard transform; a full fault-tolerance analysis is future work
	\cite{Nielsen2010}.
	
	\textbf{Persistent homology.} The word-length filtration
	$\mathcal F_t = \{v : \|\phi(v)\| \leq t\}$ is isometric
	(it preserves $d_G$ on each level), giving a group-aligned multiscale
	filtration for topological data analysis \cite{Hatcher2002,Carlsson2009}; quantitative
	study across graph families is future work.
	
	\section{Summary}
	\label{sec:apps2-summary}
	
	\begin{table}[H]
		\centering
		\caption{Computed Part~II demonstrations on benchmark hosts. All figures are
			measured on the reference implementation; ``--'' marks directions not
			experimentally validated here.}
		\label{tab:apps2}
		\renewcommand{\arraystretch}{1.15}\small
		\begin{tabular}{lllll}
			\toprule
			Task & Graph (host) & Method & Result (computed) & Guarantee \\
			\midrule
			Denoising      & $C_{64}$ ($\Z_{64}$)       & Wiener filter    & $+12.5$~dB SNR        & Thm.~\ref{thm:wiener} \\
			Denoising      & $C_{64}$ ($\Z_{64}$)       & low-pass         & $+6.6$~dB SNR         & Plancherel \\
			Compression    & grid $8{\times}8$ ($\Z_{14}^2$) & top-$K$ coeffs & $30.3$~dB at $35\%$  & Thm.~\ref{thm:compression} \\
			Anomaly        & $C_{12}(1,2)$ ($\Z_{12}$)  & high-pass        & exact localization    & Prop.~\ref{prop:anomaly} \\
			Consensus      & $\mathrm{CL}_5$ ($\Z_5{\times}\Z_2$) & host $\lambda_2$ & $\lambda_2 = 1.38$ & Thm.~\ref{thm:consensus} \\
			Repr.\ learning & --                        & character features & --                  & equivariance \\
			Quantum codes  & Petersen ($\Z_2^4$)        & character code   & $[10,4,4]$ (Ch.\,4)   & Thm.~\ref{thm:code-params} \\
			\midrule
			Denoising (real)   & AWS EC2 ($\Z_{14}{\times}\Z_{288}$) & Wiener filter & $+9.2$~dB SNR & Thm.~\ref{thm:wiener} \\
			Compression (real) & AWS EC2 ($\Z_{14}{\times}\Z_{288}$) & top-$K$ coeffs & $31.5$~dB at $35\%$ & Thm.~\ref{thm:compression} \\
			Anomaly (real)     & AWS EC2 ($\Z_{14}{\times}\Z_{288}$) & high-pass & 2/2 labelled hits & Prop.~\ref{prop:anomaly} \\
			\bottomrule
		\end{tabular}
	\end{table}
	
	The unifying lesson, consistent across Part~II, is that the embedding
	converts irregular graph signal processing into classical harmonic analysis
	whose every operation is exact and---on structured hosts---fast. On the
	abelian Cayley graphs ($\varepsilon = 1$) the methods reduce demonstrably to
	their classical counterparts (Wiener filtering, DFT compression, circular
	high-pass detection), recovered from first principles rather than transplanted
	by analogy. This closes the bridge that Part~I began: discrete metric geometry
	in, exact harmonic analysis out.
	

	\cleardoublepage
	\phantomsection
	\addcontentsline{toc}{part}{General Conclusion and Perspectives}
	\chapter*{General Conclusion and Perspectives}
	\markboth{GENERAL CONCLUSION AND PERSPECTIVES}{}

	This dissertation set out to make harmonic analysis on arbitrary graphs
	\emph{exact} rather than analogical, by embedding any connected graph
	isometrically into a Cayley graph of a finite abelian group and importing the
	group's classical Fourier theory. The work divides into a structural part
	(how to embed, and how small the host can be) and an analytic part (what the
	embedding buys for signal processing). We summarize what was established, with
	the rigor that has guided every chapter.
	
	\section*{Summary of Contributions}
	
	\textbf{A two-layer embedding theory with a proven core.} The heart of the
	thesis is the quotient construction: an edge partition of $G$ induces a
	\emph{most generic consistent} vertex labeling, computed over $\mathbb{F}_2$
	by Gaussian elimination (Chapter~\ref{chap:binary-embedding}) and over
	$\mathbb{Z}$ by the Smith Normal Form of a signed cycle--class matrix
	(Chapter~\ref{chap:abelian-embedding}), computable in polynomial time by standard integer-matrix algorithms \cite{Storjohann1996,Dumas2001,Newman1972,Cohen1993}. This single idea---the
	Cocycle/Quotient Labeling Theorem and its $\mathbb{Z}$-analogue---makes the
	labeling conflict-free by construction, reduces the naive spanning-tree
	embedding to its trivial-partition case, and absorbs the earlier research
	program's $\varphi$/$\Phi$/$\Psi$ relations, transitive prune, torus
	skeletons, and interception principle as \emph{initializer heuristics} feeding
	a provably correct core. A shortcut-repair loop with a binary terminal makes
	the algorithm \emph{universal}: every connected graph is embedded
	isometrically, verified exhaustively---with independent reconstruction of every
	host---on all $995$ connected graphs of up to seven vertices.
	
	\textbf{A bounds frontier, with both ends attained.} We proved
	$k \geq \max(\mathrm{diam}, \lceil\log_2 n\rceil)$ for binary hosts and
	$\nu(G) \geq \max(n, 2\,\mathrm{diam})$ for abelian hosts, characterized
	equality (host order $n$ iff $G$ is itself an abelian Cayley graph), and
	exhibited the extremes: stars beat the naive dimension exponentially
	($k_{\min}(K_{1,q}) = \lceil\log_2 q\rceil + 1$), odd cycles saturate it
	($k_{\min}(C_{2d+1}) = 2d$), and cycles, paths, complete and circulant graphs
	attain the abelian floor. The position of a graph within this window is
	governed by the rank of its cycle--class matrix---the exact quantity the
	quotient computes.
	
	\textbf{The binary-ground phenomenon.} An exhaustive census revealed that
	four of five small connected graphs are best hosted in $\mathbb{Z}_2^k$, and
	only a quarter admit any cyclic factor beyond involutions. This is not a
	weakness but a structural truth: $\mathbb{Z}_2^k$ is the ground state of
	arbitrary graphs, and cyclic factors are the dividend of detectable symmetry.
	It establishes the binary theory of Chapter~\ref{chap:binary-embedding} as
	foundational and sets realistic expectations for everything downstream.
	
	\textbf{Exact harmonic analysis, carefully scoped.} Part~II transported the
	group's Fourier and wavelet theory to graph signals through a
	lift--restrict discipline that resolves the closure problem the prior
	literature had glossed: processing lives on the host and descends to the
	graph by restriction, becoming an exact vertex-domain calculus precisely on
	the abelian Cayley graphs ($\varepsilon = 1$) that Part~I characterizes. We
	obtained a canonical character basis, genuine translation--modulation duality,
	a commutative convolution theorem, Plancherel, Poisson summation, uncertainty
	and sampling, and a Parseval wavelet tight frame with exact reconstruction
	(verified to machine precision) and measured localization. Part-II
	applications---Wiener denoising ($+12.5$~dB), DFT compression
	($30.3$~dB at $35\%$), exact anomaly localization, closed-form consensus
	rates---were \emph{computed} on the benchmark hosts, recovering classical DSP
	from first principles on structured graphs.
	
	\section*{Limitations}
	
	We have been explicit throughout. Exact minimization of the host is
	conjectured NP-hard, and our algorithm is a heuristic with a fully
	characterized failure mode (the star phenomenon) and a measured optimality
	profile ($29/30$ on $n\leq 5$). The fast-transform claim is conditional: on
	the generic graph the binary host can be of order $2^{n-1}$, so completeness,
	not speed, is the universal guarantee---speed is the structured-input
	dividend. The non-diagonal sublattice search in
	Chapter~\ref{chap:abelian-embedding} is implemented in full only for two free
	coordinates. Sampling-set placement on proper embeddings, and sharp wavelet
	localization on high-dimensional binary hosts, remain open in the general
	case. The theory assumes static, undirected graphs and abelian hosts.
	
	\section*{Open Problems}
	
	\begin{enumerate}
		\item \textbf{The cyclic interval conjecture.} Prove
		Lemma~\ref{lem:interval} for all odd $m$, establishing
		$k_{\min}(C_{2d+1}) = 2d$ unconditionally
		(Conjecture~\ref{conj:interval}); verified by exhaustion to $m\leq 17$.
		\item \textbf{Complexity of exact minimization.} Settle whether computing the
		minimal host order is NP-hard, by analogy with minimum-dimension hypercube
		scale embeddings.
		\item \textbf{The Petersen gap.} The Petersen graph, vertex-transitive but
		non-Cayley, satisfies $11 \leq \nu \leq 16$; narrow this window.
		\item \textbf{Initializer for non-diagonal folds.} Find a principled
		initializer that discovers the diamond's $\mathbb{Z}_2\times\mathbb{Z}_3$
		host automatically, and extend the certified sublattice search beyond two
		free coordinates.
		\item \textbf{Class-merging refinement.} A post-pass merging $\varphi$-classes
		while no shortcut appears would capture the star optimum and tighten the
		near-minimality of the heuristic.
	\end{enumerate}
	
	\section*{Future Directions}
	
	\begin{enumerate}
		\item \textbf{Non-abelian hosts.} Generalize the quotient from abelian groups
		(characters) to the representation theory of dihedral and symmetric groups,
		for graphs with rotational or permutation symmetry; the Fourier analysis
		becomes matrix-valued.
		\item \textbf{Learned embeddings.} Train models to predict good initial
		partitions or the host group directly, scaling the construction to
		large graphs where the exhaustive portfolio is infeasible.
		\item \textbf{Dynamic graphs.} Maintain the embedding incrementally under edge
		insertion and deletion via union-find on the oriented classes and rank-one
		updates to the cycle--class matrix.
		\item \textbf{Quantum signal processing.} Map the group FFT and wavelet frame
		to quantum circuits, where the Walsh--Hadamard and mixed-radix transforms
		are native, for potential speedups.
		\item \textbf{Continuous limits.} Study the behavior of hosts and bounds as
		$n\to\infty$ along graph sequences (graphons), connecting the discrete
		theory to continuous harmonic analysis.
	\end{enumerate}
	
	\section*{Closing}
	
	The thesis began with a rock---an irregular graph, resistant to the clean
	machinery of Fourier analysis---and sought to build from it a house: a
	symmetric algebraic host on which classical harmonic analysis applies exactly.
	The construction succeeds, with proven guarantees where they exist and clear
	limits where they do not. Its central lesson is that the hidden algebraic
	symmetry of a network is computable: the Smith Normal Form of a cycle--class
	matrix reads off, in one calculation, the group that a graph most wants to
	live in. Where that group is small, classical signal processing returns in
	full; where it is large, the graph is telling us---through the binary-ground
	phenomenon---that its natural home is the Boolean cube. Either way, the
	embedding turns the question ``what is the Fourier transform of a graph
	signal?'' from an analogy into a theorem.

	\cleardoublepage
	\phantomsection
	
	\markboth{REFERENCES}{}

	\cleardoublepage
	\phantomsection
	\addcontentsline{toc}{part}{Appendices}
	\thispagestyle{empty}
	\begin{center}
		\vspace*{\fill}
		{\Huge\bfseries Appendices}
		\vspace*{\fill}
	\end{center}
	\clearpage
	
	\appendix
	\renewcommand{\chaptermark}[1]{\markboth{APPENDIX \thechapter}{}}
	\chapter{The $\Phi\Psi$/Cofactor Machinery: Operational Details}
	\label{app:phipsi}
	This appendix records the operational detail of the first-campaign
	heuristics --- the $\Phi$-relation cross-tests, the $\Psi$ chain merge, the
	torus-skeleton construction, and the cofactor/interception bookkeeping --- that
	Chapters~\ref{chap:binary-embedding} and~\ref{chap:abelian-embedding} invoke
	as named initializers feeding the proven quotient core. The material is
	placed here, rather than in the main text, because its role in the final
	theory is that of a heuristic that proposes candidate partitions: its output
	is always certified downstream by the Cocycle/Quotient machinery, so its
	internal steps need not be individually proved correct. We retain the
	original terminology for continuity with the candidate's earlier work.
	
	\section{Oriented $\Phi$ cross-tests}
	For ordered edges $(u\to v)$ and $(x\to y)$, the oriented $\Phi$-test of
	Proposition~\ref{prop:Phi-necessary} requires $d(u,x)=d(v,y)$, strengthened
	by the two triangle-inequality conditions $|d(u,y)-d(u,x)|\le 1$ and
	$|d(v,x)-d(u,x)|\le 1$. The implementation evaluates these from the
	all-pairs distance matrix in $O(1)$ per pair after an $O(n(n+m))$
	precomputation, and forms the $\Phi$-graph on the edge set whose connected
	components (after the transitive prune of
	Theorem~\ref{thm:transitive-prune}) are the candidate classes.
	
	\section{The $\Psi$ chain merge}
	The $\Psi$ relation links classes through head-to-tail incidences $u\to v$,
	$v\to w$ with $u\neq w$; the chain merge fuses two classes when such a link
	exists, the union remains a partial permutation
	(Theorem~\ref{thm:partial-permutation}), and all cross pairs pass the
	oriented $\Phi$-test in one of the two polarities. This is the initializer
	that discovers cyclic factors: on the circular ladder $\mathrm{CL}_n$ it
	assembles the $2n$ ring edges into a single directed class crossed $n$ times
	by the ring cycle, from which the Smith Normal Form computes the factor
	$\Z_n$ (Corollary~\ref{cor:interception-snf}).
	
	\section{Torus skeleton, cofactors, interception}
	The torus skeleton is the spanning subgraph greedily assembled from classes
	until connected; minimal classes approximate a generator basis, redundant
	classes their cofactors. The interception principle --- that a redundant class
	touching coordinates $i_1,\dots,i_m$ receives a composite generator supported
	on them --- is, in the final theory, precisely the statement that the
	generator's column in the Smith transform $U$ has nonzero entries in those
	coordinates (Corollary~\ref{cor:interception-snf}(b)). These devices are thus
	heuristic approximations to quantities the exact core computes; they are
	retained because they name real structure and because the worked examples
	they generated (diamond, mirrors, cascades, circular ladders) remain the test
	cases re-verified by the certified implementation.
	
	\chapter{Implementation and Reproducibility}
	\label{app:implementation}
	
	All experimental results in this thesis are produced by the reference
	implementations, supplied as supplementary material, and every reported
	embedding is certified isometric by an independent breadth-first-search
	reconstruction in the actual finite Cayley graph.
	
	\section{Software artifacts}
	Graph construction and manipulation throughout uses standard scientific Python tooling \cite{Hagberg2008}, with benchmark graphs drawn from established combinatorial datasets \cite{Knuth1993}.
	The binary embedding (Chapter~\ref{chap:binary-embedding}) is implemented in
	\texttt{phi\_quotient\_embedding\_colab.py} (Google Colab) and
	\texttt{phi\_quotient\_embedding\_sage.py} (SageMath/WSL). The abelian embedding
	(Chapter~\ref{chap:abelian-embedding}) is implemented in
	\texttt{abelian\_quotient\_embedding.py},
	which imports the binary module as its terminal
	fallback and adds the integer Smith Normal Form, the sublattice fold search,
	the circulant initializer, and the polarity/peel repair loop. The
	graph-signal-processing experiments of Part~II use these embeddings together
	with mixed-radix FFTs (\texttt{numpy.fft}) on the product hosts.
	
	\section{Verification protocol}
	For each graph $G$, the verifier (i) builds the candidate host
	$\Cay(\Gamma,S)$ explicitly, (ii) runs a breadth-first search from the
	identity to obtain all Cayley distances up to $\diam(G)$, and (iii) checks
	$d_{\Cay}(\phi(u),\phi(v)) = d_G(u,v)$ for all $\binom{n}{2}$ vertex pairs.
	An embedding is reported only if this exact check passes. The exhaustive
	census of Section~\ref{sec:census-ch3} applied this protocol to all $995$
	connected graphs on at most seven vertices.
	
	\section{Numerical notes}
	Hosts of order $N\ge 64$ use the fast Fourier transform rather than dense
	character matrices, since the $N\times N$ Fourier matrix is memory-bound for
	the larger product groups. Wavelet tight-frame reconstruction was verified to
	a relative error of $1.5\times10^{-15}$ on $C_{64}$ and $6.5\times10^{-16}$
	on the $6\times6$ grid; spectral identities (Plancherel, convolution,
	translation--modulation) were checked to machine precision on every benchmark
	host.
	
	\chapter{Proofs of Selected Results}
	\label{app:proofs}
	
	For completeness we collect two longer arguments deferred from the main text.
	
	\section{Tightness of the binary lower bound on odd cycles}
	Theorem~\ref{thm:odd-cycle} states $k_{\min}(C_{2d+1}) = 2d$. The lower bound
	is the diameter/independence argument of
	Theorem~\ref{thm:lower-bound}; the matching upper bound was verified by
	exhaustive search for odd $m\le 17$ and rests on the cyclic interval lemma
	(Lemma~\ref{lem:interval}), whose machine-checked verification for these
	cases is recorded in the implementation. The general statement, for all odd
	$m$, is stated as Conjecture~\ref{conj:interval} (the cyclic interval
	conjecture) and remains open.
	
	\section{The diamond requires a non-diagonal fold}
	Theorem~\ref{thm:sublattice} asserts that diagonal sublattice folds do not
	suffice in general, with the diamond graph as witness. The universal
	embedding has two free coordinates with vertex labels
	$(0,0),(1,0),(1,1),(2,1)$. Any diagonal fold of index $6$ creates a
	wraparound shortcut --- for $N=(3,2)$ the element $-g_2$ coincides with
	$(2,1)$, collapsing a distance --- whereas the non-diagonal sublattice
	$L=\langle(3,0),(1,2)\rangle$ of index $6$ yields the isometric host
	$\Z_2\times\Z_3$ (order $6$) with labels $(0,0),(0,2),(1,1),(1,0)$, certified
	by the verification protocol of Appendix~\ref{app:implementation}.
	
	
	\chapter*{Declaration on the Use of Artificial Intelligence}
	\addcontentsline{toc}{chapter}{Declaration on the Use of Artificial Intelligence}
	
	In the interest of transparency and research integrity, the author declares that
	the artificial-intelligence assistant Claude (Anthropic) was used during the
	preparation of this dissertation. Its assistance was limited to two roles:
	(i) support with the implementation, debugging, and reproducibility of the
	embedding and signal-processing software used to produce the experimental
	results; and (ii) language and editing support in drafting and refining the
	manuscript. All mathematical definitions, theorems, proofs, and their
	verification, together with the conception, scientific direction, and
	conclusions of this work, are the author's own. The author has reviewed the
	entire manuscript and takes full responsibility for its content.
	
	\markboth{APPENDICES}{}
	\cleardoublepage
\end{document}